\documentclass[a4paper,10pt,twoside,english]{scrartcl}
\usepackage[utf8x]{inputenc}
\usepackage{mathrsfs}  
\usepackage{dsfont} 
\usepackage[T1]{fontenc}
\usepackage{lmodern}
\usepackage{graphicx}
\usepackage{enumitem}
\usepackage{wrapfig}
\usepackage{setspace}
\usepackage{color}
\usepackage{xspace} 
\usepackage{amsmath,amssymb,amsthm}
\usepackage{blkarray}
\usepackage[subnum]{cases}
\usepackage{xifthen}
\usepackage{caption}
\usepackage{subcaption}
\usepackage[top=1in, bottom=1in, left=2cm, right=2cm]{geometry}

\usepackage[unicode=true]{hyperref}

\usepackage{pgfpages}
\usepackage{pgfplots}
  \pgfplotsset{compat=newest}
  \usetikzlibrary{plotmarks}
  \usetikzlibrary{arrows.meta}
  \usepgfplotslibrary{patchplots}
  \usepackage{grffile}

\newcommand{\titel}{The Dyson equation with linear self-energy: spectral bands, edges and cusps}
\title{\titel} 
\author{
Johannes Alt\footnote{\hspace{0.15cm}Partially funded by ERC Advanced Grant RANMAT No. 338804.}
\addtocounter{footnote}{-1}\addtocounter{Hfootnote}{-1} \\
{\small \begin{tabular}{c}{University of Geneva}\\{johannes.alt@unige.ch} \end{tabular}} 
\and László Erd\H{o}s\footnotemark \\
{\small \begin{tabular}{c} IST Austria\\ lerdos@ist.ac.at\end{tabular}} 
\and Torben Krüger\footnote{\hspace{0.15cm}Partially supported by the Hausdorff Center for Mathematics.\newline Date: \today}\\
{\small \begin{tabular}{c} University of Bonn\\ torben.krueger@uni-bonn.de\end{tabular}}
}
\date{}

\hypersetup{
	pdftitle={\titel},
	pdfauthor={Johannes Alt, Laszlo Erdos, Torben Krüger},
	colorlinks={false},
	pdfborderstyle={/S/U/W 1}
}

\numberwithin{equation}{section}

\newtheoremstyle{test}
  {}
  {}
  {\itshape}
  {}
  {\bfseries}
  {.}
  { }
  {}

\newtheoremstyle{testDef}
  {}
  {}
  {}
  {}
  {\bfseries}
  {.}
  { }
  {}

\theoremstyle{testDef}
\newtheorem{definition}{Definition}[section]
\newtheorem{assums}[definition]{Assumptions}

\newtheorem{remark}[definition]{Remark}

\newtheorem*{remark*}{Remark}   
\newtheorem*{ex*}{Example}   
\newtheorem*{def*}{Definition}

\theoremstyle{test}
\newtheorem{theorem}[definition]{Theorem}
\newtheorem{lemma}[definition]{Lemma}
\newtheorem{corollary}[definition]{Corollary}
\newtheorem{convention}[definition]{Convention}
\newtheorem{proposition}[definition]{Proposition}
\newtheorem*{proposition*}{Proposition} 
\newtheorem*{corollary*}{Corollary}
\newtheorem*{theorem*}{Theorem}

\renewcommand{\bf}[1]{\boldsymbol{\mathrm{#1}}} 
\renewcommand{\cal}{\mathcal} 
 
\renewcommand{\frak}{\mathfrak} 
\newcommand{\ol}[1]{\overline{#1} \!\,} 
\newcommand{\wh}{\widehat}
\newcommand{\wt}{\widetilde}

\newcommand{\ord} {\mathcal{O}}

\renewcommand{\P}{\mathbb{P}}

\newcommand{\ii}{\mathrm{i}} 
\newcommand{\dd}{\mathrm{d}}

\newcommand{\pb}[1]{\bigl({#1}\bigr)}

\newcommand{\pbb}[1]{\biggl({#1}\biggr)}
\newcommand{\pBB}[1]{\Biggl({#1}\Biggr)}

\newcommand{\cb}[1]{\bigl\{{#1}\bigr\}}

\newcommand{\cbb}[1]{\biggl\{{#1}\biggr\}}
\newcommand{\cBB}[1]{\Biggl\{{#1}\Biggr\}}

\newcommand{\abs}[1]{\lvert #1 \rvert}

\newcommand{\absB}[1]{\Big\lvert #1 \Big\rvert}
\newcommand{\absbb}[1]{\bigg\lvert #1 \bigg\rvert}
\newcommand{\absBB}[1]{\Bigg\lvert #1 \Bigg\rvert}

\newcommand{\norm}[1]{\lVert #1 \rVert}

\newcommand{\normbb}[1]{\bigg\lVert #1 \bigg\rVert}

\newcommand{\norma}[1]{\left\lVert #1 \right\rVert}

\newcommand{\avg}[1]{\langle #1 \rangle}

\newcommand{\scalar}[2]{\langle{#1} \mspace{2mu}, {#2}\rangle}

\newcommand{\scalarbb}[2]{\bigg\langle{#1} \,\mspace{2mu},\, {#2}\bigg\rangle}

\DeclareMathOperator{\diag}{diag}

\DeclareMathOperator{\tr}{Tr}
\DeclareMathOperator{\Tr}{Tr}

\DeclareMathOperator{\re}{Re}
\DeclareMathOperator{\im}{Im}

\DeclareMathOperator*{\essinf}{ess\,inf}	
\DeclareMathOperator{\dist} {dist}                
\DeclareMathOperator*{\sign}{sign}						
\DeclareMathOperator*{\spec}{Spec}						

\newcommand{\1} {\mspace{1 mu}}
\newcommand{\2} {\mspace{2 mu}}
\newcommand{\msp}[1] {\mspace{#1 mu}}

\newcommand{\genarg} {{\,\cdot\,}}  

\newcommand{\R}{\mathbb{R}}  
\ifdefined\C\renewcommand{\C}{\mathbb{C}}\else\newcommand{\C}{\mathbb{C}}\fi 
\renewcommand{\Im}{\mathrm{Im}\,} 
\renewcommand{\Re}{\mathrm{Re}\,} 
\newcommand{\N}{\mathbb{N}}  
\newcommand{\E}{\mathbb{E}}  
\renewcommand{\P}{\mathbb{P}}  
\newcommand{\Z}{\mathbb{Z}}  
\newcommand{\di}{\mathrm{d}} 
\newcommand{\eps}{\varepsilon} 
\newcommand*{\defeq}{\mathrel{\vcenter{\baselineskip0.5ex \lineskiplimit0pt\hbox{\scriptsize.}\hbox{\scriptsize.}}}=}

\newcommand{\pt}{\partial}

\DeclareMathOperator{\supp}{supp}

\DeclareMathOperator{\ran}{ran}

\newcommand{\normone}[1]{\lVert #1 \rVert_{1}}

\newcommand{\normtwo}[1]{\lVert #1 \rVert_{2}}

\newcommand{\normtwoinf}[1]{\lVert #1 \rVert_{2\to\norm{\cdot}}}

\newcommand{\imz}[1][]{%
 \ifthenelse{\equal{#1}{}}{\Im z}{(\Im z)^{#1}}
}

\newcommand{\Hb}{\mathbb H}

\newcommand{\HbI}{{{\mathbb H}_{I,\eta_*}}}
\newcommand{\HbIclosed}{{\overline{\mathbb H}_{I,\eta_*}}}

\newcommand{\alg}{\mathcal{A}}
\newcommand{\algsa}{{\mathcal{A}_\mathrm{sa}}}
\newcommand{\algnon}{\overline{\mathcal{A}}_+}
\newcommand{\algpos}{\mathcal{A}_+}

\newcommand{\Ltwo}{L^2} 

\newcommand{\Mb}{{\mathbb M}} 

\newcommand{\id}{\mathds{1}}
\newcommand{\Id}{\mathrm{Id}}
\renewcommand{\char}{\ensuremath{\mathbf{1}}} 

\newcommand{\err}{e} 
\newcommand{\errt}{\tilde{\err}}
\newcommand{\diff}{\Delta} 
	
\renewcommand{\rm}{\mathrm} 

\newcommand{\Dmin}{\mathbb{D}}

\newcommand{\Psie}{\Psi_{\mathrm{edge}}}
\newcommand{\Psim}{\Psi_{\mathrm{min}}}

\newcommand{\spone}{\textbf{SP1}\xspace}
\newcommand{\sptwo}{\textbf{SP2}\xspace}
\newcommand{\spthree}{\textbf{SP3}\xspace}
\newcommand{\spfour}{\textbf{SP4'}\xspace}

\newcommand{\dens}{\rho}

\newcommand{\shapeJ}{J}

\begin{document}

\maketitle

\vspace*{-1.2cm} 
\begin{abstract}
We study the unique solution $m$ of the Dyson equation
\[ -m(z)^{-1} = z\id - a + S[m(z)] \]
on a  von Neumann algebra $\alg$ with  the constraint $\im m\ge 0$. Here, $z$ lies in the complex upper half-plane, $a$ is a self-adjoint element of $\alg$ and $S$ is a positivity-preserving
linear operator 
on $\alg$. We show that $m$ is the Stieltjes transform of a compactly supported $\alg$-valued measure on $\R$. 
Under suitable assumptions, we establish that this measure has a uniformly  $1/3$-H\"older continuous density with respect to the Lebesgue measure, which is supported on finitely many intervals, called bands. 
In fact,  the density is analytic inside the bands with a square-root growth at the edges and internal cubic root cusps whenever the gap between two bands vanishes. The shape of these singularities is universal and
no other singularity may occur. We give a precise asymptotic description of $m$ near the 
singular points. These asymptotics  generalize the analysis at the regular edges given 
 in the companion paper on the Tracy-Widom universality for the edge eigenvalue statistics
for correlated random matrices \cite{AltEdge} and they play a key role  
in the proof of  the Pearcey universality at the cusp for Wigner-type matrices  \cite{Erdos2018Cusp2,Erdos2018Cusp}.  
We also extend the finite dimensional  band mass formula from  \cite{AltEdge} to the von Neumann algebra setting 
by showing that the spectral mass of the bands 
is topologically rigid under deformations and  we
conclude that these masses are  quantized in some important cases. 
\end{abstract}


\tableofcontents

\section{Introduction}

An important task in random matrix theory is to determine the eigenvalue distribution of a random matrix as 
its size tends to infinity. Similarly, in free probability theory, the scalar-valued distribution of operator-valued
semicircular elements is of particular interest. In both cases, the distribution can be obtained from the corresponding \emph{Dyson equation} 
\begin{equation} \label{eq:dyson_intro}
 -m(z)^{-1} = z \id - a + S[m(z)] 
\end{equation}
on some von Neumann algebra $\alg$ with a unit $\id$ and a tracial state $\avg{\genarg}$. 
Here, $z$ lies in $\Hb \defeq \{w \in \C \colon \Im w >0\}$, the complex upper half-plane, $a = a^* \in \alg$ and $S \colon \alg \to \alg$ is a positivity-preserving linear operator. 
There is a unique solution $m \colon \Hb \to \alg$ of \eqref{eq:dyson_intro} under the assumption that
$\Im m(z) \defeq (m(z)-m(z)^*)/(2\ii)$ is a strictly positive element of $\alg$ for all $z \in \Hb$ \cite{Helton01012007}. 
For suitably chosen $a$ and $S$ as well as $\alg$, this solution characterizes the distributions in the applications mentioned above. 
In fact, in both cases, the distribution will be the measure $\dens$ on $\R$ whose Stieltjes transform is given by $z\mapsto\avg{m(z)}$. 
The measure $\dens$ is called the \emph{self-consistent density of states} and its support is the \emph{self-consistent spectrum}. 
This terminology stems from the physics literature on the Dyson equation, where $z$ is often called \emph{spectral parameter} and $S[m]$ 
 the self-energy. The linearity of the \emph{self-energy operator} $S$ is a distinctive feature of our setup. 

We first explain the connection between the eigenvalue density of a large random matrix and the Dyson equation.
Let $H \in \C^{n\times n}$ be a $\C^{n\times n}$-valued random variable, $n \in \N$, such that $H=H^*$. 
A central objective is the analysis of the  \emph{empirical spectral measure} $\mu_H\defeq n^{-1} \sum_{i=1}^n \delta_{\lambda_i}$, or its expectation, the \emph{density of states}, for large $n$, 
where $\lambda_1, \ldots, \lambda_n$ are the eigenvalues of $H$. Clearly, $n^{-1}\Tr(H-z)^{-1}$ is the Stieltjes transform of $\mu_H$ at $z\in \Hb$. 
Therefore, the resolvent $(H-z)^{-1}$ is commonly studied to obtain information about $\mu_H$.
In fact, for many random matrix ensembles, in particular models with decaying correlations among the entries, 
the resolvent $(H-z)^{-1}$
is well-approximated for large $n$ by the solution $m(z)$ of the Dyson equation \eqref{eq:dyson_intro}. 
 Here, we choose $\alg = \C^{n\times n}$ equipped with the operator norm induced by the 
Euclidean distance on $\C^n$ and the normalized trace $\avg{\genarg} = n^{-1} \Tr(\genarg)$  as tracial state as well as 
\begin{equation} \label{eq:choice_a_S_random_matrix} 
a \defeq\E H, \qquad \qquad S[x]\defeq \E[(H-a)x (H-a)], \quad x \in \C^{n\times n}. 
\end{equation}
If $(H-z)^{-1}$ is well-approximated by $m(z)$ for large $n$ then $\mu_H$ will be well-approximated by the deterministic measure $\dens$, whose Stieltjes transform is 
given by $z\mapsto\avg{m(z)}$. 
The importance of the Dyson equation \eqref{eq:dyson_intro} for random matrix theory has been realized by many authors on various levels of 
generality \cite{Anderson2006,Berezin1973,GUIONNET2002341,KhorunzhyPastur1994,Shlyakhtenko1996,Wegner1979}, see also the monographs \cite{girko2012theory,PasturShcherbina2011}
 and the more recent works \cite{Ajankirandommatrix,AjankiCorrelated,AltSingGram,AltKronecker,AltGram,Erdos2017Correlated,He2017,Knowles2017}.

Secondly, we relate the Dyson equation to free probability theory  by noticing that 
the Cauchy transform of a shifted operator-valued semicircular element is  given by $m$.  More precisely, 
let $\mathcal{B}$ be a unital $C^*$-algebra, $\alg \subset \mathcal{B}$ be a $C^*$-subalgebra with the same unit $\id$
and 
$E \colon \mathcal{B} \to \alg$  is a conditional expectation  (we refer to Chapter 9 in~\cite{mingo2017free} for notions from free probability theory).
Pick an  $a=a^*\in \alg$ and an operator-valued semicircular element $s=s^* \in \mathcal{B}$. Then  
 $G(z) \defeq E[(z-s-a)^{-1}]$ is the \emph{Cauchy-transform} of $s+a$. In this case, $m(z) = - G(z)$ satisfies \eqref{eq:dyson_intro} with $S[x] \defeq E[sxs]$ for all $x \in \alg$ \cite{Voiculescu1995}. 
If $\alg$ is a von Neumann algebra with a tracial state, then our results yield information about the scalar-valued distribution $\rho=\rho_{s+a}$  of $s+a$ with respect to this state.
The study of qualitative regularity properties for this distribution has a long history in free probability.  For example, the question of whether $\rho$ has atoms or not is intimately related to non-commutative identity testing (see \cite{Garg2016ADP,MAI20171080} and references therein) 
and the notions of free entropy and Fischer information (see \cite{Voiculescu1994,Voiculescu1998} and the survey \cite{doi:10.1112/S0024609301008992}).   
We also refer to the recent preprint \cite{Mai2018}, where the distribution of rational functions in noncommutative random variables is studied with the help of
 linearization ideas from \cite{Haagerup2005,Haagerup2006} and \cite{Helton2018}. 
Under certain assumptions, our results provide extremely detailed information about
   the regularity properties of  $\rho$, thus complementing these more general insights. 
 In particular, we show that $\rho_s$ is absolutely continuous with respect to the Lebesgue measure away from zero for any operator-valued semicircular element~$s$.
For  other  applications of the Dyson equation \eqref{eq:dyson_intro} in free probability theory, we refer to \cite{Helton01012007,Speicher1998,Voiculescu1995,Voiculescu2000}
and the recent monograph~\cite{mingo2017free}.

In this paper, we analyze the regularity properties of the self-consistent density of states $\dens$ in detail. More precisely, under suitable assumptions on $S$, we show that 
the boundedness of $m$ already implies that  $\dens$ has a $1/3$-H\"older continuous density  $\rho(\tau)$  with respect to the Lebesgue measure.   We provide a broad class of models for which the boundedness of $m$ is ensured. 
  Furthermore,  the set  where the density is positive,  $\{\tau : \rho(\tau)>0\}$,  
  splits into finitely many connected components, called \emph{bands}. The  density is real-analytic inside the bands with a square root growth behavior at the edges. If two bands touch, however, a cubic root cusp  emerges. These are the only possible types of singularities.
In fact, $m(z)$ is the Stieltjes transform of a positive operator-valued measure $v$ and we establish the 
properties mentioned above for $v$ as well. We also  extend the \emph{band mass formula} from \cite{AltEdge} expressing
the masses that $\rho$ assigns to the bands. 
  We use it to infer a certain quantization of the band masses that we 
call \emph{band rigidity}, because it is invariant under small perturbations of the data $a$ and $S$ of the Dyson equation.  In particular, we extend a quantization result from \cite{Haagerup2006} and 
 \cite{freebandrigidity} to cover limits of Kronecker random matrices.  
We remark that for the analogous phenomenon  in the context of random matrices 
the term ``exact separation of eigenvalues'' was coined in \cite{Bai1999}.

In the commutative setup, the band structure and singularity behavior of the density  have been obtained in \cite{AjankiQVE,AjankiCPAM}, where  a detailed  analysis of the regularity of $\dens$ was initiated. 
In the special noncommutative situation $\alg = \C^{n\times n}$ and $\avg{\genarg} =n^{-1} \Tr(\genarg)$, it has been shown that $\dens$ is H\"older-continuous and real-analytic 
wherever it is positive \cite{AjankiCorrelated}. 
 Recently, in the same setup, the precise behavior of $\dens$ near the spectral edges was obtained in \cite{AltEdge},
where it was a key ingredient in the proof of the  Tracy-Widom universality of the local spectral statistics near
the spectral edges for random matrices with general correlation structure. However, this analysis works
only at edges that are well separated from each other (so called \emph{regular edges}), i.e. away from the \emph{cusp}  where two edges merge
and away from the \emph{almost cusps}, i.e.  regions with  small spectral gaps or small but nonzero minima of the density. 
 The main novelty of the current work is to give an effective
regularity analysis for the general noncommutative case  with  a precise quantitative description of all singularities including 
the almost cusps.  
One of the main applications is the proof of the eigenvalue rigidity on optimal scale throughout the entire spectrum. This is a key input 
for the recent proof of the
 local spectral universality at the cusp for general Wigner-type matrices, i.e. the Pearcey statistics for the
complex hermitian case in  \cite{Erdos2018Cusp} and its real symmetric counterpart in \cite{Erdos2018Cusp2}.
We remark that cusp universality settles the third and last ubiquitous spectral universality regime after the bulk and edge universalities studied
 extensively earlier, see \cite{Erdos2017book} and references therein.

The key strategy behind the current paper as well as its predecessors \cite{AjankiQVE,AjankiCPAM,AjankiCorrelated,AltEdge}   
is a refined
stability analysis of the Dyson equation~\eqref{eq:dyson_intro} against small perturbations.
 It turns out that the equation 
is stable in the bulk regime, i.e., where $\rho(\re z)$ is separated away from zero, but is unstable near the 
points, where the density vanishes.  Even the stability in the bulk requires an unconventional idea; it relies on  rewriting
the stability operator, i.e., the derivative of the Dyson equation with respect to $m$,   through the use of a  positivity-preserving symmetric map, called the \emph{saturated self-energy operator}, $F$. We then
  extract information on the   spectral gap of $F$ by
a Perron-Frobenius argument using the positivity of $\im m$  \cite{AjankiQVE,AjankiCPAM}.  In the non-commutative setup
this transformation was based on a novel \emph{balanced polar decomposition} formula~\cite{AjankiCorrelated}.  In the small density
regime, in particular near the  regular edges studied in \cite{AltEdge},  
the stability deteriorates due to an \emph{unstable direction}, which is related to the Perron-Frobenius eigenvector of $F$.
The analysis boils down to a scalar quantity, $\Theta$, the overlap between the solution and
the {unstable direction}. For the commutative case in  \cite{AjankiQVE,AjankiCPAM}, it is shown that $\Theta$ approximately satisfies a cubic equation.
The structural property of this cubic equation is its \emph{stability}, i.e., that the coefficients of the cubic and quadratic terms
do not simultaneously vanish. This guarantees that higher order terms are negligible and the order of any singularity
is either cubic root or square root. 

Now we synthesize both analyses in the previous works to study the small density regime in the most general 
setup. The major obstacle is the noncommutativity that already substantially complicated the bulk  analysis~\cite{AjankiCorrelated},
but there the  saturated self-energy operator, $F$, governed all estimates. However,
in the regime of small density  the  {unstable direction}  is identified via the top eigenvector of a non-symmetric operator 
that coincides with the symmetric $F$ only in the commutative case.  Thus we need to perform a non-symmetric perturbation
expansion that requires precise control on the resolvent of the non-selfadjoint stability operator in the entire complex plane. 
We still
work with a cubic equation for $\Theta$, but the analysis  of its  coefficients is considerably more involved  than in \cite{AjankiQVE,AjankiCPAM}.

 The situation is much simpler near the regular edges, where the cubic equation simplifies to 
a quadratic equation; this analysis was performed in \cite{AltEdge} at least in the finite dimensional non-commutative case.
The main novelty of the present paper lies in handling the most complicated case, the cusps and almost cusps, where
we need to deal with a genuine cubic equation. The second goal of the paper is to give a unified treatment of
all spectral regimes in the general von Neumann algebraic setup. A few arguments pertaining the regular edges
are relatively simple extensions from \cite{AltEdge}  to the infinite dimensional case. We will indicate these instances
but for the reader's convenience we chose to include these proofs since in the current paper
we work under weaker conditions and in a more general setup than in \cite{AltEdge}.

 We stress that  along all  estimates, the noncommutativity is a permanent enemy; in some cases it
 can be treated perturbatively, but for the most critical parts new non-perturbative proofs are needed.
 Most critically, the stability of the cubic equation is proven with a new method.

Another novelty of the current paper,
in addition to handling the non-commutativity  and lack of symmetry, is that we present 
 the cubic analysis in a conceptually clean way that will be used in future works. Our analysis strongly suggests
 that our cubic equation for $\Theta$ is the key to any detailed singularity analysis of Dyson-type equations
 and its remarkable structure is responsible for the universal behavior of the singularities in the density. 
 
As a final remark we compare our  self-consistent density of states $\rho$, obtained from the Dyson equation, 
with the equilibrium density $\rho_V$ considered in invariant matrix ensembles with
an external potential $V$. Recall that $\rho_V$ is the solution of a variational principle  \cite{DEIFT1998388}.
Both densities approximate the empirical density
of states of a prominent class of random matrix ensembles, but they have quite different singularity structures
at the vanishing points. Our classification theorem shows that $\rho$ has only square root and cusp singularities.
On the other hand, if $V\in C^2$ then $\rho_V$ is  1/2-H\"older continuous, in particular it cannot have any
cusp singularity. Moreover $\rho_V$ may vanish at the edges of its support not necessarily  as a square root, see 
e.g. a behaviour   $\rho_V(x)\approx (x_+)^{5/2}$ in  Example 1.2 of \cite{Claeys2010}. In general,  only 
powers $\alpha = 2k$ and
$\alpha = 2k+\frac{1}{2}$, $k\in \N$ are possible for the vanishing behavior $\rho_V(x)\approx (x_+)^{\alpha}$.
These patterns  persist under small additive perturbations with an independent GUE matrix,
moreover, at critical coupling, a cusp singularity similar to our case appears as well \cite{Claeys2018}. 
A summary of known behaviours of $\rho_V$ near its vanishing points in relation with $V$ is found in Section 1.3 of \cite{Bekerman2017}.
The complexity of these patterns indicates that a concise classification theorem  of  singularities,
similar to our result on $\rho$ with merely two types of singularities, does not hold  for 
$\rho_V$.

\section{Main results}
Let $\alg$ be a finite von Neumann algebra with unit $\id$ and norm $\norm{\cdot}$.
We recall that a von Neumann algebra $\alg$ is called \emph{finite} if there is a state $\avg{\genarg} \colon \alg \to \C$ which is (i) \emph{tracial}, i.e., $\avg{xy} = \avg{yx}$ for all $x, y \in \alg$, 
(ii) \emph{faithful}, i.e., $\avg{x^*x}=0$ for some $x \in \alg$ implies $x=0$, and (iii) \emph{normal}, i.e., continuous with respect to the weak$^*$ topology. 
In the following, $\avg{\genarg}$ will always denote such state. 
The tracial state defines a scalar product $\alg \times \alg \to \C$ through
\begin{equation} \label{eq:definition_scalar_product_alg}
\scalar{x}{y}\,\defeq\,\avg{x^*y}
\end{equation}
for $x, y \in \alg$. The induced norm is denoted by $\normtwo{x} \defeq \scalar{x}{x}^{1/2}$ for $ x \in \alg$. 
Clearly, $\normtwo{x} \leq \norm{x}$ for all $x \in \alg$. 
We follow the convention that small letters are elements of $\alg$ while capital letters denote linear operators on $\alg$. 
The spectrum of $x \in\alg$ is denoted by $\spec x$, i.e., $\spec x = \C \setminus \{ z \in \C \colon (x-z)^{-1} \in \alg\}$. 

For an operator $T \colon \alg \to \alg$, we will work with three norms. 
We denote these norms by $\norm{T}$, $\normtwo{T}$ and $\normtwoinf{T}$ if $T$ is considered as an 
operator $(\alg, \norm{\genarg})\to(\alg,\norm{\genarg})$, $(\alg, \normtwo{\genarg})\to(\alg,\normtwo{\genarg})$ 
or $(\alg, \normtwo{\genarg})\to(\alg,\norm{\genarg})$, respectively.

We denote by $\alg_{\rm{sa}}$ the self-adjoint elements of $\alg$, by $\alg_+$ the cone of positive definite elements of $\alg$, i.e., 
\[
\alg_{\rm{sa}}\,\defeq\, \{x \in \alg\colon\2 x^*=x\}, \qquad \algpos\,\defeq\, \{x \in \alg_{\rm{sa}}\colon x>0\},
\]
and by $\algnon$, the $\norm{\genarg}$-closure of $\algpos$, the cone of positive semidefinite elements (or positive elements).
We now introduce two classes of linear operators on $\alg$ that preserve the cone $\algnon$. Such operators are called \emph{positivity-preserving} (or \emph{positive maps}). 
We define
\begin{subequations}
\begin{align}
\Sigma&\, \defeq\,\{S\colon \alg \to \alg\2\colon \2 S\text{ is linear, symmetric wrt.~\eqref{eq:definition_scalar_product_alg} and preserves the cone } \ol{\alg}_+\}, \label{eq:def_self_energy} \\ 
\Sigma_{\rm{flat}}& \,\defeq\,\bigg\{S\in \Sigma\colon \eps\id \leq \inf_{x \in \alg_+}\frac{S[x]}{\avg{x}} \2\le\2 \sup_{x \in \alg_+}\frac{S[x]}{\avg{x}} \2\leq \2 \eps^{-1} \id \text{ for some } \eps >0 \bigg\}.
\label{eq:def_self_energy_flat}
\end{align}
\end{subequations}
Moreover, if $S\colon \alg \to \alg$ is a positivity-preserving operator, 
then $S$ is bounded, i.e., $\norm{S}$ is finite (see e.g.~Proposition~2.1 in~\cite{Paulsen2002}). 

Let $a\in \alg_{\rm{sa}}$ be a self-adjoint element and $S\in\Sigma$.
For the \emph{data pair} $(a,S)$, we consider the associated \emph{Dyson equation}
\begin{equation} \label{eq:dyson}
-m(z)^{-1}\,=\, z\id-a+S[m(z)]\,,
\end{equation}
with spectral parameter $z \in \Hb \defeq \{ w \in \C \colon \Im w >0\}$, 
for a function $m\colon \Hb \to \cal{A}$ such that its imaginary part is positive definite,
\[
\im m(z)\,=\, \frac{1}{2\ii}(m(z)-m(z)^*) \,\in\, \cal{A}_+\,.
\]
There always exists a unique solution $m$ to the Dyson equation \eqref{eq:dyson} satisfying $\Im m(z) \in \alg_+$ \cite{Helton01012007}. Moreover, this solution is holomorphic in $z$ 
\cite{Helton01012007}. 
For Dyson equations in the context of renormalization theory, $a$ is called the \emph{bare matrix} and $S$ the \emph{self-energy (operator)}. 
In applications to free probability theory, $S$ is usually denoted by $\eta$ and called the \emph{covariance mapping} or \emph{covariance matrix} \cite{mingo2017free}.  

We now introduce positive operator-valued measures with values in $\algnon$. 
If $v$ maps Borel sets on $\R$ to elements of $\algnon$ such that $\scalar{x}{v(\genarg)x}$ is a positive measure for all $x \in\cal{A}$ then we say that 
$v$ is a \emph{measure on $\R$ with values in $\algnon$} or an \emph{$\algnon$-valued measure on $\R$}. 

First, we list a few propositions that are necessary to state our main theorem. They will be proven in 
Section~\ref{sec:solution_dyson_equation}, Section \ref{subsec:proof_regularity_density_states} and Section \ref{subsec:proof_hoelder_1_3}, respectively. 

\begin{proposition}[Stieltjes transform representation] \label{pro:stieltjes_representation}
Let $(a,S)\in \cal{A}_{\rm{sa}} \times \Sigma$ be a data pair and $m$ the solution to the associated Dyson equation. Then there exists a measure $v$ on $\R$ 
with values in $\ol{\cal{A}}_+$ such that $v(\R) = \id$ and 
\begin{equation} \label{eq:Stieltjes_representation}
m(z)\,=\, \int_\R \frac{v(\dd \tau)}{\tau-z}
\end{equation}
for all $z \in \Hb$. The support of $v$ and the spectrum of $a$ satisfy the following inclusions
\begin{subequations} 
\begin{align}
\supp v & \subset \spec a + [-2\norm{S}^{1/2},2\norm{S}^{1/2}], \label{eq:supp_v_subset_spec_a}\\
\spec a & \subset \supp v + [-\norm{S}^{1/2},\norm{S}^{1/2}]. \label{eq:spec_a_subset_supp_v}
\end{align}
\end{subequations}
Furthermore, for any $z \in \Hb$,  $m(z)$ satisfies the bound
\begin{equation} \label{eq:L2_bound}
\norm{m(z)}_2\,\le\, \frac{2}{\dist(z, \rm{Conv} \spec a) }\,,
\end{equation} 
where $\rm{Conv} \spec a$ denotes the convex hull of $\spec a$. 
\end{proposition}

Our goal is to obtain regularity results for the measure $v$. We first present some regularity results on the self-consistent density of states introduced in the following definition.

\begin{definition}[Density of states] \label{def:density_of_states}
Let $(a,S) \in \alg_\rm{sa}\times \Sigma$ be a data pair, $m$ the solution to the associated Dyson equation, \eqref{eq:dyson}, and $v$ the $\algnon$-valued measure of Proposition \ref{pro:stieltjes_representation}. 
The positive measure $\dens = \avg{v}$ on $\R$ is called the \emph{self-consistent density of states} or short \emph{density of states}. 
\end{definition}
We have $\supp \dens = \supp v$ due to the faithfulness of $\avg{\genarg}$.
Moreover, the Stieltjes transform of $\dens$ is given by $\avg{m}$ since, by \eqref{eq:dyson}, for any $z \in \Hb$, we have
\[
\avg{m(z)}\,=\, \int_\R \frac{\dens(\dd \tau)}{\tau-z}. 
\]

\begin{proposition}[Regularity of density of states] \label{pro:regularity_density_of_states}
Let $(a,S)$ be a data pair with $S \in \Sigma_{\rm{flat}}$ and $\dens_{a,S}$ the corresponding density of states. Then $\dens_{a,S}$ has a uniformly Hölder-continuous, compactly supported density 
with respect to the Lebesgue measure,
\[
\dens_{a,S}(\dd \tau)\,=\, \dens_{a,S}(\tau)\dd \tau\,.
\]
Furthermore, there exists a universal constant $c>0$ such that the function $\dens\colon \cal{A}_{\rm{sa}} \times  \Sigma_{\rm{flat}} \times \R \to [0,\infty), (a,S,\tau)\mapsto \dens_{a,S}(\tau)$ is locally H\"older-continuous 
with H\"older exponent $c$ and analytic whenever it is positive, i.e., for any  $(a,S,\tau) \in \cal{A}_{\rm{sa}} \times  \Sigma_{\rm{flat}} \times \R$ such that $\dens_{a,S}(\tau)>0$ the function $\dens$ is analytic in a 
neighbourhood of $(a,S,\tau)$.
Here, $\alg_\rm{sa}$ and $\Sigma_\rm{flat}$ are equipped with the metrics induced by $\norm{\genarg}$ on $\alg$ and its operator norm on $\alg \to \alg$, respectively. 
\end{proposition}

The following proposition is stated under a boundedness assumption on $m$ (see \eqref{eq:boundedness_for_Hoelder} below).
In the random matrix context, in Section \ref{sec:Kronecker}, we provide a sufficient condition for this assumption to hold purely expressed in 
terms of $a$ and $S$ for a large class of random matrix models.  
In the finite dimensional case, where $\cal{A} = \C^{N \times N}$ and $\avg{\1\cdot\1}=\frac{1}{N}\tr(\1\cdot\1) $,  Proposition~\ref{pro:Hoelder_1_3} has already been established in \cite[Corollary 4.5]{AltEdge} and the arguments there remain valid in our more general setup. Nevertheless, we will present its proof to keep the current work self-contained.
\begin{proposition}[Regularity of $m$] \label{pro:Hoelder_1_3}
Let $(a,S)$ be a data pair with $S \in \Sigma_{\rm{flat}}$ and $m$ the solution to the associated Dyson equation. Suppose that for a nonempty open interval
$I \subset \R$ we have 
\begin{equation} \label{eq:boundedness_for_Hoelder}
\limsup_{\eta \downarrow 0} \sup_{ \tau \in I}\norm{m(\tau+\ii\1\eta)}\,<\, \infty\,.
\end{equation}
Then $m$ has a $1/3$-H\"older continuous extension (also denoted by $m$) to any closed interval $I'\subset I$, i.e., 
\begin{equation} \label{eq:pro_Hoelder_1_3_bound}
\sup_{z_1,z_2 \in I'\times \ii[0,\infty)}\msp{-6}\frac{\norm{m(z_1)-m(z_2)}}{\abs{z_1-z_2}^{1/3}}\,<\, \infty\,.
\end{equation}
Moreover, $m$ is real-analytic in $I$ wherever $\dens$ is positive. 
\end{proposition}

The purpose of the interval $I$ in Proposition \ref{pro:Hoelder_1_3} (see also Theorem \ref{thm:singularities_flat} below) is to demonstrate the local nature
of these statements and their proofs; if $m$ is bounded on $I$ in the sense of \eqref{eq:boundedness_for_Hoelder} then we will prove regularity of $m$ and later its behaviour close to singularities 
on a genuine subinterval $I' \subset I$. At first reading, the reader may ignore this subtlety and assume $I' = I = \R$.  

In Proposition \ref{pro:analyticity_of_m} below, we provide a quantitative version of \eqref{eq:pro_Hoelder_1_3_bound} under slightly weaker conditions than those of Proposition \ref{pro:Hoelder_1_3}. 

For the following main theorem, we remark that if $m$ has a continuous extension to an interval $I \subset \R$ then the restriction of the measure $v$ from \eqref{eq:Stieltjes_representation} to $I$ 
has a density with respect to the Lebesgue measure, i.e., for each Borel set $A \subset I$, we have 
\begin{equation} \label{eq:v_density_Im_m}
 v(A) = \frac{1}{\pi} \int_A \Im m(\tau) \di \tau. 
\end{equation}
The existence of a continuous extension can be guaranteed by \eqref{eq:boundedness_for_Hoelder} in Proposition \ref{pro:Hoelder_1_3}. 

\begin{theorem}[$\Im m$ close to its singularities] \label{thm:singularities_flat}
Let $(a,S)$ be a data pair with $S \in \Sigma_{\rm{flat}}$ and $m$ the solution to the associated Dyson equation. Suppose $m$ has a continuous extension to a nonempty open interval $I \subset \R$. Then any $\tau_0 \in \supp \dens \cap I$ with $\dens(\tau_0)=0$ belongs to exactly one of the following cases:
\begin{itemize}[leftmargin=1.7cm]
\item[Edge: ] The point $\tau_0$ is a left/right edge of the density of states, i.e., there is some $\eps>0$ such that $\Im m (\tau_0 \mp \omega)=0$ for $\omega \in [0,\eps]$ and for some $v_0 \in \cal{A}_+$ we have 
\[
\Im m(\tau_0\pm\omega)\,=\, v_0 \2\omega^{1/2} + \ord(\omega)\,,\qquad \omega \downarrow 0\,.
\]
\item[Cusp: ] The point $\tau_0$ lies in the interior of $\supp \dens$ and for some $v_0 \in \cal{A}_+$ we have
\[
\Im m(\tau_0+\omega)\,=\, v_0 \2\abs{\omega}^{1/3} + \ord(\abs{\omega}^{2/3})\,,\qquad \omega \to  0\,.
\]
\end{itemize}
Moreover, $\supp \dens \cap I = \supp v \cap I$ is a finite union of closed intervals with nonempty interior. 
\end{theorem}

Theorem \ref{thm:singularities_flat} is a simplified version of our more detailed and quantitative Theorem \ref{thm:behaviour_v_close_sing} below.  
We can treat all small local minima of $\dens$ on $\supp \dens\cap I$ -- not only those ones, where $\dens$ vanishes -- 
and provide precise expansions corresponding to those in Theorem \ref{thm:singularities_flat} which are valid in some neighbourhood of $\tau_0$. 
Moreover, the coefficients $v_0$ in Theorem \ref{thm:singularities_flat} are bounded from above and below in terms of the basic parameters of the model. 
By applying $\avg{\genarg}$ to the results of Theorem~\ref{thm:singularities_flat} and Theorem~\ref{thm:behaviour_v_close_sing}, we also obtain an expansion of the self-consistent density of states $\dens$ near small local 
minima in Theorem~\ref{thm:behaviour_dens} below.

Finally, we present our quantization result. 
This result has appeared in \cite[Proposition 5.1]{AltEdge} for the  simpler  setting $\cal{A}=\C^{N \times N}$ 
 and under the flatness condition $S \in \Sigma_{\rm{flat}}$.  In the current work we will follow the same strategy of proof when $\cal{A}$ is a general von Neumann algebra
 with certain adjustments to treat the  possibly infinite dimension and  the lack of flatness.

\begin{proposition}[Band mass formula] \label{pro:band_formula} 
Let $(a,S) \in \alg_\rm{sa} \times \Sigma$ be a data pair and $m$ the solution to the associated Dyson equation,~\eqref{eq:dyson}. 
We assume that there is a constant $C>0$ such that $S[x] \leq C\avg{x} \id$ for all $x \in \algnon$. Then we have 
\begin{enumerate}[label=(\roman*)]
\item For each $\tau \in \R \setminus \supp \dens$, there is $m(\tau) \in \alg_\rm{sa}$ such that $\lim_{\eta\downarrow 0} \,\norm{m(\tau + \ii\eta)-m(\tau)} = 0$. Moreover, $m(\tau)$ determines 
the mass of $(-\infty,\tau)$ and $(\tau,\infty)$ with respect to $\dens$ in the sense that 
\begin{equation} \label{eq:band_formula}
 \dens( (-\infty, \tau)) = \avg{\char_{(-\infty,0)}(m(\tau))}, 
\end{equation}
where $\char_{(-\infty,0)}$ denotes the characteristic function of the interval $(-\infty,0)$. 
\item If $\pi \colon \alg \to \C^{n\times n}$ is a faithful representation such that $\avg{x} = n^{-1} \Tr(\pi(x))$ for all $x \in \alg$ and $J\subset \supp \dens$ is a 
connected component of $\supp \dens$ then we have 
\[n \dens(J) \in \{1, \ldots, n\}. \] 
In particular, $\supp\dens$ has at most $n$ connected components. 
\end{enumerate}
\end{proposition}

We will prove Proposition \ref{pro:band_formula} in Section \ref{sec:bumps} below.  
A result similar to part (ii) has been obtained by a different method in \cite{Haagerup2006}, see also \cite{freebandrigidity}. 
In fact, we will use the band mass formula, \eqref{eq:band_formula}, in Corollary \ref{coro:band_mass_quant_Kronecker} below 
to strengthen the quantization result in (ii) 
for a large class of random matrices (Kronecker matrices, see Section~\ref{sec:Kronecker}). 
In Section~\ref{sec:pert_data_pair}, we study the stability of the Dyson equation, \eqref{eq:dyson}, under small general pertubations of the data pair $(a,S)$.

\subsection{Examples} \label{subsec:examples}

\enlargethispage*{.7\baselineskip}

\begin{wrapfigure}{r}{0.31\textwidth}
\begin{tikzpicture}[scale = 1.0]
\draw[draw=black] (0,0) rectangle (2,2);
\draw[draw=black] (0,0) rectangle (2,1.6);
\draw[draw=black] (0,0) rectangle (0.4,2);
\node at (-0.5,1) {$r_\alpha=$};
\node at (0.2,0.8) {$1$};
\node at (1.2,1.8) {$1$};
\node at (0.2,1.8) {$\alpha$};
\node at (1.2,0.8) {$\alpha$};
\end{tikzpicture}
\caption{Structure of $r_\alpha\in \C^{n\times n}$.}
\label{fig:def_r_alpha} 
\vspace*{-0.10cm} 
\end{wrapfigure} 
We now present some examples that show the different types of singularities described by Theorem \ref{thm:singularities_flat}. 
These examples are obtained by considering the Dyson equation, \eqref{eq:dyson}, on $\C^{n\times n}$ with $\avg{\genarg}= n^{-1} \Tr$ for large $n$ and choosing $a = 0$ as well as $S=S_\alpha$, where
\[ S_\alpha[x] \defeq \frac{1}{n}\diag(r_\alpha \diag(x)) \] 
for any $x \in \C^{n\times n}$. 
Here, for $x \in\C^{n\times n}$, $\diag(x)$ denotes the vector of diagonal entries, $r_\alpha \in \C^{n\times n}$ is the symmetric block matrix from Figure \ref{fig:def_r_alpha} with $\alpha \in(0,\infty)$.  
All elements in each block are the indicated constants. 
Moreover, we write $\diag(v)$ with $v \in \C^{n}$ to denote the diagonal matrix in $\C^{n\times n}$ with $v$ on its diagonal. 
In fact, this example can also be realized on $\C^2$ with entrywise multiplication. Here, we choose $\avg{(x_1, x_2)} = \delta x_1 + (1-\delta)x_2$, where $\delta$ is the 
relative block size of the small block in the definition of $r_\alpha$. In this setup on $\C^2$, the Dyson equation can be written as 
\begin{equation} \label{eq:dyson_C_two}
-\begin{pmatrix} m_1^{-1} \\ m_2^{-1}\end{pmatrix} = z \begin{pmatrix} 1 \\ 1\end{pmatrix} + R_\alpha \begin{pmatrix} m_1\\ m_2\end{pmatrix}, \qquad \qquad R_\alpha = \begin{pmatrix} \alpha\delta & 1-\delta \\ \delta & \alpha(1-\delta) \end{pmatrix} 
\end{equation}
for $(m_1, m_2) \in \C^2$.  We remark that $R_\alpha$ is symmetric with respect to the scalar product \eqref{eq:definition_scalar_product_alg} induced by $\avg{\genarg}$. 
Figure \ref{fig:examples_gap_cusp_minimum} contains the graphs of some self-consistent densities of states $\dens$ obtained from \eqref{eq:dyson_C_two} for $\delta = 0.1$ and different values of $\alpha$. 
As the self-consistent density of states is symmetric around zero in these cases, only the part of the density on $[0,\infty)$ is shown. 
The density in Figure \ref{fig:examples_gap_cusp_minimum} (a) has a small internal gap with square root edges on both sides of this gap. 
Figure \ref{fig:examples_gap_cusp_minimum} (b) contains a cusp which is transformed, by increasing $\alpha$, into an internal nonzero local minimum 
in Figure \ref{fig:examples_gap_cusp_minimum} (c). This nonzero local minimum is covered by Theorem~\ref{thm:behaviour_v_close_sing} (d) below.

\newlength{\fwidth}
\newlength{\fheight}
\setlength{\fwidth}{.25\textwidth}
\setlength{\fheight}{.9\fwidth}

\begin{figure}[ht!]
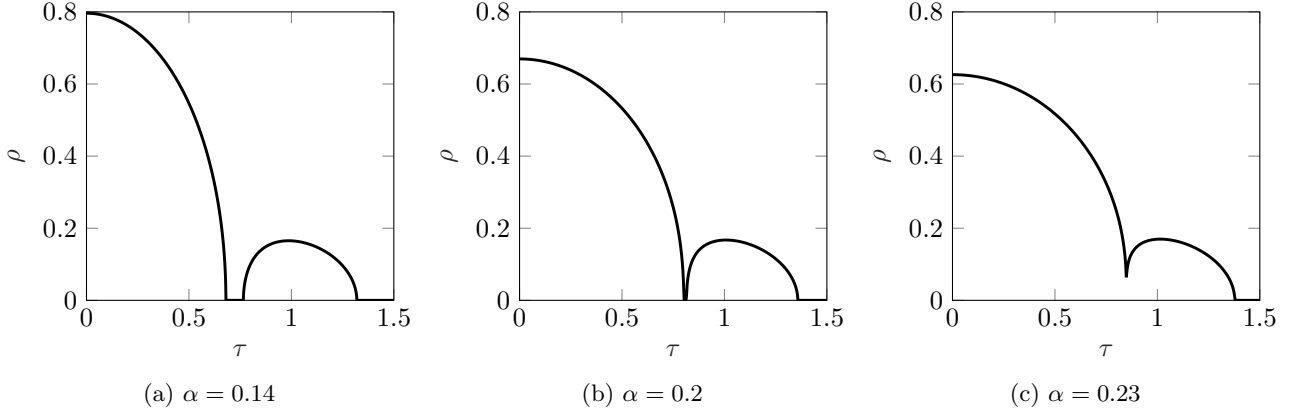

\begin{subfigure}{.33\textwidth}
\centering
\input{wigner_bigger_gap_oneside.tikz}
\caption{$\alpha = 0.14$} 
\end{subfigure}
\begin{subfigure}{.33\textwidth}
\centering
\input{wigner_cusp_0_8_oneside.tikz}
\caption{$\alpha = 0.2$} 
\end{subfigure}
\begin{subfigure}{.33\textwidth}
\centering
\input{wigner_nonzero_minimum_oneside.tikz}
\caption{$\alpha = 0.23$} 
\end{subfigure}
\caption{Examples of the self-consistent density of states $\dens$ from \eqref{eq:dyson_C_two} for $\delta = 0.1$ and several values of $\alpha$.}
\label{fig:examples_gap_cusp_minimum}
\end{figure}

\subsection{Main ideas of the proofs} 

In this subsection, we informally summarize several key ideas in the proofs of Proposition \ref{pro:Hoelder_1_3} and Theorem \ref{thm:singularities_flat}. 

\paragraph{Hölder-continuity of $m$.} 
To simplify the notation, we assume in this outline that $\norm{m(z)} \lesssim 1$ for all $z \in \Hb$, i.e., we assume \eqref{eq:boundedness_for_Hoelder} 
with $I = \R$. 
We first show that $\Im m(z)$ is $1/3$-Hölder continuous and then conclude the same regularity for $m=m(z)$. 
To that end, we now control $\pt_z \Im m(z)$ by differentiating the Dyson equation, \eqref{eq:dyson}, with respect to $z$. 
This yields 
\[ 2\ii \pt_z \Im m = (\Id - C_m S)^{-1}[m^2].  \] 
Here, $\Id$ denotes the identity map on $\alg$ and $C_m \colon \alg \to \alg$ is defined by $C_m[x] \defeq m x m$ for any $x \in \alg$. 

In order to control the norm of  the inverse $(\Id - C_mS)^{-1}$ of the stability operator,  we rewrite it in a more symmetric form. 
We find an invertible $V$ with $\norm{V},\norm{V^{-1}} \lesssim 1$, a unitary operator $U$ and a self-adjoint operator $T$ acting on $\alg$ such that
\[ \Id - C_mS = V^{-1}( U - T)V. \] 
The Rotation-Inversion Lemma from \cite{AjankiCPAM} (see Lemma \ref{lem:rotation_inversion} below) is designed to 
control $(U-T)^{-1}$ for a unitary operator $U$ and a self-adjoint operator $T$ with $\normtwo{T} \leq 1$. 
Applying this lemma in our setup yields $\norm{(\Id- C_mS)^{-1}} \lesssim \norm{\Im m}^{-2}$. 

Since $\norm{m} \lesssim 1$, we thus obtain   
\begin{equation} \label{eq:key_ideas_aux1}
 \norm{\pt_z \Im m} \lesssim \norm{\Im m}^{-2}. 
\end{equation}
This bound 
implies that $(\Im m)^3\colon \Hb \to \alg_+$ is uniformly Lipschitz-continuous. 
Hence, we can extend $\Im m$ to a $1/3$-Hölder continuous function on $\R\cup \Hb$ and we obtain 
\[ m(z) = \frac{1}{\pi} \int_\R \frac{\Im m(\tau) \di \tau}{\tau - z}. \] 
This also implies that $m$ is uniformly $1/3$-Hölder continuous on $\R \cup \Hb$. 
Furthermore, $m(\tau)$ and $\Im m(\tau)$ are real-analytic in $\tau$ around $\tau_0\in \R$, wherever $\dens(\tau_0)$ is positive. 

\paragraph{Behaviour of $\Im m$ where it is not analytic.} 

Owing to \eqref{eq:key_ideas_aux1}, some unstable behaviour of the 
Dyson equation is expected close to points $\tau_0\in\R$, where $\Im m(\tau_0)$ is zero or small. 
In order to analyze this behaviour of $\Im m(\tau)$, 
we compute $\Delta \defeq m(\tau_0 + \omega) - m(\tau_0)$ 
from the Dyson equation, \eqref{eq:dyson}. Since $m$ has a continuous extension to $\R$, \eqref{eq:dyson} 
holds true for $z \in \R$ as well.  
We evaluate \eqref{eq:dyson} at $z=\tau_0$ and $z=\tau_0 + \omega$ and obtain the quadratic $\alg$-valued equation
\begin{equation} \label{eq:stab_equation}
  B[\Delta] = m S[\Delta] \Delta + \omega m \Delta + \omega m^2, \qquad B \defeq \Id - C_mS. 
\end{equation}
The blow-up of  the inverse $B^{-1}$ of the stability operator $B$  close to $\tau_0$ requires analyzing the contributions of $\Delta$ 
in the unstable direction of $B^{-1}$ separately. In fact, $B$ possesses precisely one unstable direction 
denoted by $b$ since we will show that $\normtwo{T}$ is a non-degenerate eigenvalue of $T$. 
We decompose $\Delta$ into $\Delta = \Theta b + r$, where $\Theta$ is the scalar contribution of $\Delta$ in 
the direction $b$ and $r$ lies in the spectral subspace of $B$ complementary to $b$. 

We view $\tau_0$ as fixed and consider $\omega \ll 1$ as the main variable. Projecting \eqref{eq:stab_equation} onto $b$ and its complement yield
the scalar-valued cubic equation
\begin{equation} \label{eq:cubic}
 \psi \Theta(\omega)^3 + \sigma \Theta(\omega)^2 + \pi \omega = \ord(\abs{\omega}\abs{\Theta(\omega)} + \abs{\Theta(\omega)}^4)
\end{equation}
with two parameters $\psi \geq 0$ and $\sigma \in \R$. 
In fact, the $1/3$-Hölder continuity of $m$ implies $\Theta = \ord(\abs{\omega}^{1/3})$ and, hence, 
the right-hand side of \eqref{eq:cubic} is indeed of lower order than the terms on the left-hand side. 
Analyzing \eqref{eq:cubic} instead of \eqref{eq:stab_equation} is a more tractable problem since 
we have reduced a quadratic $\alg$-valued equation, \eqref{eq:stab_equation}, to the scalar-valued cubic 
equation,~\eqref{eq:cubic}. 

The essential feature of the cubic equation \eqref{eq:cubic} is its stability. 
By this, we mean that there exists a constant $c>0$ such that 
\[ \psi + \sigma^2 \geq c. \] 
This bound will follow from the structure of the Dyson equation and prevents any singularities of higher order than $\omega^{1/2}$ or $\omega^{1/3}$. 
Obtaining more detailed information about $\Theta$ from~\eqref{eq:cubic} requires 
applying Cardano's formula with an error term. Therefore, we switch to \emph{normal coordinates}, $(\omega, \Theta(\omega)) 
\to (\lambda, \Omega(\lambda))$, in~\eqref{eq:cubic}.  
We will study four normal forms, one quadratic $\Omega(\lambda)^2 + \Lambda(\lambda) = 0$, and three cubics, $\Omega(\lambda)^3 + \Lambda(\lambda)=0$ 
and $\Omega(\lambda)^3 \pm 3\Omega(\lambda) + 2 \Lambda(\lambda) = 0$, 
where $\Lambda(\lambda)$ is a perturbation of the identity map $\lambda \mapsto \lambda$. 
The first case corresponds to the square root singularity of the isolated edge, the second is the cusp. 
The last two cases describe the situation of \emph{almost cusps}, see later.

The correct branches in Cardano's formula are identified with the help of four \emph{selection principles}
for the solution $\Omega(\lambda)$ corresponding to $\Theta$ of the cubic equation in normal form (see 
\spone to \spfour at the beginning of Section \ref{subsec:cubic_normal_form} below).
These selection principles are special properties of $\Omega$ which originate from the continuity of $m$, 
$\Im m \geq 0$ and the Stieltjes transform representation, \eqref{eq:Stieltjes_representation}, of $m$.  
Once the correct branch is chosen, we obtain the precise behaviour of $\Im m$ around $\tau_0$, 
where $\tau_0 \in \supp \dens$ satisfies $\dens(\tau_0)=0$ or even $\dens(\tau_0) \ll 1$, from Cardano's formula and careful estimates
of $r$ in the decomposition $\Delta = \Theta b + r$ (see Theorem \ref{thm:behaviour_v_close_sing} below).

\section{The solution of the Dyson equation}\label{sec:solution_dyson_equation}

In this section, we first introduce some notations used in the proof of Proposition \ref{pro:stieltjes_representation}, then prove the proposition and finally give a few further properties of $m$. 

For $x,y \in\alg$, we introduce the bounded operator $C_{x,y}\colon \alg \to \alg$ defined through  $C_{x,y}[h] \defeq x h y$ 
for $h \in \alg$. We set $C_{x} \defeq C_{x,x}$. 
For $x,y  \in \alg$, the operator $C_{x,y}$ satisfies the simple relations
\[ C^*_{x,y} = C_{x^*,y^*}, \qquad C_{x,y}^{-1} = C_{x^{-1},y^{-1}},  \]
where $C^*_{x,y}$ is the adjoint with respect to the scalar product defined in \eqref{eq:definition_scalar_product_alg}. 
Here, the second identity holds if $x$ and $y$ are invertible in $\alg$. In fact, $C_{x,y}$ is invertible if and only if $x$ and $y$ are invertible in $\alg$. 

In the following, we will often use the functional calculus for normal elements of $\alg$.
As we will explain now, our setup allows for a direct way to represent $\alg$ as a subalgebra of the bounded operators on a Hilbert space. 
Therefore, one can think of the functional calculus being performed on this Hilbert space. 
The Hilbert space is the completion of $\alg$ equipped with the scalar product defined in \eqref{eq:definition_scalar_product_alg} and denoted by $\Ltwo$. 
In order to represent $\alg$ as subalgebra of the bounded operators $B(\Ltwo)$ on $\Ltwo$, we 
 denote by $\ell_x$ for $x \in \alg$ the left-multiplication on $\Ltwo$ by $x$, 
i.e., $\ell_x \colon \Ltwo \to \Ltwo$, $\ell_x(y) = xy$ for $ y \in \Ltwo$. The inclusion $\alg \subset \Ltwo$ and the Cauchy-Schwarz inequality yield the well-definedness of $\ell_x$ and $\ell_x \in B(\Ltwo)$, 
the bounded linear operators on $\Ltwo$. In fact, 
\[ \alg \to B(\Ltwo), \qquad x \mapsto \ell_x \]
defines a faithful representation of $\alg$ as a von Neumann algebra in $B(\Ltwo)$ \cite[Theorem 2.22]{TakesakiBook}.

We now introduce the \emph{balanced polar decomposition} of $m$. If $w=w(z)\in \alg$, $q=q(z) \in \alg$ and $u = u(z) \in \alg$ 
are defined through 
\begin{equation} \label{eq:def_q_u}
 w \defeq (\Im m)^{-1/2} (\Re m) (\Im m)^{-1/2} + \ii \id, \qquad q \defeq \abs{w}^{1/2} (\Im m)^{1/2}, \qquad u \defeq\frac{w}{\abs{w}}
\end{equation}
via the spectral calculus of the self-adjoint operator $(\Im m)^{-1/2} (\Re m) (\Im m)^{-1/2}$ 
then we have 
\begin{equation} \label{eq:m_representation_q_star_u_q}
m(z) = \Re m(z) + \ii \Im m(z) = q^* u q. 
\end{equation}
Here, $u$ is unitary and commutes with $w$. 
The decomposition $m = q^* u q$ was already introduced and also called balanced polar decomposition in \cite{AjankiCorrelated} in the special setting of matrix algebras. 
 The operators $\abs{w}^{1/2}$, $q$ and $u$ correspond to $\mathbf{W}$, 
$\mathbf{W}\sqrt{\Im \mathbf{M}}$ and $\mathbf{U}^*$ in the notation of \cite{AjankiCorrelated}, respectively.
With the definitions in \eqref{eq:def_q_u}, \eqref{eq:dyson} reads as 
\begin{equation} \label{eq:dyson_second_version}
 - u^* = q(z-a)q^* + F[u],
\end{equation}
where we introduced the \emph{saturated self-energy operator}
\begin{equation} \label{eq:def_F}
F \defeq C_{q,q^*}S C_{q^*,q}. 
\end{equation} 
It is positivity-preserving as well as symmetric, $F=F^*$, and corresponds to the saturated self-energy operator $\cal{F}$ in \cite{AjankiCorrelated}.

\begin{proof}[Proof of Proposition \ref{pro:stieltjes_representation}]
The existence of $v$ will be a consequence of the following lemma which will be proven in Appendix \ref{app:positive_operator_valued_measure} below. 

\begin{lemma} \label{lem:existence_measure}
Let $\alg$ be a von Neumann algebra with unit $\id$ and a tracial, faithful, normal state 
$\avg{~} \colon \alg \to \C$. If $ h \colon \Hb \to \alg$ is a holomorphic function satisfying $\Im h(z) \in \alg_+$ 
for all $z \in \Hb$ and 
\begin{equation} \label{eq:Stieltjes_transform_normalization}
 \lim_{\eta \to \infty} \ii \eta h(\ii\eta ) = -\id
\end{equation}
then there exists a unique measure $v \colon \mathcal B \to \alg$ on the Borel sets $\mathcal B$ of $\R$ with values in $\algnon$ such that 
\begin{equation} \label{eq:Stieltjes_representation_h}
 h(z) = \int_\R \frac{v(\di \tau)}{\tau -z}
\end{equation}
 for all $z \in \Hb$ and $v(\R) =\id$. 
\end{lemma}

In order to apply Lemma \ref{lem:existence_measure}, we have to verify \eqref{eq:Stieltjes_transform_normalization}
for $h=m$. 
To that end, we take the imaginary part of \eqref{eq:dyson} and use $\Im m \geq 0$ as well as $S \in \Sigma$ to conclude
\[ - \im m^{-1}(z) = \imz \id + S[\Im m] \geq \imz \id. \] 
Hence, $\norm{m(z)} \leq (\Im z)^{-1}$ as for any $x \in \alg$ we have $\norm{x} \leq 1$ if $x$ is invertible and $\Im x^{-1} \geq \id$.  
Therefore, evaluating~\eqref{eq:dyson} at $z =\ii \eta$, $\eta>0$, and multiplying the result by $m$ from the left yield 
\[ \ii\eta m(\ii\eta ) = -\id  + m(\ii\eta)a - m(\ii\eta)S[m(\ii\eta)] \to -\id \]
for $\eta \to \infty$ as $S$ is bounded. Hence, Lemma \ref{lem:existence_measure} implies the existence of $v$, i.e., 
the Stieltjes transform representation of $m$ in~\eqref{eq:Stieltjes_representation}. 
 
This representation has the following well-known bounds as a direct consequence (e.g.~\cite{AjankiQVE,AjankiCorrelated,AltKronecker}). 

\begin{lemma} \label{lem:stieltjes_represenation}
Let $v$ be the measure in Proposition \ref{pro:stieltjes_representation} and $\dens = \avg{v}$. Then, for any $z \in 
\Hb$, we have
\begin{equation} \label{eq:stieltjes_bound_m}
 \norm{m(z)} \leq \frac{1}{\dist(z,\supp \dens)}, \qquad \Im m(z) \leq \frac{\imz}{\dist(z,\supp \dens)^{2}}\id. \qedhere
\end{equation}
\end{lemma}

For the proofs of \eqref{eq:supp_v_subset_spec_a} and \eqref{eq:spec_a_subset_supp_v}, we refer to 
the proofs of Proposition 2.1 in \cite{AjankiCorrelated} and (3.4) in \cite{AltKronecker} in the matrix setup, the same argument works for 
our general setup as well. 

We now prove \eqref{eq:L2_bound}.  
Taking the imaginary part of the Dyson equation, \eqref{eq:dyson_second_version}, yields
\[
\im u\,=\, \imz[ ] qq^* + F[\im u] \,\ge\, \max\{ \imz[ ] qq^*, F[\im u] \} \,.
\]
Thus, $\im u \geq \imz[ ] \norm{(qq^*)^{-1}}^{-1}\id$. We remark that $qq^*$ is invertible since $\Im m(z)>0$ for $z \in \Hb$.
Therefore, the following Lemma \ref{lem:positive_symmetric_map_bounded} with $h =\Im u/\normtwo{\Im u}$ implies $\normtwo{F} \leq 1$. 

\begin{lemma} \label{lem:positive_symmetric_map_bounded}
Let $T \colon \alg \to \alg$ be a positivity-preserving operator which is symmetric with respect to \eqref{eq:definition_scalar_product_alg}. If there are $h \in \alg$ and $\eps>0$ such that 
$h\geq \eps \id$ and $Th \leq h$ then $\normtwo{T} \leq 1$.  
\end{lemma}

\begin{proof}
The argument in the proof of Lemma 4.6 in \cite{AjankiQVE} also yields this lemma in our current setup. 
\end{proof} 

We rewrite the Dyson equation \eqref{eq:dyson_second_version} in the form
\begin{equation} \label{solve for qq*}
q(a-z)q^*\,=\, u^*+F[u]\,.
\end{equation}
We take the $\normtwo{\genarg}$-norm on both sides of \eqref{solve for qq*} and use that $\normtwo{u}=1$ (since it is unitary) and $\normtwo{F}\le 1$ to find
\begin{equation} \label{upper bound on qq*}
\normtwo{q(a-z)q^*}\,\le\, 2\,.
\end{equation}
Then we use the polar decomposition $m=q^*uq$ again and with $z=\tau +\ii \eta$ find
\[
\scalar{m}{(C_{a-\tau}+\eta^2)m} = \re \2\scalar{m}{C_{a-z,(a-z)^*}m} \le \abs{\scalar{m}{C_{a-z,(a-z)^*}m}}
= \abs{\scalar{q(a-z)q^*}{C_{u^*,u}[q(a-z)q^*]}}\le 4\,,
\]
where the last step holds because of \eqref{upper bound on qq*}. Recall that $a=a^*$. Since $\spec(C_{a-\tau}) = \{\lambda\mu: \lambda,\mu \in \spec(a-\tau)\}$ we have $\inf \spec(C_{a-\tau}) \ge \dist (\tau,\rm{Conv}\spec a)^2$, provided $\tau \not \in \rm{Conv}\spec a$. Thus in this case \eqref{eq:L2_bound} follows. In case $\tau  \in \rm{Conv}\spec a$ we simply use the trivial bound $\norm{m}_2 \le \norm{m} \le \eta^{-1}$ from the first inequality of \eqref{eq:stieltjes_bound_m} and \eqref{eq:L2_bound} still holds.
\end{proof}

From now on until the end of Section \ref{subsec:proof_regularity_density_states}, we will always assume that $S$ is \emph{flat}, i.e., $S \in \Sigma_{\rm{flat}}$ (cf.~\eqref{eq:def_self_energy_flat}). 
In fact, all of our estimates will be uniform in all data pairs $(a,S)$ that satisfy 
\begin{equation} \label{eq:estimates_data_pair}
 c_1\avg{x}\id \leq S[x] \leq c_2 \avg{x}\id, \qquad \norm{a} \leq c_3 
\end{equation}
for all $x \in \alg_+$ with the some fixed constants $c_1, c_2, c_3 >0$. 
Therefore, the constants $c_1, c_2, c_3$ from \eqref{eq:estimates_data_pair} are called \emph{model parameters} and we introduce the following convention.

\begin{convention}[Comparison relation] \label{conv:comparsion_1}
Let $x, y \in \alg_{\rm{sa}}$. We write $x \lesssim y$ if there is $c>0$ depending only on the model parameters $c_1, c_2, c_3$ from \eqref{eq:estimates_data_pair} such that 
$cy - x$ is positive definite, i.e., $cy-x \in \algnon$. We define $x \gtrsim y$ and $x \sim y$ accordingly. 
We also use this notation for scalars $x,y$. Moreover, we write $x = y + \ord(\alpha)$ for $x, y \in \alg$ and $\alpha >0$ if $\norm{x-y} \lesssim \alpha$. 
\end{convention}

We remark that we will choose a different set of model parameters later and redefine $\sim$ accordingly (cf.~Convention \ref{conv:comparison_2}).

\begin{proposition}[Properties of the solution] \label{pro:prop_solution_flatness} 
Let $(a,S)$ be a data pair satisfying \eqref{eq:estimates_data_pair} and $m$ be the solution to the associated Dyson equation, \eqref{eq:dyson}. 
We have
\begin{align}
\normtwo{m(z)} & \lesssim 1, \label{eq:m_Ltwo_bound}\\ 
 \norm{m(z)} ~& \lesssim   \frac{1}{\avg{\Im m(z)} + \dist(z, \supp \dens)}, \label{eq:m_bound_dens}\\
 \norm{m(z)^{-1}}~ & \lesssim  1 + \abs{z},\label{eq:m_inverse_bound} \\ 
 \avg{\Im m(z)}\id & \lesssim  \Im m(z) \lesssim (1+\abs{z}^2)\norm{m(z)}^2\avg{\Im m(z)}\id \label{eq:Im_m_sim_avg_only_flat}
\end{align}
uniformly for $z \in \Hb$.
\end{proposition}

These bounds are immediate consequences of the flatness of $S$ exactly as in the proof of Proposition 4.2 in~\cite{AjankiCorrelated} using $\supp \dens = \supp v$ by 
the faithfulness of $\avg{\genarg}$. We omit the details. 

Note that \eqref{eq:m_inverse_bound} implies a lower bound $\norm{m(z)} \gtrsim (1 + \abs{z})^{-1}$ 
since $\norm{m}\norm{m^{-1}} \geq 1$. 

\section{Regularity of the solution and the density of states} \label{sec:regularity}

In this section, we will prove Proposition \ref{pro:regularity_density_of_states} and Proposition \ref{pro:Hoelder_1_3}. 
Their proofs are based on a bound on  the inverse of the stability operator $\Id - C_mS$  of the Dyson equation, \eqref{eq:dyson}, which will 
be given in Proposition~\ref{pro:linear_stability} below. 

\subsection{Linear stability of the Dyson equation}
For the formulation of the following proposition, we introduce 
the harmonic extension of the density of states $\dens$ defined in Definition \ref{def:density_of_states} to $\Hb$.
The harmonic extension at $z \in \Hb$ is denoted by $\dens(z)$ and given by 
\[ \dens(z) \defeq \frac{1}{\pi} \avg{\Im m(z)}. \]

\begin{proposition}[Linear Stability] \label{pro:linear_stability}
There is a universal constant $C > 0$ such that, for the solution $m$ to \eqref{eq:dyson} associated to any $a \in \alg_{\rm{sa}}$ and $S \in \Sigma$ satisfying \eqref{eq:estimates_data_pair}, we have 
\begin{equation} \label{eq:linear_stability}
 \normtwo{(\Id - C_{m(z)} S)^{-1}} \lesssim 1 + \frac{1}{(\dens(z) + \dist(z, \supp \dens))^C}
\end{equation}
uniformly for all $z \in \Hb$. 
\end{proposition}

Before proving Proposition \ref{pro:linear_stability}, we will explain how the linear stability yields the Hölder-continuity and analyticity of $\dens$ in Proposition~\ref{pro:regularity_density_of_states}. 
Indeed, assuming that $m$ depends differentiably on $(z,a,S)$, we can compute the directional derivative $\nabla_{(\delta,d,D)}$ at $(z,a,S)$ of both 
sides in \eqref{eq:dyson}. The result of this computation is 
\[
(\Id-C_m S)[\nabla_{(\delta,d,D)}m]\,=\,m(\delta -d + D[m])m.  
\]
Using the bound in Proposition \ref{pro:linear_stability} and $\dens(z) = \pi^{-1} \avg{\Im m(z)}$, we conclude 
from \eqref{eq:m_bound_dens} that 
\begin{equation} \label{eq:dens_general_derivatives}
\abs{\nabla_{(\delta,d,D)}\dens}\,\le\, \frac{1}{\dens^C}(\abs{\delta}+ \norm{d} + \norm{D})
\end{equation}
with a possibly larger $C$. Therefore, it is clear that the control on $(\Id - C_m S)^{-1}$ will be the key input in the proof of Proposition \ref{pro:regularity_density_of_states}.

In order to prove Proposition \ref{pro:linear_stability}, we will use the representation 
\begin{equation} \label{eq:B_alternative_representation_F}
 \Id - C_m S = C_{q^*,q}C_u ( C_u^* - F)C_{q^*,q}^{-1}, 
\end{equation}
where $q$, $u$ and $F$ were defined in \eqref{eq:def_q_u} and \eqref{eq:def_F}, respectively. This representation has the advantage that $C_u^*$ is unitary and $F$ is symmetric. Hence, it is much 
easier to obtain some spectral properties for $C_u^* - F$ compared to $\Id - C_mS$. Now, we will first analyze $q$ and $F$ in the following two lemmas and 
then use this knowledge to verify Proposition~\ref{pro:linear_stability}.

\begin{lemma} \label{lem:norm_q}
If \eqref{eq:estimates_data_pair} holds true then we have
\[ \norm{q(z)}\lesssim (1 + \abs{z})^{1/2} \norm{m(z)}, \qquad \norm{q(z)^{-1}} \lesssim ( 1 + \abs{z}) \norm{m(z)}^{1/2}  \]
uniformly for $z \in \Hb$. 
\end{lemma} 

\begin{proof}
For $q = q(z)$, we will show below that 
\begin{equation} \label{eq:norm_q_first_step}
 \frac{A^{1/2}}{B^{1/2}} \norm{m(z)^{-1}}^{-1} \id \leq q^* q \leq \frac{B^{1/2}}{A^{1/2}} \norm{m(z)}\id   
\end{equation}
if $A \id \leq \Im m(z) \leq B \id$ for some $A, B \in (0,\infty)$ and $z \in \Hb$. Choosing $A$ and $B$ according to \eqref{eq:Im_m_sim_avg_only_flat}, 
using the $C^*$-property of $\norm{\genarg}$, $\norm{q^*q}=\norm{q}^2$, and \eqref{eq:m_inverse_bound}, we immediately obtain Lemma \ref{lem:norm_q}. 

For the proof of \eqref{eq:norm_q_first_step}, we set $g \defeq \Re m$ and $ h\defeq \Im m$. Using the monotonicity
of the square root, we compute 
\[ \begin{aligned}
q^* q & = h^{1/2}\Big( \id + h^{-1/2} g h^{-1} g h^{-1/2}\Big)^{1/2} h^{1/2} \\ 
 & \leq A^{-1/2} h^{1/2} \Big(h^{-1/2} ( h^2 + g^2 ) h^{-1/2}\Big)^{1/2} h^{1/2} \\ 
& \leq  \norm{m} A^{-1/2} h ^{1/2}.  
\end{aligned} 
\]
Here, we employed $h^{-1} \leq A^{-1} \id$ as well as $\id \leq A^{-1} h$  in the first step and $(\Re m)^2 + (\Im m)^2 = (m^*m + mm^*)/2 \leq \norm{m}^2$ in the second step. 
Thus, $h \leq B \id$ yields the upper bound in \eqref{eq:norm_q_first_step}. 
Similar estimates using $\id \geq B^{-1} h$ and $\norm{m^{-1}}^{-2} \leq (m^*m + mm^*)/2$ prove the lower bound 
in \eqref{eq:norm_q_first_step} which completes the proof of the lemma. 
\end{proof}

\begin{lemma}[Properties of $F$] \label{lem:prop_F}
If the bounds in \eqref{eq:estimates_data_pair} are satisfied then $\normtwo{F}$ is a simple eigenvalue of $F \colon \alg \to \alg$ 
defined in \eqref{eq:def_F}. 
Moreover, there is a unique eigenvector 
$f \in \alg_+$ such that $F[f] = \normtwo{F} f$ and $\normtwo{f} = 1$. This eigenvector satisfies 
\begin{equation} \label{eq:normtwo_F}
 1- \normtwo{F} = (\Im z) \frac{\scalar{f}{qq^*}}{\scalar{f}{\Im u}}. 
\end{equation}
In particular, $\normtwo{F}\leq 1$. 
Furthermore, the following properties hold true uniformly for $z \in \Hb$ satisfying $\abs{z}\leq 3 ( 1 + \norm{a} + \norm{S}^{1/2})$ and $\normtwo{F(z)} \geq 1/2$: 
\begin{enumerate}[label=(\roman*)]
\item The eigenvector $f$ has upper and lower bounds 
\begin{equation}
 \norm{m}^{-4}\id \lesssim f \lesssim \norm{m}^{4}\id.  \label{eq:f_sim_norm_m} 
\end{equation}
\item The operator $F$ has a spectral gap $\vartheta\in (0,1]$ satisfying $\vartheta \gtrsim \norm{m}^{-28}$ and 
\begin{equation}
 \spec(F/\normtwo{F}) \subset [-1 + \vartheta, 1- \vartheta] \cup \{ 1\}.\label{eq:spec_F}
\end{equation}
\end{enumerate}
\end{lemma} 

\begin{proof} 
 The definition of $F$ in \eqref{eq:def_F}, \eqref{eq:estimates_data_pair} and Lemma \ref{lem:norm_q} imply
\begin{equation}  \label{eq:proof_prop_F_aux1}
 (1+ \abs{z})^{-4}\norm{m(z)}^{-2} \avg{a} \id \lesssim F[a] \lesssim (1 + \abs{z})^2\norm{m(z)}^4 \avg{a} \id  
\end{equation}
for all $a \in \algpos$ and all $z \in \Hb$. 
We will use Lemma \ref{lem:positive_symmetric_maps_aux} (ii) from Appendix \ref{app:positive_symmetric_map}. 
The condition \eqref{eq:T_flat_aux} with $T = F$ is satisfied by \eqref{eq:proof_prop_F_aux1} with constants depending on $\norm{m}$ and $\abs{z}$. 
Hence, Lemma \ref{lem:positive_symmetric_maps_aux} (ii) implies the existence and uniqueness of the eigenvector $f$. 
We compute the scalar product of $f$ with the imaginary part of \eqref{eq:dyson_second_version}. Since $F$ is symmetric, this immediately 
yields \eqref{eq:normtwo_F}. 

We now assume that $z \in \Hb$ satisfies $\abs{z}\leq 3 ( 1 + \norm{a} + \norm{S}^{1/2})$ and $\normtwo{F(z)} \geq 1/2$.
Then $\abs{z} \lesssim 1$ and, by using this in \eqref{eq:proof_prop_F_aux1}, we thus obtain \eqref{eq:f_sim_norm_m} and \eqref{eq:spec_F} from Lemma \ref{lem:positive_symmetric_maps_aux} (ii) 
since $\norm{m} \gtrsim 1$ by \eqref{eq:m_inverse_bound}. 
\end{proof}

The following proof of Proposition \ref{pro:linear_stability} proceeds similarly to the one of Proposition 4.4 in \cite{AjankiCorrelated}. 

\begin{proof}[Proof of Proposition \ref{pro:linear_stability}]
We will distinguish several cases. If $\abs{z} \geq 3 ( 1 + \kappa)$ with $\kappa \defeq \norm{a} + 2 \norm{S}^{1/2}$ then we conclude from \eqref{eq:Stieltjes_representation} 
and $\supp \dens \subset [-\kappa, \kappa]$ by \eqref{eq:supp_v_subset_spec_a} that $\norm{m(z)} \leq ( \abs{z} - \kappa)^{-1}$. Thus, 
\[ \normtwo{C_{m(z)} S} \leq \frac{\normtwo{S}}{( \abs{z} -\kappa)^2} \leq \frac{\norm{S}}{4 ( 1+ \kappa)^2} \leq \frac{1}{4}. \] 
Here, we used $\normtwo{S} \leq \norm{S}$ since $S$ is symmetric and $\kappa \geq \norm{S}^{1/2}$. This shows \eqref{eq:linear_stability} for large $\abs{z}$. 

Next, we assume $\abs{z} \leq 3 ( 1 + \kappa)$. In this regime, we use the alternative representation of $\Id - C_mS$ in 
\eqref{eq:B_alternative_representation_F} and the spectral properties of $F$ from Lemma \ref{lem:prop_F}. 
Indeed, from \eqref{eq:B_alternative_representation_F} and Lemma \ref{lem:norm_q}, we conclude 
\begin{equation} \label{eq:proof_linear_stability_aux1}
 \normtwo{(\Id - C_m S)^{-1}} \lesssim \norm{m}^3 \normtwo{(C_u^*- F)^{-1}}\lesssim \frac{1}{(\dens(z)+\dist(z,\supp\dens))^3} 
\normtwo{(C_u^* - F)^{-1}} 
\end{equation}
as $u \in \alg$ is unitary. Here, we used \eqref{eq:m_bound_dens} in the last step. 
If $\normtwo{F(z)} \leq 1/2$ then this immediately yields \eqref{eq:linear_stability} as $\normtwo{C_u}=1$. We now assume $\normtwo{F(z)} \geq 1/2$. 
In this case, we will use the following lemma. 

\begin{lemma}[Rotation-Inversion Lemma]\label{lem:rotation_inversion}
 Let $U$ be a unitary operator on $\Ltwo$ and $T$ a symmetric operator on $\Ltwo$. We assume that 
there is a constant $\theta >0$ such that 
\[ \spec T \subset [ -\normtwo{T} +\theta, \normtwo{T}-\theta] \cup \{ \normtwo{T} \} \]
with a non-degenerate eigenvalue $\normtwo{T} \leq 1$. Then there is a universal constant $C>0$ such that 
\[ \normtwo{(U-T)^{-1}} \leq \frac{C}{\theta \abs{ 1 - \normtwo{T} \scalar{t}{U[t]}}}, \] 
where $t \in \Ltwo$ is the normalized, $\normtwo{t}=1$, eigenvector of $T$ corresponding to $\normtwo{T}$.  
\end{lemma}
The proof of this lemma is identical to the proof of Lemma 5.6 in \cite{AjankiCPAM}, where a result of this type was first 
applied in the context of vector Dyson equations. 

We start from the estimate \eqref{eq:proof_linear_stability_aux1}, use the Rotation-Inversion Lemma, Lemma \ref{lem:rotation_inversion}, with $U=C_u^*$ and $T = F$ as well as 
\eqref{eq:spec_F} and \eqref{eq:m_bound_dens} and obtain 
\[ \normtwo{(\Id- C_m S)^{-1}} \lesssim \frac{(\dens(z) + \dist(z,\supp\dens))^{-31}}{\abs{1 - \normtwo{F} \scalar{f}{C_u^*[f]}}} 
\leq \frac{(\dens(z) + \dist(z,\supp\dens))^{- 31}}{\max\{ 1- \normtwo{F}, \abs{1 - \avg{fC_u^*[f]}}\}}. \] 

In order to complete the proof of \eqref{eq:linear_stability}, we now show that 
\begin{equation} \label{eq:proof_linear_stability_aux2}
 \max\{  1- \normtwo{F}, \abs{1 - \avg{fC_u^*[f]}}\} \gtrsim ( \dens(z) + \dist(z, \supp \dens))^C 
\end{equation}
for some universal constant $C >0$. 
We first prove auxiliary upper and lower bounds on $\Im u= (q^*)^{-1} (\Im m) q^{-1}$. 
We have 
\begin{equation} \label{eq:proof_linear_stability_aux3}
 \dens(z)(\dens(z) + \dist(z,\supp\dens))^2\id \lesssim \Im u \lesssim \frac{\imz \norm{m} }{\dist(z,\supp\dens)^2}\id .  
\end{equation}
For the lower bound, we used the lower bound in \eqref{eq:Im_m_sim_avg_only_flat}, Lemma \ref{lem:norm_q} 
and \eqref{eq:m_bound_dens}. The upper bound is a direct consequence of \eqref{eq:stieltjes_bound_m} as well as Lemma \ref{lem:norm_q}. 
Since $\scalar{f}{qq^*}\geq \norm{(qq^*)^{-1}}^{-1} \avg{f} \gtrsim \norm{m} \avg{f}$ by Lemma \ref{lem:norm_q}, the relation \eqref{eq:normtwo_F} and the upper bound in \eqref{eq:proof_linear_stability_aux3} yield 
\[ 1- \normtwo{F} \gtrsim \dist(z,\supp\dens)^2. \] 
As $1- \avg{fC_{\Re u}[f]} \geq 0$ and $\avg{f^2} = 1$, we obtain from the lower bound in \eqref{eq:proof_linear_stability_aux3} that  
\begin{equation} \label{eq:proof_linear_stability_aux4}
 \abs{1 - \avg{fC_u^*[f]}} \geq \Re[ 1- \avg{f C_u^*[f]}] = 1 - \avg{fC_{\Re u }[f]} + \avg{f C_{\Im u}[f]} \gtrsim \dens(z)^2(\dens(z) + \dist(z,\supp\dens))^{4}. 
\end{equation}
This completes the proof of \eqref{eq:proof_linear_stability_aux2} and hence of Proposition \ref{pro:linear_stability}. 
\end{proof}

\subsection{Proof of Proposition \ref{pro:regularity_density_of_states}} \label{subsec:proof_regularity_density_states}

The following proof of Proposition \ref{pro:regularity_density_of_states} is similar to the one of Proposition 2.2 in \cite{AjankiCorrelated}. 

\begin{proof}[Proof of Proposition \ref{pro:regularity_density_of_states}]
We first show that $\dens \colon \Hb \to (0,\infty)$ has a uniformly Hölder-continuous extension to $\overline{\Hb}$, which we will also denote by $\dens$. 
This extension restricted to $\R$ will be the density of the measure $\dens$ from Definition \ref{def:density_of_states}. 
Since $\Id-C_mS$ is invertible for each $z \in \Hb$ by \eqref{eq:linear_stability}, the implicit function theorem allows us to differentiate \eqref{eq:dyson} 
with respect to $z$. This yields 
\begin{equation} \label{eq:proof_pro_regularity_aux1}
 (\Id- C_m S)[\pt_z m] = m^2. 
\end{equation}
Since $z \mapsto \avg{m(z)}$ is holomorphic on $\Hb$ as remarked below \eqref{eq:dyson}, we have $2\pi\ii \pt_z \dens(z) = 2\ii \pt_z \Im \avg{m(z)} = \pt_z \avg{m(z)}$. 
Thus, we obtain from \eqref{eq:proof_pro_regularity_aux1} that 
\begin{equation} \label{eq:proof_pro_regularity_aux2}
\abs{\pt_z \dens} \lesssim \normtwo{\pt_z m} \leq \normtwo{(\Id- C_mS)^{-1}} \norm{m}^2 \lesssim \dens^{-(C+2)} 
\end{equation}
Here, we used \eqref{eq:linear_stability}, $\dens(z) \lesssim \normtwo{m(z)} \lesssim 1$ by \eqref{eq:m_Ltwo_bound}
and \eqref{eq:m_bound_dens} 
in the last step. Hence, $\dens^{C+3}$ is a uniformly Lipschitz-continuous function on $\Hb$.
Therefore, $\dens$ defines uniquely a uniformly $1/(C+3)$-Hölder continuous function on $\R$ which is a density of the measure $\dens$ from Definition \ref{def:density_of_states}
with respect to the Lebesgue measure on~$\R$. 

Next, we show the Hölder-continuity with respect to $a$ and $S$. 
As before in \eqref{eq:dens_general_derivatives}, we compute the derivatives and use \eqref{eq:m_bound_dens} and~\eqref{eq:linear_stability} to obtain 
\[ \abs{\nabla_{(d,D)}\dens_{(a,S)}(z)} \lesssim \abs{\avg{\nabla_{(d,D)}m}} \lesssim \frac{\norm{d} + \norm{D}}{\dens^{C+3}}. \] 
Since the constants in \eqref{eq:linear_stability} and \eqref{eq:m_bound_dens} depend on the constants in \eqref{eq:estimates_data_pair}, 
we conclude that $\dens$ is also a locally $1/(C+4)$-Hölder continuous function of $a$ and $S$. 

We are left with showing that $\dens$ is real-analytic in a neighbourhood of $(\tau_0,a,S)$ if $\dens_{a,S}(\tau_0) >0$. 
Since $\dens(\tau_0)>0$, we can extend $m$ to $\tau_0$ by \eqref{eq:proof_pro_regularity_aux2}. Moreover, $m(\tau_0)$ is invertible as $\Im m(\tau_0)>0$ and, thus, 
solves \eqref{eq:dyson} with $z=\tau_0$. Since \eqref{eq:dyson} depends analytically on $z=\tau$, $a$ and $S$ in a small neighbourhood of $(\tau_0,a,S)$, the 
solution $m$ and thus $\dens$ will depend analytically on $(\tau,a,S)$ in this neighbourhood by the implicit function theorem. 
This completes the proof of Proposition~\ref{pro:regularity_density_of_states}.
\end{proof}

\subsection{Proof of Proposition \ref{pro:Hoelder_1_3}} \label{subsec:proof_hoelder_1_3}

For $I \subset \R$ and $\eta_*>0$, we define
\begin{equation} \label{eq:def_HbI}
 \HbI \defeq \{ z \in \Hb \colon \Re z \in I, ~ \Im z \in (0,\eta_*] \}
\end{equation}
and its closure $\HbIclosed$.

\begin{assums} \label{assums:general}
Let $m$ be the solution of \eqref{eq:dyson} for $a=a^*\in\alg$ satisfying $\norm{a} \leq k_1$ with a positive constant $k_1$ and  
$S \in \Sigma$ satisfying $\normtwoinf{S} \leq k_2$ for some positive constant $k_2$. 
For an interval $I \subset \R$ and some $\eta_* \in (0,1]$, we assume that 
\begin{enumerate}[label=(\roman*)]
\item There are positive constants $k_3$, $k_4$ and $k_5$ such that 
\begin{align} 
 \norm{m(z)} & \leq k_3, \label{eq:m_sup_bound}\\
 k_4 \avg{\Im m(z)} \id \leq \Im m(z)  & \leq  k_5\avg{\Im m(z)}\id, \label{eq:Im_m_sim_avg_Im_m}
\end{align}
uniformly for all $z \in \HbI$. 
\item The operator $F\defeq C_{q,q^*}SC_{q^*,q}$ has a simple eigenvalue $\normtwo{F}$ with eigenvector $f \in \algpos$ that satisfies \eqref{eq:normtwo_F} for all $z \in \HbI$. 
Moreover, \eqref{eq:spec_F} holds true and there are positive constants $k_6$, $k_7$ and $k_8$ such that 
\begin{equation} \label{eq:eigenvector_F_sim_1}
 k_6 \id \leq f \leq k_7 \id, \qquad \vartheta \geq k_8.  
\end{equation}
uniformly for all $z \in \HbI$. 
\end{enumerate}
\end{assums}

We remark that $S \in \Sigma_{\rm{flat}}$ is not necessarily required in Assumptions \ref{assums:general}.
In fact, we will show in Lemma \ref{lem:q_bounded_Im_u_sim_avg} below that $S \in \Sigma_\rm{flat}$ and \eqref{eq:m_sup_bound} imply all other 
conditions in Assumptions \ref{assums:general}. 

\begin{convention}[Model parameters, Comparison relation] \label{conv:comparison_2}
For the remainder of the Section \ref{sec:regularity} as well as Section \ref{sec:stability_operator} and Section \ref{sec:cubic_equation}, 
we will only consider $k_1, \ldots, k_8$ as \emph{model parameters} and understand the comparison relation $\sim$ from Convention~\ref{conv:comparsion_1} 
with respect to this set of model parameters. 
\end{convention}

We remark that all of our estimates will be uniform in $\eta_* \in (0,1]$. Therefore, $\eta_*$ is not considered a model parameter. 
At the end of this section, we will directly conclude Proposition \ref{pro:Hoelder_1_3} from the following proposition. 

\begin{proposition}[Regularity of $m$] \label{pro:analyticity_of_m} 
Let Assumptions \ref{assums:general} hold true on an interval $I \subset \R$ for some $\eta_* \in (0,1]$. 

Then, for any $\theta\in(0,1]$, $m$ can be uniquely extended to $I_\theta \defeq \{\tau \in I \colon \dist(\tau, \pt I) \geq \theta\}$
such that it is uniformly $1/3$-Hölder continuous, indeed, 
\begin{equation} \label{eq:Hoelder_1_3_general_assums}
 \norm{m(z_1) - m(z_2)} \lesssim \theta^{-4/3}\abs{z_1-z_2}^{1/3} 
\end{equation}
for all $z_1, z_2 \in I_\theta\times\ii[0,\infty)$. 
Moreover, if $\dens(\tau_0) >0$, $\tau_0 \in I$, then $m$ is real-analytic in a neighbourhood of $\tau_0$ and 
\begin{equation} \label{eq:m_Lipschitz_bound_positive_density}
\norm{\pt_\tau m(\tau_0)} \lesssim \dens(\tau_0)^{-2}. 
\end{equation}
\end{proposition}

We remark that the bound in \eqref{eq:m_Lipschitz_bound_positive_density} will be extended to higher derivatives in Lemma~\ref{lem:derivatives_m} below.

In the following lemma, we establish a very helpful consequence of (i) in Assumptions \ref{assums:general}. 
Moreover, part (ii) of the following lemma shows that all conditions in Assumptions \ref{assums:general} are satisfied if we assume \eqref{eq:m_sup_bound} and the flatness of $S$. 

\begin{lemma} \label{lem:q_bounded_Im_u_sim_avg}
Let $m$ be the solution to \eqref{eq:dyson} for some data pair $(a,S) \in \alg_\rm{sa}\times \Sigma$. We have
\begin{enumerate}[label=(\roman*)]
\item Let $\norm{a} \lesssim 1$, $\norm{S} \lesssim 1$ and $U \subset \Hb$ such that $\sup\{\abs{z}\colon z \in U\} \lesssim 1$. 
If \eqref{eq:m_sup_bound} and \eqref{eq:Im_m_sim_avg_Im_m} hold true uniformly for $z \in U$ then, uniformly for $z \in U$, we have 
\begin{equation} \label{eq:q_sim_1_im_u_sim_dens_consequence_assum}
 \norm{q}, \norm{q^{-1}} \sim 1, \qquad \Im u \sim \avg{\Im u }\id \sim \dens\id.
\end{equation}
\item 
Let $I \subset [-C, C]$ for some $C \sim 1$ and \eqref{eq:m_sup_bound} hold true uniformly for all $z \in \HbI$.
If $S \in \Sigma_{\rm{flat}}$ and $\norm{a} \lesssim 1$ then $\normtwoinf{S} \lesssim 1$, \eqref{eq:Im_m_sim_avg_Im_m} holds true uniformly for all $z \in \HbI$ 
and part (ii) of Assumptions \ref{assums:general} is satisfied. 
\item If Assumptions \ref{assums:general} hold true then, uniformly for $z \in \HbI$, we have
\begin{equation} \label{eq:norm_stab_operator_m_sim_one}
 \normtwo{(\Id-C_{m(z)}S)^{-1}} + \norm{(\Id- C_{m(z)} S)^{-1}} \lesssim \dens(z)^{-2} . 
\end{equation}
\end{enumerate}
\end{lemma} 

\begin{proof}[Proof of Lemma \ref{lem:q_bounded_Im_u_sim_avg}]
For the proof of (i), we use $\norm{a} \lesssim 1$, $\norm{S} \lesssim 1$ and \eqref{eq:dyson} to show $\norm{m(z)^{-1}} \lesssim 1$ uniformly for all $z \in U$.
Thus, following the proof of Lemma \ref{lem:norm_q} immediately yields the estimates on $q$ and $q^{-1}$ in 
\eqref{eq:q_sim_1_im_u_sim_dens_consequence_assum} due to \eqref{eq:m_sup_bound} and~\eqref{eq:Im_m_sim_avg_Im_m}. 
Thus, as $\norm{q}, \norm{q^{-1}} \sim 1$, we obtain the missing relations in \eqref{eq:q_sim_1_im_u_sim_dens_consequence_assum} from \eqref{eq:Im_m_sim_avg_Im_m} since
\[ \Im u = (q^*)^{-1} (\Im m) q^{-1} \sim \Im m \sim \avg{\Im m} \sim \avg{\Im u}. \] 
We now show (ii). By Lemma \ref{lem:norm_two_to_inf} (i), the upper bound in the definition of flatness, \eqref{eq:estimates_data_pair}, implies $\normtwoinf{S} \lesssim 1$. 
Owing to \eqref{eq:m_sup_bound} and \eqref{eq:m_inverse_bound}, we have $\norm{m(z)} \sim 1$ for all $z\in \HbI$. 
Hence, \eqref{eq:Im_m_sim_avg_Im_m} follows from \eqref{eq:Im_m_sim_avg_only_flat} since $\abs{z} \leq C + 1$ for $z \in \HbI$. 
Moreover, (ii) in Assumptions \ref{assums:general} is a consequence of Lemma \ref{lem:prop_F}. 

To prove \eqref{eq:norm_stab_operator_m_sim_one}, we follow the proof of Proposition \ref{pro:linear_stability} and replace the use of \eqref{eq:m_bound_dens} 
as well as \eqref{eq:f_sim_norm_m} and \eqref{eq:spec_F} from Lemma \ref{lem:prop_F} by \eqref{eq:m_sup_bound} and \eqref{eq:eigenvector_F_sim_1}, respectively. This yields 
\begin{equation} \label{eq:proof_norm_stab_operator_last_step}
 \normtwo{(\Id- C_mS)^{-1}} \lesssim 1 + \abs{1- \normtwo{F} \avg{fC_u^*[f]}}^{-1}\lesssim \abs{1- \normtwo{F}\avg{fC_u^*[f]}}^{-1},  
\end{equation}
where we used in the last step that \eqref{eq:m_sup_bound} implies $\dens(z) \lesssim 1$ on $\HbI$. 
Since $\Im u \sim \dens$ by \eqref{eq:q_sim_1_im_u_sim_dens_consequence_assum} and $\normtwo{F} \leq 1$ by \eqref{eq:normtwo_F} that holds under Assumptions \ref{assums:general} (ii), we conclude 
\[ \abs{1- \normtwo{F} \avg{fC_u^*[f]}}^{-1} \lesssim \abs{1- \avg{fC_u^*[f]}}^{-1} \lesssim \dens^{-2} \]
as in \eqref{eq:proof_linear_stability_aux4} in the proof of Proposition \ref{pro:linear_stability}. 
This shows $\normtwo{(\Id- C_mS)^{-1}} \lesssim \dens(z)^{-2}$. 
Using $\normtwoinf{S} \lesssim 1$ and Lemma \ref{lem:norm_two_to_inf} (ii), we obtain the missing $\norm{\genarg}$-bound in 
\eqref{eq:norm_stab_operator_m_sim_one}. This completes the proof of Lemma \ref{lem:q_bounded_Im_u_sim_avg}. 
\end{proof} 

\begin{proof}[Proof of Proposition \ref{pro:analyticity_of_m}] 
Similarly to the proof of Proposition \ref{pro:regularity_density_of_states}, we obtain 
\begin{equation} \label{eq:proof_pro_hoelder_1_3_aux1}
 \norm{\pt_z \Im m(z)} \lesssim \norm{\pt_z m(z)} \leq \norm{(\Id- C_mS)^{-1}} \norm{m(z)}^2 \lesssim \dens(z)^{-2}\sim \norm{\Im m(z)}^{-2} 
\end{equation}
for $z \in \HbI$ from \eqref{eq:m_sup_bound}, \eqref{eq:norm_stab_operator_m_sim_one} and \eqref{eq:Im_m_sim_avg_Im_m}.
By the submultiplicativity of $\norm{\genarg}$, $(\Im m(z))^3 \colon \HbI \to (\alg, \norm{\genarg})$ is a uniformly Lipschitz-continuous function. 
Hence, $\Im m(z)$ is uniformly $1/3$-Hölder continuous on $\HbI$ (see e.g.~Theorem X.1.1 in \cite{Bhatia_matrix_analysis}) 
and, thus, has a uniformly $1/3$-Hölder continuous extension to $\HbIclosed$. 
We conclude that the measure $v$ restricted to $I$ has a density with respect to the Lebesgue 
measure on $I$, i.e., \eqref{eq:v_density_Im_m} holds true for all measurable $A \subset I$. 
Now, \eqref{eq:Hoelder_in_close_support} in Lemma \ref{lem:Stieltjes_transform_Hölder} implies the uniform $1/3$-Hölder continuity of $m$ on $I_\theta \times \ii(0,\infty)$. In particular, $m$ can be uniquely extended to a uniformly 
$1/3$-Hölder continuous function on $I_\theta \times \ii[0,\infty)$ such that \eqref{eq:Hoelder_1_3_general_assums} holds true. 

To prove the analyticity of $m$, we refer to the proof of the analyticity of $\dens$ in Proposition \ref{pro:regularity_density_of_states}. 
The bound \eqref{eq:m_Lipschitz_bound_positive_density} can be read off from \eqref{eq:proof_pro_hoelder_1_3_aux1}. 
This completes the proof of the proposition.
\end{proof}

\begin{proof}[Proof of Proposition \ref{pro:Hoelder_1_3}]
By \eqref{eq:boundedness_for_Hoelder}, there are $C_0>0$ and $\eta_* \in (0,1]$ such that $\norm{m(\tau +\ii\eta)} \leq C_0$ for all $\tau \in I$ and $\eta \in (0,\eta_*]$. 
Hence, by Lemma \ref{lem:q_bounded_Im_u_sim_avg} (ii), the flatness of $S$ 
implies Assumptions \ref{assums:general} on $I \cap [-C,C]$ for $C \defeq 3(1 + \norm{a} + \norm{S}^{1/2})$, i.e., $C \sim 1$. 
Therefore, Proposition \ref{pro:analyticity_of_m} yields Proposition \ref{pro:Hoelder_1_3} on $I \cap [-C,C]$. 
 
Owing to \eqref{eq:stieltjes_bound_m} and $\supp v = \supp\dens$, we have $\dist(\tau, \supp v) \geq 1$ for $\tau \in I$
satisfying $\tau \notin [-C+1,C-1]$. Hence, for these $\tau$, the Hölder-continuity follows immediately from 
\eqref{eq:Hoelder_away_from_support} in Lemma \ref{lem:Stieltjes_transform_Hölder}. 
By \eqref{eq:supp_v_subset_spec_a}, we have $\Im m(\tau) = 0$ for $\tau \in I$ satisfying $\tau \notin [-C,C]$. 
Therefore, the statement about the analyticity is trivial outside of $[-C,C]$. 
This completes the proof of Proposition~\ref{pro:Hoelder_1_3}. 
\end{proof}

\section{Spectral properties of the stability operator for small self-consistent density of states}  \label{sec:stability_operator}

In this section, we study the  stability operator $B=B(z)\defeq \Id - C_{m(z)} S$,  when $\dens=\dens(z)$ is small and Assumptions \ref{assums:general} hold true. 
Note that we do not require $S$ to be flat, i.e., to satisfy \eqref{eq:estimates_data_pair}.
We will view  $B$  as a perturbation of the operator $B_0$, which we introduce now. We define 
\begin{equation} \label{eq:def_s_B_0_E}
 s \defeq \sign \Re u,  \qquad B_0 \defeq C_{q^*,q}( \Id - C_s F) C_{q^*,q}^{-1}, \qquad E \defeq (C_{q^*sq} - C_m)S=C_{q^*,q}(C_s - C_u)FC_{q^*,q}^{-1},  
\end{equation}
with $u$ and $q$ defined in \eqref{eq:def_q_u} and $F$ defined in \eqref{eq:def_F}.
Note $B_0 = \Id - C_{q^*sq}S$, i.e., in the definition of $B$, $u$ in $m=q^*uq$ is replaced by $s$. 
Thus, we have $B = B_0 + E$. 
Under Assumptions \ref{assums:general}, \eqref{eq:q_sim_1_im_u_sim_dens_consequence_assum} holds true which we will often use in the following. 
Since $\id - \abs{\Re u } = \id - \sqrt{\id - (\Im u)^2} \leq (\Im u)^2 \lesssim \dens^2$, we also obtain  
\begin{equation} \label{eq:Re_u_s}
 \Re u = s + \ord(\dens^2), \qquad \im u \,=\, \ord(\rho)\,,\qquad \Re m = q^* s q + \ord(\dens^2) 
\end{equation}
and with $C_s-C_u =\ord(\norm{s-u})= \ord(\rho)$ we get 
\begin{equation} \label{eq:E=ord rho}
E\, =\,\ord(\dens)\,.
\end{equation}
Here, we use the notation $R = T + \ord(\alpha)$ for operators $T$ and $R$ on $\alg$ and $\alpha>0$ if $\norm{R - T} \lesssim \alpha$.
We introduce 
\begin{equation} \label{eq:def_f_u}
f_u \defeq \dens^{-1} \Im u. 
\end{equation}
By the functional calculus for the normal operator $u$, $\Re u$, $s$ and $f_u$ commute. Hence, $C_s[f_u] = f_u$. 
From the imaginary part of \eqref{eq:dyson_second_version} and \eqref{eq:q_sim_1_im_u_sim_dens_consequence_assum}, we conclude that 
\begin{equation} \label{eq:F_f_u}
 (\Id - F) [f_u] = \dens^{-1} \imz qq^* = \ord(\dens^{-1}\imz). 
\end{equation}
The following technical lemma provides control on the resolvent of the stability operator $B$ and its relatives. It has been stated for the finite dimensional situation $\cal{A}=\C^{N \times N}$ in \cite[Corollary 4.8]{AltEdge}. For the reader's convenience we present its proof following the same line of reasoning as in \cite{AltEdge}.
For  $z\in \C$ and $\eps>0$, we denote by $D_\eps(z) \defeq \{ w \in \C \colon \abs{z-w} < \eps\}$ the disk in $\C$ of radius $\eps$ around $z$. 

\begin{lemma}[Spectral properties of stability operator] \label{lem:prop_F_small_dens} 
Let $T \in \{\Id-F, \Id-C_s F,B_0,B, \Id - C_{m^*,m}S \}$. 
If Assumptions~\ref{assums:general} are satisfied on an interval $I \subset \R$ for some $\eta_* \in (0,1]$, then there
are $\dens_* \sim 1$ and $\eps \sim 1$ such that
\begin{equation} \label{eq:B_0_resolvent_bound}
 \norm{( T-\omega\2 \Id)^{-1}}_2 +\norm{( T-\omega\2 \Id)^{-1}} +\norm{(  T^*-\omega\2\Id)^{-1}} \lesssim 1 
\end{equation}
uniformly for all $z \in\HbI$ satisfying $\dens(z)+\dens(z)^{-1} \imz\leq \dens_*$ and for all $\omega \in \C$ with $\omega \not \in D_{\eps}(0) \cup D_{1-2 \eps}(1) $. 
Furthermore, there is a single simple (algebraic multiplicity $1$) eigenvalue $\lambda$ in the disk around $0$, i.e., 
\begin{equation}
\label{eq:nondegeneracy for CF}
\spec(T) \cap D_{\eps}(0)\,=\, \{\lambda\}\quad \text{and} \quad \rm{rank}\2 P_T\,=\, 1\,, \quad \text{where} \quad 
P_T\,\defeq \, -\frac{1}{2\pi\ii}\int_{\partial D_\eps(0)}(T-\omega\Id)^{-1} \di\omega \,.
\end{equation}
\end{lemma}
If Assumptions \ref{assums:general} are satisfied on $I$ for some $\eta_* \in (0,1]$ then we have
\begin{equation} \label{eq:f_u_sim_1}
 f_u = \dens^{-1} \Im u \sim 1. 
\end{equation}
uniformly for $z \in \HbI$ due to \eqref{eq:q_sim_1_im_u_sim_dens_consequence_assum}. This fact will often be used in the following without mentioning it.

\begin{proof} 
 First, we notice that for each choice of the operator $T$ from the lemma, the bound $\normtwoinf{\Id-T} \lesssim 1$ holds because of $\normtwoinf{S} \lesssim 1$, \eqref{eq:m_sup_bound} and \eqref{eq:q_sim_1_im_u_sim_dens_consequence_assum}. Therefore invertibility of $T-\omega\2 \Id$ as an operator on $L^2$ implies invertibility as an operator on $\alg$, as long as  $\omega$ stays away from $1$, due to Lemma \ref{lem:norm_two_to_inf} (ii). It suffices thus to show the bound on the $\norm{\1\cdot\1}_2$-norm from \eqref{eq:B_0_resolvent_bound} and \eqref{eq:nondegeneracy for CF}. For $T=\Id-F$ both assertions hold by Lemma~\ref{lem:prop_F}. In particular, we find 
\begin{equation} \label{eq:f_u_approx_f}
f  = \normtwo{f_u}^{-1} f_u+\ord(\dens^{-1} \imz)\,,
\end{equation}
where $f$ is the single top eigenvector of $F$, $Ff= \normtwo{F}f$ (see Lemma \ref{lem:prop_F}).
The proof of \eqref{eq:f_u_approx_f} follows from \eqref{eq:F_f_u} and $\normtwo{F}= 1 +\ord(\dens^{-1}\imz)$ (cf.~\eqref{eq:normtwo_F}) by straightforward perturbation theory of the simple isolated eigenvalue~$\normtwo{F}$.

We will now prove \eqref{eq:nondegeneracy for CF} and the $\norm{\1\cdot\1}_2$-norm bound
\begin{equation} \label{eq:T 2 norm bound}
 \norm{( T-\omega\2 \Id)^{-1}}_2 \,\lesssim\, 1\,,  \qquad \omega \not \in D_{\eps}(0) \cup D_{1-2 \eps}(1) 
\end{equation}
 for the choices $T=\Id-C_sF ,B_0,B,\Id - C_{m^*,m}S$ in this order. We start with $T=\Id-C_s F$. We introduce the interpolation $T_t:=\Id-V_tF$ between $T_0=\Id-F$ and $T_1=\Id-C_sF$ by setting 
\[ V_t \defeq  (1-t) \Id + t\2C_s\,,\qquad t \in [0,1]\,.\]
Once we have established \eqref{eq:T 2 norm bound} with $T=T_t$ for all $t \in[0,1]$, 
the assertion about the single isolated eigenvalue \eqref{eq:nondegeneracy for CF} also follows for $T=T_t$. Indeed, the rank of the spectral projection $P_{T_t}$ is a
continuous function of $t$ and thus $\rm{rank}\2P_{T_t}=\rm{rank}\2P_{T_0}=1$ by what we have already shown.

In order to show \eqref{eq:T 2 norm bound} we consider two regimes. On the one hand, for $\abs{\omega}\ge 3$ we simply use $\normtwo{F} \le 1$ and  $\normtwo{V_t} \leq 1$. On the other hand, for $\abs{\omega}\le 3$ we estimate the norm of $((1-\omega)\Id - V_t F)[x]$ from below for any $x \in L^2$. For this purpose we decompose $x = \alpha f + y$ according to the top eigenvector $f$ of $F$, with $y \perp f$ and $\alpha \in \C$. Then we find
\begin{equation} \label{eq:estimate_T_omega} 
\begin{split}
\normtwo{( (1-\omega)\Id-V_t F  )[x]}^2 &= \, \abs{\alpha}^2 \abs{\omega}^2+\normtwo{( (1-\omega)\Id-V_t F )[y]}^2 +\ord\big( \dens^{-1}\imz\norm{x}_2^2\big)  
\\
&\ge \abs{\alpha}^2 \eps^2+(\vartheta-2\eps)^2(\norm{x}_2^2-\abs{\alpha}^2)+\ord\big( \dens^{-1}\imz\norm{x}_2^2\big)  
\,,
\end{split}
\end{equation} 
where $ \vartheta \sim 1 $ is the spectral gap of $F$ from \eqref{eq:spec_F}. In the equality of \eqref{eq:estimate_T_omega} we used that 
$V_t F[f] = f + \ord( \dens^{-1}\imz)$ and $FV_t [f] = f + \ord( \dens^{-1}\imz)$ due to \eqref{eq:f_u_approx_f}, $V_t[f_u] = f_u$ and $\normtwo{F}= 1 +\ord(\dens^{-1}\imz)$, as well as the orthogonality of $y$ and $f$.
For the inequality in \eqref{eq:estimate_T_omega} we estimated $\abs{\omega} \ge \eps$ and used
\[
\normtwo{( (1-\omega)\Id-V_t F )[y]}^2\ge (\abs{1-\omega}-\norm{F}_2(1-\vartheta))^2\normtwo{y}^2 \ge (\vartheta-2\eps)^2(\norm{x}_2^2-\abs{\alpha}^2)\,.
\]
From \eqref{eq:estimate_T_omega} we now conclude $\normtwo{( (1-\omega)\Id-V_t F  )[x]}^2 \gtrsim \normtwo{x}^2$ by choosing $\eps$ and $\rho_\ast$ small enough. 

Since we have established the claim of the lemma for $T=\Id-C_sF$ it also follows for $T= B_0$ because of the definition of $B_0$ in \eqref{eq:def_s_B_0_E}  and \eqref{eq:q_sim_1_im_u_sim_dens_consequence_assum}. Thus $B_0$ has a simple isolated eigenvalue in $D_\eps(0)$ and we can use analytic perturbation theory to establish the lemma for the choices $T= B,\Id - C_{m^*,m}S$. Note that in either case $T=B_0 +\ord(\rho)$ due to $\norm{s-u} \lesssim \rho$ (cf. \eqref{eq:Re_u_s}). 
\end{proof}

If $z \in \HbI$ satisfies $\dens(z) + \dens(z)^{-1}\imz \leq \dens_*$ for $\dens_* \sim 1$ from Lemma \ref{lem:prop_F_small_dens} then we denote by 
$P_{s,F}$ the spectral projection corresponding to the isolated eigenvalue of $\Id - C_s F$, i.e., $P_{s,F}$ equals $P_T$ in \eqref{eq:nondegeneracy for CF} with $T=\Id - C_s F$. 
We also set $Q_{s,F}\defeq \Id - P_{s,F}$. Moreover, for such $z$, we define $\psi$ and $\sigma$ by 
\begin{equation}\label{eq:def_psi_sigma}
 \psi(z) \defeq \scalar{sf_u^2}{(\Id + F)(\Id - C_s F)^{-1}Q_{s,F}[sf_u^2]}, \qquad \sigma(z) \defeq \avg{sf_u^3}. 
\end{equation}

 In the following corollary we consider $B$ as a perturbation of $B_0$ and correspondingly expand its isolated eigenvalue and eigenvectors. In \cite[Corollary 4.8]{AltEdge} a simpler expansion has been performed in the vicinity of an edge point, i.e., where $\im m$ follows the square root behaviour from Theorem~\ref{thm:singularities_flat}. However, here we have to expand to higher order because we cover the neighbourhood of any cubic root cusp  from Theorem~\ref{thm:singularities_flat} as well.
\begin{corollary} \label{coro:eigenvector_expansion}
Let $z \in \HbI$ satisfy $\dens(z) + \dens(z)^{-1}\imz \leq \dens_*$ for $\dens_* \sim 1$ from Lemma \ref{lem:prop_F_small_dens}. 
 Let $\beta_0$ and $\beta$ be the isolated eigenvalues in $D_\eps(0)$ of $B_0$ and $B$, respectively (cf. Lemma~\ref{lem:prop_F_small_dens}). 
We denote by $P_0$ and $P$ the spectral projections corresponding to $\beta_0$ and $\beta$, respectively. 
Then with $Q_0\defeq \Id -P_0$ and $Q\defeq \Id-P$ we have 
\begin{equation} \label{eq:B_inverse_Q_norm}
\norm{B^{-1} Q}  + \normtwo{B^{-1}Q}  +\norm{B_0^{-1} Q_0}  \lesssim 1.
\end{equation}
Furthermore, we set $b_0\defeq P_0C_{q^*,q}[f_u]$ and $l_0 \defeq P_0^*C_{q,q^*}^{-1}[f_u]$. 
Then $b_0$ and $l_0$ are right and left eigenvectors of $B_0$  associated to $\beta_0$  and we have  
\begin{subequations} 
\begin{align} 
 b_0 & = C_{q^*,q}[f_u] + \ord(\dens^{-1} \imz), \qquad \qquad l_0 = C_{q,q^*}^{-1} [f_u] + \ord(\dens^{-1} \imz ), \label{eq:b_0_l_0_approx}\\ 
\beta_0 &=\frac{\imz}{\dens} \frac{\pi}{\avg{f_u^2}} +\ord(\dens^{-2}\imz[2]) = \ord(\dens^{-1} \imz)\,.& \label{eq:beta_0_approx}
\end{align}
\end{subequations}
 The definitions $b \defeq P[b_0]$ and $l \defeq P^*[l_0]$ yield right and left eigenvectors of $B$  associated to $\beta$  which satisfy  
\begin{subequations}
\label{eq:expansion of beta b l}
\begin{align}
b \,&=\, b_0 + 2\1\ii\1\dens\2C_{q^*,q}( \Id - C_s F)^{-1}Q_{s,F}[    sf_u^2]  + \ord(\dens^2+\im z)\,, \label{eq:expansion_b}
\\
l \,&=\, l_0 - 2\1\ii\1\dens\2C_{q,q^*}^{-1}( \Id -  FC_s)^{-1}Q_{s,F}^*F[sf_u^2]  + \ord(\dens^2+\im z)\,, \label{eq:expansion_l}
\\
\beta\scalar{l}{b}\,&=\, \pi \dens^{-1} \imz - 2\ii\dens\sigma + 2 \dens^2 \bigg( \psi + \frac{\sigma^2}{\avg{f_u^2}}\bigg) + \ord( \dens^3 + \imz + \dens^{-2}\imz[2])\,. 
\label{eq:expansion_beta_scalar_l_b} 
\end{align}
\end{subequations}
Moreover, we have 
\begin{equation} \label{eq:b_l_bounded} 
\norm{b} \lesssim 1, \qquad \qquad \norm{l} \lesssim 1.
\end{equation}
\end{corollary}

For later use, we record some identities here. 
From \eqref{eq:f_u_approx_f} in the proof of Lemma \ref{lem:prop_F_small_dens} with $C_s[f_u] = f_u$, we obtain the first relation in 
\begin{equation}  \label{eq:expansion_P_sF}  
 P_{s,F}  = \frac{\scalar{f_u}{\genarg}}{\avg{f_u^2}} f_u +\ord( \dens^{-1} \imz),  \quad 
 P_{s,F}^* = P_{s,F} + \ord(\dens^{-1} \imz), \quad Q_{s,F}^* = Q_{s,F} + \ord(\dens^{-1}\imz).
\end{equation}
This first relation together with $f_u = f_u^*$ implies the second and third one. 
Moreover, the definitions of $B_0$ and $Q_0$ yield
\begin{equation}
 B_0^{-1} Q_0  = C_{q^*,q} ( \Id- C_s F)^{-1} Q_{s,F} C_{q^*,q}^{-1}.  \label{eq:B_0_inverse_Q_0}
\end{equation} 
By a direct computation starting from the definition of $f_u$ in \eqref{eq:def_f_u}
and the balanced polar decomposition, $m = q^* u q$, we obtain
\begin{equation} \label{eq:f_u_qq_star}
 \avg{f_uqq^*} = \rho^{-1}\avg{\im m}=\pi.
 \end{equation}

\begin{proof} 
 The bounds in \eqref{eq:B_inverse_Q_norm} follow directly from the analytic functional calculus and Lemma \ref{lem:prop_F_small_dens}. 
The expressions \eqref{eq:b_0_l_0_approx} for the right and left eigenvectors, $b_0$ and $l_0$, corresponding to the simple isolated eigenvalue $\beta_0$, follow by simple perturbation theory from 
\begin{equation} \label{eq:B_0C_q^*q_f_u} 
B_0^*C_{q,q^*}^{-1}[f_u]\,=\, \dens^{-1} \imz[ ]\id \,, \qquad B_0C_{q^*,q}[f_u]\,=\, \ord(\dens^{-1}\imz)\,, 
\end{equation}
which in turn is a consequence of \eqref{eq:F_f_u} and $C_s[f_u]=f_u$.
For  \eqref{eq:beta_0_approx} we take the scalar product with $b_0$ on both sides of the first equation in \eqref{eq:B_0C_q^*q_f_u}. Then we use \eqref{eq:b_0_l_0_approx} and \eqref{eq:f_u_qq_star}.

Now we show \eqref{eq:expansion_b} and \eqref{eq:expansion_l}. By analytic perturbation theory of $B$ around $B_0$ we find $b = b_0 + b_1 + \ord(\dens^2)$ and $l=l_0 + l_1 + \ord(\dens^2)$ with  
$b_1 \defeq - (B_0 -\beta_0\Id)^{-1} Q_0 E[b_0]$ and $l_1 \defeq - (B_0^* - \bar{\beta}_0\Id)^{-1}Q_0^*E^*[l_0]$ (cf.~Lemma~\ref{lem:non_hermitian_perturbation_theory} with $E$ satisfying \eqref{eq:E=ord rho}). 
Here the invertibility of $B_0 -\beta_0\Id$ on the range of $Q_0$ is seen from the second part of Lemma~\ref{lem:prop_F_small_dens} with $T=B_0$. 
In fact, 
\begin{equation} \label{eq:B_0_minus_beta_0_inverse}
(B_0 - \beta_0\Id)^{-1}Q_0 = B_0^{-1}Q_0 + \ord(\beta_0). 
\end{equation} 
Furthermore, we use \eqref{eq:b_0_l_0_approx} and obtain the first equalities below: 
 \begin{subequations} \label{eq:E_b_0_Estar_l_0}
\begin{align}
 E[b_0] & = C_{q^*,q} (C_s - C_u) F[f_u] + \ord(\imz) = -2 \ii \dens C_{q^*,q} [sf_u^2] +2\dens^2C_{q^*,q}[f_u^3] +  \ord(\dens^3+ \imz), \label{eq:E_b_0}\\ 
 E^*[l_0]&  = C_{q,q^*}^{-1} F (C_s - C_{u}^*)[f_u] + \ord(\imz) = 2 \ii \dens C_{q,q^*}^{-1}F[s f_u^2] + 2\dens^2 C_{q,q^*}^{-1} F[f_u^3] + \ord(\dens^3 + \imz). \label{eq:E_star_l_0}
\end{align}
\end{subequations}
In the second equality of \eqref{eq:E_b_0}, we applied $ (C_s - C_u)[f_u] = 2( \Im u -\ii \Re u )(\Im u)f_u = -2 \ii \dens s f_u^2  + 2 \dens^2 f_u^3 +\ord(\dens^3)$, $\norm{C_s - C_u}= \ord(\dens)$ (cf. \eqref{eq:Re_u_s}) and \eqref{eq:F_f_u}. 
For the second equality in \eqref{eq:E_star_l_0}, we applied $( C_s - C_{u}^*)[f_u] = 2\ii\dens s f_u^2 + 2 \dens^2 f_u^3 + \ord(\dens^3)$. 

For the proof of \eqref{eq:expansion_beta_scalar_l_b}, we start from \eqref{eq:expansion_beta_scalar_l_b_abstract}, use $E = \ord(\dens)$ and obtain 
\begin{equation} \label{eq:beta_scalar_l_b_lemma}
 \beta \scalar{l}{b} = \beta_0 \scalar{l_0}{b_0} + \scalar{l_0}{E[b_0]} - \scalar{l_0}{EB_0(B_0-\beta_0\Id)^{-2}Q_0E[b_0]} + \ord(\dens^3).  
\end{equation}
Each of the terms on the right-hand side is computed individually. 
For the first term, we use $\scalar{l_0}{b_0} = \avg{f_u^2} + \ord(\dens^{-1} \imz)$ due to \eqref{eq:b_0_l_0_approx} and thus obtain from \eqref{eq:beta_0_approx} that 
\begin{align*}
 \beta_0 \scalar{l_0}{b_0} 
= \pi \dens^{-1} \imz + \ord(\dens^{-2}\imz[2]).  
\end{align*}
Using \eqref{eq:b_0_l_0_approx} and \eqref{eq:E_b_0_Estar_l_0} yields for the second term
\[ \scalar{l_0}{E[b_0]} = - 2\ii \dens\avg{sf_u^3} + 2\dens^2\avg{f_u^4} + \ord(\dens^3 + \imz)
= -2 \ii \dens \sigma + 2 \dens^2 \bigg(\frac{\sigma^2}{\avg{f_u^2}} + \scalar{sf_u^2}{Q_{s,F}[sf_u^2]}\bigg) + \ord(\dens^3 + \imz) 
, \] 
where we used $\Id = P_{s,F} + Q_{s,F}$ and $\scalar{sf_u^2}{P_{s,F}[sf_u^2]} = \sigma^2/\avg{f_u^2} +\ord(\dens^{-1}\imz)$ by \eqref{eq:expansion_P_sF} in the last step.

For the third term, we use \eqref{eq:beta_0_approx} and $E = \ord(\dens)$ which yields 
\begin{align*}
 \scalar{l_0}{EB_0(B_0-\beta_0\Id)^{-2}Q_0 E[b_0]} & = \scalar{E^*[l_0]}{(B_0-\beta_0\Id)^{-1}Q_0 E[b_0]} + \ord(\beta_0 \norm{E}^2) \\
 & = \scalar{E^*[l_0]}{B_0^{-1}Q_0E[b_0]} + \ord(\dens\imz) \\
& = - 4 \dens^2 \scalar{sf_u^2}{F(\Id - C_sF)^{-1} Q_{s,F}[sf_u^2]} + \ord(\dens\imz + \dens^3). 
\end{align*} 
Here, we used \eqref{eq:B_0_minus_beta_0_inverse} in the second step and \eqref{eq:E_b_0_Estar_l_0} as well as \eqref{eq:B_0_inverse_Q_0} in the last step. 
Collecting the results for the three terms in \eqref{eq:beta_scalar_l_b_lemma} and using $C_s = C_s^*$ as well as $C_s[sf_u^2] = sf_u^2$ yield \eqref{eq:expansion_beta_scalar_l_b}.

 The bounds in \eqref{eq:b_l_bounded} are directly implied by \eqref{eq:expansion_b} and \eqref{eq:expansion_l}, respectively. This finishes the proof of the corollary. 
\end{proof}

The following corollary has appeared prior to this work in \cite[Proposition 4.4]{AltEdge}. We include its short proof for the reader's convenience.

\begin{corollary}[Improved bound on $B^{-1}$]  \label{coro:B_inverse_improved_bound}
Let Assumptions \ref{assums:general} hold true on an interval $I\subset \R$ for some $\eta_* \in (0,1]$. 
Then, uniformly for all $z \in \HbI$, we have
\begin{equation} \label{eq:B_inverse_improved_bounds}
 \normtwo{B^{-1}(z)} + \norm{B^{-1}(z)}\, \lesssim \, \frac{1}{\dens(z)(\dens(z) + \abs{\sigma(z)}) + \dens(z)^{-1} \imz}. 
\end{equation}
\end{corollary}
\begin{proof}
If $\dens \geq \dens_*$ for some $\dens_* \sim 1$ then \eqref{eq:B_inverse_improved_bounds} have been shown in \eqref{eq:norm_stab_operator_m_sim_one} as $\abs{\sigma} \lesssim 1$. 
Therefore, we prove \eqref{eq:B_inverse_improved_bounds} for $\dens \leq \dens_*$ and a sufficiently small $\dens_* \sim 1$. 
By $\normtwoinf{S} \lesssim 1$ and Lemma \ref{lem:norm_two_to_inf} (ii), it suffices to show the bound for $\normtwo{\genarg}$. 
We follow the proof of \eqref{eq:norm_stab_operator_m_sim_one} until \eqref{eq:proof_norm_stab_operator_last_step}. 
Hence, for the improved bound, we have to show that 
\begin{equation} \label{eq:proof_B_inverse_improved_aux1}
 \abs{1- \normtwo{F} \avg{fC_u^*[f]}} \gtrsim \dens(\dens + \abs{\sigma}) + \dens^{-1} \imz. 
\end{equation}
 We have $\abs{1- \normtwo{F} \avg{fC_u^*[f]}} \gtrsim \max\{ 1- \normtwo{F}, \abs{1- \avg{fC_u^*[f]}}\} \gtrsim \dens^{-1}\imz + \abs{1- \avg{fC_u^*[f]}}$ by \eqref{eq:normtwo_F}. 
We continue 
\[ \abs{1- \avg{fC_u^*[f]}}=\abs{1- \avg{fu^*fu^*}} \gtrsim \avg{f \Im u f \Im u} + \abs{\avg{f\Im u f \Re u}} \gtrsim \dens^2 + \dens \abs{\sigma} + \ord(\dens^3 + \imz). \] 
Here, we used $1 \geq \avg{f\Re u f \Re u}$ due to $\normtwo{f}=1$, \eqref{eq:q_sim_1_im_u_sim_dens_consequence_assum} as well as $\avg{f\Im u f \Re u } = \dens\normtwo{f_u}^{-2} \avg{f_u^3s} + \ord(\dens^{3} + \imz)$ by \eqref{eq:f_u_approx_f} and \eqref{eq:Re_u_s}. 
By possibly shrinking $\dens_* \sim 1$, we thus obtain \eqref{eq:proof_B_inverse_improved_aux1}. 
This completes the proof of \eqref{eq:B_inverse_improved_bounds}.
\end{proof}

The remainder of this section is devoted to several results about the behaviour of $\dens(z)$, 
$\sigma(z)$ and $\psi(z)$ close to the real axis. They will be applied in the next section. 
We now prepare these results by extending $q$, $u$, $f_u$ and $s$ to the real axis.

\newcommand{\Hbtheta}{\Hb_{I_\theta,\eta_*}}
\newcommand{\Hbthclosed}{\overline{\Hb}_{I_\theta,\eta_*}}

\begin{lemma}[Extensions of $q$, $u$, $f_u$ and $s$] \label{lem:q_u_f_u_extension}
Let $I \subset \R$ be an interval, $\theta \in (0,1]$ and Assumptions \ref{assums:general} hold true on $I$ for some $\eta_* \in (0,1]$. 
We set $I_\theta \defeq \{ \tau \in I \colon \dist(\tau, \pt I) \geq \theta\}$.  Then we have
\begin{enumerate}[label=(\roman*)]
\item The functions $q$, $u$ and $f_u$ have unique uniformly $1/3$-Hölder continuous extensions to $\Hbthclosed$. 
\item The function $z \mapsto \dens(z)^{-1} \imz$ has a unique uniformly $1/3$-Hölder continuous extension to $\Hbthclosed$. In particular, we have 
\begin{equation} \label{eq:dens_reciprocal_imz_equal_zero} 
 \lim_{z \to \tau_0} \dens(z)^{-1} \imz = 0 
\end{equation}
for all $\tau_0 \in \supp \dens \cap I_\theta$. 
Moreover, for $z \in\Hbthclosed$, we have 
\[ \dist(z,\supp\dens) \gtrsim 1 \qquad \Longleftrightarrow \qquad \dens(z)^{-1}\imz \gtrsim 1. \]
\item There is a threshold $\dens_* \sim 1$ such that $s=\sign(\Re u)$ has a unique uniformly $1/3$-Hölder continuous extension to 
$\{ w \in \Hbthclosed : \dens(w) \leq \dens_*\}$. 
\end{enumerate}
\end{lemma}

\begin{proof} 
For the proof of (i), we will show below that 
\[f_m(z) \defeq \dens(z)^{-1} \Im m(z)\]
 is uniformly $1/3$-Hölder continuous on $\Hbtheta$. Indeed, this suffices to obtain the Hölder-continuity of $q$ and $u$ since their definitions in \eqref{eq:def_q_u} can be rewritten as
\begin{equation} \label{eq:def_u_q_zero_Im_m} 
\begin{aligned}
 q & = \abs{h^{-1/2} g h^{-1/2} + \ii \id }^{1/2} h^{1/2} = \Big(\dens(z)^2\id + f_m^{-1/2} g f_m^{-1} g f_m^{-1/2}\Big)^{1/4} f_m^{1/2},\\ 
 u & = \frac{ \dens(z) w}{\abs{\dens(z)w}} = \frac{\ii \dens(z)\id + f_m^{-1/2} g f_m^{-1/2}}{\abs{\ii \dens(z)\id + f_m^{-1/2} g f_m^{-1/2}}}, 
\end{aligned}
\end{equation} 
where $g = \Re m$, $h= \Im m$, $w$ is defined in \eqref{eq:def_q_u} and $z \in \Hb$ is arbitrary. 
 Since $\abs{\dens(z)w} \sim 1$ and $f_m \sim 1$ on $\Hbtheta$ by \eqref{eq:q_sim_1_im_u_sim_dens_consequence_assum} as well as \eqref{eq:Im_m_sim_avg_Im_m} 
and $m$, hence $\dens$ and $\Re m$ are Hölder-continuous on $I_\theta\times \ii[0,\infty)$ (Proposition~\ref{pro:analyticity_of_m}), 
it thus suffices to show that $f_m$ is uniformly Hölder-continuous to conclude from \eqref{eq:def_u_q_zero_Im_m} that 
$q$ and $u$ are Hölder-continuous. 
As $f_u = \dens^{-1} \Im u = (q^*)^{-1} f_m q^{-1}$, the Hölder-continuity of $f_m$, the Hölder-continuity of $q$ and 
 the upper and lower bounds on $q$ from \eqref{eq:q_sim_1_im_u_sim_dens_consequence_assum} imply that $f_u$ can be extended to a $1/3$-Hölder continuous function on $\Hbthclosed$.

Therefore, we now complete the proof of (i) by showing the $1/3$-Hölder continuity of $f_m$. To that end, we distinguish three subsets of $\Hbtheta$. 

\emph{Case 1:} On the set $\{ z \in \Hbtheta: \dens(z) \geq \dens_*\}$ for any $\dens_*\sim 1$, the uniform $1/3$-Hölder continuity of $f_m$ follows from $\dens(z) \gtrsim 1$ and the $1/3$-Hölder continuity of $m$ 
from Proposition~\ref{pro:analyticity_of_m}. 

\emph{Case 2:}
In order to analyze $f_m$ on the set $\{z \in \Hbtheta : \dens(z) \leq \dens_*\}$ for some $\dens_* \sim 1$ to be chosen later, we take the imaginary part of the Dyson equation, \eqref{eq:dyson}, at $z \in \Hb$ and obtain 
\begin{equation} \label{eq:im_m_from_dyson}
 B_*[\Im m] = \imz[ ]m^* m, \qquad B_* \defeq \Id - C_{m^*,m}S, 
\end{equation}
where $m=m(z)$. 
From $m = q^* u q$, we obtain the representation 
\[ \Id - C_{m^*,m}S = C_{q^*,q} (\Id - C_{u^*,u} F) C_{q^*,q}^{-1}. \] 
Hence, \eqref{eq:normtwo_F}, Lemma~\ref{lem:q_bounded_Im_u_sim_avg} (ii) and Lemma~\ref{lem:norm_two_to_inf} (ii) yield the 
invertibility of $B_*$ for each $z \in \HbI$ as well as 
\begin{equation} \label{eq:norm_B_star_inverse}
\normtwo{B_*^{-1}(z)} + \norm{B_*^{-1}(z)}\lesssim \frac{1}{1- \normtwo{F}} \lesssim \frac{\dens(z)}{\imz} 
\end{equation}
for all $z \in\HbI$ (compare the proof of \eqref{eq:norm_stab_operator_m_sim_one}). 
Owing to the invertibility of $B_*$,   
we conclude from \eqref{eq:im_m_from_dyson} that 
\begin{equation} \label{eq:f_m_inverse_B_star}
 f_m(z) = \pi \frac{\Im m(z)}{\avg{\Im m(z)}} = \pi \frac{B_*^{-1}[m^*m]}{\avg{B_*^{-1}[m^*m]}} 
\end{equation}
for all $z \in \Hbtheta$. 

On the set $\{ z \in \Hbtheta :\dens(z)^{-1} \imz \geq \dens_*\}$ for any $\dens_* \sim 1$, $B_*^{-1}[m^*m]$ is uniformly $1/3$-Hölder continuous due to \eqref{eq:norm_B_star_inverse} and the $1/3$-Hölder continuity 
of $m$. 
Moreover, from \eqref{eq:normtwo_F} and $\Im u \sim \dens \id$, 
we see that $1- \normtwo{F} \sim 1$ if $\dens(z)^{-1} \imz \gtrsim 1$. Hence, by Lemma~\ref{lem:inverse_positivity_preserving} in Appendix~\ref{app:positive_symmetric_map} below, 
$(\Id-C_{u^*,u}F)^{-1}$ is positivity-preserving and satisfies 
\begin{equation} \label{eq:inverse_Id_minus_C_u_star_u_F_lower_bound} 
 (\Id - C_{u^*,u}F)^{-1}[xx^*] \geq xx^* 
\end{equation}
for any $x \in \alg$. 
We conclude that $B_*^{-1} = C_{q^*,q}(\Id - C_{u^*,u}F)^{-1} C_{q^*,q}^{-1}$ is positivity-preserving. Together with \eqref{eq:q_sim_1_im_u_sim_dens_consequence_assum}, 
 \eqref{eq:inverse_Id_minus_C_u_star_u_F_lower_bound} implies $\avg{B^{-1}_*[m^*m]} \gtrsim 1$ as  $\norm{m(z)^{-1}} \lesssim 1$ by $\norm{a} \lesssim 1$, $\norm{S} \lesssim 1$ and \eqref{eq:dyson}. 
Thus, \eqref{eq:f_m_inverse_B_star} yields the uniform $1/3$-Hölder continuity of $f_m$ on $\{ z \in \Hbtheta: \dens(z)^{-1} \imz \geq \dens_*\}$ for any $\dens_* \sim 1$. 

\emph{Case 3:} We now show that $f_m$ is Hölder-continuous on $\{ z \in \Hbtheta: \dens(z) + \dens(z)^{-1} \imz \leq \dens_*\}$ for some sufficiently small $\dens_* \sim 1$. 
In fact, Lemma~\ref{lem:prop_F_small_dens} applied to $T=B_*$ yields the existence of a unique eigenvalue $\beta_*$ of $B_*$ of smallest modulus. 
Inspecting the proof of Corollary~\ref{coro:eigenvector_expansion} for $B$ reveals that this proof only used $B=B_0 + \ord(\dens)$ about $B$. Therefore, the same argument works if 
$B$ is replaced by $B_*$ since $B_* = B_0 + \ord(\dens)$ (compare the proof of Lemma~\ref{lem:prop_F_small_dens}). 
We thus find a right eigenvector $b_*$ and a left eigenvector $l_*$ of $B_*$ associated to $\beta_*$, i.e., 
\[ B_*[b_*] = \beta_* b_*, \quad \qquad (B_*)^*[l_*] = \overline{\beta_*} l_*, \] 
which satisfy 
\begin{subequations} 
 \begin{align} 
 b_* & = b_0 + \ord(\dens) = q^* f_u q  + \ord( \dens + \dens^{-1} \imz), \label{eq:expansion_b_*}\\ 
l_* & = l_0 + \ord(\dens) = q^{-1} f_u (q^*)^{-1} + \ord(\dens + \dens^{-1} \imz), \label{eq:expansion_l_*}\\
\beta_* \scalar{l_*}{b_*} & = \pi\dens^{-1} \imz + \ord(\dens + \dens^{-2} \imz[2]). \label{eq:expansion_beta_*_scalar_l_*_b_*} 
\end{align} 
\end{subequations}  
Moreover, we have 
\begin{equation} \label{eq:norm_inverse_B_star_Q_star} 
 \norm{B_*^{-1} Q_*} + \normtwo{B^{-1}_* Q_*} \lesssim 1,  
\end{equation}
where $Q_*$ denotes the spectral projection of $B_*$ to the complement of the spectral subspace of $\beta_*$.

Therefore, as $\beta_* \neq 0$ (cf. \eqref{eq:norm_B_star_inverse}) if $\imz >0$, we obtain 
\[ \Im m = \imz[ ] B_*^{-1}[m^* m] = \imz[ ] \bigg ( \beta_*^{-1} \frac{\scalar{l_*}{m^*m}}{\scalar{l_*}{b_*}} b_* + B_*^{-1} Q_* [m^* m]\bigg ). \] 
Consequently, as $\Im m >0$, we have
\begin{equation} \label{eq:Im_m_avg_Im_m_rep}
 \frac{\Im m }{\avg{\Im m}} = \frac{\scalar{l_*}{m^*m} b_* + \beta_* \scalar{l_*}{b_*} B_*^{-1} Q_* [m^*m]}{\scalar{l_*}{m^*m}\avg{b_*} + \beta_* \scalar{l_*}{b_*} \avg{B_*^{-1} Q_* [m^* m]}}, 
\end{equation}
which together with \eqref{eq:f_m_inverse_B_star} shows that $f_m$ is uniformly $1/3$-Hölder continuous on $\{ z \in\Hbtheta: \dens(z) + \dens(z)^{-1} \imz \leq \dens_*\}$. Here, we used that $B_*$ and, thus, $\beta_*$, $l_*$, $b_*$ and $B_*^{-1}Q_*$ are $1/3$-Hölder 
continuous and the denominator in \eqref{eq:Im_m_avg_Im_m_rep} is $\gtrsim 1$ due to 
\[ \begin{aligned} 
\scalar{l_*}{m^*m} & = \avg{q^{-1} f_u (q^*)^{-1} q^*u^* q q^* u q} + \ord( \dens + \dens^{-1} \imz) \\ 
 & = \dens^{-1} \Im \avg{q^* u u u^* q} + \ord( \dens + \dens^{-1} \imz) = \pi + \ord(\dens + \dens^{-1}\imz)
\end{aligned}\] 
by \eqref{eq:expansion_b_*} and \eqref{eq:expansion_l_*} as well as $\avg{b_*} = \pi + \ord(\dens + \dens^{-1} \imz)$ by \eqref{eq:f_u_qq_star}. 
Here, we also used \eqref{eq:expansion_beta_*_scalar_l_*_b_*} and \eqref{eq:norm_inverse_B_star_Q_star}. 
This completes the proof of~(i). 

For the proof of (ii), we multiply \eqref{eq:im_m_from_dyson} by $\dens(z)^{-1} (m^*m)^{-1}$ which yields 
\[ \dens(z)^{-1} \imz = (m^*m)^{-1} B_*[f_m]. \] 
Owing to $m^*m \geq \norm{m^{-1}}^{-2} \gtrsim 1$ as well as the $1/3$-Hölder continuity of $m$, $B_*$ and $f_m$, we obtain the same regularity for $z \mapsto \dens(z)^{-1} \imz$. 
Since $\lim_{\eta \downarrow 0} \dens(\tau + \ii \eta)^{-1} \eta = 0$ for $\tau \in \supp \dens \cap I_\theta$ satisfying $\dens(\tau) >0$, the continuity of $\dens(z)^{-1} \imz$ 
directly implies \eqref{eq:dens_reciprocal_imz_equal_zero}. 
If $\dist(z,\supp\dens) \gtrsim 1$ then $\dens(z)^{-1} \imz \gtrsim 1$ as $\dens(z) \leq \imz/\dist(z,\supp\dens)^2$ which can be seen by applying $\avg{\genarg}$ to the 
second bound in \eqref{eq:stieltjes_bound_m}. Conversely, if $\dist(z,\supp\dens) \lesssim 1$ then the Hölder-continuity of $\dens(z)^{-1} \imz$ and \eqref{eq:dens_reciprocal_imz_equal_zero} 
imply $\dens(z)^{-1} \imz \lesssim 1$. 

We now turn to the proof of (iii). 
Owing to the first relation in \eqref{eq:Re_u_s}, there is $\dens_* \sim 1$ such that $\abs{\Re u} \geq \frac{1}{2}\id$ if $z \in \Hbtheta$ satisfies $\dens(z) \leq \dens_*$. 
Therefore, we find a smooth function $\varphi\colon \R \to [-1, 1]$ such that $\varphi(t) = 1$ for all $t \in [1/2,\infty)$, $\varphi(t) = -1$ for all $t \in (-\infty,-1/2]$ and 
$s(z) = \sign(\Re u(z)) = \varphi(\Re u(z))$ for all $z \in \Hbtheta$ satisfying $\dens(z) \leq \dens_*$. 
Since $\varphi$ is smooth, we conclude that $\varphi$ is an \emph{operator Lipschitz function} \cite[Theorem 1.6.1]{Aleksandrov2016}, 
i.e., $\norm{\varphi(x) - \varphi(y)} \leq C \norm{x-y}$ for all self-adjoint $x, y \in \alg$. 
Hence, we conclude
\[\norm{s(z_1) - s(z_2)} = \norm{\varphi(\Re u(z_1)) - \varphi(\Re u (z_2))} \lesssim \norm{z_1 - z_2}^{1/3},  \] 
where we used that $\varphi$ is operator Lipschitz and $u$ is $1/3$-Hölder continuous in the last step. 
This completes the proof of Lemma \ref{lem:q_u_f_u_extension}.
\end{proof}

\begin{lemma}[Properties of $\psi$ and $\sigma$] \label{lem:stability_cubic_equation}
Let $I \subset \R$ be an interval and $\theta \in (0,1]$.
If $m$ satisfies Assumptions \ref{assums:general} on $I$ for some $\eta_* \in (0,1]$ 
then there is a threshold $\dens_* \sim 1$ such that, with 
\[\Hb_\rm{small} \defeq \{ z \in \Hbtheta \colon \dens(z) + \dens(z)^{-1} \imz \leq \dens_*\}, \]
we have 
\begin{enumerate}[label=(\roman*)] 
\item The functions $\sigma$ and $\psi$ defined in \eqref{eq:def_psi_sigma} have unique uniformly $1/3$-Hölder continuous extensions to $\{ z \in \Hbthclosed: \dens(z) \leq \dens_*\}$ 
and $\overline{\Hb}_\rm{small}$, respectively. 
\item 
Uniformly for all $z \in \overline{\Hb}_\rm{small}$, we have
\begin{equation} \label{eq:psi_plus_sigma_sim_1}
 \psi(z) + \sigma(z)^2 \sim 1.
\end{equation}
\end{enumerate}
\end{lemma} 

\begin{proof} 
For the proof of (i), we choose $\dens_*\sim 1$ so small that all parts of Lemma \ref{lem:q_u_f_u_extension} are applicable. 
Thus, Lemma~\ref{lem:q_u_f_u_extension} and $\sigma = \avg{sf_u^3}$ yield (i) for $\sigma$. 
Similarly, since $q$ is now defined on $\Hbthclosed$, we can define $F$ via \eqref{eq:def_F} on this set as well. 
Moreover, owing to the uniform $1/3$-Hölder continuity of $q$ from Lemma \ref{lem:q_u_f_u_extension}, $F$ is uniformly $1/3$-Hölder continuous on $\Hbthclosed$. 
Hence, using Lemma~\ref{lem:prop_F_small_dens} for $T = \Id - C_s F$, the Hölder-continuity of $s$ and $f_u$, the function $\psi$ has a unique $1/3$-Hölder continuous extension to $\overline{\Hb}_\mathrm{small}$.
This completes the proof of (i) for~$\psi$. 

We now turn to the proof of (ii). In fact, we will show \eqref{eq:psi_plus_sigma_sim_1} only on $\{ w \in \Hbtheta  \colon \dens(w) + \dens(w)^{-1} \Im w \leq \dens_*\}$, where $\dens_* \sim 1$ 
is chosen small enough such that Lemma \ref{lem:prop_F_small_dens} is applicable. By the continuity of $\sigma$ and $\psi$, 
the bound \eqref{eq:psi_plus_sigma_sim_1} immediately extends to the closure of this set.
Instead of \eqref{eq:psi_plus_sigma_sim_1}, we will prove that 
\begin{equation} \label{eq:psi_plus_sigma_sim_1_general}
 \scalar{x}{(\Id + F)(\Id - C_s F)^{-1} Q_{s,F}[x]} + \scalar{f_u}{x}^2 \sim \normtwo{x}^2 
\end{equation}
for all $x \in \alg$ satisfying $C_s[x] = x$ and $x=x^*$. Since these conditions are satisfied by $x = sf_u^2$, \eqref{eq:psi_plus_sigma_sim_1_general} immediately implies \eqref{eq:psi_plus_sigma_sim_1}.
In fact, the upper bound in \eqref{eq:psi_plus_sigma_sim_1_general} follows from $\normtwo{(\Id - C_sF)^{-1}Q_{s,F}} \lesssim 1$ by Lemma \ref{lem:prop_F_small_dens}, $\normtwo{F} \leq 1$ 
and $f_u \sim 1$ due to \eqref{eq:f_u_sim_1}. 

From $C_s[x] = x$, we conclude 
\begin{equation} \label{eq:psi_x_first_step}
\begin{aligned}
\scalar{x}{(\Id + F)(\Id- C_s F)^{-1} Q_{s,F}[x]} & = \scalar{x}{( \Id + C_sF)(\Id -C_s F)^{-1}Q_{s,F}[x]} \\ 
 & =\scalar{x}{((C_s F -\Id ) + 2\Id)  (\Id- C_s F)^{-1}  Q_{s,F}[x]} \\ 
 & = \scalar{x}{(-\Id +2 ( \Id - C_s F)^{-1}) Q_{s,F}[x]}. 
\end{aligned}
\end{equation}
Using \eqref{eq:expansion_P_sF} and $C_s[f_u] = f_u$, we see that 
\begin{equation} \label{eq:expansion_C_s_P}
 C_s P_{s,F}[x] = P_{s,F}[x] + \ord(\dens^{-1} \imz), \qquad C_s Q_{s,F}[x] = Q_{s,F}[x] + \ord(\dens^{-1}\imz)  
\end{equation}
for $x \in \alg$ satisfying $C_s[x] = x$. 
When applied to \eqref{eq:psi_x_first_step}, the expansion \eqref{eq:expansion_C_s_P} and $(\Id - FC_s)^{-1} = C_s (\Id - C_s F)^{-1}C_s$ yield 
\begin{equation} \label{eq:psi_x_second_step}
\begin{aligned}
&\scalar{x}{(\Id + F)(\Id- C_s F)^{-1}) Q_{s,F}[x]} \\ 
&\hspace{2.8cm} = \scalar{Q_{s,F}[x]}{( - \Id + (\Id - C_s F)^{-1} + (\Id - FC_s)^{-1} )Q_{s,F}[x]} + \ord(\normtwo{x}^2\dens^{-1}\imz) \\ 
&\hspace{2.8cm} = \scalar{Q_{s,F}[x]}{(\Id - F C_s)^{-1}(\Id - F^2) (\Id - C_sF)^{-1} Q_{s,F}[x]} + \ord(\normtwo{x}^2\dens^{-1}\imz) \\ 
&\hspace{2.8cm} = \scalar{(\Id - C_s F)^{-1}Q_{s,F}[x]}{Q_f(\Id - F^2)Q_f (\Id - C_sF)^{-1} Q_{s,F}[x]} + \ord(\normtwo{x}^2\dens^{-1}\imz) \\ 
&\hspace{2.8cm} \gtrsim \normtwo{Q_f(\Id -C_sF)^{-1} Q_{s,F}[x]}^2 + \ord(\normtwo{x}^2\dens^{-1} \imz) \\ 
&\hspace{2.8cm} \gtrsim \normtwo{Q_{s,F}[x]}^2 + \ord(\normtwo{x}^2\dens^{-1} \imz). 
\end{aligned}
\end{equation}
Here, in the first step, we also used the second and third relation in \eqref{eq:expansion_P_sF}. 
In the third step, we then defined the orthogonal projections $P_f \defeq \scalar{f}{\cdot} f$ and $Q_f \defeq \Id - P_f$, where $Ff = \normtwo{F}f$ (cf.~Assumptions~\ref{assums:general}~(ii)),
 and inserted $Q_f$ using  
\begin{equation}  \label{eq:P_f_times_Q_sF}
 P_f Q_{s,F} = \ord(\dens^{-1} \imz)  
\end{equation}
which follows from \eqref{eq:f_u_approx_f} and \eqref{eq:expansion_P_sF}. We also used that $Q_{s,F}$ commutes with $(\Id - C_sF)^{-1}$. The fourth step is a consequence of \eqref{eq:spec_F} and \eqref{eq:eigenvector_F_sim_1}. 
In the last step, we employed $Q_f Q_{s,F} = Q_{s,F} + \ord(\dens^{-1} \imz)$ by \eqref{eq:P_f_times_Q_sF} and $\normtwo{\Id - C_s F} \leq 2$.  

By \eqref{eq:expansion_P_sF}, we have $\normtwo{P_{s,F}[x]}^2 = \scalar{f_u}{x}^2 + \ord(\normtwo{x}^2 \dens^{-1}\imz)$ if $x =x^*$.
Combining this observation with \eqref{eq:psi_x_second_step} proves \eqref{eq:psi_plus_sigma_sim_1_general} up to terms of order $\ord(\normtwo{x}^2 \dens^{-1} \imz)$. 
Hence, possibly shrinking $\dens_* \sim 1$ and requiring $\dens(z)^{-1} \imz \leq \dens_*$ complete the proof of the lemma. 
\end{proof}

\begin{remark}[Auxiliary quantities as functions of $m$]  \label{rem:derived_quantities_cont_function_of_m} 
Inspecting the proofs of Lemma \ref{lem:q_u_f_u_extension} and Lemma \ref{lem:stability_cubic_equation} reveals that $q$, $u$, $f_u$ and $s$ as well as $\sigma$ and $\psi$ 
are Lipschitz-continuous functions of $m$. More precisely, we have the following statements: 
\begin{enumerate}[label=(\roman*)]
\item 
Let $c_1, c_2, c_3>0$ satisfy $c_1 < c_2$
and $\mathcal{M}^{(1)}=\mathcal{M}^{(1)}(c_1, c_2,c_3) \subset \alg$ be a nonempty subset of $\alg$ satisfying that 
\begin{equation} \label{eq:conditions_cal_M_1}
 \Im m_1 \in \algpos, \qquad c_1 \avg{\Im m_1}\id \leq \Im m_1 \leq c_2 \avg{\Im m_1} \id , \qquad 
\norma{ \frac{\Im m_1}{\avg{\Im m_1}} - \frac{\Im m_2}{\avg{\Im m_2}}} \leq c_3 \norm{m_1 - m_2} 
\end{equation}
hold true for all $m_1, m_2 \in \mathcal{M}^{(1)}$. Then $q$, $u$ and $f_u$ are uniformly Lipschitz-continuous functions of 
$m$ on $\mathcal{M}^{(1)}$. 
\item For some $\dens_*>0$, let $\mathcal{M}^{(2)} = \mathcal{M}^{(2)}(c_1,c_2,c_3,\dens_*)\subset \alg$ be a 
subset of $\alg$ satisfying \eqref{eq:conditions_cal_M_1} for all $m_1, m_2 \in \mathcal{M}^{(2)}$ 
and $\avg{\Im m} \leq \pi \dens_*$ for all $m \in \mathcal{M}^{(2)}$.
Then there is a (small) $\dens_* \sim 1$ such that $s$ and $\sigma$ are uniformly Lipschitz-continuous functions 
of $m$ on $\mathcal{M}^{(2)} \subset \alg$. 
\item 
Fix $c_4>0$. 
 Let $\mathcal{M}^{(3)}=\mathcal{M}^{(3)}(c_1, c_2, c_3, c_4, \dens_*)$ be a subset of a set $\mathcal{M}^{(2)}$ from (ii) with $\dens_*\sim 1$ chosen as in (ii) such that, 
for any $m \in \mathcal{M}^{(3)}$, the operator $\Id - C_{s(m)}F(m)$ has a unique eigenvalue of smallest modulus 
and this eigenvalue is simple (recall  that $F =C_{q,q^*}S C_{q^*,q}$
is a function of $m$ via $q=q(m)$).
Let $Q_m$ denote the spectral projection of $\Id - C_{s(m)}F(m)$ onto the complement of this eigenvalue.
Moreover, we require that
\begin{equation} \label{eq:conditions_cal_M_3}
\norma{(\Id-C_{s(m_1)}F(m_1))^{-1} Q_{m_1} - (\Id - C_{s(m_2)}F(m_2))^{-1} Q_{m_2}} \leq c_4 \norm{m_1 - m_2}
\end{equation}
holds true for any $m_1$, $m_2 \in \mathcal{M}^{(3)}$.  
Then $\psi$ is a uniformly Lipschitz-continuous function of $m$ on $\mathcal{M}^{(3)}$. 
\end{enumerate}
We always consider $\mathcal{M}^{(i)}$, $i=1,2,3$, with the metric induced by the norm $\norm{\genarg}$ on $\alg$.
The constants in the Lipschitz-continuity estimates
as well as $\dens_*$ given in (ii) 
only depend on the control parameters $c_1$, $c_2$, $c_3$ and $c_4$.
\end{remark} 

The careful analysis of the operator $B$ and its inverse allows for the precise bounds on the derivatives of $m$ in the following lemma. 

\begin{lemma}[Derivatives of $m$]  \label{lem:derivatives_m} 
Let $I \subset \R$ be an open interval and $\theta \in (0,1]$. 
If Assumptions~\ref{assums:general} hold true on $I$ for some $\eta_* \in (0,1]$ then
there is $C \sim 1$ such that 
\[ \norm{\pt_z^k m(\tau)} \lesssim \frac{C^kk!}{\dens(\tau)^{2k -1}(\dens(\tau) + \abs{\sigma(\tau)})^{k}} \] 
uniformly for all $\tau \in I_\theta$ satisfying $\dens(\tau)>0$ and all $k \in \N$ satisfying $k \geq 1$. 
Here, we set $\abs{\sigma(\tau)} \defeq 0$ if $\dens(\tau) > \dens_*$ with $\dens_*$ as in Lemma~\ref{lem:stability_cubic_equation}. 
\end{lemma}

\begin{proof}
To indicate the mechanism, we first prove that, for all $\tau \in I_\theta$ satisfying $\dens(\tau)>0$, we have 
\begin{equation} \label{eq:derivative_m_bound_first_three_deriviatives}
  \norm{\pt_z m(\tau)} \lesssim \dens^{-1}(\dens + \abs{\sigma})^{-1},\qquad \norm{\pt_z^2 m(\tau)} \lesssim \dens^{-3}(\dens + \abs{\sigma})^{-2}, 
\qquad \norm{\pt_z^{3} m(\tau)} \lesssim \dens^{-5}(\dens + \abs{\sigma})^{-3}, 
\end{equation}
where $\dens \defeq \dens(\tau)$ and $\sigma \defeq \sigma(\tau)$. 

Since $\dens(\tau)>0$, $m$ is real analytic around $\tau$ by Proposition~\ref{pro:analyticity_of_m} and we can differentiate the Dyson equation, \eqref{eq:dyson}, with respect 
to $z$ and evaluate at $z = \tau$. 
Differentiating \eqref{eq:dyson} iteratively yields 
\begin{equation} \label{eq:derivatives_m}
 \begin{aligned} 
 B[\pt_z m] & = m^2, \qquad \qquad B[\pt_z^2 m] = 2 (\pt_z m) m^{-1} (\pt_z m), \\ 
B[\pt_z^3 m]  & = - 6 (\pt_z m) m^{-1} (\pt_z m) m^{-1} (\pt_z m) + 3 (\pt_z^2 m)m^{-1} (\pt_z m) + 3 (\pt_z m)m^{-1} (\pt_z^2 m)
  \end{aligned} 
\end{equation} 
where $B = \Id - C_mS$ and $m \defeq m(\tau)$. 
Since $\dens(\tau)>0$, $B$ is invertible by \eqref{eq:B_inverse_improved_bounds}, \eqref{eq:dens_reciprocal_imz_equal_zero} and the $1/3$-Hölder continuity of $m$ by Proposition~\ref{pro:analyticity_of_m}. 

We set $\dens \defeq \dens(\tau)$. 
If $\dens > \dens_*$ for some $\dens_* \sim 1$ then \eqref{eq:derivative_m_bound_first_three_deriviatives} follows trivially from \eqref{eq:derivatives_m}, $\norm{B^{-1}} \lesssim 1$ by \eqref{eq:B_inverse_improved_bounds} and 
$\norm{m} + \norm{m^{-1}} \lesssim 1$. 

We now prove \eqref{eq:derivative_m_bound_first_three_deriviatives} for $\dens \leq \dens_*$ and some sufficiently small $\dens_* \sim 1$. 
Under this assumption, Lemma~\ref{lem:prop_F_small_dens} and Corollary~\ref{coro:eigenvector_expansion} are applicable. 
In the remainder of this proof, the eigenvalue $\beta$, the eigenvectors $l$ and $b$ as well as the spectral projections $P$ and $Q$ are understood to be evaluated at $\tau$. 
We will now estimate the image of $B^{-1}$ applied to the right-hand sides of \eqref{eq:derivatives_m} in order to prove \eqref{eq:derivative_m_bound_first_three_deriviatives}. 

Inserting $P+Q=\Id$ on the right-hand side of the first identity in \eqref{eq:derivatives_m}, inverting $B$ and using 
\[ P = \frac{\scalar{l}{\genarg}}{\scalar{l}{b}} b \] 
 as well as $B^{-1}[b] = \beta^{-1}b$  
yield 
\begin{equation} \label{eq:pt_z_m_expansion}
 \pt_z m = \frac{\scalar{l}{m^2}}{\beta \scalar{l}{b}} b + B^{-1} Q[m^2]. 
\end{equation}
We will now estimate $\scalar{l}{m^2}$ and $\beta \scalar{l}{b}$. 
From $m = q^* s q + \ord(\dens)$ by \eqref{eq:Re_u_s}, \eqref{eq:b_0_l_0_approx}, \eqref{eq:expansion_l} and \eqref{eq:dens_reciprocal_imz_equal_zero}, we obtain 
\begin{equation}
 \scalar{l}{m^2} = \avg{f_u s qq^* s}  + \ord(\dens) = \pi + \ord(\dens), 
\end{equation}
where we used $sf_u s = f_u s^2 = f_u$ and \eqref{eq:f_u_qq_star} in the last step. 

From \eqref{eq:expansion_beta_scalar_l_b} and \eqref{eq:dens_reciprocal_imz_equal_zero}, we conclude 
\begin{equation} \label{eq:beta_scalar_l_b_aux_derivative_m} 
 \beta \scalar{l}{b} = - 2\ii \dens \sigma + \dens^2 \bigg( \psi + \frac{\sigma^2}{\avg{f_u^2}} \bigg) + \ord(\dens^{3}). 
\end{equation}
Here and in the remainder of the proof, $\sigma$, $\psi$, $f_u$, $q$ and $s$ are understood to be evaluated at $\tau$. 

Since $\sigma$ and $\psi$ are real, we conclude $\abs{\beta \scalar{l}{b}} \sim \dens(\dens+\abs{\sigma})$ for $\dens_* \sim 1$ sufficiently small. 
As $\norm{B^{-1}Q} \lesssim 1$ and $\norm{b} \lesssim 1$, we thus obtain $\norm{\pt_z m} \lesssim \dens^{-1}(\dens + \abs{\sigma})^{-1}$ from \eqref{eq:pt_z_m_expansion}. 

Using \eqref{eq:derivatives_m}, \eqref{eq:pt_z_m_expansion}, $\norm{\pt_z m} \lesssim \dens^{-1}(\dens + \abs{\sigma})^{-1}$ and $\norm{B^{-1}} \lesssim \dens^{-1}(\dens + \abs{\sigma})^{-1}$  
by Corollary~\ref{coro:B_inverse_improved_bound} yield 
\begin{equation} \label{eq:pt_z^2_m_expansion}
 \pt_z^2 m = 2 \frac{\scalar{l}{m^2}^2 \scalar{l}{bm^{-1}b}}{(\beta \scalar{l}{b})^3} b + \ord(\dens^{-2}(\dens + \abs{\sigma})^{-2}) = \ord(\dens^{-3}(\dens + \abs{\sigma})^{-2}). 
\end{equation}
Here, in the last step, we used $\norm{b} \lesssim 1$ and $\abs{\scalar{l}{bm^{-1}b}} \lesssim \abs{\sigma} + \dens$ due to the expansion 
\begin{equation}\label{eq:derivatives_m_aux3}
 \scalar{l}{bm^{-1} b } = \avg{q^{-1} f_u (q^*)^{-1} q^* f_u q q^{-1} s (q^*)^{-1} q^*f_u q} + \ord(\dens) = \sigma + \ord(\dens) 
\end{equation}
as well as $\abs{\beta \scalar{l}{b}} \sim \dens(\dens + \abs{\sigma})$ and $\scalar{l}{m^2} = \ord(1)$. 
The proof of \eqref{eq:derivatives_m_aux3} is a consequence of \eqref{eq:b_0_l_0_approx}, \eqref{eq:expansion_b}, \eqref{eq:expansion_l}, \eqref{eq:dens_reciprocal_imz_equal_zero},
$m^{-1} = q^{-1}s(q^*)^{-1} + \ord(\dens)$ by \eqref{eq:Re_u_s} as well as $q \sim 1$. 

Similarly, owing to \eqref{eq:derivatives_m}, \eqref{eq:pt_z_m_expansion} and \eqref{eq:pt_z^2_m_expansion},  we obtain 
\[ \pt_z^3 m = 12 \frac{\scalar{l}{m^2}^3 \scalar{l}{bm^{-1}b}^2}{(\beta \scalar{l}{b})^5} b + \ord(\dens^{-5}(\dens + \abs{\sigma})^{-3}) = \ord(\dens^{-5}(\dens + \abs{\sigma})^{-3}). \] 

We now estimate $\pt_z^k m(z)$ for $k > 3$. To that end, we will fix a parameter $\alpha >1$ and prove that there are $\dens_* \sim 1$, $C_1 \sim_\alpha 1$ and $C_2 \sim_\alpha 1$ such that, for $k \in \N$, we have 
\begin{equation} \label{eq:derivative_m_aux_statement} 
 m^{(k)} \defeq \pt^k_z m =  \beta_k b + q_k, 
\end{equation}
where $m=m(\tau)$ for $\tau \in I_\theta$ satisfying $\dens\defeq \dens(\tau) \leq \dens_*$ and  $\beta_k \in \C$ and $q_k \in \ran Q$ satisfy 
\begin{equation} \label{eq:bound_beta_k_q_k_derivative_m}
 \abs{\beta_k} \leq \frac{k!C_1 C_2^{k-1}} {k^\alpha} \dens^{-2k + 1}(\dens + \abs{\sigma})^{-k}, \qquad \qquad  \norm{q_k} \leq \frac{k!C_1 C_2^{k-1}}{k^\alpha} \dens^{-2k + 2}(\dens + \abs{\sigma})^{-k}. 
\end{equation}
Here, $\sim_\alpha$ indicates that the constants in the definition of the comparison relation $\sim$ will depend on $\alpha$. 

Before we prove \eqref{eq:derivative_m_aux_statement} below, we note two auxiliary statements. 
First, as $\pt_z m^{-1} = - m^{-1} (\pt_z m) m^{-1}$ it is easy to check the following version of the usual Leibniz-rule: 
\begin{equation} \label{eq:derivative_k_m_inverse}
 \pt_z^k m^{-1} = \sum_{n=1}^k\, \sum_{\substack{a_1 + \ldots + a_n = k\\1 \leq a_i \leq k}} \frac{k!}{a_1!\ldots a_n!}\, (-1)^{n} \, m^{-1} m^{(a_1)} m^{-1} m^{(a_2)} \ldots m^{-1} m^{(a_n)} m^{-1} 
\end{equation}
for any $k \in \N$. Here, in the sum over $a_1 + \ldots + a_n = k$, the order of $a_1, \ldots, a_n$ has to be taken into account since $m^{-1}$ and $m^{(a)}$ do not commute in general. 

Second, we also have the following auxiliary bound. 
For all $k \in \N$, $n \in \N$ with $n \leq k$ and $\alpha >1$, we have 
\begin{equation} \label{eq:bound_riemann_zeta} 
 \sum_{\substack{a_1 + \ldots + a_n = k\\1 \leq a_i \leq k}} \frac{1}{a_1^\alpha\cdots a_n^\alpha} \leq \frac{(2^{\alpha + 1}\zeta(\alpha))^{n-1}}{k^\alpha}, 
\end{equation}
where $\zeta(\alpha)=\sum_{n=1}^\infty n^{-\alpha}$ is Riemann's zeta function. The bound in \eqref{eq:bound_riemann_zeta} can be proven by induction. 

We now show \eqref{eq:derivative_m_aux_statement} and \eqref{eq:bound_beta_k_q_k_derivative_m} by induction on $k$. The initial step of the induction with $k=1$ has been established in \eqref{eq:pt_z_m_expansion} with $\beta_1 = \scalar{l}{m^2}/(\beta\scalar{l}{b})$, 
 $q_1 = B^{-1}Q[m^2]$ and some sufficiently large $C_1 \sim 1$. 
Next, we establish the induction step by proving \eqref{eq:derivative_m_aux_statement} and \eqref{eq:bound_beta_k_q_k_derivative_m} under the assumption that they hold true for all 
derivatives of lower order. From the induction hypothesis, we conclude 
\begin{equation} \label{eq:conseq_ind_hypo_derivative_m} 
 \norm{m^{(a)}} \leq \frac{k!C_1 C_2^{a-1}}{k^\alpha} \frac{\norm{b} + \dens }{\dens^{2a-1}(\dens + \abs{\sigma})^a}  
\end{equation}
for all $ a \in \N$ satisfying $1 \leq a \leq k-1$.

For $k \geq 2$, we differentiate \eqref{eq:dyson} $k$-times and obtain 
\begin{equation} \label{eq:derivative_m_general_formula_B}
 B[\pt^k_z m] = r_k \defeq \pt_z^k m + m \Big(\pt_z^k m^{-1} \Big) m.  
\end{equation}
By separating the contributions for $n=1$ and $n \geq 2$ in \eqref{eq:derivative_k_m_inverse}, we conclude 
\begin{equation} \label{eq:decomposition_r_k} 
 r_k = \sum_{n=3}^k \,\sum_{\substack{a_1 + \ldots + a_n = k\\1 \leq a_i < k-1}} \frac{k!}{a_1!\ldots a_n!}\, (-1)^{n} \, m^{(a_1)} m^{-1} \ldots m^{-1} m^{(a_n)}
+ \sum_{a=1}^{k-1} \frac{k!}{a! (k-a)!} m^{(a)} m^{-1} m^{(k-a)}. 
\end{equation}
Since $n$ is at least 3 in the first sum, we obtain from \eqref{eq:conseq_ind_hypo_derivative_m} and \eqref{eq:bound_riemann_zeta} that 
\begin{equation} \label{eq:r_k_geq_3} 
 \sum_{n=3}^k \,\sum_{\substack{a_1 + \ldots + a_n = k\\1 \leq a_i < k-1}} \frac{k!}{a_1!\ldots a_n!}\norm{m^{(a_1)} m^{-1} \ldots m^{-1} m^{(a_n)}} \leq \frac{k!}{k^\alpha}
\frac{\norm{b} + \dens} {\dens^{2k-3}(\dens+\abs{\sigma})^k} \sum_{n=3}^k C_1^n M_\alpha^{n-1} C_2^{k-n}, 
\end{equation} 
where $M_\alpha \defeq 2^{\alpha+2}\zeta(\alpha)\norm{m^{-1}}( \norm{b} + \dens)$. 
A similar argument yields 
\[ \sum_{a=1}^{k-1} \frac{k!}{a! (k-a)!} \norm{m^{(a)} m^{-1} m^{(k-a)}} \leq \frac{k!}{k^\alpha}\frac{\norm{b} + \dens}{\dens^{2k - 2}(\dens+ \abs{\sigma})^{k}} 
C_1^2 M_\alpha C_2^{k-2}. \] 
Thus, we choose $C_2 \geq 2 M_\alpha C_1$ and conclude 
\[ \norm{r_k} \leq \frac{k!}{k^\alpha} \frac{\norm{b} + \dens}{\dens^{2k-2}(\dens + \abs{\sigma})^{k}}\frac{M_\alpha C_1^2 C_2^k}{C_2^2( 1 - M_\alpha C_1/C_2)}. \]
Therefore, we obtain the bound on $\norm{q_k}$ in \eqref{eq:bound_beta_k_q_k_derivative_m} for $C_2 \sim 1$ sufficiently large since $q_k = Q[\pt_z^k m] = B^{-1}Q[r_k]$ 
and $\norm{B^{-1}Q} \lesssim 1$. 

Moreover, $\beta_k = \scalar{l}{r_k}/(\beta\scalar{l}{b})$. Hence, by using the decomposition of $r_k$ in \eqref{eq:decomposition_r_k} and \eqref{eq:r_k_geq_3}, we obtain 
\[ \abs{\beta_k} \leq  
\frac{k!C_1 C_2^{k-1}}{k^\alpha} \frac{\norm{b} +\dens}{\dens^{2k-1}(\dens + \abs{\sigma})^k} \frac{\norm{l}\dens^2}{\abs{\beta\scalar{l}{b}}} \frac{C_1^2 M_\alpha^2}{C_2^2(1-M_\alpha C_1/C_2)}
+  \sum_{a=1}^{k-1} \frac{k!}{a! (k-a)!} \frac{\abs{\scalar{l}{m^{(a)} m^{-1} m^{(k-a)}}}}{\abs{\beta \scalar{l}{b}}} \]
We use \eqref{eq:derivative_m_aux_statement} for $m^{(a)}$ and $m^{(k-a)}$ in the argument of the last sum, which yields 
\[\begin{aligned} 
 \frac{1}{a! (k-a)!} \frac{\abs{\scalar{l}{m^{(a)} m^{-1} m^{(k-a)}}}}{\abs{\beta\scalar{l}{b}}} 
\leq \, &  \frac{\abs{\beta_a}}{a!}\frac{\abs{\beta_{k-a}}}{(k-a)!} \frac{\abs{\scalar{l}{bm^{-1}b}}}{\abs{\beta\scalar{l}{b}}} +\frac{C_1^2 C_2^{k-2}}{a^\alpha(k-a)^\alpha \dens^{2k-1}(\dens+ \abs{\sigma})^k} 
\frac{\dens^{2} \norm{l}\norm{m^{-1}}} {\abs{\beta\scalar{l}{b}}} ( 2 \norm{b} + \dens) \\ 
 \leq \, &\frac{C_1^2 C_2^{k-2}}{a^\alpha(k-a)^\alpha \dens^{2k-1}(\dens + \abs{\sigma})^k} \frac{\dens(\dens + \abs{\sigma})}{\abs{\beta \scalar{l}{b}}} \bigg( \frac{\abs{\scalar{l}{bm^{-1}b}}}{\dens + \abs{\sigma}} + \norm{l}\norm{m^{-1}} ( 2 \norm{b} + \dens)\bigg)
\end{aligned} \]
Here, we applied \eqref{eq:bound_beta_k_q_k_derivative_m} to estimate $q_a$ and $q_{k-a}$ as well as $\beta_a$ and $\beta_{k-a}$. 
Since $\abs{\beta \scalar{l}{b}} \sim \dens(\dens + \abs{\sigma})$ as shown below \eqref{eq:beta_scalar_l_b_aux_derivative_m} and 
$\abs{\scalar{l}{bm^{-1}b}}\lesssim \abs{\sigma} + \dens$ due to \eqref{eq:derivatives_m_aux3}, we obtain the bound on $\abs{\beta_k}$ in \eqref{eq:bound_beta_k_q_k_derivative_m} by using \eqref{eq:bound_riemann_zeta} 
to perform the summation over $a$. 
This completes the induction argument, which yields \eqref{eq:derivative_m_aux_statement} and \eqref{eq:bound_beta_k_q_k_derivative_m} for all $k \in \N$ by possibly increasing $C_2 \sim 1$.  
By choosing, say, $\alpha = 2$, we immediately conclude Lemma~\ref{lem:derivatives_m} for $\tau \in I_\theta$ satisfying $\dens(\tau) \leq \dens_*$. 
If $\dens(\tau)>\dens_*$ then $\norm{B^{-1}} \lesssim 1$. Hence, a simple induction argument using \eqref{eq:derivative_m_general_formula_B} and \eqref{eq:decomposition_r_k}, which hold true for $\dens(\tau) >\dens_*$ as well,
 yields some $C \sim 1$ such that 
\[  \norm{\pt_z^k m(\tau)} \lesssim k! C^k \] 
for all $k \in \N$ satisfying $k \geq 1$.  
Since $\dens(\tau) \lesssim 1$ for all $\tau \in I_\theta$, we obtain Lemma~\ref{lem:derivatives_m} in the missing regime. 
\end{proof}

\section{The cubic equation} \label{sec:cubic_equation} 

The following Proposition \ref{pro:cubic_for_dyson_equation} is the main result of this section. 
It asserts that $m$ is determined by the solution to a cubic equation, \eqref{eq:cubic_sing_dens} below, 
close to points $\tau_0\in\supp\dens$ of small density $\dens(\tau_0)$. In Section \ref{sec:shape_analysis}, this cubic equation will allow for a classification of the small local minima of $\tau \mapsto \dens(\tau)$. 
To have a short notation for the elements of $\supp\dens$ of small density, we introduce the set 
 \[ \Dmin_{\eps,\theta} \defeq \{\tau \in \supp \dens\cap I \colon \dens(\tau) \in [0,\eps], ~\dist(\tau,\pt I) \geq \theta \} \] 
for $\eps >0$ and $\theta >0$.

The leading order terms of the cubic and quadratic coefficients in \eqref{eq:cubic_sing_dens} are given by $\psi(\tau_0)$ and $\sigma(\tau_0)$, respectively. 
For their definitions, we refer to Lemma \ref{lem:stability_cubic_equation} (i) and \eqref{eq:def_psi_sigma}.

\begin{proposition}[Cubic equation for shape analysis] \label{pro:cubic_for_dyson_equation}
Let $I \subset \R$ be an open interval and $\theta \in (0,1]$. 
If Assumptions \ref{assums:general} hold true on $I$ for some $\eta_* \in (0,1]$ 
then there are thresholds $\dens_* \sim 1$ and $\delta_* \sim 1$ such that, for all $\tau_0 \in \Dmin_{\dens_*,\theta}$, the following hold true:
\begin{enumerate}[label=(\alph*)] 
\item 
For all $\omega\in [-\delta_*, \delta_*]$, we have   
\begin{equation} \label{eq:def_admiss_decom_diff_m}
 m(\tau_0 + \omega) -m(\tau_0) = \Theta(\omega) b + r(\omega), 
\end{equation}
where $\Theta \colon [-\delta_*, \delta_* ] \to \C$ and $r \colon [-\delta_*, \delta_*] \to \alg$ are defined by 
\begin{equation} \label{eq:def_Theta_r_shape_analysis}
 \Theta(\omega) \defeq \scalarbb{\frac{l}{\scalar{b}{l}}}{m(\tau_0 +\omega) - m(\tau_0)}, \qquad r(\omega) \defeq Q[m(\tau_0 +\omega) - m(\tau_0)]. 
\end{equation}
Here, $l=l(\tau_0)$, $b=b(\tau_0)$ and $Q=Q(\tau_0)$ are the eigenvectors and spectral projection of $B(\tau_0)$ introduced in Corollary~\ref{coro:eigenvector_expansion}.
We have $b = b^* + \ord(\dens)$ and $l = l^* + \ord(\dens)$ as well as $b +b^* \sim 1$ and $l + l^* \sim 1$
with $\dens = \dens(\tau_0) = \avg{\Im m(\tau_0)}/\pi$. 
\item The function $\Theta$ satisfies the cubic equation 
\begin{equation} \label{eq:cubic_sing_dens}
 \mu_3 \Theta^3(\omega) + \mu_2 \Theta^2(\omega) + \mu_1\Theta(\omega) + \omega \Xi(\omega) = 0 
\end{equation}
for all $\omega \in [-\delta_*, \delta_*]$. The complex coefficients $\mu_3$, $\mu_2$, $\mu_1$ and $\Xi$ in \eqref{eq:cubic_sing_dens} fulfill 
\begin{subequations}
\begin{align}
\mu_3 & = \psi + \ord(\dens), \label{eq:mu3_imz_0}\\
\mu_2 & = \sigma + \ii\dens\bigg(3\psi + \frac{\sigma^2}{\avg{f_u^2}} \bigg) + \ord(\dens^2), \label{eq:mu2_imz_0}\\ 
\mu_1 & = 2\ii\dens \sigma - 2 \dens^2 \bigg( \psi+ \frac{\sigma^2}{\avg{f_u^2}} \bigg) + \ord(\dens^3),\label{eq:mu1_imz_0} \\ 
\Xi(\omega) & = \pi(1 + \nu(\omega))+\ord(\rho), \label{eq:Xi_omega}
\end{align}
\end{subequations}
where $\sigma=\sigma(\tau_0)$ as well as $\psi= \psi(\tau_0)$. 
For the error term $\nu(\omega)$, we have 
\begin{equation}
\abs{\nu(\omega)} \lesssim \abs{\Theta(\omega)} +\abs{\omega} \lesssim \abs{\omega}^{1/3}  \label{eq:bound_nu}.
\end{equation}
for all $\omega \in [-\delta_*, \delta_* ]$. 
Uniformly for $\tau_0 \in \Dmin_{\dens_*,\theta}$, we have
\begin{equation} \label{eq:def_admiss_scaling_parameters}
 \psi + \sigma^2 \sim 1.
\end{equation}
\item Moreover, $\Theta(\omega)$ and $r(\omega)$ are bounded by 
\begin{subequations}
\begin{align}
 \abs{\Theta(\omega)} & \lesssim \min \bigg\{ \frac{\abs{\omega}}{\dens^2}, \abs{\omega}^{1/3} \bigg\}, \label{eq:def_admiss_bound_Theta} \\ 
\norm{r(\omega)} & \lesssim \abs{\Theta(\omega)}^2 +\abs{\omega}, \label{eq:def_admiss_bound_r}
\end{align}
\end{subequations}
uniformly for all $\omega \in [-\delta_*, \delta_*] $. 
\end{enumerate}
\begin{enumerate}[label=(d\arabic*)]
\item If $\dens >0$ then $\Theta$ and $r$ are differentiable in $\omega$ at $\omega=0$. 
\item If $\dens = 0$ then we have
\begin{equation} \label{eq:def_admiss_dens_zero_Im_Theta_bounds}
 \Im \Theta(\omega) \geq 0, \qquad \abs{\Im \nu(\omega)} \lesssim \Im \Theta(\omega), \qquad \norm{\Im r(\omega)} \lesssim (\abs{\Theta(\omega)} + \abs{\omega})\Im \Theta(\omega), 
\end{equation}
for all $\omega \in [-\delta_*, \delta_* ]$ and $\Re \Theta$ is non-decreasing on the connected components of $\{ \omega \in [-\delta_*, \delta_* ] \colon \Im \Theta(\omega) = 0 \}$. 
\end{enumerate}
\begin{enumerate}[label=(\alph*)]
\setcounter{enumi}{4}
\item The function $\sigma \colon \Dmin_{\dens_*,\theta} \to \R$ is uniformly $1/3$-Hölder continuous. 
\end{enumerate}
\end{proposition}

The previous proposition is the analogue of Lemma 9.1 in \cite{AjankiQVE}.  It should also be compared to \cite[Proposition~4.12]{AltEdge}, where the shape analysis was performed only in a neighbourhood of an edge and thus a lower order accuracy was sufficient.
The cubic equation for $\Theta$, \eqref{eq:cubic_sing_dens}, will be obtained from an $\alg$-valued quadratic equation for $\diff \defeq m(\tau_0 +\omega) - m(\tau_0)$
and the results of Section \ref{sec:stability_operator}. In fact, we have  
\begin{equation} \label{eq:stability_dyson}
  (\Id - C_mS)[\diff] =  \omega m^2 + \frac{\omega}{2} \Big( m \diff + \diff m\Big)  + \frac{1}{2}\Big( m S[\diff]\diff + \diff S[\diff]m\Big), 
\end{equation}
where $\tau_0, \tau_0 +\omega \in I_\theta \defeq \{ \tau \in I \colon \dist(\tau, \pt I) \geq \theta\}$ and $m \defeq m(\tau_0)$ (see the proof of Proposition \ref{pro:cubic_for_dyson_equation} 
in Section~\ref{subsec:proof_proposition_cubic_for_dyson_equation} below for a derivation of \eqref{eq:stability_dyson}).
Projecting \eqref{eq:stability_dyson} onto the direction $b$ and its complement, where $b$ is the unstable direction of $B$ defined in Corollary \ref{coro:eigenvector_expansion}, 
yields the cubic equation, \eqref{eq:cubic_sing_dens}, for the contribution $\Theta$ of $\diff$ parallel with $b$. 
In the next subsection, this derivation is presented in a more abstract and transparent setting of a general $\alg$-valued quadratic equation. 
After that, the coefficients of the cubic equation are computed in Lemma \ref{lem:coefficients_general_cubic} in the setup of \eqref{eq:stability_dyson} before we prove Proposition \ref{pro:cubic_for_dyson_equation} 
in Section~\ref{subsec:proof_proposition_cubic_for_dyson_equation}.

\subsection{General cubic equation}  \label{subsec:general_cubic_equation} 

Let $B, T \colon \cal{A} \to \cal{A}$ be linear maps, $A \colon \cal{A} \times \cal{A} \to \cal{A}$ a bilinear map and $K \colon \alg \times \alg \to \alg$ a map. 
For $\diff, \err \in \cal{A}$, we consider the quadratic equation
\begin{equation} \label{eq:quadratic_eq}
B[\diff] - A[\diff,\diff]- T[\err] - K[\err,\diff]\,=\, 0\,.
\end{equation}
We view this as an equation for $\diff$, where $e$ is a (small) error term. 
This quadratic equation is a generalization of the stability equation \eqref{eq:stability_dyson} for the Dyson equation, \eqref{eq:dyson} 
(see \eqref{eq:B_A_general_dyson} and \eqref{eq:T_K_density_of_states_analysis} below for the concrete choices of $B$,$T$, $A$ and $K$ in the setting of \eqref{eq:stability_dyson}). 

Suppose that $B$ has a non-degenerate isolated eigenvalue $\beta$ and a corresponding eigenvector $b$, i.e., $B[b]\,=\, \beta b$ and $D_r(\beta) \cap \spec(B) = \{\beta\}$ for some $r >0$. 
We denote the spectral projection corresponding to $\beta$ and its complementary projection by $P$ and $Q$, respectively, i.e.,
\begin{equation} \label{eq:def_projections_P_Q}
P\,\defeq\, -\frac{1}{2\pi\ii}\oint_{\partial D_r(\beta)} (B - \omega\Id)^{-1} \dd \omega = \frac{\scalar{l}{\genarg}}{\scalar{l}{b}} b \,,\qquad Q\,\defeq\, \Id-P\,.
\end{equation}
Here, $l\in \cal{A}$ is an eigenvector of $B^*$ corresponding to its eigenvalue $\ol{\beta}$, i.e., $B^*[l]\,=\, \ol{\beta}l$.
In the following, we will assume that
\begin{equation} \label{eq:assums_derivation_cubic_equation} 
\hspace{-0.22cm}\norm{B^{-1}Q[x]} \lesssim \norm{x}, \quad \abs{\scalar{l}{b}}^{-1}+\norm{b}+\norm{l} \lesssim 1, \quad \norm{A[x,y]} \lesssim \norm{x} \norm{y}, \quad \norm{T[\err]} \lesssim \norm{\err}, \quad \norm{K[\err,y]} \lesssim \norm{\err}\norm{y} 
\end{equation}
for all $x, y \in \alg$ and the $\err \in \alg$ from \eqref{eq:quadratic_eq}. 
The guiding idea is that the main contribution in the decomposition 
\begin{equation} \label{eq:def_Theta}
\diff\,=\, \Theta \2 b+Q[\diff], \qquad \Theta \defeq \frac{\scalar{l}{\diff}}{\scalar{l}{b}} 
\end{equation}
is given by $\Theta$, i.e., the coefficient of $\diff$ in the direction $b$, under the assumption that $\diff$ is small. 
If $A = K = 0$ then this would be a simple linear stability analysis of the equation $B[\diff]= \, small\,$ around an isolated eigenvalue of $B$. The presence of the quadratic terms in \eqref{eq:quadratic_eq} 
requires to follow second and third order terms carefully. 
In the following lemma, we show that the behaviour of $\Theta$ is governed by a scalar-valued cubic equation (see \eqref{eq:general_cubic_equation} below) and that $Q[\diff]$ is indeed dominated by $\Theta$.
The implicit constants in \eqref{eq:assums_derivation_cubic_equation} are the model parameters in Section \ref{subsec:general_cubic_equation}.
\begin{lemma}[General cubic equation] \label{lem:general_cubic_equation} 
Let $\beta$ be a non-degenerate isolated eigenvalue of $B$. 
Let $\diff \in \alg$ and $\err \in\alg$ satisfy \eqref{eq:quadratic_eq}, $\Theta$ be defined as in \eqref{eq:def_Theta} and the conditions in \eqref{eq:assums_derivation_cubic_equation} hold true. 
Then there is $\eps \sim 1$ such that if $\norm{\diff} \leq \eps$ then $\Theta$ satisfies the cubic equation 
\begin{equation} \label{eq:general_cubic_equation}
\mu_3\2 \Theta^3+\mu_2 \2\Theta^2 + \mu_1 \2\Theta + \mu_0\,=\, \errt, 
\end{equation}
with some $\errt =\ord(\abs{\Theta}^4+\abs{\Theta}\norm{\err}+\norm{\err}^2)$ and with 
coefficients
\begin{equation} \label{eq:Coefficients mu3 to mu0}
\begin{aligned}
\mu_3\,&=\,\scalar{l}{A[b,B^{-1}QA[b,b]]+A[B^{-1}QA[b,b],b]}, \\
\mu_2\,&=\,\scalar{l}{A[b,b]}, \\
\mu_1\,&=\,-\beta \scalar{l}{b}, \\
\mu_0\,&=\,\scalar{l}{T[\err]}.
\end{aligned}
\end{equation}
Moreover, we have
\begin{equation} \label{eq:estimate_tilde_r}
 Q[\diff] = B^{-1} Q T[e] + \ord(\abs{\Theta}^2 + \norm{\err}^2). 
\end{equation}
If we additionally assume that $\Im \diff \in \algnon$, $l = l^*$ and $b=b^*$ as well as 
\begin{equation} \label{eq:assums_general_cubic_Im_part_estimates}
 B[x]^* = B[x^*], \quad A[x,y]^* = A[x^*, y^*], \quad T[\err]^* = T[e], \quad K[\err,y]^* = K[\err, y^*] 
\end{equation}
for all $x, y \in \alg$ then there are $\eps \sim 1$ and $\delta \sim 1$ such that $\norm{\diff} \leq \eps$ and $\norm{e} \leq \delta$ also imply
\begin{subequations} 
\begin{align} 
\norm{\Im Q[\diff]} & \lesssim (\abs{\Theta} + \norm{\err}) \Im \Theta, \label{eq:Im_Q_lesssim_Im_Theta}\\ 
\abs{\Im \tilde{\err}} & \lesssim ( \abs{\Theta}^3 + \norm{\err} ) \Im \Theta. \label{eq:Im_tilde_err_lesssim_Im_Theta} 
\end{align} 
\end{subequations} 
\end{lemma} 

\begin{proof} 
Setting $r\defeq Q[\diff]$, the quadratic equation \eqref{eq:quadratic_eq} reads as 
\begin{equation} \label{eq:quadratic with Theta}
 \Theta \beta  b +Br \,=\,  T[\err]+A[\diff,\diff]+ K[\err,\diff].
\end{equation}
By applying $Q$ and afterwards $B^{-1}$ to the previous relation, we conclude that
\begin{equation} \label{eq:R2 expansion}
 r = B^{-1} Q T[\err] + \Theta^2 B^{-1} Q A[b,b] + \err_1, \qquad 
\err_1 \defeq \Theta B^{-1} Q (A[b,r] + A[r,b]) + B^{-1}Q A[r,r] + B^{-1} Q K[\err, \diff]. 
\end{equation}
We have $\norm{e_1}  \lesssim   \norm{r} \abs{\Theta} + \norm{r}^2 + \norm{e} \norm{\diff}$ and $\norm{r} \lesssim \norm{e} + \abs{\Theta}^2 + \norm{e_1}$. 
From the second bound in \eqref{eq:assums_derivation_cubic_equation}, we conclude $\norm{P} + \norm{Q} \lesssim 1$ and, thus, $\norm{r} \lesssim \norm{\diff}$. 
By choosing $\eps \sim 1$ small enough, assuming $\norm{\diff} \leq \eps$ and using $\norm{r} \lesssim \norm{\diff}$, we obtain 
\begin{equation} \label{eq:bound_r_bound_e_1}
 \norm{r} \lesssim \abs{\Theta}^2 + \norm{e}, \qquad \norm{e_1} \lesssim \abs{\Theta}^3 + \norm{e} \abs{\Theta} + \norm{e}^2. 
\end{equation}
This proves \eqref{eq:estimate_tilde_r}. 
Defining $ e_2 \defeq e_1 + B^{-1}Q T[e]$ yields $\diff = \Theta b + \Theta^2 B^{-1}Q A[b,b] + e_2$. By plugging this into \eqref{eq:quadratic with Theta} and 
computing the scalar product with $\scalar{l}{\genarg}$, we obtain 
\begin{subequations}
\begin{align}
 \Theta \beta \scalar{l}{b}  & = \scalar{l}{T[e]} + \Theta^2 \scalar{l}{A[b,b]} + \Theta^3 \scalar{l}{A[b,B^{-1}QA[b,b]] + A[B^{-1}QA[b,b],b]} - \tilde{e}, \\
\tilde{e}  & \defeq -\scalar{l}{K[e,\diff] + \Theta^4A[B^{-1}QA[b,b],B^{-1}QA[b,b]] + A[\diff, e_2] + A[e_2, \diff] - A[e_2, e_2]}. \label{eq:def_tilde_e}
\end{align} 
\end{subequations}
Since $\norm{e_2} \lesssim \abs{\Theta}^3 + \norm{e}$ and $\norm{\diff} \lesssim \abs{\Theta} + \norm{e}$ by \eqref{eq:bound_r_bound_e_1} and \eqref{eq:estimate_tilde_r}, 
we conclude $\tilde{e} = \ord(\abs{\Theta}^4 + \abs{\Theta} \norm{e} + \norm{e}^2)$. 
Therefore, $\Theta$ satisfies \eqref{eq:general_cubic_equation} with the coefficients from \eqref{eq:Coefficients mu3 to mu0}. 

For the rest of the proof, we additionally assume that the relations in \eqref{eq:assums_general_cubic_Im_part_estimates} hold true.
Taking the imaginary part of \eqref{eq:R2 expansion} and arguing similarly as after \eqref{eq:R2 expansion} yield 
\[ \norm{\Im e_1}  \lesssim ( \norm{r} + \abs{\Theta} + \norm{e} ) (\Im \Theta + \norm{\Im r}), \qquad  \norm{\Im r } \lesssim \abs{\Theta} \Im \Theta + \norm{\Im e_1}. \]  
Hence, \eqref{eq:Im_Q_lesssim_Im_Theta} and $\norm{\Im e_1} \lesssim (\abs{\Theta} + \norm{e}) \Im \Theta$ follow for $\norm{\diff} \leq \eps$ and $\norm{e} \leq \delta$ with some sufficiently small $\eps \sim 1$ and $\delta \sim 1$. 
From this and taking the imaginary part in \eqref{eq:def_tilde_e}, we conclude \eqref{eq:Im_tilde_err_lesssim_Im_Theta} as $\norm{\Im \diff} \lesssim \Im \Theta$ by \eqref{eq:Im_Q_lesssim_Im_Theta} 
and $\Im e_2 = \Im e_1$. This completes the proof of Lemma \ref{lem:general_cubic_equation}.
\end{proof}

\subsection{Cubic equation associated to Dyson stability equation}

Owing to \eqref{eq:Coefficients mu3 to mu0}, the coefficients $\mu_3$, $\mu_2$ and $\mu_1$ are completely determined by the bilinear map $A$ and the operator $B$. 
For analyzing the Dyson equation, \eqref{eq:dyson}, owing to \eqref{eq:stability_dyson}, the natural choices for $A$ and $B$ are 
\begin{equation} \label{eq:B_A_general_dyson}
 B \defeq \Id - C_mS, \qquad A[x,y] \defeq \frac{1}{2}( mS[x] y  + y S[x] m) 
\end{equation}
with $x, y \in \alg$. 
In particular, $Q$ in \eqref{eq:def_projections_P_Q} has to be understood with respect to $B = \Id - C_m S$. 
In the next lemma, we compute $\mu_3$, $\mu_2$ and $\mu_1$ with these choices.  
This computation involves the inverse of $\Id - C_sF$. 

In order to directly ensure its invertibility, we will assume $\imz >0$. This assumption will be removed in the proof of Proposition \ref{pro:cubic_for_dyson_equation} 
in Section \ref{subsec:proof_proposition_cubic_for_dyson_equation} below. 

\begin{lemma}[Coefficients of the cubic for Dyson equation] \label{lem:coefficients_general_cubic} Let $A$ and $B$ be defined as in \eqref{eq:B_A_general_dyson}. 
If Assumptions~\ref{assums:general} hold true on an interval $I \subset \R$ for some $\eta_* \in (0,1]$ 
then there is a threshold $\dens_* \sim 1$ such that, for $z \in \HbI$ satisfying $\dens(z) + \dens(z)^{-1} \imz \leq \dens_*$,
the coefficients of the cubic \eqref{eq:general_cubic_equation} have the expansions
\begin{subequations} \label{eq:coefficients_dyson_general}
\begin{align}
\mu_3 & =  \psi + \ord( \dens + \dens^{-1} \imz), \label{eq:mu_3_dyson_general}\\ 
\mu_2 & =  \sigma + \ii \dens \bigg( 3 \psi + \frac{\sigma^2}{\avg{f_u^2}}\bigg) + \ord(\dens^2 + \dens^{-1} \imz), \label{eq:mu_2_dyson_general} \\ 
\mu_1 & =  -\pi \dens^{-1} \imz + 2\ii\dens\sigma - 2 \dens^2 \bigg( \psi + \frac{\sigma^2}{\avg{f_u^2}}\bigg) + \ord( \dens^3 + \imz + \dens^{-2}\imz[2]). 
\label{eq:mu_1_dyson_general}
\end{align}
\end{subequations}
Moreover, we also have
\begin{equation} \label{eq:scalar_l_m_Sb_b} 
\scalar{l}{mS[b]b} =\sigma + \ii \dens \bigg( 3 \psi + \frac{\sigma^2}{\avg{f_u^2}}\bigg) + \ord(\dens^2 + \dens^{-1} \imz). 
\end{equation}
\end{lemma}

\begin{proof} 
In this proof, we use the convention that concatenation of maps on $\alg$ and evaluation of these maps in elements of $\alg$ are prioritized 
before the multiplication in $\alg$, i.e., 
\[ AB[b]c \defeq (A[B[b]])c \] 
if $A$ and $B$ are maps on $\alg$ and $b, c \in \alg$.  
We will obtain all expansions in \eqref{eq:coefficients_dyson_general} from \eqref{eq:Coefficients mu3 to mu0} by using the special choices for $A$ and $B$ from \eqref{eq:B_A_general_dyson}. 
Before starting with the proof of \eqref{eq:mu_3_dyson_general}, we establish a few identities. 
Recalling $m = q^* uq$ from \eqref{eq:m_representation_q_star_u_q} and \eqref{eq:def_F}, we first notice the following alternative expression for $A$ 
\begin{equation} \label{eq:A_x_y_alternative}
 A[x,y] = \frac{1}{2} C_{q^*,q}\Big[ u F C_{q^*,q}^{-1}[x] C_{q^*,q}^{-1} [y]+ C_{q^*,q}^{-1}[y] FC_{q^*,q}^{-1}[x] u \Big]  
\end{equation}
with $x, y \in \alg$. 
Owing to \eqref{eq:q_sim_1_im_u_sim_dens_consequence_assum}, the operators $C_{q^*,q}$ and $C_{q^*,q}^{-1}$ are bounded. 
We choose $\dens_* \sim 1$ small enough so that Lemma \ref{lem:prop_F_small_dens} is applicable. 
By using $u = s + \ii \Im u + \ord(\dens^2)$ due to \eqref{eq:Re_u_s} as well as \eqref{eq:def_f_u}, \eqref{eq:F_f_u} and \eqref{eq:b_0_l_0_approx} 
in \eqref{eq:A_x_y_alternative}, we obtain 
\begin{equation} \label{eq:A_b_0_b_0}
 A[b_0,b_0] = C_{q^*,q}[s f_u^2 + \ii \dens f_u^3] + \ord( \dens^2 + \dens^{-1} \imz). 
\end{equation}
Combining \eqref{eq:A_b_0_b_0} and \eqref{eq:B_0_inverse_Q_0} implies
\begin{equation*} 
 B_0^{-1}Q_0 A[b_0,b_0] = C_{q^*,q}(\Id - C_sF)^{-1}Q_{s,F}[sf_u^2] + \ord(\dens + \dens^{-1} \imz). 
\end{equation*}
We now prove the expansion \eqref{eq:mu_3_dyson_general} for $\mu_3$ by starting from \eqref{eq:Coefficients mu3 to mu0} and using $l=l_0 + \ord(\dens)$, $b=b_0 + \ord(\dens)$ by \eqref{eq:expansion of beta b l}, 
$B^{-1}Q = B_0^{-1}Q_0 + \ord(\dens)$ due to $B=B_0 + \ord(\dens)$ and Lemma \ref{lem:prop_F_small_dens} 
and the previous identities. This yields 
\begin{align*}  
\mu_3 = \, & \scalar{l_0}{A[B_0^{-1}Q_0A[b_0,b_0],b_0]+ A[b_0,B_0^{-1}Q_0 A[b_0,b_0]]} + \ord(\dens)  \\
= &\,  \scalar{f_u}{uF(\Id - C_sF)^{-1} Q_{s,F}[sf_u^2]f_u + uF[f_u](\Id-C_sF)^{-1}Q_{s,F}[sf_u^2]}  + \ord(\dens + \dens^{-1} \imz) \\ 
= &\,   \scalar{sf_u^2}{(\Id + F)(\Id - C_s F)^{-1}Q_{s,F}[sf_u^2]} + \ord(\dens + \dens^{-1} \imz). 
\end{align*} 
Here, we also used $F[f_u] = f_u + \ord(\dens^{-1}\imz)$ by \eqref{eq:F_f_u} and $u= s + \ord(\dens)$ by \eqref{eq:Re_u_s}. 
This shows \eqref{eq:mu_3_dyson_general}. 

In order to compute $\mu_2$, we define 
\[b_1 \defeq 2\ii\dens C_{q^*,q}(\Id- C_sF)^{-1}Q_{s,F}[sf_u^2], \qquad l_1 \defeq - 2\ii\dens C_{q,q^*}^{-1}(\Id-FC_s)^{-1} Q_{s,F}^*F[sf_u^2]. \]
Then we use \eqref{eq:expansion_b} as well as \eqref{eq:expansion_l} and obtain 
\begin{align*}
 \scalar{l}{A[b,b]} = \, & \scalar{l_0}{A[b_0,b_0]} + \scalar{l_1}{A[b_0,b_0]} + \scalar{l_0}{A[b_1,b_0]} + \scalar{l_0}{A[b_0,b_1]} + \ord(\dens^2+\imz) \\ 
= \, &  \avg{sf_u^3} + \ii \dens \avg{f_u^4} + 2 \ii \dens \scalar{sf_u^2}{(\Id + 2F)(\Id - C_s F)^{-1} Q_{s,F}[sf_u^2] } 
+ \ord(\dens^2+\dens^{-1} \imz )  \\ 
= \, & \sigma + \ii  \dens \bigg( 3 \psi + \frac{\sigma^2}{\avg{f_u^2}}\bigg) + \ord(\dens^2+ \dens^{-1} \imz ).
\end{align*}
Here, in the second step, we used \eqref{eq:b_0_l_0_approx}, \eqref{eq:A_b_0_b_0} and the definition of $l_1$ to compute the first and second term, 
\eqref{eq:b_0_l_0_approx}, the definition of $b_1$ and \eqref{eq:A_x_y_alternative} to compute the third and fourth term.
In the last step, we then employed 
\[\begin{aligned}
&\avg{f_u^4} + \scalar{sf_u^2}{2 ( \Id + 2 F)(\Id - C_s F)^{-1}Q_{s,F}[sf_u^2]}  \\ 
&\hspace{4cm} = \,  \scalar{sf_u^2}{ ( \Id + 2 ( \Id+ 2 F)(\Id - C_s F)^{-1} )Q_{s,F}[sf_u^2]} + \scalar{sf_u^2}{P_{s,F}[sf_u^2]}  \\ 
 &\hspace{4cm} = \,  3 \scalar{sf_u^2}{(\Id + F)(\Id - C_s F)^{-1} Q_{s,F}[sf_u^2]}  + \frac{\sigma^2}{\avg{f_u^2}} + \ord(\dens^{-1} \imz). 
\end{aligned}  \] 
Here, we applied \eqref{eq:expansion_P_sF}, $C_s = C_s^*$ and $C_s[sf_u^2] = sf_u^2$.
Since $\mu_2 = \scalar{l}{A[b,b]}$ by \eqref{eq:Coefficients mu3 to mu0}, this completes the proof of \eqref{eq:mu_2_dyson_general}. 
A similar computation as the one for $\mu_2$ yields \eqref{eq:scalar_l_m_Sb_b}. 

Since $\mu_1=-\beta\scalar{l}{b}$ by \eqref{eq:Coefficients mu3 to mu0}, the expansion in \eqref{eq:expansion_beta_scalar_l_b} immediately yields \eqref{eq:mu_1_dyson_general}. 
This completes the proof of the lemma. 
\end{proof}

\subsection{The cubic equation for the shape analysis} 
\label{subsec:proof_proposition_cubic_for_dyson_equation} 

In this subsection, we will prove Proposition \ref{pro:cubic_for_dyson_equation} by using Lemma \ref{lem:general_cubic_equation} and Lemma \ref{lem:coefficients_general_cubic}. 
Therefore, in addition to the choices of $A$ and $B$ in~\eqref{eq:B_A_general_dyson}, we choose $\diff = m(\tau_0 + \omega) - m(\tau_0)$, $\tau_0, \tau_0 + \omega \in I$, $\err = \omega\id$ and
\begin{equation} \label{eq:T_K_density_of_states_analysis} 
T[x] = x m^2, \quad K[x, y] = \frac{1}{2} (x m y + y m x) 
\end{equation}
for $x, y \in \alg$ with $m=m(\tau_0)$ in \eqref{eq:quadratic_eq}.

\begin{proof}[Proof of Proposition \ref{pro:cubic_for_dyson_equation}]
We choose $\dens_*\sim 1$ such that Lemma \ref{lem:prop_F_small_dens} and Corollary \ref{coro:eigenvector_expansion} are applicable. 
We fix $\tau_0 \in \Dmin_{\dens_*,\theta}$ and set $m = m(\tau_0)$. 
The statements about $l$ and $b$ in (a) of Proposition \ref{pro:cubic_for_dyson_equation} follow from Corollary~\ref{coro:eigenvector_expansion}. In particular, 
$\abs{\scalar{l}{b}} \sim 1$.  
Thus, the conditions in \eqref{eq:assums_derivation_cubic_equation} are a direct consequence of 
Assumptions~\ref{assums:general}, \eqref{eq:q_sim_1_im_u_sim_dens_consequence_assum}, Lemma \ref{lem:prop_F_small_dens} and Corollary \ref{coro:eigenvector_expansion}. 
Furthermore, if $\dens = 0$ then we have $m=m^*$ and, thus, \eqref{eq:assums_general_cubic_Im_part_estimates} follows. 
For $\omega \in [-\delta_*, \delta_*]$, $\delta_* \defeq \theta/2$, we set $\Delta = m(\tau_0 +\omega) - m$. Since $\Theta(\omega)b = P[\Delta]$, $r(\omega) = Q[\Delta]$ and $P + Q = \Id$, 
we immediately obtain \eqref{eq:def_admiss_decom_diff_m}. This proves~(a). 

Next, we derive \eqref{eq:stability_dyson} for $\Delta\defeq m(z_0+\omega) - m(z_0)$ and $m\defeq m(z_0)$ with $z_0 \defeq \tau_0 + \ii\eta$, $\tau_0\in\Dmin_{\dens_*,\theta}$, 
$\omega \in [-\delta_*,\delta_*]$ and $\eta \in (0,\eta_*]$. 
We subtract \eqref{eq:dyson} evaluated at $z = z_0$ from \eqref{eq:dyson} evaluated at $z = z_0 + \omega$ and obtain \eqref{eq:stability_dyson} with $\Delta$ and $m$ defined at $z_0 = \tau_0 + \ii \eta$.
Directly taking the limit $\eta \downarrow 0$ yields \eqref{eq:stability_dyson} with the original choices of $\Delta$ and $m$ at $z_0 = \tau_0$ by the Hölder-continuity of $m$ 
on $\overline{\Hb}_{I',\eta_*}$, $I' \defeq \{ \tau \in I \colon \dist(\tau, \pt I) \geq \theta/2\}$, due to Proposition \ref{pro:analyticity_of_m}. 

Lemma \ref{lem:general_cubic_equation} is applicable for $\abs{\omega} \leq \delta_*$ with some sufficiently small $\delta_* \sim 1$ since this guarantees $\norm{\Delta} \leq \eps$ owing to the Hölder-continuity of $m$. 
Hence, Lemma \ref{lem:general_cubic_equation} yields a cubic equation for $\Theta$ as defined in \eqref{eq:def_Theta_r_shape_analysis} with $l=l(z_0)$, $b=b(z_0)$ and $z_0 = \tau_0 + \ii\eta$.  
The coefficients of this cubic equation are given in Lemma \ref{lem:general_cubic_equation}. 
Owing to the uniform $1/3$-Hölder continuity of $z \mapsto m(z)$ on $\overline{\Hb}_{I',\eta_*}$,
we conclude from the definition of $\Theta$ and $r\defeq Q[\diff]$ in \eqref{eq:def_Theta_r_shape_analysis}, 
the boundedness of $Q$ and $B^{-1}Q$ as well as \eqref{eq:estimate_tilde_r} that
$\abs{\Theta(\omega)} \lesssim \abs{\omega}^{1/3}$, i.e., the second bound in \eqref{eq:def_admiss_bound_Theta}, and \eqref{eq:def_admiss_bound_r} uniformly for $\eta \in [0,\eta_*]$.  

We now compute the coefficients of the cubic in \eqref{eq:cubic_sing_dens} for $\tau_0 \in \Dmin_{\dens_*,\theta}$. Set $z_0 \defeq \tau_0 + \ii\eta$. 
Note that for $\eta = \Im z_0 >0$ these coefficients were already given in \eqref{eq:coefficients_dyson_general}, 
so the only task is to check their limit behaviour as $\eta \downarrow 0$. 
Owing to \eqref{eq:dens_reciprocal_imz_equal_zero}, the expansions in \eqref{eq:mu3_imz_0}, \eqref{eq:mu2_imz_0} and \eqref{eq:mu1_imz_0} 
follow from \eqref{eq:mu_3_dyson_general}, \eqref{eq:mu_2_dyson_general} and \eqref{eq:mu_1_dyson_general}, respectively, 
using the continuity of $\sigma$, $\psi$ and $f_u$ on $\overline{\Hb}_\rm{small}$ by Lemma \ref{lem:stability_cubic_equation} and Lemma \ref{lem:q_u_f_u_extension}, respectively. 
We now show \eqref{eq:Xi_omega}. 
With the definitions of $\errt$ and $\mu_0$ from Lemma \ref{lem:general_cubic_equation} (see \eqref{eq:def_tilde_e} and \eqref{eq:Coefficients mu3 to mu0}, respectively), we set $\Xi(\omega) \defeq \omega^{-1}(\mu_0- \errt)$ for arbitrary $\abs{\omega} \leq \delta_*$. 
Since $l = C_{q,q^*}^{-1}[f_u] + \ord(\dens + \dens^{-1}\eta)$ due to \eqref{eq:b_0_l_0_approx} and \eqref{eq:expansion_l}, as well as 
$m^2 = (\Re m)^2 + \ord(\dens) = C_{q^*,q}C_s[qq^*] +\ord(\dens)$ due to $\Im m \sim \dens\id$ and \eqref{eq:Re_u_s},  we have 
\begin{equation} \label{eq:mu0_omega_inverse}
 \omega^{-1}\mu_0 = \avg{l^*m^2} = \avg{f_u qq^*} + \ord( \dens + \dens^{-1}\eta) = \pi + \ord(\dens + \dens^{-1} \eta).  
\end{equation}
Here, we also used $C_s[f_u] = f_u$ in the second step and \eqref{eq:f_u_qq_star} in the last step. 
We set $\nu(\omega) \defeq - (\omega\pi)^{-1} \errt$. We recall $e = \omega\id$. 
Since $\errt = \ord(\abs{\Theta(\omega)}^4 + \abs{\Theta(\omega)}\abs{\omega} + \abs{\omega}^2)$ and $\abs{\Theta(\omega)} \lesssim \abs{\omega}^{1/3}$, 
we obtain \eqref{eq:bound_nu}. 
This yields \eqref{eq:Xi_omega} by using \eqref{eq:dens_reciprocal_imz_equal_zero} in \eqref{eq:mu0_omega_inverse}.
Since \eqref{eq:psi_plus_sigma_sim_1} implies \eqref{eq:def_admiss_scaling_parameters}, this completes the proof of (b) for $\tau_0 \in \Dmin_{\dens_*,\theta}$ and we assume $\eta =0$ in the following.

If $\dens = \dens(\tau_0)>0$ then \eqref{eq:m_Lipschitz_bound_positive_density} 
yields the missing first bound in \eqref{eq:def_admiss_bound_Theta} completing the proof of part (c). 
Moreover, in this case, the definitions of $\Theta$ and $r$ imply their differentiability at $\omega =0$ due to Proposition \ref{pro:analyticity_of_m}. 
This shows~(d1). 

We now verify (d2). Since $\dens =0$, we have $\Im m(\tau_0) = 0$ and thus $\Im \Theta(\omega) \geq 0$ by the positive semidefiniteness of $\Im m(\tau_0 +\omega)$. 
Since $\mu_0$ is real as $l$ and $T[e]$ are self-adjoint, we obtain the second bound in \eqref{eq:def_admiss_dens_zero_Im_Theta_bounds} directly from \eqref{eq:Im_tilde_err_lesssim_Im_Theta}
 and $\abs{\Theta(\omega)}\lesssim \abs{\omega}^{1/3}$.  
The third bound in \eqref{eq:def_admiss_dens_zero_Im_Theta_bounds} follows from \eqref{eq:Im_Q_lesssim_Im_Theta} and $e = \omega \id$. 
Since $\dens=0$ and hence $b = C_{q^*,q}[f_u]$ by \eqref{eq:expansion_b} and $l = C_{q,q^*}^{-1}[f_u]$ by \eqref{eq:expansion_l} are positive definite elements of $\alg$, 
 $\Re \Theta(\omega)  + \scalar{l}{m(\tau_0)}/\scalar{l}{b}$ is the real part of the Stieltjes transform of a positive measure $\mu$ evaluated on the real axis. 
The real part of a Stieltjes transform is non-decreasing on the connected components of the complement in $\R$ of the support of its defining measure. 
Therefore, as the support of $\mu$ is contained in $\R \setminus \text{int}( \{ \omega\in [-\delta_*,\delta_* ] \colon \Im \Theta(\omega) = 0 \})$, where $\text{int}$ denotes the interior,  due to $\Im m(\tau_0)=0$, we conclude that $\Re \Theta(\omega)$ 
is non-decreasing on the connected components of $\{\omega \in [-\delta_*,\delta_*] \colon \Im \Theta(\omega) = 0\}$. 

Lemma \ref{lem:stability_cubic_equation} (i) directly implies the Hölder-continuity in (e), which 
 completes the proof of Proposition~\ref{pro:cubic_for_dyson_equation}.
\end{proof}

\section{Cubic analysis} \label{sec:shape_analysis} 

The main result of this section, Theorem \ref{thm:behaviour_v_close_sing} below, implies Theorem \ref{thm:singularities_flat} and gives even effective error terms. 
Theorem \ref{thm:behaviour_v_close_sing} describes the behaviour of $\Im m$ close to local minima of $\dens$ inside of $\supp \dens$.
This behaviour is governed by the universal shape functions $\Psie \colon [0,\infty) \to \R$ and $\Psim \colon \R \to \R$ defined by  
\begin{subequations} \label{eq:def_Psim_Psie} 
\begin{align}
\Psie(\lambda) & \defeq \frac{\sqrt{(1+\lambda)\lambda}}{\big( 1+ 2 \lambda + 2 \sqrt{(1+\lambda)\lambda}\big)^{2/3} + \big( 1+ 2 \lambda - 2 \sqrt{(1+\lambda)\lambda}\big)^{2/3} +1 }, \label{eq:def_Psie} \\
\Psim(\lambda) & \defeq \frac{\sqrt{1+\lambda^2}}{\big(\sqrt{1+\lambda^2} + \lambda\big)^{2/3} + \big(\sqrt{1+\lambda^2} -\lambda\big)^{2/3} -1}-1. \label{eq:def_Psim}
\end{align}
\end{subequations}

For the definition of the comparison relation $\lesssim$, $\gtrsim$ and $\sim$ in the following Theorem \ref{thm:behaviour_v_close_sing}, we refer to Convention~\ref{conv:comparsion_1} and remark that the 
model parameters in Theorem \ref{thm:behaviour_v_close_sing} are given by $c_1$, $c_2$ and $c_3$ in \eqref{eq:estimates_data_pair}, $k_3$ in \eqref{eq:m_sup_bound} and $\theta$ in the definition of $I_\theta$ 
in \eqref{eq:def_I_theta} below.

\begin{theorem}[Behaviour of $\Im m$ close to local minima of $\dens$] \label{thm:behaviour_v_close_sing} 
Let $(a,S)$ be a data pair such that \eqref{eq:estimates_data_pair} is satisfied. Let $m$ be the solution to the associated Dyson equation \eqref{eq:dyson} and assume that \eqref{eq:m_sup_bound} holds true on $\HbI$ for some interval 
$I \subset \R$ and some $\eta_*\in (0,1]$. 
We write $v \defeq \pi^{-1} \Im m$ and, for some $\theta \in (0,1]$, we  set 
\begin{equation} \label{eq:def_I_theta}
I_\theta \defeq \{ \tau \in I \colon \dist(\tau, \pt I ) \geq \theta\}.
\end{equation} 

Then there are thresholds $\dens_* \sim 1$ and $\delta_* \sim 1$ such that if $\tau_0\in\supp\dens\cap I_\theta$ is a 
local minimum of $\dens$ and $\dens(\tau_0) \leq \dens_*$ then 
\begin{equation} \label{eq:expansion_local_minima_general}
 v(\tau_0 + \omega) = v(\tau_0) + h \Psi(\omega) + \ord\bigg( \dens(\tau_0) \abs{\omega}^{1/3}\char(\abs{\omega} \lesssim \dens(\tau_0)^3) + \Psi(\omega)^2\bigg)
\end{equation}
for $\omega \in [-\delta_*,\delta_*] \cap D$ with some $h=h(\tau_0) \in \alg$ satisfying $h \sim 1$. Moreover, the set $D$ and the function $\Psi$ depend only on the type of $\tau_0$ in the following way:
\begin{enumerate}[label=(\alph*)]
\item \emph{Left edge:} If $\tau_0 \in (\pt \supp \dens) \setminus\{ \inf \supp \dens\}$ is the infimum of a connected component of $\supp \dens$ and the lower edge of the corresponding gap is in $I_\theta$, i.e.,
 $\tau_1\defeq \sup( (-\infty, \tau_0)\cap \supp \dens)\in I_\theta$, 
then \eqref{eq:expansion_local_minima_general} holds true with $v(\tau_0)=0$, $D = [0,\infty)$ and 
\[ \Psi(\omega)  = \Delta^{1/3} \Psie\bigg( \frac{\omega}{\Delta} \bigg) \] 
where $\Delta \defeq \tau_0 - \tau_1$. 
If $\tau_0 = \inf \supp \dens$, or more generally $\dens(\tau)=0$ for all $\tau \in [\tau_0 - \eps, \tau_0]$
with some $\eps \sim 1$, then the same conclusion holds true with $\Delta \defeq 1$. 
\item \emph{Right edge:} If $\tau_0 \in\pt \supp \dens$ is the supremum of a connected component 
then a similar statement as in the case of a left edge holds true. 
\item \emph{Cusp:} If $\tau_0 \notin \pt\supp\dens$ and $\dens(\tau_0)=0$ then 
\eqref{eq:expansion_local_minima_general} holds true with $D = \R$ and $\Psi(\omega) = \abs{\omega}^{1/3}$. 
\item \emph{Internal minimum:} 
If $\tau_0 \notin \pt\supp\dens$ and $\dens(\tau_0)>0$ then there is $\wt{\dens} \sim \dens(\tau_0)$ such that 
\eqref{eq:expansion_local_minima_general} holds true with $D = \R$ and 
\[ \Psi(\omega)  = \wt{\dens}\Psim\bigg( \frac{\omega}{\wt{\dens}^3} \bigg). \]
\end{enumerate}
\end{theorem}

If the conditions of Theorem~\ref{thm:behaviour_v_close_sing} hold true, i.e., the data pair $(a,S)$ satisfies \eqref{eq:estimates_data_pair} and $m$ satisfies \eqref{eq:m_sup_bound} on $\HbI$, 
then Assumptions~\ref{assums:general} are fulfilled on $\HbI$ (compare Lemma~\ref{lem:q_bounded_Im_u_sim_avg} (ii)). 
In fact, Theorem~\ref{thm:behaviour_v_close_sing} holds true under Assumptions~\ref{assums:general} which will become apparent from the proof.

Theorem \ref{thm:behaviour_v_close_sing} contains the most important results of the shape analysis. When considering $\dens=\avg{v}$ instead of $v$ the 
coefficient in front of $\Psi(\omega)$ in \eqref{eq:expansion_local_minima_general} can be precisely identified as demonstrated in part (i) of Theorem~\ref{thm:behaviour_dens} below.
Moreover, Theorem~\ref{thm:behaviour_dens} contains additional information on the size of the connected components
of $\supp \dens$ and the distance between local minima; these are collected in part (ii). 
Note that the same information were also proven in the commutative setup in Theorem 2.6 of \cite{AjankiQVE}
and Theorem~\ref{thm:behaviour_dens} shows that they are also available in our general von Neumann algebra setup.

We remark that $\Psim(\omega) = \Psim(-\omega)$ for $\omega \in \R$ and, for $\omega >0$, $\Delta >0$ and $\wt{\rho} >0$, we have 
\begin{subequations}  
\begin{align}
 \Delta^{1/3} \Psie\bigg(\frac{\omega}{\Delta}\bigg) & \sim \min \bigg\{ \frac{\omega^{1/2}}{\Delta^{1/6}}, \omega^{1/3} \bigg\}, \label{eq:scaling_psie} \\
\wt{\rho} \Psim\bigg(\frac{\omega}{\wt{\rho}^3} \bigg) & \sim \min\bigg \{ \frac{\omega^2}{\wt{\rho}^5}, \omega^{1/3} \bigg\}. \label{eq:scaling_psim} 
\end{align}
\end{subequations}

The comparison relations $\sim$, $\lesssim$ and $\gtrsim$ in the following Theorem~\ref{thm:behaviour_dens} are understood with respect to the constants $k_1, \ldots, k_8$ from Assumptions~\ref{assums:general} 
and $\theta$ in the definition of $I_\theta$ in \eqref{eq:def_I_theta}. 

\begin{theorem}[Behaviour of $\dens$ near almost cusp points; Structure of the set of minima of $\dens$]   \label{thm:behaviour_dens} 
Let $I \subset \R$ be an open interval and $\theta \in (0,1]$. 
If Assumptions~\ref{assums:general} hold true on $I$ for some $\eta_* \in (0,1]$ (in particular, if the data pair $(a,S)$ satisfies \eqref{eq:estimates_data_pair} and $m$ satisfies \eqref{eq:m_sup_bound} on $\HbI$)
then the following statements hold true 
\begin{enumerate}[label=(\roman*)]
\item There are thresholds $\dens_* \sim 1$, $\sigma_* \sim 1$ and $\delta_* \sim 1$ such that if $\tau_0 \in \supp \dens\cap I_\theta$ is a local minimum of $\dens$ satisfying 
$\dens(\tau_0) \leq \dens_*$ then we set 
$\Gamma \defeq \sqrt{27} \pi/(2\psi)$ with $\psi=\psi(\tau_0)$ defined as in Lemma~\ref{lem:stability_cubic_equation} 
and have 
\begin{subequations}
\begin{enumerate}[label=(\alph*)]
\item (Left edge with small gap) If $\tau_0 \in \pt\supp \dens\setminus \{ \inf \supp\dens\}$ is the infimum of a connected component of $\supp \dens$, $\abs{\sigma(\tau_0)} \leq \sigma_*$ and the lower edge of the gap lies in $I_\theta$, 
i.e., $\tau_1 \defeq \sup ( ( -\infty, \tau_0 ) \cap \supp \dens)\in I_\theta$, then 
\begin{equation} \label{eq:dens_expansion_left_edge} 
 \dens(\tau_0 + \omega) = \big(4\Gamma \big)^{1/3} \Psi(\omega) + \ord\left( \abs{\sigma(\tau_0)} \Psi(\omega) + \Psi(\omega)^2 \right), \qquad \Psi(\omega) \defeq \Delta^{1/3}\Psie\bigg( \frac{\omega}{\Delta}\bigg) 
\end{equation}
for all $ \omega \in [0,\delta_*]$.  
Here, $\Gamma \sim 1$ and $\psi \sim 1$. 
\item (Right edge with small gap) If $\tau_0 \in \pt \supp \dens \setminus \{ \sup\supp \dens\}$ is the supremum of a connected component then a similar statement as in the case of a left edge holds true. 
\item (Cusp) If $\tau_0 \notin \pt\supp \dens$ and $\dens(\tau_0) =0$ then 
\begin{equation} \label{eq:dens_expansion_cusp}
 \dens(\tau_0 + \omega) =  \frac{\Gamma^{1/3}}{4^{1/3}} \abs{\omega}^{1/3} + \ord\Big(\abs{\omega}^{2/3}\Big) 
\end{equation}
for all $\omega \in [-\delta_*, \delta_*]$. 
Here, $\Gamma \sim 1$ and $\psi \sim 1$. 
\item (Nonzero local minimum) 
If $\tau_0\notin \pt \supp \dens$ and $\dens(\tau_0) >0$ then 

\begin{equation} \label{eq:multiplicative_expansion_nonzero_minimum_general_assums}
\begin{aligned} 
\dens(\tau_0 + \omega ) & = \dens(\tau_0) + \begin{cases} \Gamma^{1/3}\Psi(\omega) \bigg( 1 + \ord(\dens(\tau_0)^{1/2} ) \bigg), \quad & \text{if }\,\abs{\omega} \lesssim \dens(\tau_0)^{7/2}, \\
\Gamma^{1/3} \Psi(\omega) \bigg( 1 + \ord\bigg( \frac{\dens(\tau_0)^4}{\abs{\omega}} \bigg)\bigg), & \text{if }\, \dens(\tau_0)^{7/2} \lesssim \abs{\omega} \lesssim \dens(\tau_0)^{3}, \\ 
\Gamma^{1/3} \Psi(\omega) \bigg( 1 + \ord( \Psi(\omega))\bigg), & \text{if }\,\dens(\tau_0)^3 \lesssim \abs{\omega} \leq \delta_*,
\end{cases}
 \\
\Psi(\omega) & \defeq \wt{\dens} \Psim\left(\frac{\omega}{\wt{\dens}^3} \right), \qquad \qquad \wt{\dens} \defeq \frac{\dens(\tau_0)}{\Gamma^{1/3}} 
\end{aligned} 
\end{equation}
for all $\omega \in \R$. Here, $\Gamma \sim 1$ and $\psi \sim 1$. 
\end{enumerate}
\end{subequations}
\item If $\supp \dens \cap I_\theta \neq \varnothing$ then $\supp \dens \cap I_\theta$ consists of $K \sim 1$ intervals, i.e., there are 
$\alpha_1, \ldots, \alpha_{K}\in \pt\supp \dens\cup \pt I_\theta$ and $\beta_1, \ldots, \beta_{K} \in \pt\supp\dens\cup \pt I_\theta$, $\alpha_i < \beta_i < \alpha_{i+1}$, such that 
\begin{equation} \label{eq:decomposition_support} 
 \supp \dens \cap \overline{I_\theta} = \bigcup_{i=1}^{K} [\alpha_i, \beta_i] 
\end{equation}
and $\beta_i - \alpha_i \sim 1$ if $\beta_i \neq \sup I_\theta$ and $\alpha_i \neq \inf I_\theta$. 

For $\dens_*>0$, we define the set $\Mb_{\dens_*}$ of small local minima $\tau$ of $\dens$ which are not edges of $\supp\dens$, i.e., 
\begin{equation} \label{eq:def_Mb_dens_ast}
 \Mb_{\dens_*} \defeq \{ \tau \in (\supp \dens \setminus\pt \supp \dens) \cap I_\theta \colon \dens(\tau) \leq \dens_*, ~\dens \text{ has a local minimum at }\tau \}. 
\end{equation}
There is a threshold $\dens_* \sim 1$ such that, for all $\gamma_1, \gamma_2 \in \Mb_{\dens_*}$ satisfying $\gamma_1 \neq \gamma_2$ and for all $i=1, \ldots, K$, 
we have 
\begin{equation} 
\abs{\gamma_1 - \gamma_2} \sim 1, \qquad  \abs{\alpha_i-\gamma_1} \sim 1, \qquad \abs{\beta_i - \gamma_1} \sim 1
\end{equation}
if $\alpha_i \neq \inf I_\theta$ and $\beta_i \neq \sup I_\theta$. 
\end{enumerate}
\end{theorem} 

The factors $4^{1/3}$ and $4^{-1/3}$ in the cases (a) and (c) of part (i) of Theorem~\ref{thm:behaviour_dens} can be
eliminated by redefining $\Gamma$, $\Psie$ and $\Psim$ to bring the leading term on the right-hand sides  
into the uniform $\Gamma^{1/3}\Psi(\omega)$ form. 
We have not used these redefined versions of 
$\Gamma$, $\Psie$ and $\Psim$ here in order to be consistent with \cite{AjankiQVE}. 

We remark that part (i) (a) and (b) of Theorem~\ref{thm:behaviour_dens} cover only the case of $\tau_0 \in \pt\supp\dens$ with sufficiently small $\abs{\sigma(\tau_0)}$. 
We will establish later that the smallness of $\abs{\sigma(\tau_0)}$ corresponds to the smallness of the adjacent gap $\tau_0 - \tau_1$ (see Lemma~\ref{lem:large_gap} below). 
Together with the cusps and the small nonzero local minima, they form the \emph{almost cusp points} of $\dens$ 
(see the set $P_\mathrm{cusp}$  in \eqref{eq:def_P_cusp}  later). 
At the extreme edges $\inf \supp \dens$ and $\sup \supp \dens$, $\sigma$ is not so small. Thus, the exclusion of these 
edges in the statement of Theorem~\ref{thm:behaviour_dens} (a) and (b) is in fact superfluous. 
If $\abs{\sigma(\tau_0)}$ is not so small then $\dens(\tau_0 + \omega)$ is well approximated by a rescaled version of $(\omega_\pm)^{1/2}$ (positive and negative part of $\omega$ for left and right edge, respectively).  
The precise statement and scaling are given in  Proposition~\ref{pro:char_regular_edge} below.    

\begin{remark}[Scaling relations for $\dens(z)$] \label{rmk:scaling im m}
Let $I \subset \R$ be an open interval, $\theta \in (0,1]$ and $\dens(z) \defeq \avg{\Im m(z)}/\pi$ for $z \in \Hb$. 
If Assumptions~\ref{assums:general} hold true on $I$ with $\eta_*=1$ then there are $\eps \sim 1$ and $\dens_* \sim 1$ such that 
\begin{enumerate}[label=(\roman*)]
\item (Inside support around an edge with small gap) 
Let $\tau_0, \tau_1 \in \supp \dens \cap I_\theta$ satisfy $\tau_0 < \tau_1$ and $(\tau_0, \tau_1) \cap \supp \dens = \varnothing$. 
We set $\Delta \defeq \tau_1 - \tau_0$. 
For $\omega \in [0,\eps]$, we have 
\[ \dens(\tau_0 - \omega + \ii\eta) \sim \dens(\tau_1+\omega + \ii\eta) \sim \frac{(\omega + \eta)^{1/2}}{(\Delta + \omega + \eta)^{1/6}} \] 
\item (Inside a gap) 
Let $\tau_0, \tau_1 \in \supp \dens \cap I_\theta$ satisfy $\tau_0 < \tau_1$ and $(\tau_0, \tau_1) \cap \supp \dens = \varnothing$. 
We set $\Delta \defeq \tau_1 - \tau_0$. 
Then, for $\tau \in [\tau_0, \tau_1]$ and $\eta \in [0,\eps]$, we have 
\[ \dens(\tau + \ii \eta) \sim \frac{\eta}{(\Delta + \eta)^{1/6}} \bigg( \frac{1}{(\tau_1 - \tau + \eta)^{1/2} }
+ \frac{1}{(\tau -\tau_0 + \eta)^{1/2}} \bigg). \] 
\item (Around a left edge with large gap)
Let $\tau_0 \in I_\theta \cap \pt \supp\dens$ satisfy $\dens(\tau) = 0$ for all $\tau \in [\tau_0 -\delta, \tau_0]$  and some $\delta \sim 1$. Then, for 
$\omega \in [0,\eps]$, we have
\[ \begin{aligned} 
 \dens(\tau_0 + \omega + \ii \eta) & \sim (\omega + \eta)^{1/2}, \\ 
\dens(\tau_0 - \omega + \ii \eta) & \sim \frac{\eta}{(\omega + \eta)^{1/2}}.
\end{aligned} \] 
A similar statement holds true for a right edge $\tau_0$, i.e., if $\dens(\tau)=0$ for all $\tau \in [\tau_0,\tau_0+\delta]$ and some $\delta \sim 1$. 
\item (Close to a local minimum) 
If $\tau_0 \in \supp \dens \setminus \pt\supp\dens$ is a local minimum of $\dens$ such that $\dens(\tau_0) \leq \dens_*$ then, for all $ \omega \in [-\eps, \eps]$ and $\eta \in [0,\eps]$, we have
\[ \dens(\tau_0 + \omega + \ii \eta)  \sim \dens(\tau_0) + ( \abs{\omega} + \eta)^{1/3}.\] 
\end{enumerate}
These scaling relations for $\dens(z)=\avg{\im m(z)}/\pi$ are proven in the same way as the corresponding ones in Corollary~A.1 of \cite{AjankiQVE}. The proof in \cite{AjankiQVE} simply relied on the fact that $\avg{\im m(z)}$ is the harmonic extension of $\pi \1\rho$ to the complex upper half-plane and the behavior of $\rho$ close to its local minima and thus is applicable equally well in the current situation, due to Theorem~\ref{thm:behaviour_dens}. 
\end{remark}

\subsection{Shape regular points} 

In the following definition, we introduce the notion of a \emph{shape regular point} which collects the properties of $m$ necessary for the proof of Theorem \ref{thm:behaviour_v_close_sing}. 
Proposition \ref{pro:behaviour_v_close_sing_weak_conditions} below explains how the statements of Theorem \ref{thm:behaviour_v_close_sing} are transferred to this more general setup. 
In fact, Lemma \ref{lem:q_bounded_Im_u_sim_avg} (ii) and Proposition \ref{pro:cubic_for_dyson_equation} 
show that, 
under the assumptions of Theorem \ref{thm:behaviour_v_close_sing}, any point $\tau_0\in \supp\dens \cap I$ 
of sufficiently small density $\dens(\tau_0)$ is a shape regular point for $m$ in the sense of Definition \ref{def:admissible_shape_analysis} below. 
By explicitly spelling out the properties of $m$ really used in the proof of Theorem \ref{thm:behaviour_v_close_sing} we made our argument 
modular because  a similar  analysis around  shape regular points will be applied in later works as well. 

This modularity, however, requires to reinterpret the concept of comparison relations. In earlier sections we used
 the comparison relation $\sim$, $\lesssim$ and the $\ord$-notation introduced in Convention \ref{conv:comparsion_1} 
to hide irrelevant constants in various estimates that depended only on the model parameters $c_1$, $c_2$, $c_3$
 from \eqref{eq:estimates_data_pair}, $k_3$ from \eqref{eq:m_sup_bound} and $\theta$ from \eqref{eq:def_I_theta}, these are also the model parameters 
 in Theorem \ref{thm:behaviour_v_close_sing}. 
The model parameters in Theorem~\ref{thm:behaviour_dens} are given by $k_1, \ldots, k_8$ in Assumptions~\ref{assums:general} and $\theta$ in the definition of $I_\theta$. 
 
 The formulation of Definition \ref{def:admissible_shape_analysis} also involves comparison relations instead of carrying constants; in the application these constants
 depend on the original model parameters. When Proposition \ref{pro:behaviour_v_close_sing_weak_conditions}   is proven,
 the corresponding constants  directly depend  on the constants in Definition \ref{def:admissible_shape_analysis},
 hence they also indirectly depend on the original model parameters
when we apply it to the proof of  Theorem \ref{thm:behaviour_v_close_sing}.   Since these dependences are somewhat involved
and we do not want to overload the paper with different concepts of comparison relations, 
for simplicity,   for the purpose of Theorem \ref{thm:behaviour_v_close_sing}, the reader may think of the implicit constants in every $\sim$-relation depending 
only on the original model parameters $c_1$, $c_2$, $c_3$, $k_3$ and $\theta$.

\begin{definition}[Admissibility for shape analysis, shape regular points] \label{def:admissible_shape_analysis} 
Let $m$ be the solution of the Dyson equation \eqref{eq:dyson} associated to a data pair $(a,S)\in \alg_{\rm{sa}} \times \Sigma$. 
\begin{enumerate}[label=(\roman*)]
\item 
Let $\tau_0 \in \R$, $\shapeJ \subset \R$ be an open interval with $0 \in \shapeJ$,  $\Theta \colon \shapeJ \to \C$ and $r \colon \shapeJ \to \alg$ be continuous 
functions and $b \in \alg$.  
We say that $m$ is $(\shapeJ,\Theta,b,r)$-\emph{admissible for the shape analysis at $\tau_0$} if the following conditions are satisfied: 
\begin{enumerate}[label=(\alph*)]
\item The function $m \colon \Hb \to \alg$ has a continuous extension to $\tau_0 + \shapeJ$, which we also denote by $m$. 
The relation \eqref{eq:def_admiss_decom_diff_m} and the bounds \eqref{eq:def_admiss_bound_Theta} as well as \eqref{eq:def_admiss_bound_r} hold true for all $\omega \in \shapeJ$. 
\item The function $\Theta$ satisfies the cubic equation \eqref{eq:cubic_sing_dens} for all $\omega \in \shapeJ$ with the coefficients 
\begin{align*}
 \mu_3 & = \psi + \ord(\rho), \\
 \mu_2 & = \sigma + \ii 3\psi \rho + \ord(\rho^2 + \rho \abs{\sigma}), \\
\mu_1 & = - 2 \rho^2 \psi + \ii \kappa_1 \rho \sigma  +  \ord(\rho^3+\rho^2 \abs{\sigma}),\\ 
\Xi(\omega) & = \kappa ( 1+ \nu(\omega)) + \ord(\rho), 
\end{align*}
where $\rho \defeq \avg{\Im m(\tau_0)}/\pi$
and $\psi, \kappa \geq 0$ as well as $\sigma, \kappa_1 \in \R$ are some parameters satisfying \eqref{eq:def_admiss_scaling_parameters} and $\kappa, \abs{\kappa_1}\sim1$. 
The function $\nu\colon \shapeJ \to \C$ satisfies~\eqref{eq:bound_nu}. 
\item The element $b \in \alg$ in \eqref{eq:def_admiss_decom_diff_m} fulfils $b=b^* + \ord(\rho)$ and $b+ b^* \sim 1$. 
\end{enumerate}
\begin{enumerate}[label=(d\arabic*)]
\item If $\dens >0$ then $\Theta$ and $r$ are differentiable in $\omega$ at $\omega=0$. 
\item If $\dens = 0$ then \eqref{eq:def_admiss_dens_zero_Im_Theta_bounds} holds true for all $\omega \in \shapeJ$ 
and $\Re \Theta$ is non-decreasing on the connected components of $\{ \omega \in \shapeJ \colon \Im \Theta(\omega) = 0 \}$. 
\end{enumerate}
\item
Let $\tau_0 \in \R$ and $\shapeJ \subset \R$ be an open interval with $0 \in \shapeJ$. 
We say that $\tau_0$ is a \emph{shape regular point for $m$ on $\shapeJ$} if $m$ is $(\shapeJ,\Theta,b,r)$-admissible for the shape analysis at $\tau_0$ for 
 some continuous functions $\Theta\colon \shapeJ \to \C$ and $r \colon \shapeJ \to \alg$ as well as $b \in \alg$. 
\end{enumerate}
\end{definition}

The key technical step in the proof of Theorem \ref{thm:behaviour_v_close_sing} is the following Proposition \ref{pro:behaviour_v_close_sing_weak_conditions}; it
shows that Theorem \ref{thm:behaviour_v_close_sing} holds under more general weaker conditions, in fact
shape admissibility is sufficient. For the proof of Theorem \ref{thm:behaviour_v_close_sing} we will first check shape regularity
from Proposition \ref{pro:cubic_for_dyson_equation} and then we will prove Proposition \ref{pro:behaviour_v_close_sing_weak_conditions};   both steps are done in Section \ref{subsec:proof_behaviour_v_small} below.

\begin{proposition}[Theorem \ref{thm:behaviour_v_close_sing} under weaker assumptions] \label{pro:behaviour_v_close_sing_weak_conditions} 
For the solution $m$ to the Dyson equation \eqref{eq:dyson}, we write $v \defeq \pi^{-1} \Im m$, $\dens = \avg{v}$. 

There are thresholds $\dens_* \sim 1$ and $\delta_* \sim 1$ such that if $\dens(\tau_0) \leq \dens_*$ and $\tau_0\in\supp\dens$ is a local minimum of $\dens$ as well as a shape regular point for $m$ on $\shapeJ$ 
with an open interval $\shapeJ \subset \R$ satisfying $0 \in \shapeJ$ then 
\eqref{eq:expansion_local_minima_general} holds true for all $\omega \in [-\delta_*,\delta_*] \cap \shapeJ \cap D$. 
Here, as in Theorem \ref{thm:behaviour_v_close_sing}, $h=h(\tau_0) \in \alg$ with $h \sim 1$ and $D$ as well as $\Psi$ depend only on the type of $\tau_0$ in the following way:

Suppose that $\tau_0 \in \pt \supp \dens$ is the infimum of a connected component of $\supp \dens$. 
If $\dens(\tau)=0$ for all $\tau \in [\tau_0 - \eps, \tau_0]$ with some $\eps \sim 1$ (e.g.~$\tau_0 =\inf\supp \dens$) and $\abs{\inf \shapeJ} \gtrsim 1$, 
then the conclusion of case (a) in Theorem \ref{thm:behaviour_v_close_sing} holds true with $\Delta = 1$ and $v(\tau_0)=0$. 

If $\tau_0\neq \inf \supp\dens$ and $\tau_1\defeq \sup( (-\infty, \tau_0)\cap \supp \dens)$ is a shape regular point for $m$, $\Delta \lesssim 1$ with $\Delta \defeq \tau_0 - \tau_1$ 
and $\abs{\sigma(\tau_0) - \sigma(\tau_1)} \lesssim \abs{\tau_0-\tau_1}^{1/3}$ 
then the conclusion of case (a) in Theorem \ref{thm:behaviour_v_close_sing} holds true with this choice of $\Delta$ as well as $v(\tau_0)=0$.  

Similarly to (a), the statement of case (b) in Theorem \ref{thm:behaviour_v_close_sing} can be translated to the current setup. 
The cases (c) and (d) of Theorem \ref{thm:behaviour_v_close_sing}, cusp and internal minimum, respectively, hold true without any changes. 

Furthermore, suppose that $\tau_0 \in \supp \dens$ is a shape regular point for $m$ and $\dens(\tau_0)=0$, then $\tau_0$ is a cusp if $\sigma(\tau_0)=0$ and $\tau_0$ is an edge, in particular $\tau_0 \in \pt\supp \dens$, 
if $\sigma(\tau_0)\neq 0$. 
\end{proposition}

Similarly, the following Proposition~\ref{pro:behaviour_dens_weak_conditions} is the analogue of Theorem~\ref{thm:behaviour_dens} under the sole requirement of shape admissibility. 
Owing to the weaker assumptions, the error term in \eqref{eq:expansion_nonzero_local_minimum_general_assums} as well as the result in \eqref{eq:distance_of_internal_minima} of 
Proposition~\ref{pro:behaviour_dens_weak_conditions} are weaker than the corresponding results in Theorem~\ref{thm:behaviour_dens}. 
We will first show Proposition~\ref{pro:behaviour_dens_weak_conditions} and then conclude Theorem~\ref{thm:behaviour_dens} by using extra arguments for the stronger conclusions; 
both proofs will be presented in Section~\ref{subsec:proof_behaviour_dens} below. 

At a shape regular point $\tau_0\in \R$, we set $\Gamma \defeq \sqrt{27} \kappa/(2\psi)$ (cf. Theorem \ref{thm:abstract_cubic_equation} (i) below), where $\kappa = \kappa(\tau_0)$ and $\psi = \psi (\tau_0)$ 
are defined as in Definition~\ref{def:admissible_shape_analysis} (i) (b). 

\begin{proposition}[Behaviour of $\dens$ near almost cusp points, set of minima of $\dens$ under weaker assumptions]\label{pro:behaviour_dens_weak_conditions}
Let $m$ be the solution to the Dyson equation, \eqref{eq:dyson}, and $\dens=\pi^{-1} \avg{\Im m}$. 
The following statements hold true 
\begin{enumerate}[label=(\roman*)]
\item There are thresholds $\dens_* \sim 1$, $\sigma_* \sim 1$ and $\delta_* \sim 1$ such that if $\tau_0 \in \supp \dens$ is a shape regular point for $m$ on an open interval $J \subset \R$ with $0 \in J$, 
$\dens(\tau_0) \leq \dens_*$ and $\tau_0$ is a local minimum of $\dens$ then we have 
\begin{enumerate}[label=(\alph*)]
\item (Left edge with small gap) If $\tau_0 \in \pt\supp \dens \setminus \{ \inf \supp \dens \}$ is the infimum of a connected component of $\supp \dens$, $\abs{\sigma(\tau_0)} \leq \sigma_*$ and $\tau_1 \defeq \sup ( ( -\infty, \tau_0 ) \cap \supp \dens)$ 
is a shape regular point satisfying $\Delta \lesssim 1$ for $\Delta \defeq \tau_0 - \tau_1$ and $\abs{\sigma(\tau_0) - \sigma(\tau_1)} \lesssim \abs{\tau_0 - \tau_1}^{1/3}$ then 
\eqref{eq:dens_expansion_left_edge} for all $\omega \in [0,\delta_*]\cap J$.  
\item (Right edge with small gap) If $\tau_0 \in \pt \supp \dens \setminus \{ \sup \supp \dens\}$ is the supremum of a connected component then a similar statement as in the case of a left edge holds true. 
\item (Cusp) If $\tau_0 \notin \pt\supp \dens$ and $\dens(\tau_0) =0$ then \eqref{eq:dens_expansion_cusp} holds true for all $\omega \in [-\delta_*, \delta_*] \cap J$. 
\item (Internal minimum) If $\tau_0 \notin \pt\supp\dens$ and $\dens(\tau_0)>0$ then 
\begin{equation} \label{eq:expansion_nonzero_local_minimum_general_assums} 
 \begin{aligned} 
\dens(\tau_0 + \omega) & = \dens(\tau_0) +\Gamma^{1/3}\Psi(\omega)  + \ord\left(\frac{\abs{\omega}}{\dens(\tau_0)}  \char( \abs{\omega} \lesssim \dens(\tau_0)^3) + \Psi(\omega)^2 \right), \\
\Psi(\omega) & \defeq \wt{\dens} \Psim\left(\frac{\omega}{\wt{\dens}^3} \right), \qquad \qquad \wt{\dens} \defeq \frac{\dens(\tau_0)}{\Gamma^{1/3}} 
\end{aligned} 
\end{equation} 
for all $\omega \in [-\delta_*, \delta_*] \cap J$. 
\end{enumerate} 
\item 
Let $I \subset \R$ be an open interval with $\supp \dens \cap I \neq \varnothing$ and $\abs{I} \lesssim 1$ 
 and let $m$ have a continuous extension to the closure $\overline{I}$ of $I$. 
Let $\shapeJ \subset \R$ be an open interval with $0 \in \shapeJ$ and $\dist(0,\pt \shapeJ) \gtrsim 1$ such that $J + (\pt \supp \dens) \cap I \subset I$. 
We assume that all points in $(\pt \supp \dens) \cap I$ are shape regular points for $m$ on $\shapeJ$ and all estimates in Definition \ref{def:admissible_shape_analysis} hold true uniformly on $(\pt \supp \dens) \cap I$. 
If $\abs{\sigma(\tau_0) - \sigma(\tau_1)} \lesssim \abs{\tau_0 - \tau_1}^{1/3}$  
uniformly for all $\tau_0, \tau_1 \in (\pt \supp \dens) \cap I$ 
then $\supp \dens \cap I$ consists of $K \sim 1$ intervals, i.e., there are 
$\alpha_1, \ldots, \alpha_{K}\in \pt\supp \dens\cup \pt I$ and $\beta_1, \ldots, \beta_{K} \in \pt\supp\dens\cup \pt I$, $\alpha_i < \beta_i < \alpha_{i+1}$, such that 
\eqref{eq:decomposition_support} holds true with $\overline{I_\theta}$ replaced by $\overline{I}$ 
and $\beta_i - \alpha_i \sim 1$ if $\beta_i \neq \sup I$ and $\alpha_i \neq \inf I$. 

If $\Mb_{\dens_*}$ is defined as in \eqref{eq:def_Mb_dens_ast} then there is a threshold $\dens_* \sim 1$ such that if, in addition to the previous conditions in (ii), all points of $(\Mb_{\dens_*} \cup \pt\supp\dens) \cap I$ 
 are shape regular points for $m$ on $\shapeJ$ and all estimates in Definition \ref{def:admissible_shape_analysis} hold true 
uniformly on $(\Mb_{\dens_*}\cup \pt\supp\dens)\cap I$ then, for $\gamma \in \Mb_{\dens_*}$, we have $\abs{\alpha_i - \gamma} \sim 1$ and $\abs{\beta_i - \gamma} \sim 1$
if $\alpha_i \neq \inf I$ and $\beta_i \neq \sup I$. 
Moreover, for any $\gamma_1, \gamma_2 \in \Mb_{\dens_*}$, we have either 
\begin{equation} \label{eq:distance_of_internal_minima}
 \abs{\gamma_1 - \gamma_2} \sim 1, \qquad \text{or} \qquad \abs{\gamma_1 - \gamma_2} \lesssim \min\{ \dens(\gamma_1), \dens(\gamma_2)\}^4.  
\end{equation}
If $\dens(\gamma_1) = 0$ or $\dens(\gamma_2) = 0$ then, for $\gamma_1 \neq \gamma_2$, only the first case occurs.
\end{enumerate}
\end{proposition}

An important step towards Theorem \ref{thm:behaviour_v_close_sing} and Proposition \ref{pro:behaviour_v_close_sing_weak_conditions} will be to prove similar behaviours for $\Theta$ as 
$\Im \Theta$ is the leading term in $v$. These behaviours are collected in the following theorem, Theorem \ref{thm:abstract_cubic_equation}. 
It has weaker assumptions than those of Theorem \ref{thm:behaviour_v_close_sing} and those required in Proposition \ref{pro:behaviour_v_close_sing_weak_conditions} -- 
 in particular, on the coefficient $\mu_1$ in the cubic equation \eqref{eq:cubic_sing_dens}. 
However, these assumptions will be sufficient for the purpose of Theorem \ref{thm:abstract_cubic_equation}. 

\begin{theorem}[Abstract cubic equation] \label{thm:abstract_cubic_equation}
Let $\Theta(\omega)$ be a continuous solution to the cubic equation
\begin{equation} \label{eq:abstract_cubic}
 \mu_3 \Theta(\omega)^3 + \mu_2 \Theta(\omega)^2 + \mu_1 \Theta(\omega) + \omega\Xi(\omega) = 0   
\end{equation}
for $\omega \in \shapeJ $, where $\shapeJ  \subset \R$ is an open interval with $0 \in \shapeJ $. 
We assume that the coefficients satisfy 
\[ \begin{aligned} \mu_3 & = \psi + \ord(\rho), \\ 
\mu_2 & = \sigma + 3 \ii \psi\rho + \ord(\rho^2 + \rho\abs{\sigma} ), \\ 
\mu_1 & = -2 \rho^2 \psi + \ord( \rho^3 + \rho\abs{\sigma}) ,\\
 \Xi(\omega) & = \kappa ( 1 + \nu(\omega))+\ord(\rho)
 \end{aligned} \] 
with some fixed parameters $\psi\geq 0$, $\rho \geq 0$, $\sigma \in \R$ and $\kappa \sim 1$. The cubic equation is assumed to be stable in the sense that  
\begin{equation}\label{eq:abstract_cubic_psi_sigma_sim_1}
 \psi + \sigma^2 \sim 1. 
\end{equation}
Moreover, for all $\omega \in \shapeJ $, we require the following bounds on $\nu$ and $\Theta$: 
\begin{subequations}
\begin{align} 
 \abs{\nu(\omega)} & \lesssim \abs{\omega}^{1/3},  \label{eq:bound_nu_cubic}\\
\abs{\Theta(\omega)} & \lesssim \abs{\omega}^{1/3}.  \label{eq:bound_theta_abstract_cubic}
\end{align}
\end{subequations}
Then the following statements hold true: 
\begin{enumerate}[label=(\roman*)]
\item ($\rho>0$) For any $\Pi_* \sim 1$, there is a threshold $\rho_* \sim 1$ such that if $\rho \in (0,\rho_*]$ and $\abs{\sigma} \leq \Pi_* \rho^2$ 
then we have 
\begin{equation} \label{eq:Im_Theta_nonzero_minimum}
 \Im \Theta (\omega) = \rho\Psim\bigg(\Gamma\frac{\omega}{\rho^3}\bigg) +\ord\Big(\min \{ \rho^{-1} \abs{\omega}, \abs{\omega}^{2/3}\}\Big),  
\end{equation}
with $\Gamma\defeq \sqrt{27} \kappa/(2\psi)$. Note that $\Gamma \sim 1$ if $\dens_* \sim 1$ is small enough. 
\item ($\rho = 0$) If $\rho = 0$ and we additionally assume $\Im \Theta(\omega) \geq 0$ for $\omega \in \shapeJ $, $\Re \Theta$ is non-decreasing on the connected components of $\{ \omega \in \shapeJ  \colon \Im \Theta(\omega)=0\}$ as well as
\begin{equation} \label{eq:abstract_cubic_Im_nu_lesssim_Im_Theta}
 \abs{\Im \nu(\omega)} \lesssim \Im \Theta(\omega) 
\end{equation}
for all $\omega \in \shapeJ $ then we have 
\begin{enumerate}[label=(\alph*)]
\item If $\sigma = 0$ then $\Im \Theta(\omega)$ has a cubic cusp at $\omega=0$, i.e., 
\begin{equation} \label{eq:cusp_abstract_cubic}
 \Im \Theta (\omega) = \frac{\sqrt 3}{2} \bigg( \frac{\kappa}{\psi}\bigg)^{1/3} \abs{\omega}^{1/3} + \ord( \abs{\omega}^{2/3}). 
\end{equation}
\item 
If $\sigma \neq 0$ then $\Im \Theta(\omega)$ has a square root edge at $\omega=0$, i.e., there is  $c_* \sim 1$ such that 
\begin{equation} \label{eq:abstract_cubic_small_gap}
\Im \Theta(\omega) = \begin{cases} \displaystyle  c \wh{\Delta}^{1/3} \Psie\bigg( \frac{\abs{\omega}}{\wh{\Delta}} \bigg) + \ord\Big(  (\abs{\nu(\omega)} + \eps(\omega))\eps(\omega)\Big),
& \text{ if }  \sign \omega  = \sign \sigma,\\ 
0, & \text{ if } \omega \in - \sign \sigma[0,c_* \abs{\sigma}^3], \end{cases}
\end{equation}
where $\wh{\Delta} \in (0,\infty)$, $c\in (0,\infty)$ and $\eps \colon \R \to [0,\infty)$ are defined by 
\begin{equation} \label{eq:def_wh_Delta_abstract}
 \wh{\Delta} \defeq \min\bigg\{\frac{4}{27 \kappa} \frac{\abs{\sigma}^3}{\psi^2}, 1 \bigg\}, \qquad c \defeq 3 \sqrt{\kappa} \frac{\wh{\Delta}^{1/6}}{\abs{\sigma}^{1/2}}, 
\qquad \eps(\omega) \defeq \min \bigg\{ \frac{\abs{\omega}^{1/2}}{\wh{\Delta}^{1/6}}, \abs{\omega}^{1/3} \bigg\}.  
\end{equation}
We have $\wh{\Delta} \sim \abs{\sigma}^3$ and $c \sim 1$. Moreover, for $\sign \omega = \sign \sigma$, we have
\begin{equation} \label{eq:abstract_small_gap_abs_Theta_bound}
 \abs{\Theta(\omega)} \lesssim \eps(\omega). 
\end{equation}
\end{enumerate}
\end{enumerate}
\end{theorem}

\subsection{Cubic equations in normal form} \label{subsec:cubic_normal_form}

The core of the proof of Theorem \ref{thm:abstract_cubic_equation} is to bring \eqref{eq:abstract_cubic} into a normal form by a change of variables. We will first explain the analysis
of these normal forms, especially the mechanism of choosing the right branch of the solution based upon selection principles that will be derived from the constraints on $\Theta$ given 
in Theorem \ref{thm:abstract_cubic_equation}. Then, in Section \ref{subsec:proof_abstract_cubic}, we show how to bring \eqref{eq:abstract_cubic} to these normal forms. 

In the following proposition, we study a special solution $\Omega(\lambda)$ to a one-parameter family of cubic equations in normal forms with constant term $\Lambda(\lambda)$ (or $2\Lambda(\lambda)$), 
where $\Lambda(\lambda)$ is a perturbation of the identity map $\lambda \mapsto \lambda$. 
Here, a-priori, the real parameter $\lambda$ is always contained in an (possibly unbounded) interval around $0$. 
This range of definition will not be explicitly indicated in the statements but will be explicitly restricted 
for their conclusions.
We compare the solution to this perturbed cubic equation with the solution to the cubic equation with constant term $\lambda$. 
Depending on the precise type of the cubic equation, the choice of the solution is based on some of the following \emph{selection principles} 
\begin{enumerate}[label=\textbf{SP\arabic*~}]
\item $\lambda \mapsto \Omega(\lambda)$ is continuous
\item $\Omega(0) = \Omega_0$ for some given $\Omega_0 \in \C$ 
\item[\textbf{SP3~}] $\Im ( \Omega(\lambda) - \Omega(0)) \geq 0$, 
\item[\textbf{SP4'}] $\abs{\Im \Lambda(\lambda)} \leq \gamma \abs{\lambda} \Im \Omega(\lambda)$ for some $\gamma >0$ and $\Re \Omega(\lambda)$ is non-decreasing on the connected components of $\{ \lambda \colon \Im \Omega(\lambda) = 0 \}$. 
\end{enumerate}
We use the notation \spfour to distinguish this selection principle from \textbf{SP-4} which was introduced in Lemma~9.9 
of \cite{AjankiQVE}. 

We will make use of the following standard convention for complex powers. 

\begin{definition}[Complex powers] 
We define $\C \setminus (-\infty,0) \to \C,~\zeta \mapsto \zeta^\gamma$ for $\gamma \in \C$ by $\zeta^\gamma \defeq \exp(\gamma \log \zeta)$, 
where $\log \colon \C \setminus (-\infty,0) \to \C$ is a continuous branch of the complex logarithm with $\log 1=0$. 
\end{definition}

With this convention, we record Cardano's formula as follows: 
\begin{proposition}[Cardano] 
The three roots of $\Omega^3  - 3 \Omega + 2 \zeta$, $\zeta \in \C$, are $\wh{\Omega}_+(\zeta)$, $\wh{\Omega}_-(\zeta)$ and $\wh{\Omega}_0(\zeta)$ which are defined by 
\begin{equation}  \label{eq:def_cubic_root_small_gap_normal_form}
\wh{\Omega}_\pm(\zeta) \defeq \frac{1}{2} ( \Phi_+(\zeta) + \Phi_-(\zeta)) \pm \frac{\ii \sqrt{3}}{2} ( \Phi_+(\zeta)-\Phi_-(\zeta)),
\qquad \wh{\Omega}_0(\zeta) \defeq - (\Phi_+ (\zeta) + \Phi_-(\zeta)), 
\end{equation} 
where 
\[ \Phi_\pm(\zeta) =  \begin{cases} ( \zeta \pm \sqrt{ \zeta^2 - 1})^{1/3}, & \text{if }\Re \zeta \geq 1, \\ ( \zeta \pm \ii \sqrt{ 1 - \zeta^2})^{1/3},&  \text{if } \abs{\Re \zeta} < 1, \\ 
- ( - \zeta \mp \sqrt{\zeta^2 - 1})^{1/3}, & \text{if } \Re \zeta \leq - 1. \end{cases} \] 
\end{proposition}

\begin{proposition}[Solution to the cubic in normal form] \label{pro:solution_cubic_normal_form}
Let $\Omega(\lambda)$ satisfy \spone and \sptwo.
\begin{enumerate}[label=(\roman*)]
\item (Non-zero local minimum) 
Let $\Omega_0 =\sqrt 3( \ii + \chi_1)$ in \sptwo and $\Omega(\lambda)$ satisfy
\begin{equation} \label{eq:cubic_normal_nonzero_minimum}
 \Omega(\lambda)^3 + 3 \Omega(\lambda) + 2 \Lambda(\lambda) = 0, \qquad  \Lambda(\lambda) = (1 + \chi_2 + \mu(\lambda)) \lambda + \chi_3, 
\end{equation}
with $\abs{\mu(\lambda)} \lesssim \alpha \abs{\lambda}^{1/3}$, $\alpha>0$. 
Then there exist $\delta \sim 1$ and $\chi_* \sim 1$ such that if $\alpha, \abs{\chi_1}, \abs{\chi_2}, \abs{\chi_3} \leq \chi_*$ then 
\begin{equation} \label{eq:normal_form_expansion_local_minimum}
 \Omega(\lambda) -\Omega_0 = \wh{\Omega}(\lambda)  - \ii \sqrt 3 + \ord\Big( (\alpha+\abs{\chi_2}+ \abs{\chi_3}) \min\{\abs{\lambda}, \abs{\lambda}^{2/3} \} 
\Big) 
\end{equation}
for all $\lambda \in \R$ satisfying $\abs{\lambda} \leq \delta/\alpha^3$, where 
 $\wh{\Omega}(\lambda) \defeq  \Phi_{\rm{odd}}(\lambda) + \ii \sqrt{3} \Phi_{\rm{even}}(\lambda)$ and 
$\Phi_{\rm{odd}}$ and $\Phi_{\rm{even}}$ are the odd and even part of the function $\Phi \colon \C \to \C$, $\Phi(\zeta) \defeq (\sqrt{1 + \zeta^2}+ \zeta)^{1/3}$, 
respectively. 

Moreover, we have for $\abs{\lambda} \leq \delta/\alpha^3$ that 
\begin{equation} \label{eq:bound_Omega_normal_nonzero_minimum}
 \abs{\Omega(\lambda) - \Omega_0} \lesssim \min\{ \abs{\lambda}, \abs{\lambda}^{1/3}\} . 
\end{equation}
\end{enumerate}
In the following, we assume that $\Omega(\lambda)$, in addition to \spone and \sptwo, also satisfies \spthree and \spfour.
\begin{enumerate}[label=(\roman*)]
\setcounter{enumi}{1}
\item (Simple edge) 
Let $\Omega_0=0$ in \sptwo and $\Omega(\lambda)$ be a solution to 
\begin{equation} \label{eq:cubic_normal_simple_edge}
 \Omega^2(\lambda) + \Lambda(\lambda) = 0, \qquad \Lambda(\lambda) = ( 1 + \mu(\lambda))\lambda.  
\end{equation}
If $\abs{\mu(\lambda)} \leq \gamma^{2/3} \abs{\lambda}^{1/3}$ for the $\gamma>0$ of \spfour then there is $c_* \sim 1$ such that 
\begin{equation} \label{eq:normal_form_expansion_simple_edge}
 \Omega(\lambda) = \wh{\Omega}(\lambda) + \ord\Big(\abs{\mu(\lambda)}\abs{\lambda}^{1/2}\Big), \qquad 
\wh{\Omega}(\lambda ) \defeq \begin{cases} \ii \lambda^{1/2},& \text{if } \lambda \in [0,c_*\gamma^{-2}], \\ -(-\lambda)^{1/2}, & \text{if } \lambda \in [-c_*\gamma^{-2},0]. \end{cases} 
\end{equation}
Moreover, we have $\Im \Omega(\lambda) = 0$ for $ \lambda \in [-c_*\gamma^{-2}, 0]$. 
\item (Sharp cusp) Let $\Omega_0=0$ in \sptwo, $\gamma \sim 1$ in \spfour and $\Omega(\lambda)$ be a solution to 
\begin{equation} \label{eq:cubic_normal_sharp_cusp}
 \Omega^3(\lambda) + \Lambda(\lambda) = 0, \qquad \Lambda(\lambda) = (1 +\mu(\lambda)) \lambda.
\end{equation}
If $\abs{\mu(\lambda)} \lesssim \abs{\lambda}^{1/3}$ then there is $\delta \sim 1$ such that 
\begin{equation} \label{eq:normal_form_expansion_sharp_cusp}
 \Omega(\lambda) =  \wh{\Omega}(\lambda) + \ord\Big( \abs{\mu(\lambda)} \abs{\lambda}^{1/3}\Big), \qquad 
\wh{\Omega}(\lambda) \defeq \frac{1}{2} \begin{cases} (-1+\ii\sqrt{3}) \lambda^{1/3}, & \text{if } \lambda \in (0,\delta], \\ (1+\ii\sqrt{3})\abs{\lambda}^{1/3}, & \text{if } \lambda \in [-\delta, 0]. \end{cases}  
\end{equation}
\item (Two nearby edges) 
Let $\Omega_0 = s$ for some $s \in \{ \pm 1\}$ in \sptwo, $\gamma \sim 1$ in \spfour and $\Omega(\lambda)$ be a solution to 
\begin{equation} \label{eq:cubic_normal_two_nearby_edges}
 \Omega(\lambda)^3 - 3 \Omega(\lambda) + 2 \Lambda(\lambda) = 0, \qquad \Lambda(\lambda) = ( 1+ \mu(\lambda))\lambda + s. 
\end{equation}
Then there are $\delta \sim 1$, $\varrho \sim 1$ and $\gamma_* \sim 1$ such that if $\abs{\mu(\lambda)} \lesssim \wh{\gamma} \abs{\lambda}^{1/3}$ for some $\wh{\gamma} \in [0, \gamma_*]$ then 
\begin{enumerate}[label=(\alph*)]
\item We have 
\begin{equation}\label{eq:normal_form_expansion_nearby_edges}
 \Omega(\lambda) = \wh{\Omega}_+(1 +\abs{\lambda}) + \ord\Big( \abs{\mu(\lambda)}\min\{\abs{\lambda}^{1/2},\abs{\lambda}^{1/3}\}\Big),
\end{equation}
for all $\lambda \in s(0,2\delta/\wh{\gamma}^3]$. (Recall the definition of $\wh{\Omega}_+$ from \eqref{eq:def_cubic_root_small_gap_normal_form}.) 
Moreover, for all $\lambda \in s(0,2\delta/\wh{\gamma}^3]$, we have 
\begin{equation} \label{eq:normal_coordinates_small_gap_estimate_Omega}
 \abs{\Omega(\lambda) - \Omega_0} \lesssim \min\Big \{ \abs{\lambda}^{1/2}, \abs{\lambda}^{1/3}\Big \} . 
\end{equation}
\item For all $\lambda \in -s(0,2-\varrho\wh{\gamma}]$, we have
\begin{equation} \label{eq:Im_Omega_small_in_small_gap}
 \Im \Omega(\lambda) \lesssim \wh{\gamma}^{1/2}.  
\end{equation}
\item We have 
\begin{equation} \label{eq:Im_Omega_-_sign_lambda_3_positive}
 \Im \Omega(-s (2 + \varrho \wh{\gamma})) >0.  
\end{equation}
\end{enumerate}
\end{enumerate}
\end{proposition} 

The core of each part in Proposition \ref{pro:solution_cubic_normal_form} is choosing the correct cubic root. 
For the most complicated part (iv), we state this choice in the following auxiliary lemma. 
For its formulation, we 
introduce the intervals 
\begin{equation} \label{eq:def_intervals_small_gaps}
 I_1 \defeq - s [-\lambda_1,0), \qquad I_2 \defeq -s(0,\lambda_2], \qquad I_3 \defeq -s [\lambda_3,\lambda_1], 
\end{equation}
where we used the definitions 
\begin{equation} \label{eq:def_lambdas_small_gap}
\lambda_1 \defeq 2 \frac{\delta}{\wh{\gamma}^3}, \qquad \lambda_2 \defeq 2 - \varrho \wh{\gamma}, \qquad 
\lambda_3 \defeq 2 + \varrho \wh{\gamma}. 
\end{equation}
These definitions are modelled after (9.105) in \cite{AjankiQVE}. We will choose $\wh{\gamma} = \wh{\Delta}^{1/3}$ in the proof of 
Theorem \ref{thm:abstract_cubic_equation} below. Then $\lambda_1$ corresponds to an expansion range $\delta$ in the $\omega$ 
coordinate. Note that with the above choice of $\wh{\gamma}$, we obtain the same $\lambda_1$ as in (9.105) of \cite{AjankiQVE}. 
However, $\lambda_2$ and $\lambda_3$ differ slightly from those in \cite{AjankiQVE}, where $\lambda_{2,3}$ were set to be $2 \mp \varrho \abs{\sigma}$. 
Nevertheless, we will see below that $\wh{\gamma} \sim \abs{\sigma}$ but they are not equal in general. 

For given $\delta, \varrho \sim 1$, we will always choose $\gamma_* \sim 1$ so small that $\wh{\gamma} \leq \gamma_*$ implies
\[  \lambda_1 \geq 4, \qquad 1 \leq \lambda_2 < 2 < \lambda_3 \leq 3 . \] 
Therefore, the intervals in \eqref{eq:def_intervals_small_gaps} are disjoint and nonempty. 

\begin{lemma}[Choice of cubic roots in Proposition \ref{pro:solution_cubic_normal_form} (iv)] \label{lem:choice_roots_two_nearby_edges}
Under the assumptions of Proposition \ref{pro:solution_cubic_normal_form} (iv), there are $\delta, \varrho, \gamma_* \sim 1$ such that 
if $\wh{\gamma} \leq \gamma_*$ then we have 
\[ \Omega|_{I_k} = \wh{\Omega}_+ \circ \Lambda|_{I_k} \] 
for $k=1, 2,3$. 
Here, $\wh{\Omega}_+$ is defined as in \eqref{eq:def_cubic_root_small_gap_normal_form}. 
\end{lemma} 

\begin{proof}
The proof is the same as the one of Lemma 9.14 in \cite{AjankiQVE} but \textbf{SP-4} in \cite{AjankiQVE} is replaced by \spfour above.
In that proof, \textbf{SP-4} is used only in the part titled ``Choice of $a_2$''. We redo this part here. 
Recall that $a_2 = 0, \pm$ denoted the index such that $\Omega|_{I_2} = \wh{\Omega}_{a_2} \circ \Lambda|_{I_2}$ and our goal is to show $a_2 = +$. 
Similarly as in \cite{AjankiQVE}, we assume without loss of generality $s=-1$. 
Since $\lim_{\lambda \downarrow -1} \wh{\Omega}_-(\lambda) = 2$ and $\Omega(0)=-1$ by \sptwo, we conclude $a_2 \neq -$. 
(In the corresponding step in \cite{AjankiQVE}, there was a typo: $\wh{\Omega}_+(-1+0) = 2$ should have been $\wh{\Omega}_-(-1+0)=2$, 
resulting in the choice $a_2=+$. This conclusion is only used in the bound (9.137) of \cite{AjankiQVE} 
which still holds true. The rest of the proof is unaffected.)

We now prove $a_2 \neq 0$. To that end, we take the imaginary part of the cubic equation, \eqref{eq:cubic_normal_two_nearby_edges}, and obtain  
\begin{equation} \label{eq:proof_select_root_two_nearby_edges}
 3( (\Re \Omega)^2 - 1 ) \Im \Omega = -2 \lambda \Im \mu(\lambda) + (\Im \Omega)^3. 
\end{equation}
Suppose that $a_2 = 0$. 
From the definition of $\wh{\Omega}_0$, $\Lambda(\lambda) = ( 1+ \mu(\lambda))\lambda - 1$ and $\abs{\mu(\lambda)} \lesssim \wh{\gamma} \abs{\lambda}^{1/3}$ we obtain 
\begin{equation} \label{eq:nearby_edges_behaviour_hat_Omega_0}
 \Re \wh{\Omega}_0(\Lambda(\lambda)) \leq -1 - c \abs{\lambda}^{1/2} + C \wh{\gamma}^{1/2} \lambda^{2/3}, \qquad 
\abs{\Im \wh{\Omega}_0(\Lambda(\lambda))} \lesssim \wh{\gamma}^{1/2} \lambda^{2/3}, 
\end{equation}
(compare (9.120) in \cite{AjankiQVE}). 
Thus, from \eqref{eq:proof_select_root_two_nearby_edges}, we conclude 
\[ \abs{\lambda}^{1/2} \Im \Omega \lesssim \abs{\lambda} \Im \Omega \] 
for small $\lambda$ as $\abs{\Im \mu(\lambda)} \lesssim \Im \Omega$ by \spfour and $\abs{\Im\Lambda}= \abs{\lambda}\abs{\Im \mu}$. 
Hence, $\Im \Omega(\lambda) = 0$ for small enough $\abs{\lambda}$. Thus, $\Re \Omega$ is non-decreasing 
for such $\lambda$ by \spfour, but from $\Omega(0)=-1$ and the first bound in \eqref{eq:nearby_edges_behaviour_hat_Omega_0}
we conclude that $\Re \Omega$ has to be decreasing if $\Omega(\lambda) = \wh{\Omega}_0(\Lambda(\lambda))$. 
This contradiction shows $a_2 \neq 0$, hence, $a_2 = +$. The rest of the proof in \cite{AjankiQVE} is unchanged. 
\end{proof}

\begin{proof}[Proof of Proposition \ref{pro:solution_cubic_normal_form}]
For the proof of (i), we mainly follow the proof of Proposition 9.3 in \cite{AjankiQVE} with $\gamma_4 = \chi_1$, 
$\gamma_5 = \chi_2$ and $\gamma_6 = \chi_3$ in (9.35) and (9.37) of \cite{AjankiQVE}. 

Following the careful selection of the correct solution of \eqref{eq:cubic_normal_nonzero_minimum} (cf.~(9.36) in \cite{AjankiQVE}) by the selection principles till above (9.50) in \cite{AjankiQVE} 
yields $\Omega(\lambda) = \wh{\Omega}(\Lambda(\lambda))$ 
and hence, in particular, $\wh{\Omega} (\chi_3) = \Omega_0 = \sqrt{3}(\ii + \chi_1)$. ($\wh{\Omega}=\wh{\Omega}_+$ in \cite{AjankiQVE}.)
By defining 
\[ \Lambda_0(\lambda) \defeq ( 1 + \chi_2 + \mu(\lambda))\lambda \] 
and using $\abs{\mu(\lambda)}\lesssim \alpha \abs{\lambda}^{1/3}$ instead of (9.54) in \cite{AjankiQVE}, we obtain 
\[ \wh{\Omega}(\Lambda_0(\lambda)) - \wh{\Omega}(0) = \wh{\Omega}(\lambda) - \wh{\Omega}(0) + \ord\left(
(\abs{\chi_2} + \abs{\mu(\lambda)}) \frac{\abs{\lambda}}{1 + \abs{\lambda}^{2/3}} \right)
 =  \wh{\Omega}(\lambda) - \wh{\Omega}(0) + \ord( (\alpha+\abs{\chi_2}) \min\{\abs{\lambda}, \abs{\lambda}^{2/3} \}) \] 
instead of (9.56) in \cite{AjankiQVE}. 
Thus, (9.57) in the proof of Proposition 9.3 in \cite{AjankiQVE} yields 
\[ \wh{\Omega}(\chi_3 + \Lambda_0(\lambda)) - \wh{\Omega}(\chi_3) = \wh{\Omega}(\lambda) - \wh{\Omega}(0) + 
\ord( (\alpha+\abs{\chi_2}+\abs{\chi_3}) \min\{\abs{\lambda},\abs{\lambda}^{2/3}\}).\]
Thus, we obtain \eqref{eq:normal_form_expansion_local_minimum} since $\wh{\Omega}(\chi_3) = \Omega_0$ and $\wh{\Omega}(0) = \ii\sqrt{3}$. We remark that \eqref{eq:bound_Omega_normal_nonzero_minimum} is exactly (9.53) in \cite{AjankiQVE}. 

The proof of (ii) resembles the proof of Lemma 9.11 in \cite{AjankiQVE} but we replace assumption \textbf{SP-4} of \cite{AjankiQVE} by \spfour. 
Since $\Omega(\lambda)$ solves \eqref{eq:cubic_normal_simple_edge}, there is a function 
$A\colon \R \to \{ \pm \}$ such that $\Omega(\lambda) = \wt{\Omega}_{A(\lambda)}(\Lambda(\lambda))$ 
for all $\lambda \in \R$. 
Here, $\wt{\Omega}_\pm \colon \C \to \C$ denote the functions 
\[ \wt{\Omega}_\pm(\zeta) \defeq \pm \begin{cases} \ii \zeta^{1/2}, & \text{if } \Re \zeta \geq 0, \\ -(-\zeta)^{1/2}, & \text{if } \Re\zeta <0.\end{cases} \]
(Note that they were denoted by $\wh{\Omega}_\pm$ in (9.78) of \cite{AjankiQVE}). 
By assumption, there is $c_*\sim 1$ such that $\abs{\mu(\lambda)} < 1$ for all $ \abs{\lambda} \leq c_*\gamma^{-2}$. 
Hence, by \spone, we find $a_+, a_- \in \{\pm\}$ such that $A(\lambda) = a_\pm$ for $\lambda \in \pm[0,c_*\gamma^{-2}]$. 

For $\lambda \geq 0$, we have 
\[ \Im \wt{\Omega}_-(\Lambda(\lambda)) = - \lambda^{1/2} + \ord(\mu(\lambda)\lambda^{1/2}). \] 
Thus, possibly shrinking $c_*\sim 1$, we obtain $\Im \wt{\Omega}_-(\Lambda(\lambda)) <0$ for $\lambda \in (0,c_*\gamma^{-2}]$. Therefore, the choice $a_+ = -$ would contradict \spthree and we conclude $a_+ = +$. 

We now prove that $a_-=+$. Assume to the contrary that $a_- = -$. For small enough $c_* \sim 1$, we have 
\[ \begin{aligned} 
\Re \wt{\Omega}_- (\Lambda(\lambda)) & = \abs{\lambda}^{1/2} \Re ( 1+ \mu(\lambda))^{1/2} \sim \abs{\lambda}^{1/2},\\
\Im \wt{\Omega}_-(\Lambda(\lambda)) & = \abs{\lambda}^{1/2} \Im ( ( 1 + \mu(\lambda))^{1/2}) \lesssim \abs{\lambda}^{1/2}
\end{aligned} \]
for $\lambda \in [-c_*\gamma^{-2},0)$ by the definition of $\wt{\Omega}_-$ and $\Lambda$. 
Hence, taking the imaginary part of \eqref{eq:cubic_normal_simple_edge} and using \spfour yield 
\[ \abs{\lambda}^{1/2} \Im \Omega(\lambda) \lesssim \gamma\abs{\lambda} \Im \Omega(\lambda) \] 
for $\lambda \in [-c_*\gamma^{-2},0)$. By possibly shrinking $c_* \sim 1$, we obtain 
$\Im \Omega(\lambda) = 0$ for $\lambda \in [-c_*\gamma^{-2}, 0)$. Thus, \spfour implies that 
$\Re \Omega$ is non-decreasing on $[-c_*\gamma^{-2},0)$ which contradicts $\Re \wt{\Omega}_-(0)=0$ and $\Re \wt{\Omega}_- (\Lambda(\lambda)) \sim \abs{\lambda}^{1/2} >0$ 
for $\lambda \in [-c_*\gamma^{-2},0)$ with small enough $c_* \sim 1$. 
Hence, $a_-=+$ which completes the selection of the main term $\wh{\Omega}=\wt{\Omega}_+$ in \eqref{eq:normal_form_expansion_simple_edge}. The error term in \eqref{eq:normal_form_expansion_simple_edge} 
follows by estimating $\wh{\Omega}(\Lambda(\lambda))$ directly. 

For the proof of (iii), we select the correct root of \eqref{eq:cubic_normal_sharp_cusp} as in the proof of Lemma 9.12 in \cite{AjankiQVE} under \spfour instead of \textbf{SP-4}.
Since $\Omega(\lambda)$ solves \eqref{eq:cubic_normal_sharp_cusp} there is a function $A \colon \R \to 
\{ 0, \pm\}$ such that 
\[ \Omega(\lambda) = \wt{\Omega}_{A(\lambda)} ( \Lambda(\lambda)) \] 
for all $\lambda \in \R$. Here, we introduced the functions $\wt{\Omega}_a\colon \C \to \C$, $a=0,\pm$, defined by 
\[ \wt{\Omega}_0 \defeq - \begin{cases} \zeta^{1/3}, & \text{if } \Re \zeta \geq 0, \\ -(-\zeta)^{1/3}, & \text{if } \Re \zeta < 0, \end{cases}\qquad \wt{\Omega}_\pm (\zeta) \defeq \frac{1 \mp \ii \sqrt{3}}{2} \wt{\Omega}_0(\zeta). \] 
(Note that they were denoted by $\wh{\Omega}_a$, $a \in \{0,\pm\}$, in (9.87) of \cite{AjankiQVE}.)
By \spone, $A$ can only change its value at $\lambda$ if $\Lambda(\lambda) = 0$. By choosing $\delta \sim 1$ 
small enough and using $\abs{\mu(\lambda)} \lesssim \abs{\lambda}^{1/3}$, we have $A(\lambda) = a_+$ and 
$A(-\lambda) = a_-$ for some constants $a_\pm$ and for all $\lambda \in (0,\delta]$. 

We will now use \spthree and \spfour to determine the value of $a_+$ and $a_-$. 
As in (9.91) of the proof of Lemma 9.12 in \cite{AjankiQVE}, we have
\[ \pm (\sign \lambda) \Im \wt{\Omega}_\pm(\Lambda(\lambda)) = \frac{\sqrt{3}}{2} \abs{\lambda}^{1/3} 
+ \ord( \mu(\lambda) \lambda^{1/3}) \geq \abs{\lambda}^{1/3} - C \abs{\lambda}^{2/3}. \] 
By possibly shrinking $\delta \sim 1$, we conclude $\Im \wt{\Omega}_-(\Lambda(\lambda)) < 0$ for $\lambda 
\in (0,\delta]$ and $\Im \wt{\Omega}_+ (\Lambda(\lambda)) <0$ for $\lambda \in [-\delta, 0)$. 
Hence, owing to \spthree, we conclude $a_+ \neq - $ and $a _- \neq + $. 

Next, we will prove $a_+ \neq 0$. For $\lambda \geq 0$, we have 
\[ \Re \wt{\Omega}_0(\Lambda(\lambda)) \leq - \lambda^{1/3} + C \lambda^{2/3}, \qquad \Im \wt{\Omega}_0(\Lambda(\lambda)) 
\lesssim \lambda^{2/3} . \] 
Thus, assuming $\Omega(\lambda) = \wt{\Omega}_0(\Lambda(\lambda))$ and estimating the imaginary part of \eqref{eq:cubic_normal_sharp_cusp} yield 
\[ \lambda^{2/3} \Im \Omega(\lambda) \lesssim (\Im \Omega(\lambda))^3 + \abs{\Im \Lambda(\lambda)} 
\lesssim \abs{\lambda} \Im \Omega(\lambda). \] 
Hence, we possibly shrink $\delta \sim 1$ and conclude $\Im \Omega(\lambda)=0$ for $\lambda \in [0,\delta]$. 
Therefore, $\Re \Omega(\lambda)$ is non-decreasing on $[0,\delta]$ by \spfour. Combined with $\Omega_0=0$ and 
$\Re \wt{\Omega}_0(\Lambda(\lambda)) \lesssim - \lambda^{1/3}$, we obtain a contradiction. Hence, this implies $a_+ \neq 0$, i.e., $a_+ = +$. 

A similar argument excludes $a_- = 0$ and we thus obtain $a_- = -$. Now, \eqref{eq:normal_form_expansion_sharp_cusp} is obtained from the definition of $\wh{\Omega} = \wt{\Omega}_+$,
 which completes the proof of (iii). 

For the proof of (iv), we remark that all estimates follow from Lemma \ref{lem:choice_roots_two_nearby_edges} in the same way as they followed in \cite{AjankiQVE} from Lemma 9.14 in \cite{AjankiQVE}. 
Indeed, \eqref{eq:normal_form_expansion_nearby_edges} is the same as (9.129) in \cite{AjankiQVE}. The bound \eqref{eq:normal_coordinates_small_gap_estimate_Omega} is shown analogously to (9.129) and (9.130) in~\cite{AjankiQVE}. 
Moreover, \eqref{eq:Im_Omega_small_in_small_gap} is (9.137) in \cite{AjankiQVE} and \eqref{eq:Im_Omega_-_sign_lambda_3_positive} is obtained as (9.109) in~\cite{AjankiQVE}. 
This completes the proof of Proposition~\ref{pro:solution_cubic_normal_form}. 
\end{proof}

\subsection{Proof of Theorem \ref{thm:abstract_cubic_equation}}  \label{subsec:proof_abstract_cubic} 

Before we prove Theorem \ref{thm:abstract_cubic_equation}, we collect some properties of $\Psie$ and $\Psim$ 
which will be useful in the following. We recall that $\Psie$ and $\Psim$ were defined in \eqref{eq:def_Psim_Psie}. 

\begin{lemma}[Properties of $\Psim$ and $\Psie$]  \label{lem:prop_Psie_Psim}
\begin{enumerate}[label=(\roman*)]
\item Let $\wh{\Omega}$ be defined as in Proposition \ref{pro:solution_cubic_normal_form} (i). 
Then, for any $\lambda \in \R$, we have 
\begin{equation} \label{eq:representation_Psim}
 \Psim(\lambda) = \frac{1}{\sqrt 3} \Im [ \wh{\Omega}(\lambda) - \wh{\Omega}(0)].
\end{equation}
\item
Let $\wh{\Omega}_+$ be defined as in \eqref{eq:def_cubic_root_small_gap_normal_form}. Then, for any $\lambda \geq 0$, we have
\begin{equation} \label{eq:representation_Psie}
 \Psie(\lambda) = \frac{1}{2 \sqrt{3}} \Im \wh{\Omega}_+(1+2 \lambda).
\end{equation}
\item There is a function $\wt{\Psi} \colon [0,\infty) \to \R$ with uniformly bounded derivatives and $\wt{\Psi}(0)=0$ 
such that, for any $\lambda \geq 0$, we have
\begin{equation} \label{eq:Psie_close_to_square_root}
 \Psie(\lambda) = \frac{\lambda^{1/2}}{3} ( 1+ \wt{\Psi}(\lambda)), \qquad \abs{\wt{\Psi}(\lambda)} \lesssim \min\{ \lambda, \lambda^{1/3} \} . 
\end{equation}
\item There is $\eps_* \sim 1$ such that if $\abs{\eps} \leq \eps_*$ then, for any $\lambda \geq 0$, we have 
\begin{equation} \label{eq:Psie_rescale_argument}
 \Psie( (1 + \eps) \lambda ) = ( 1 + \eps)^{1/2} \Psie(\lambda) + \ord(\eps \min \{ \lambda^{3/2}, \lambda^{1/3}\}). 
\end{equation}
\end{enumerate}
\end{lemma} 

We remark that \eqref{eq:representation_Psie} was present in (9.127) of \cite{AjankiQVE} but the coefficient $1/(2\sqrt{3})$ was erroneously missing there. 
The relation in \eqref{eq:Psie_rescale_argument} is identical to (9.145) in \cite{AjankiQVE}. Moreover, we use the proof of \cite{AjankiQVE}.

\begin{proof}
The parts (i), (ii) and (iii) are direct consequences of the definitions of $\Psim$, $\wh{\Omega}$, $\Psie$ and $\wh{\Omega}_+$. 

For the proof of (iv), we choose $\eps_* \leq 1/2$ such that $1 + \eps \sim 1$ for $\abs{\eps} \leq \eps_*$. If $0 \leq \lambda \lesssim 1$ then 
\eqref{eq:Psie_rescale_argument} follows from \eqref{eq:Psie_close_to_square_root}. For $\lambda \gtrsim 1$, we choose $\eps_* = 1/3$ and then \eqref{eq:Psie_rescale_argument} is a consequence of \eqref{eq:representation_Psie} 
above as well as the stability of Cardano's solutions, (9.111) in Lemma 9.17 of \cite{AjankiQVE}. 
\end{proof}

In the following proof of Theorem \ref{thm:abstract_cubic_equation}, we will choose appropriate normal coordinates $\Omega$ and $\Lambda$ in each case
such that \eqref{eq:abstract_cubic} turns into one of the cubic equations in normal form from Proposition \ref{pro:solution_cubic_normal_form}. 
This procedure has been similarly performed in the proofs of Proposition 9.3, Lemma 9.11, Lemma 9.12 and Section 9.2.2 in \cite{AjankiQVE}. 
However, owing to the weaker error bounds here, we include the proof for the sake of completeness. 

\begin{proof}[Proof of Theorem \ref{thm:abstract_cubic_equation}]
We start with the proof of part (i) (cf.~Proposition 9.3 in \cite{AjankiQVE}).
Owing to~\eqref{eq:bound_theta_abstract_cubic} and $\abs{\Psim(\lambda)}\lesssim \abs{\lambda}^{1/3}$, the statement of 
\eqref{eq:Im_Theta_nonzero_minimum} is trivial for $\abs{\omega} \gtrsim 1$ since the error term dominates. 
Therefore, it suffices to prove \eqref{eq:Im_Theta_nonzero_minimum} for $\abs{\omega}\leq \delta$ with some $\delta\sim 1$. 

By possibly shrinking $\rho_* \sim 1$, we can assume that $\abs{\sigma} \leq \Pi_* \rho_*^2$ is small enough such that $\psi \sim 1$ by \eqref{eq:abstract_cubic_psi_sigma_sim_1}.  
In the following, we will choose $\omega$-independent complex numbers $\gamma_\nu, \gamma_0, \gamma_1, \ldots, \gamma_7 \in \C$ 
such that certain relations hold. For each choice, it is easily checked that $\abs{\gamma_k} \lesssim \rho$ for $k = \nu, 0, 1, \ldots, 7$. 
We divide \eqref{eq:abstract_cubic} by $\mu_3$ and obtain 
\begin{equation} \label{eq:cubic_sing_dens_nonzero_minimum_abstract}
 \Theta^3 + \ii 3 \rho ( 1+ \gamma_2) \Theta^2 - 2 \rho^2 ( 1+ \gamma_1) \Theta + ( 1+ \gamma_0+(1 + \gamma_\nu) \nu(\omega)) \frac{\kappa}{\psi} \omega = 0,  
\end{equation}
using $\abs{\mu_3} \sim 1$ and $\abs{\sigma} \leq \Pi_* \dens^2$. 
We introduce the normal coordinates 
\begin{equation} \label{eq:normal_coordinates_nonzero_minimum}
 \lambda \defeq \Gamma \frac{\omega}{\rho^3}, \qquad \Omega(\lambda) \defeq \sqrt{3} \bigg[ ( 1+\gamma_3) \frac{1}{\rho} \Theta \bigg ( \frac{\rho^3}{\Gamma} \lambda\bigg) + \ii + \gamma_4 \bigg], 
\end{equation}
where $\Gamma \defeq \sqrt{27} \kappa / (2 \psi)$. Note that $\Gamma \sim 1$ since $\psi \sim 1$. 
A straightforward computation starting from \eqref{eq:cubic_sing_dens_nonzero_minimum_abstract} shows that $\Omega(\lambda)$ and $\Lambda(\lambda)$ satisfy \eqref{eq:cubic_normal_nonzero_minimum} with 
\[ \Lambda (\lambda) \defeq ( 1+ \gamma_5 + \mu(\lambda))\lambda + \gamma_6, \qquad \mu(\lambda)\defeq ( 1+ \gamma_7) \nu\bigg( \frac{\rho^3}{\Gamma} \lambda \bigg), \] 
i.e., $\chi_2 = \gamma_5$, $\chi_3 = \gamma_6$ and $\alpha = \rho$ by \eqref{eq:bound_nu_cubic}. 
Hence, from \eqref{eq:normal_form_expansion_local_minimum} and \eqref{eq:normal_coordinates_nonzero_minimum}, we obtain $\delta \sim 1$
and $\chi_* \sim 1$ such that
\[ \Im \Theta(\omega) = \Im \frac{\rho}{1 + \gamma_3}\frac{1}{\sqrt{3}} [ \Omega(\lambda) - \Omega_0] = \rho \Psim\bigg(\Gamma\frac{\omega}{\rho^3} \bigg) + 
\ord\left( \rho^2 \min\{ \abs{\lambda},\abs{\lambda}^{1/3} \} + \rho^2 \min \{ \abs{\lambda}, \abs{\lambda}^{2/3} \} \right) \]
for $\abs{\lambda} \leq \delta/\rho^3$ if $\rho \leq \min \{\chi_*, \rho_*\}$.
Here, we also used \eqref{eq:bound_Omega_normal_nonzero_minimum} to expand $\rho/(1+\gamma_3)$ and \eqref{eq:representation_Psim}.  
By employing \eqref{eq:normal_coordinates_nonzero_minimum} again and replacing $\rho_*$ by $\min\{ \chi_*,\rho_*\}$, 
we conclude \eqref{eq:Im_Theta_nonzero_minimum}.

We now turn to the proof of part (ii) of Theorem \ref{thm:abstract_cubic_equation}. 
Since $\rho=0$, the cubic equation \eqref{eq:abstract_cubic} simplifies to the following equation
\begin{equation} \label{eq:cubic_vanishing_density}
 \psi \Theta(\omega)^3 + \sigma \Theta(\omega)^2 + \kappa( 1+ \nu(\omega))\omega = 0. 
\end{equation}

We now prove Theorem \ref{thm:abstract_cubic_equation} (ii) (a), i.e., the case $\sigma = 0$ (cf.~Lemma 9.12 in \cite{AjankiQVE}). 
For any $\delta \sim 1$, the assertion is trivial for $\abs{\omega} \geq \delta$ since the error term dominates $\abs{\omega}^{1/3}$ 
and $\Im \Theta(\omega)$ in this case (compare \eqref{eq:bound_theta_abstract_cubic}). Therefore, it suffices to prove the lemma 
for $\abs{\omega} \leq \delta$ with some $\delta \sim 1$. We choose the normal coordinates 
\[ \lambda \defeq \omega, \qquad {\Omega}(\lambda) \defeq \bigg(\frac{\psi}{\kappa}\bigg)^{1/3} \Theta(\lambda), \]
and notice that the cubic equation \eqref{eq:cubic_vanishing_density} becomes 
\eqref{eq:cubic_normal_sharp_cusp} with $\mu(\lambda) = \nu(\lambda)$. 
The bound \eqref{eq:bound_nu_cubic} implies $\abs{\mu(\lambda)} \lesssim \abs{\lambda}^{1/3}$. 
Thus,~\eqref{eq:cusp_abstract_cubic} is a consequence of Proposition \ref{pro:solution_cubic_normal_form} (iii). 
This completes the proof of (ii) (a). 

For the proof of Theorem \ref{thm:abstract_cubic_equation} (ii) (b), we first show the following auxiliary lemma (cf.~Lemma~9.11 in~\cite{AjankiQVE}). 

\begin{lemma}[Simple edge]  \label{lem:simple_edge}
Let the assumptions of Theorem \ref{thm:abstract_cubic_equation} (ii) hold true. 
If $\sigma \neq 0$ then there is $c_* \sim 1$ such that 
\begin{equation}  \label{eq:simple_edge}
\Im \Theta(\omega) = \begin{cases} \sqrt{\kappa} \absB{\displaystyle\frac{\omega}{\sigma}}^{1/2} + \ord\Big( \Big(\abs{\nu(\omega)} + \abs{\sigma}^{-1} \abs{\Theta(\omega)}\Big) \absB{\frac{\omega}{\sigma}}^{1/2} \Big) , 
& \text{ if }  \sign \omega = \sign \sigma, ~\abs{\omega} \leq c_* \abs{\sigma}^3, \\ 
0, & \text{ if } \sign \omega = - \sign \sigma, ~\abs{\omega} \leq c_* \abs{\sigma}^3. \end{cases}
\end{equation}
Moreover, we have $\abs{\Theta(\omega)} \lesssim \abs{\omega/\sigma}^{1/2}$ for $\abs{\omega} \leq c_* \abs{\sigma}^3$.
\end{lemma}

\begin{proof} 
Dividing \eqref{eq:cubic_vanishing_density} by $\kappa \sigma$ yields 
\begin{equation} \label{eq:proof_simple_edge_aux2}
 \bigg(1 + \frac{\psi}{\sigma}\Theta(\omega)\bigg) \frac{\Theta(\omega)^2}{\kappa} + ( 1+ \nu(\omega)) \frac{\omega}{\sigma} = 0.
\end{equation}
We introduce $\lambda$, $\Omega(\lambda)$ and $\mu(\lambda)$ defined by 
\[ \lambda \defeq \frac{\omega}{\sigma}, \qquad {\Omega}(\lambda) \defeq \frac{1}{\sqrt{\kappa}} \Theta(\sigma \lambda),\qquad 
 \mu(\lambda) \defeq \frac{1 + \nu(\sigma\lambda)}{1 + \psi\sigma^{-1} \Theta(\sigma \lambda) } -1. \] 
In the normal coordinates $\lambda$ and $\Omega(\lambda)$, 
\eqref{eq:proof_simple_edge_aux2} viewed as a quadratic equation, fulfills 
\eqref{eq:cubic_normal_simple_edge} with the above choice of $\mu(\lambda)$. 
Since $\abs{\psi\sigma^{-1} \Theta(\sigma \lambda)} \lesssim \abs{\sigma}^{-2/3} \abs{\lambda}^{1/3}$ by 
\eqref{eq:bound_theta_abstract_cubic}, there is $c_* \sim 1$ such that 
\begin{equation} \label{eq:proof_simple_edge_aux1}
 \abs{\mu(\lambda)} \lesssim \abs{\nu(\sigma \lambda)} + \abs{\sigma}^{-1} \abs{\Theta(\sigma\lambda)} \lesssim \abs{\sigma}^{-2/3}\abs{\lambda}^{1/3}, \qquad \abs{\Im \mu(\lambda)} \lesssim \abs{\sigma}^{-1} \Im \Theta(\sigma \lambda)  
\end{equation}
for $\abs{\lambda} \leq c_* \abs{\sigma}^2$ by \eqref{eq:bound_nu_cubic}, \eqref{eq:bound_theta_abstract_cubic} 
and \eqref{eq:abstract_cubic_Im_nu_lesssim_Im_Theta}. 
Hence, we apply Proposition \ref{pro:solution_cubic_normal_form} (ii) with $\gamma \sim \abs{\sigma}^{-1}$ in \spfour and obtain 
\eqref{eq:simple_edge} with an error term $\ord(\abs{\mu(\lambda)}\abs{\lambda}^{1/2})$ instead, as well as $\abs{\Theta(\omega)} 
\lesssim \abs{\sigma}^{-1/2}\abs{\omega}^{1/2}$. 
Thus, 
the first bound in \eqref{eq:proof_simple_edge_aux1} completes the proof of \eqref{eq:simple_edge}.
\end{proof} 

From the second case in \eqref{eq:simple_edge}, we conclude the second case in \eqref{eq:abstract_cubic_small_gap}. 
The first case in \eqref{eq:abstract_cubic_small_gap} and \eqref{eq:abstract_small_gap_abs_Theta_bound} are trivial if $\abs{\omega} \gtrsim 1$ due to \eqref{eq:bound_theta_abstract_cubic} and \eqref{eq:scaling_psie}. 
Hence, it suffices to prove this case for $\abs{\omega} \leq \delta$ with some $\delta \sim 1$. 
If $\abs{\sigma} \gtrsim 1$ then the first case in \eqref{eq:abstract_cubic_small_gap} also follows from \eqref{eq:simple_edge} with $\delta \defeq c_* \abs{\sigma}^3$. 
Indeed, from \eqref{eq:Psie_close_to_square_root}, we conclude 
\[ \sqrt{\kappa} \absbb{\frac{\omega}{\sigma}}^{1/2} = c \wh{\Delta}^{1/3} \Psie\bigg( \frac{\abs{\omega}}{\wh{\Delta}}\bigg) + \ord(\abs{\omega}^{3/2}), \] 
where $c$ and $\wh{\Delta}$ are defined as in \eqref{eq:def_wh_Delta_abstract}. Since $\abs{\omega} \lesssim \eps(\omega)$ for $\abs{\omega} \leq \delta$ and $\eps(\omega)$ defined as in \eqref{eq:def_wh_Delta_abstract}
we obtain the first case in \eqref{eq:abstract_cubic_small_gap} if $\abs{\sigma} \gtrsim 1$. Similarly, $\abs{\Theta(\omega)} \lesssim \abs{\omega/\sigma}^{1/2}$ by Lemma \ref{lem:simple_edge} yields 
\eqref{eq:abstract_small_gap_abs_Theta_bound} if $\abs{\omega} \leq \delta$ and $\abs{\sigma} \gtrsim 1$. 
Hence, it remains to show the first case in \eqref{eq:abstract_cubic_small_gap} and \eqref{eq:abstract_small_gap_abs_Theta_bound} if $\abs{\sigma} \leq \sigma_*$ for some $\sigma_* \sim 1$. 
In fact, we 
 choose $\sigma_* \sim 1$ so small that $\psi \sim 1$ by \eqref{eq:abstract_cubic_psi_sigma_sim_1} and $\wh{\Delta} < 1$ for $\abs{\sigma} \leq \sigma_*$. 
In order to apply Proposition \ref{pro:solution_cubic_normal_form} (iv), we introduce 
\begin{equation} \label{eq:normal_coordinates_small_gap}
 \lambda \defeq \frac{2}{\wh{\Delta}}\omega, \qquad {\Omega}(\lambda) \defeq 3 \frac{\psi}{\abs{\sigma}} \Theta\bigg( \frac{\wh{\Delta}}{2} \lambda\bigg) + \sign \sigma, \qquad \mu(\lambda) \defeq \nu\bigg( \frac{\wh{\Delta}}{2} \lambda \bigg) 
\end{equation}
(cf.~(9.96) and (9.99) in \cite{AjankiQVE}). 
The cubic \eqref{eq:cubic_vanishing_density} takes the form \eqref{eq:cubic_normal_two_nearby_edges} in the normal coordinates $\lambda$ 
and $\Omega(\lambda)$ with the above choice of $\mu(\lambda)$ and $s = \sign \sigma$ in \eqref{eq:cubic_normal_two_nearby_edges}. 
By \eqref{eq:bound_nu_cubic}, we have $\abs{\mu(\lambda)} \lesssim \wh{\Delta}^{1/3} \abs{\lambda}^{1/3}$. 
We set $\wh{\gamma} \defeq \wh{\Delta}^{1/3}$. 
Therefore, Proposition \ref{pro:solution_cubic_normal_form} (iv) and \eqref{eq:representation_Psie} yield $\delta \sim 1$ and possibly smaller 
$\sigma_*\defeq \min\{\sigma_*, \gamma_*\} \sim 1$ such that the first case in 
\eqref{eq:abstract_cubic_small_gap} holds true for $\abs{\sigma} \leq \sigma_*$ and $\abs{\omega} \leq \delta$ as $\mu(\lambda) = \nu(\omega)$ and $\wh{\Delta} \sim \abs{\sigma}^3$. 
Moreover, \eqref{eq:normal_coordinates_small_gap_estimate_Omega} implies \eqref{eq:abstract_small_gap_abs_Theta_bound} for $\abs{\omega} \leq \delta$. 
This completes the proof of (ii) (b) and hence of Theorem \ref{thm:abstract_cubic_equation}. 
\end{proof} 

\subsection{Proof of Theorem \ref{thm:behaviour_v_close_sing} and Proposition \ref{pro:behaviour_v_close_sing_weak_conditions}}  \label{subsec:proof_behaviour_v_small} 

In this section, we prove Theorem \ref{thm:behaviour_v_close_sing} and Proposition \ref{pro:behaviour_v_close_sing_weak_conditions}.
Some parts of the following proof resemble the proofs of Theorem 2.6, Proposition 9.3 and Proposition 9.8 in~\cite{AjankiQVE}. 
However, owing to the weaker assumptions, we present it here for the sake of completeness.

\begin{proof}[Proof of Theorem \ref{thm:behaviour_v_close_sing} and Proposition \ref{pro:behaviour_v_close_sing_weak_conditions}] 
We will only prove the statements in Proposition \ref{pro:behaviour_v_close_sing_weak_conditions}. Theorem \ref{thm:behaviour_v_close_sing} is a direct consequence of this proposition as well 
as Lemma \ref{lem:q_bounded_Im_u_sim_avg} (ii) and Proposition \ref{pro:cubic_for_dyson_equation}. 

Along the proof of Proposition~\ref{pro:behaviour_v_close_sing_weak_conditions}, we will shrink $\delta_* \sim 1$ such that \eqref{eq:expansion_local_minima_general} holds true for all $\omega \in [-\delta_*,\delta_*] \cap \shapeJ \cap D$.
We will transfer the expansions of $\Theta$ in Theorem \ref{thm:abstract_cubic_equation} to expansions of $v$ by means of \eqref{eq:def_admiss_decom_diff_m}. 
To that end, we take the imaginary part of \eqref{eq:def_admiss_decom_diff_m} and obtain
\begin{equation} \label{eq:proof_behaviour_v_aux1}
 v(\tau_0 + \omega) = v(\tau_0) + \pi^{-1}\Re b \Im \Theta(\omega) + \pi^{-1}\Im b \Re \Theta(\omega) + \pi^{-1}\Im r(\omega). 
\end{equation}

We first establish \eqref{eq:expansion_local_minima_general} at a shape regular point $\tau_0 \in (\supp\dens) \setminus \pt \supp \dens$ which is a local minimum of $\tau \mapsto \dens(\tau)$. 
If $\dens = \dens(\tau_0) = 0$, i.e., the case of a cusp at $\tau_0$, case (c), then $\sigma =0$. 
Indeed, if $\sigma$ were not $0$, then, by the second case in \eqref{eq:abstract_cubic_small_gap}, $\Im \Theta(\omega)$ would vanish on one side of $\tau_0$.
By the third bound in \eqref{eq:def_admiss_dens_zero_Im_Theta_bounds}, this would imply the vanishing of $\dens$ as well, contradicting to $\tau_0 \in \supp\dens \setminus \pt \supp \dens$. 
Hence, for any $\delta_* \sim 1$, \eqref{eq:cusp_abstract_cubic} and \eqref{eq:proof_behaviour_v_aux1} immediately yield 
\eqref{eq:expansion_local_minima_general} for $\omega \in [-\delta_*, \delta_*] \cap \shapeJ  \cap D$ with 
$h = (2\pi)^{-1}b \sqrt{3} (\kappa/\psi)^{1/3}$  
 using \eqref{eq:def_admiss_bound_Theta}, \eqref{eq:def_admiss_bound_r} and $b = b^*$ due to $\dens = 0$. 

We now assume $\dens>0$ which corresponds to an internal nonzero minimum at $\tau_0$, case (d). 
Thus, the following lemma implies that the condition $\abs{\sigma} \leq \Pi_*\dens^2$, $\sigma=\sigma(\tau_0)$,
needed to apply Theorem \ref{thm:abstract_cubic_equation} (i) is fulfilled. 
We will prove Lemma \ref{lem:sigma_nonzero_local_minimum} at the end of this section. 

\begin{lemma}[Bound on $\abs{\sigma}$ at nonzero local minimum] \label{lem:sigma_nonzero_local_minimum} 
There are thresholds $\dens_* \sim 1$ and $\Pi_* \sim 1$ such that 
\[ \abs{\sigma(\tau_0)} \leq \Pi_* \dens(\tau_0)^2 \] 
for each shape regular point $\tau_0 \in \supp\dens$ which is a local minimum of $\dens$ and satisfies  
 $0 < \dens(\tau_0) \leq \dens_*$. 
\end{lemma}

Hence, \eqref{eq:Im_Theta_nonzero_minimum}, \eqref{eq:proof_behaviour_v_aux1} and \eqref{eq:def_admiss_bound_r} yield 
\eqref{eq:expansion_local_minima_general} 
with $\wt{\dens} = \dens \Gamma^{-1/3}$ and $h = \pi^{-1}\Gamma^{1/3}\Re b$. Here, we also used 
\begin{equation}\label{eq:error_nonzero_local_minimum}
 \dens \abs{\Theta(\omega)} + \abs{\Theta(\omega)}^2 + \abs{\omega} + \min\{ \dens^{-1}\abs{\omega}, \abs{\omega}^{2/3}\}\lesssim \frac{\abs{\omega}}{\dens} \char ( \abs{\omega} \lesssim \dens^3) + \Psi(\omega)^2,  
\end{equation}
 which is a consequence of \eqref{eq:def_admiss_bound_Theta}, \eqref{eq:scaling_psim} for $\abs{\omega} \lesssim 1$, 
as well as $\Re b \sim 1$ and $\Im b = \ord(\dens)$. 
This completes the proof of \eqref{eq:expansion_local_minima_general} for shape regular points $\tau_0 \in (\supp \dens) \setminus \pt \supp \dens$, cases (c) and (d). 

We now turn to the proof of \eqref{eq:expansion_local_minima_general} at an edge $\tau_0$, case (a), i.e., for a shape regular point $\tau_0 \in \pt \supp \dens$. 
We first prove a version of \eqref{eq:expansion_local_minima_general} with $\wh{\Delta}$ in place of $\Delta$, \eqref{eq:proof_behaviour_v_aux2} below. 
In a second step, we then replace $\wh{\Delta}$ by $\Delta$ to obtain \eqref{eq:expansion_local_minima_general}. 

Since $\tau_0 \in \pt\supp \dens$, we have $\dens=\dens(\tau_0) = 0$. Therefore, $v(\tau_0)=0$ since $\avg{\genarg}$ is a faithful trace and $v(\tau_0)$ is positive semidefinite. 
As $\tau_0 \in \pt\supp \dens$, we have $\sigma(\tau_0) \neq 0$. Indeed, assuming $\sigma(\tau_0)=0$, using Theorem~\ref{thm:abstract_cubic_equation}~(ii)~(a), taking the imaginary part of \eqref{eq:def_admiss_decom_diff_m}
as well as applying the third bound in \eqref{eq:def_admiss_dens_zero_Im_Theta_bounds} 
and the second bound in \eqref{eq:def_admiss_bound_Theta} yield the contradiction $\tau_0 \in (\supp \dens) \setminus \pt\supp\dens$. 
Recalling the definitions of $\wh{\Delta}$ and $c$ from \eqref{eq:def_wh_Delta_abstract}, \eqref{eq:proof_behaviour_v_aux1} and~\eqref{eq:abstract_cubic_small_gap} yield
\begin{equation} \label{eq:proof_behaviour_v_aux2}
 v(\tau_0 + \omega) = \pi^{-1}c \wh{\Psi}(\omega)b + \ord(\wh{\Psi}(\omega)^2), \qquad \qquad \wh{\Psi}(\omega) \defeq \wh{\Delta}^{1/3} \Psie\bigg(\frac{\abs{\omega}}{\wh{\Delta}}\bigg)  
\end{equation}
for any $\omega \in [-\delta_*, \delta_*]\cap \shapeJ  \cap D$ with $\sign \omega = \sign \sigma$ and some $\delta_* \sim 1$. 
Here, we also used $b = b^*\sim 1$, the first bound in \eqref{eq:bound_nu}, 
\eqref{eq:abstract_small_gap_abs_Theta_bound} and $\eps(\omega) \sim \wh{\Psi}(\omega)$ by \eqref{eq:scaling_psim} to obtain 
\[ \abs{\Theta(\omega)}^2 + \abs{\omega} + ( \abs{\Theta(\omega)} + \abs{\omega}+\eps(\omega)) \eps(\omega) 
\lesssim \wh{\Psi}(\omega)^2 \] 
for any $\omega \in [-\delta_*, \delta_*]\cap \shapeJ  \cap D$ with $\sign \omega = \sign \sigma$ and some $\delta_* \sim 1$. 
This means that we have shown \eqref{eq:expansion_local_minima_general} with $\Psi$ replaced by~$\wh{\Psi}$. 

We now replace $\wh{\Delta}$ by $\Delta$ in \eqref{eq:proof_behaviour_v_aux2} to obtain \eqref{eq:expansion_local_minima_general}. 
To that end, we first assume that $\abs{\sigma} \gtrsim 1$ and $\Delta \lesssim 1$. 
The second part of \eqref{eq:abstract_cubic_small_gap} implies $\abs{\sigma}^3 \lesssim \Delta \lesssim 1$  
and thus $\abs{\sigma}^3 \sim \Delta \sim 1$. 
Since $\abs{\sigma}^3 \sim \wh{\Delta}$ we conclude $\wh{\Delta} \sim \Delta$. 
Therefore, we obtain 
\[\wh{\Delta}^{1/3}\Psie\bigg( \frac{\abs{\omega}}{\wh{\Delta}}\bigg) 
 = \bigg(\frac{\Delta}{\wh{\Delta}}\bigg)^{1/6} \Delta^{1/3}\Psie\bigg( \frac{\abs{\omega}}{\Delta} \bigg) + \ord(\min \{ \abs{\omega}^{3/2}, \abs{\omega}^{1/3}\}) .  \] 
Here, we used $\Psie(\abs{\lambda}) \lesssim \abs{\lambda}^{1/3}$ for $\abs{\lambda} \gtrsim 1$ and \eqref{eq:Psie_close_to_square_root} otherwise. 
Applying this relation to \eqref{eq:proof_behaviour_v_aux2} yields \eqref{eq:expansion_local_minima_general} for $\omega\in [-\delta_*,\delta_*]\cap \shapeJ  \cap D$ with $\sign \omega = \sign \sigma$, $\delta_* \sim 1$ 
and $h\defeq \pi^{-1}c (\Delta/\wh{\Delta})^{1/6}b\sim 1$ for $\abs{\sigma} \gtrsim 1$ and $\Delta \lesssim 1$.

The next lemma shows that $\abs{\sigma} \gtrsim 1$ at the edge of a gap of size $\Delta\gtrsim 1$.
We postpone its proof until the end of this section. 

\begin{lemma}[$\sigma$ at an edge of a large gap] \label{lem:large_gap}
Let $\tau_0\in\pt\supp \dens$ be a shape regular point for $m$ on $\shapeJ$. If $\abs{\inf \shapeJ} \gtrsim 1$ and there is $\eps \sim 1$ such that $\dens(\tau) = 0$ for 
all $\tau \in [\tau_0-\eps, \tau_0]$ then $\abs{\sigma} \sim 1$. 
We also have $\abs{\sigma} \sim 1$ if $\sup \shapeJ \gtrsim 1$ and $\dens(\tau) = 0$ for all $\tau \in [\tau_0,\tau_0+\eps]$ and some $\eps \sim 1$. 
\end{lemma}

Under the assumptions of the previous lemma, we set $\Delta \defeq 1$ and obtain trivially $\wh{\Delta} \sim 1 \sim \Delta$. Thus, 
\eqref{eq:proof_behaviour_v_aux2} implies \eqref{eq:expansion_local_minima_general} by the same argument as in the case $\Delta \lesssim 1$.

For $\abs{\sigma} \leq \sigma_*$ with some sufficiently small $\sigma_*\sim 1$, we will prove below 
with the help of the following Lemma~\ref{lem:small_gap} and \eqref{eq:Psie_rescale_argument} that replacing $\wh{\Delta}$ by $\Delta$ in \eqref{eq:proof_behaviour_v_aux2} 
yields an affordable error. 
We present the proof of Lemma \ref{lem:small_gap} at the end of this section. 

\begin{lemma}[Size of small gap] \label{lem:small_gap}
Let $\tau_0,\tau_1 \in \pt \supp \dens$, $\tau_1<\tau_0$, be two shape regular points for $m$ on $\shapeJ_0$ and $\shapeJ_1$, 
respectively, where $\shapeJ_0, \shapeJ_1 \subset \R$ are two open intervals with $0 \in \shapeJ_0 \cap \shapeJ_1$. 
We assume $\abs{\inf \shapeJ_0} \gtrsim 1$ and $\sup \shapeJ_1 \gtrsim 1$ as well as $(\tau_1, \tau_0) \cap \supp \dens = \varnothing$. 
We set $\Delta(\tau_0)\defeq \tau_0 - \tau_1$. 
Then there is $\wt{\sigma} \sim 1$ such that if $\abs{\sigma(\tau_0)}\leq \wt{\sigma}$ and $\abs{\sigma(\tau_0) - \sigma(\tau_1)} \lesssim \abs{\tau_0-\tau_1}^{1/3}$ then
\[ \frac{\Delta(\tau_0)}{\wh{\Delta}(\tau_0)} = 1 + \ord(\sigma(\tau_0)). \]
The same statement holds true when $\tau_0$ is replaced by $\tau_1$ with $\Delta(\tau_1) \defeq \tau_0 - \tau_1$. 
\end{lemma}

From Lemma \ref{lem:small_gap}, we conclude that there is $\gamma \in \C$ such that $\abs{\gamma} \lesssim 1$ and $\Delta = ( 1+ \gamma \abs{\sigma}) \wh{\Delta}$. 
By possibly shrinking $\sigma_* \sim 1$, we can assume that $\abs{\gamma \sigma} \leq \eps_*$ for $\abs{\sigma} \leq \sigma_*$, where $\eps_* \sim 1$ is chosen as in Lemma \ref{lem:prop_Psie_Psim} (iv). 
Thus, \eqref{eq:Psie_rescale_argument} yields 
\[ \wh{\Delta}^{1/3}\Psie\bigg(\frac{\abs{\omega}}{\wh{\Delta}}\bigg) = \bigg( \frac{\Delta}{\wh{\Delta}}\bigg)^{1/6} \Delta^{1/3} \Psie\bigg( \frac{\abs{\omega}}{\Delta}\bigg) 
+ \ord\bigg( \min\bigg\{\frac{\abs{\omega}^{3/2}} {\Delta^{5/6}}, \abs{\omega}^{1/3} \bigg\}\bigg). \]
Hence, choosing $h \defeq \pi^{-1}c (\Delta/\wh{\Delta})^{1/6}b$ as before and noticing $h \sim 1$ yields \eqref{eq:expansion_local_minima_general} in the missing regime. This completes the proof of  
Proposition \ref{pro:behaviour_v_close_sing_weak_conditions}. 
As we have already explained, Theorem \ref{thm:behaviour_v_close_sing} follows immediately. 
\end{proof}

The core of the proof of Lemma \ref{lem:sigma_nonzero_local_minimum} is an effective monotonicity estimate on $v$, see \eqref{eq:monotonicity_estimate2} below, which is the analogue of (9.20) in Lemma 9.2 of \cite{AjankiQVE}. 
Owing to the weaker assumptions on the coefficients of the cubic equation, we need to present an upgraded proof here. 
In fact, the bound in (9.20) of \cite{AjankiQVE} contained a typo. It should have read as 
\[ (\sign \sigma(\tau)) \pt_\tau v(\tau) \gtrsim \frac{1}{\avg{v(\tau)} ( 1 + \abs{\sigma(\tau)})} \] 
for $\tau \in \mathbb{D}_{\eps_*}$ satisfying $\Pi(\tau) \geq \Pi_*$. 
However, this does not affect the correctness of the argument in \cite{AjankiQVE}. 

\begin{proof}[Proof of Lemma \ref{lem:sigma_nonzero_local_minimum}]
In the whole proof, we will use the notation of Definition \ref{def:admissible_shape_analysis}. 
We will show below that there are $\dens_* \sim 1$ and $\Pi_* \sim 1$ such that 
\begin{equation} \label{eq:monotonicity_estimate2}
 (\sign \kappa_1\sigma(\tau)) \pt_\tau v(\tau) \gtrsim \dens(\tau)^{-1}  
\end{equation}
for all $\tau \in \R$ which satisfy $\dens(\tau)\in (0,\dens_*]$ and $\abs{\sigma(\tau)} \geq \Pi_* \dens(\tau)^2$ and 
are admissible points for the shape analysis. 

Now, we first conclude the statement of the lemma from \eqref{eq:monotonicity_estimate2} through a proof by contradiction. 
If $\tau_0$ satisfies the conditions of Lemma \ref{lem:sigma_nonzero_local_minimum} then 
 $\pt_\tau \dens(\tau_0) = 0$ as $\tau_0$ is a local minimum of $\dens$. Assuming $\abs{\sigma(\tau_0)} \geq 
\Pi_* \dens(\tau_0)^2$ and applying $\avg{\genarg}$ to  \eqref{eq:monotonicity_estimate2} yield 
the contradiction $\pt_\tau \dens(\tau_0)>0$.

For the proof of \eqref{eq:monotonicity_estimate2} we start by proving a relation for $\pt_\tau v(\tau)$. 
We divide \eqref{eq:def_admiss_decom_diff_m} by $\omega$, use $\Theta(0)=0$ and $r(0)=0$ as well as take the limit $\omega\to 0$ to obtain $\pt_\tau m(\tau) = b \pt_\omega \Theta(0) + \pt_\omega r(0)$. 
Taking the imaginary part of the previous relation yields 
\begin{equation} \label{eq:lem_monotone_aux2}
 \pi \pt_\tau v(\tau) = \Im [ b \pt_\omega\Theta(0)] + \Im \pt_\omega r(0). 
\end{equation}
We divide \eqref{eq:def_admiss_bound_r} by $\omega$, employ the first bound in \eqref{eq:def_admiss_bound_Theta} and obtain 
\[ \normbb{\frac{r(\omega)}{\omega}} \lesssim 1 + \absbb{\frac{\Theta(\omega)}{\omega}}^2 \lesssim 1 + \frac{\abs{\omega}}{\dens^4}. \] 
By sending $\omega \to 0$ and using $r(0) = 0$, we conclude 
\begin{equation} \label{eq:lem_monotone_aux3}
 \norm{\Im \pt_\omega r(0)} \lesssim1.
\end{equation}
We divide \eqref{eq:cubic_sing_dens} by $\mu_1 \omega$, take the limit $\omega\to 0$ and use $\lim_{\omega\to 0} \Theta(\omega)=\Theta(0)=0$ to obtain 
\begin{equation} \label{eq:lem_monotone_aux4}
\begin{aligned} 
\pt_\omega \Theta(0) = -\frac{\Xi(0)\bar{\mu}_1}{\abs{\mu_1}^2} 
= & \frac{(\kappa+\ord(\dens))(\ii \kappa_1\dens\sigma + 2\dens^2\psi + \ord(\dens^3+\dens^2\abs{\sigma}))}{4 \dens^4\abs{\psi+ \ord( \dens + \abs{\sigma})}^2 + \dens^2 \abs{\kappa_1\sigma + \ord(\dens^2+\dens\abs{\sigma})}^2 } \\ 
= & \frac{\kappa}{\dens} \frac{\ii \kappa_1\sigma + 2\dens\psi + \ord(\dens^2 + \dens \abs{\sigma})}{4 \dens^2\abs{\psi+ \ord(\dens+\abs{\sigma})}^2 + \abs{\kappa_1\sigma + \ord(\dens^2+\dens\abs{\sigma})}^2 }, 
\end{aligned} 
\end{equation}
where we employed $\abs{\mu_1}^2 =  4 \dens^4\abs{\psi+ \ord(\dens+\abs{\sigma})}^2 + \dens^2 \abs{\kappa_1\sigma + \ord(\dens^2 + \dens\abs{\sigma})}^2$ as $\rho,\psi,\kappa_1, \sigma \in \R$.
Thus, we obtain 
\begin{equation} \label{eq:lem_monotone_aux5}
 \dens \abs{\Re \pt_\omega \Theta(0)} 
\lesssim \frac{\dens + \dens\abs{\sigma}}{\dens^2\abs{\psi+ \ord( \dens + \abs{\sigma})}^2 + \abs{\kappa_1\sigma + \ord(\dens^2+\dens\abs{\sigma})}^2}. 
\end{equation}
Therefore, using $b = b^* + \ord(\dens)$, $b + b^* \sim 1$, $\kappa \sim 1$ and $\abs{\kappa_1} \sim 1$ yields
\[ (\sign \kappa_1 \sigma)\Im [ b \pt_\omega \Theta(0)] \gtrsim \frac{\dens^{-1}\abs{\sigma} + \ord(\dens + \abs{\sigma}) + \ord( \dens + \dens \abs{ \sigma}) }{\abs{\sigma+ \ord(\dens^2 + \dens\abs{\sigma})}^2 + \dens^2 \abs{\psi + \ord(\dens+\abs{\sigma})}^2} 
\gtrsim \frac{\abs{\sigma}}{\abs{\sigma}^2 + \dens^2} 
\frac{1}{\dens}. \] 
Here, in the first step, the error term $\ord(\dens+ \dens \abs{\sigma})$ in the numerator originates from the second term in 
\begin{equation} \label{eq:lem_monotone_aux6}
\begin{aligned}
 (\sign \kappa_1 \sigma)\Im [ b \pt_\omega\Theta(0)] & = (\sign \kappa_1 \sigma)\Big( \Re b \Im \pt_\omega\Theta(0) + \Im b \Re \pt_\omega\Theta(0) \Big)\\ 
 & \gtrsim  (\sign \kappa_1 \sigma) \Im \pt_\omega\Theta(0) - \dens\abs{\Re \pt_\omega\Theta(0)} 
\end{aligned}
\end{equation}
and applying \eqref{eq:lem_monotone_aux5} to it. We applied \eqref{eq:lem_monotone_aux4} to the first term on the right-hand side of \eqref{eq:lem_monotone_aux6}. 
In the last estimate, we used $\psi, \abs{\sigma}, \dens \lesssim 1$ and $\abs{\sigma} \geq \Pi_* \dens^2$ for some large $\Pi_* \sim 1$ as well as $\dens \leq \dens_*$ for some small $\dens_* \sim 1$. 
Employing $\abs{\sigma} \geq \Pi_* \dens^2$ once more, the factor $\abs{\sigma}/(\abs{\sigma}^2 + \dens^2)$ on the right-hand side scales like $(1 + \abs{\sigma})^{-1} \gtrsim 1$.
Hence, we conclude from \eqref{eq:lem_monotone_aux2} and \eqref{eq:lem_monotone_aux3} that 
\[ (\sign \kappa_1\sigma) \pt_\tau v(\tau) \gtrsim\frac{1}{\dens} + \ord(1). \] 
By choosing $\dens_* \sim 1$ sufficiently small, we obtain \eqref{eq:monotonicity_estimate2}. This completes the proof of Lemma \ref{lem:sigma_nonzero_local_minimum}.
\end{proof}

\begin{proof}[Proof of Lemma \ref{lem:large_gap}]
We prove both cases, $\dens(\tau) = 0$ for all $\tau \in [\tau_0 - \eps, \tau_0]$ 
or for all $\tau \in [\tau_0, \tau_0 +\eps]$, in parallel. 
We can assume that $\abs{\sigma} \leq \wt{\sigma}$ for any $\wt{\sigma} \sim 1$ as the statement trivially holds true otherwise. 
We choose $(\delta,\varrho,\gamma_*)$ as in Proposition \ref{pro:solution_cubic_normal_form} (iv), $\wh{\Delta}$ as in \eqref{eq:def_wh_Delta_abstract}, normal coordinates $(\lambda, \Omega(\lambda))$ 
as in \eqref{eq:normal_coordinates_small_gap} as well as $\wh{\gamma} = \wh{\Delta}^{1/3}$ and $s = \sign \sigma$. 
We set $\lambda_3\defeq 2 + \varrho \wh{\Delta}^{1/3}$ (cf.~\eqref{eq:def_lambdas_small_gap}) and $\omega_3 \defeq \wh{\Delta}\lambda_3 /2$. 
There is $\wt{\sigma} \sim 1$ such that $\wh{\Delta} \leq \gamma_*^3$ for $\abs{\sigma} \leq \wt{\sigma}$ due to $\wh{\Delta} \sim \abs{\sigma}^3$ by \eqref{eq:def_admiss_scaling_parameters} and the 
definition of $\wh{\Delta}$ in \eqref{eq:def_wh_Delta_abstract}. 
Hence, $\omega_3\leq C \abs{\sigma}^3$ and, by possibly shrinking $\wt{\sigma}\sim 1$, we obtain 
$-\omega_3\sign \sigma\in\shapeJ$ for $\abs{\sigma} \leq \wt{\sigma}$ due to the assumption on $\shapeJ$
($\abs{\inf \shapeJ} \gtrsim 1$ or $\sup \shapeJ \gtrsim 1$). 
From \eqref{eq:Im_Omega_-_sign_lambda_3_positive}, we obtain $\Im \Omega(-\lambda_3\sign \sigma) >0$. 
Hence, $\Im \Theta(-\omega_3\sign\sigma)>0$. 
From the third bound in \eqref{eq:def_admiss_dens_zero_Im_Theta_bounds}, the second bound in 
\eqref{eq:def_admiss_bound_Theta} and $\omega_3 \lesssim \abs{\sigma}^3$, we conclude 
$v(-\omega_3\sign\sigma) >0$ for $\abs{\sigma} \leq \wt{\sigma}$ and sufficiently small $\wt{\sigma} \sim 1$. 
Thus, $\dens(-\omega_3 \sign\sigma)>0$ which implies $\omega_3>\eps$. 
Therefore, $\abs{\sigma}^3 \gtrsim \omega_3 > \eps \sim 1$ which completes the proof of Lemma~\ref{lem:large_gap}.
\end{proof} 

We finish this section by proving Lemma \ref{lem:small_gap}. It is similarly proven as Lemma 9.17 in \cite{AjankiQVE}. We present the proof due to the weaker assumptions of Lemma \ref{lem:small_gap}. 
The main difference is the proof of \eqref{eq:size_small_gap_aux2} below (cf.~(9.138) in \cite{AjankiQVE}). In \cite{AjankiQVE}, 
$\Theta$ could be explicitly represented in terms of $m$, i.e, 
\[ \Theta(\omega) = \scalar{f}{m(\tau_0 + \omega)- m(\tau_0)} \]
(cf.~(9.8) and (8.10c) in \cite{AjankiQVE} with $\alpha =0$). In our setup, $b$ and $r$ do not necessarily define an orthogonal decomposition (cf.~\eqref{eq:def_admiss_decom_diff_m}). 

\begin{proof}[Proof of Lemma \ref{lem:small_gap}]
Let $(\delta, \varrho, \gamma_*)$ be chosen as in Proposition \ref{pro:solution_cubic_normal_form} (iv). 
We choose $\wh{\Delta}$ as in \eqref{eq:def_wh_Delta_abstract} and normal coordinates as 
in \eqref{eq:normal_coordinates_small_gap} as well as $\wh{\gamma} = \wh{\Delta}^{1/3}$ and $s = \sign \sigma$. 
We assume $\wh{\Delta} \leq \gamma_*^3$ in the following and 
define $\lambda_3$ as in \eqref{eq:def_lambdas_small_gap}.  
By using $\abs{\inf \shapeJ_0} \gtrsim 1$ as in the proof of Lemma \ref{lem:large_gap}, we find $\wt{\sigma} \sim 1$ 
such that $-\omega_3 \in \shapeJ_0$ for $\omega_3 \defeq \lambda_3 \wh{\Delta}/2$ and $\abs{\sigma} \leq \wt{\sigma}$. 
Thus, $-\Delta = \tau_1- \tau_0 \in \shapeJ_0$. 
We set 
\[ \lambda_0 \defeq \inf\{ \lambda >0 \colon \Im \Omega(\lambda) >0 \} \] 
and remark that $\lambda_0 = 2 \Delta/\wh{\Delta}$ due to the definition of $\Delta$ and the third bound in \eqref{eq:def_admiss_dens_zero_Im_Theta_bounds}.
From \eqref{eq:Im_Omega_-_sign_lambda_3_positive}, we conclude $\lambda_0 \leq \lambda_3$. Thus, 
$\Delta \leq \wh{\Delta}(1 + \ord(\wh{\gamma})) = \wh{\Delta}(1 + \ord(\abs{\sigma}))$ as $\varrho \sim 1$ and 
$\wh{\gamma} \sim \abs{\sigma}$. 
Therefore, it suffices to show the opposite bound, 
\begin{equation} \label{eq:size_small_gap_aux1}
 \Delta \geq \wh{\Delta}( 1+ \ord(\abs{\sigma})). 
\end{equation}
If $\lambda_0 \geq \lambda_2 \defeq 2 - \varrho \wh{\Delta}^{1/3}$ (cf.~\eqref{eq:def_lambdas_small_gap}) then we have \eqref{eq:size_small_gap_aux1} as $\wh{\Delta}^{1/3} \sim 
\abs{\sigma}$ and $\varrho \sim 1$. If $\lambda_0 < \lambda_2$ then we will prove below that 
\begin{equation} \label{eq:size_small_gap_aux2}
 \Im \Omega(\lambda_0 + \xi) \gtrsim \xi^{1/2} 
\end{equation}
for $\xi \in [0,1]$. From \eqref{eq:Im_Omega_small_in_small_gap}, we then conclude 
\[ c_0(\lambda_2 - \lambda_0)^{1/2}  \leq \Im \Omega(\lambda_2) \leq C_1 \abs{\sigma}^{1/2}  \] 
as $\wh{\gamma} \sim \abs{\sigma}$. Hence, 
\[ \lambda_0 \geq \lambda_2 - (C_1/c_0)^2 \abs{\sigma} \geq 2 - C \abs{\sigma}, \] 
where we used $\lambda_2 = 2 - \varrho \wh{\gamma}$ and $\varrho \sim 1$ in the last step. This shows \eqref{eq:size_small_gap_aux1} also in the case $\lambda_0 < \lambda_2$. 
Therefore, the proof of the lemma will be completed once \eqref{eq:size_small_gap_aux2} is proven. 

In order to prove \eqref{eq:size_small_gap_aux2}, we translate it into the coordinates $\omega$ relative to $\tau_0$ and $v$. 
From $\lambda_0 < \lambda_2$, we obtain 
\begin{equation} \label{eq:size_small_gap_aux3}
 \Delta < ( 1 - \varrho \wh{\Delta}^{1/3}) \wh{\Delta} \lesssim \abs{\sigma}^3.  
\end{equation}
Since 
\[ \pi v(\tau_0 - \Delta - \wt{\omega}) = b \Im \Theta( -\Delta- \wt{\omega}) + \Im r( - \Delta - \wt{\omega}), \] 
the bound \eqref{eq:size_small_gap_aux2} would follow from 
\begin{equation} \label{eq:size_small_gap_aux5}
 v (\tau_0 - \Delta - \wt{\omega}) \gtrsim \wh{\Delta}(\tau_0)^{-1/6} \abs{\wt{\omega}}^{1/2} 
\end{equation}
for sufficiently small $\Delta \lesssim \abs{\sigma}^3 \leq \wt{\sigma}^3$ and $\wt{\omega} \leq \wt{\delta}$ due to
the third bound in \eqref{eq:def_admiss_dens_zero_Im_Theta_bounds}.
Since $v(\tau_1) = 0$ and $\tau_1 = \tau_0 - \Delta$ is a shape regular point, we conclude from \eqref{eq:proof_behaviour_v_aux2} that 
\[ v(\tau_1 - \wt{\omega}) \gtrsim \wh{\Delta}(\tau_1)^{-1/6} \abs{\wt{\omega}}^{1/2} \] 
for $\abs{\wt{\omega}} \leq \delta$. 
Therefore, it suffices to show that 
\begin{equation} \label{eq:size_small_gap_aux4}
 \wh{\Delta}(\tau_1) \lesssim \wh{\Delta}(\tau_0) 
\end{equation}
in order to verify \eqref{eq:size_small_gap_aux5}. 
Owing to $\abs{\sigma(\tau_0)-\sigma(\tau_1)} \lesssim \Delta^{1/3}$ and \eqref{eq:size_small_gap_aux3}, we have 
\[ \abs{\sigma(\tau_1)} \lesssim \abs{\sigma(\tau_0)} + \Delta^{1/3} \lesssim \abs{\sigma(\tau_0)}. \]
We allow for a smaller choice of $\wt{\sigma} \sim 1$ and assume $\psi(\tau_1) \sim \psi(\tau_0) \sim 1$ by \eqref{eq:def_admiss_scaling_parameters}. 
Assuming without loss of generality $\wh{\Delta}(\tau_0) < 1$ and $\wh{\Delta}(\tau_1) < 1$, we obtain \eqref{eq:size_small_gap_aux4} 
by the definition of $\wh{\Delta}$ in \eqref{eq:def_wh_Delta_abstract}. 
We thus get \eqref{eq:size_small_gap_aux4} and hence \eqref{eq:size_small_gap_aux5}. This proves \eqref{eq:size_small_gap_aux2} and completes the proof of Lemma \ref{lem:small_gap}.
\end{proof}

\subsection{Proofs of Theorem~\ref{thm:behaviour_dens} and Proposition~\ref{pro:behaviour_dens_weak_conditions}} \label{subsec:proof_behaviour_dens} 

\begin{proof}[Proof of Proposition~\ref{pro:behaviour_dens_weak_conditions}] 
We start with the proof of part (i). 
We apply $\avg{\genarg}$ to \eqref{eq:expansion_local_minima_general}, use $\dens=\avg{v}$ and obtain $\avg{h}$ from the definitions of $h$ in 
the four cases given in the proof of Proposition \ref{pro:behaviour_v_close_sing_weak_conditions}. 
Indeed, by using the relations 
\begin{equation} \label{eq:behaviour_dens_aux1}
 \avg{b} = \pi + \ord(\rho), \qquad c^3 = 4 \Gamma, 
\end{equation}
which are proven below, 
as well as Lemma~\ref{lem:small_gap} in the cases (a) and (b) and the stronger error estimate
\eqref{eq:error_nonzero_local_minimum} in case (d), we conclude part (i) of Proposition~\ref{pro:behaviour_dens_weak_conditions} up to the proof of \eqref{eq:behaviour_dens_aux1}. 

The first relation in \eqref{eq:behaviour_dens_aux1} follows from applying $\avg{\genarg}$ to \eqref{eq:expansion_b} 
and using \eqref{eq:b_0_l_0_approx}, Corollary~\ref{coro:characterization_support} with $\tau_0 \in \supp\dens$, the cyclicity of $\avg{\genarg}$ and \eqref{eq:f_u_qq_star}. 
The second relation in \eqref{eq:behaviour_dens_aux1} is a consequence of the definition of $c$ in \eqref{eq:def_wh_Delta_abstract} and the definition of $\Gamma$ in Theorem~\ref{thm:abstract_cubic_equation} (i). 
This completes the proof of part (i).

We now turn to the proof of part (ii) of Proposition \ref{pro:behaviour_dens_weak_conditions} and assume that all points of $(\pt \supp\dens) \cap I$ are shape regular for $m$ and all estimates in Definition \ref{def:admissible_shape_analysis} 
hold true uniformly on this set. As in the proof of Proposition \ref{pro:behaviour_v_close_sing_weak_conditions}, we conclude $\sigma(\tau_0) \neq 0$ for all $\tau_0 \in (\pt\supp\dens) \cap I$. 
Owing to $\dist(0,\pt\shapeJ) \gtrsim 1$ and the Hölder-continuity of $\sigma$ on $(\pt \supp \dens) \cap I$, Proposition \ref{pro:behaviour_v_close_sing_weak_conditions} is applicable to every $\tau_0 \in (\pt \supp\dens) \cap I$. 
Hence, \eqref{eq:scaling_psie} and $\dist(0,\pt\shapeJ) \gtrsim 1$ imply the existence of $\delta_1, c_1 \sim 1$ such that 
\begin{equation} \label{eq:proof_behaviour_v_aux3}
 \dens(\tau_0 +\omega) \geq c_1 \abs{\omega}^{1/2}  
\end{equation}
for all $\omega \in -\sign\sigma(\tau_0)[0,\delta_1]$ and $\tau_0 \in (\pt\supp\dens)\cap I$. 
In particular, $\tau_0 - \sign\sigma(\tau_0)[0,\delta_1] \subset \supp \dens$ for all $\tau_0 \in (\pt\supp\dens) \cap I$. 
Since $\abs{I} \lesssim 1$, this implies that $\supp \dens \cap I$ consists of finitely many intervals $[\alpha_i,\beta_i]$ 
with lengths $\gtrsim 1$, and, thus, their number $K$ satisfies $K\sim 1$ as $\delta_1 \sim 1$ and $\beta_i - \alpha_i \geq \delta_1$
if $\beta_i \neq \sup I$ and $\alpha_i \neq \inf I$. 

Additionally, we now assume that the elements of $\Mb_{\dens_*}$ are shape regular points for $m$ on $\shapeJ$ and all estimates in Definition \ref{def:admissible_shape_analysis} hold true uniformly on $\Mb_{\dens_*}$. 
By possibly shrinking $\dens_* \sim 1$, we conclude from \eqref{eq:proof_behaviour_v_aux3} that $\abs{\alpha_i - \gamma} \sim 1$ and $\abs{\beta_i-\gamma} \sim 1$ for any $i= 1, \ldots, K$ and $\gamma \in \Mb_{\dens_*}$. 

Suppose now that $\tau_0 \in \Mb_{\dens_*}$ with $\dens(\tau_0)=0$. Then part (i) and $\dist(0,\pt \shapeJ) \gtrsim 1$ yield the existence of $\delta_2, c_2 \sim 1$ such that 
\[ \dens(\tau_0 + \omega) \geq c_2 \abs{\omega}^{1/3} \] 
for all $\abs{\omega} \leq \delta_2$. By possibly further shrinking $\rho_*\sim 1$, we thus obtain $\abs{\tau_0 - \gamma} \sim 1$ for all $\gamma \in \Mb_{\dens_*}\setminus\{\tau_0\}$. 
We thus conclude \eqref{eq:distance_of_internal_minima} in this case. 

Finally, let $\gamma_1, \gamma_2 \in \Mb_{\dens_*}$ with $\dens(\gamma_1),\dens(\gamma_2)>0$. 
Then applying (i) with $\tau_0 = \gamma_1$ and $\tau_0 =\gamma_2$ yields 
\[ \Psi_1(\omega) + \Psi_2(\omega) \lesssim \abs{\omega}^{1/3}\Big(\dens(\gamma_1)\char (\abs{\omega} \lesssim \dens(\gamma_1)^3) + \dens(\gamma_2)\char(\abs{\omega} \lesssim \dens(\gamma_2)^3)\Big) + 
\Psi_1(\omega)^2 + \Psi_2(\omega)^2, \] 
where we defined $\omega = \gamma_2 - \gamma_1$ and 
\[ \Psi_1(\omega) \defeq \wt{\dens}_1 \Psim\bigg(\frac{\abs{\omega}}{\wt{\dens}_1^3}\bigg), \qquad \Psi_2(\omega) \defeq \wt{\dens}_2 \Psim\bigg(\frac{\abs{\omega}}{\wt{\dens}_2^3}\bigg)\]  
with $\wt{\dens}_1 \sim \dens(\gamma_1)$ and $\wt{\dens}_2 \sim \dens(\gamma_2)$
(cf.~Corollary 9.4 in \cite{AjankiQVE}). 
Thus, we obtain either $\abs{\omega} \sim 1$ or $\abs{\omega} \lesssim \min \{ \dens(\gamma_1), \dens(\gamma_2)\}^4$. 
This completes the proof of \eqref{eq:distance_of_internal_minima} and hence the one of Proposition 
\ref{pro:behaviour_dens_weak_conditions}. 
\end{proof}

Finally, we use Proposition~\ref{pro:behaviour_dens_weak_conditions} and a Taylor expansion of $\dens$ around a nonzero local minimum $\tau_0$ 
to obtain the stronger conclusions of Theorem~\ref{thm:behaviour_dens}. 

\begin{proof}[Proof of Theorem~\ref{thm:behaviour_dens}]
We start with the proof of part (i). Let $\tau_0 \in \supp\dens \cap I_\theta$ satisfy the conditions of Theorem~\ref{thm:behaviour_dens} (i). 
Then, by Proposition~\ref{pro:cubic_for_dyson_equation}, the conditions of Proposition~\ref{pro:behaviour_dens_weak_conditions} (i) are fulfilled and all 
conclusions in Theorem~\ref{thm:behaviour_dens} (i) apart from the case $\abs{\omega} \lesssim \dens(\tau_0)^{7/2}$ in \eqref{eq:multiplicative_expansion_nonzero_minimum_general_assums} follow from 
Proposition~\ref{pro:behaviour_dens_weak_conditions} (i) and \eqref{eq:scaling_psim}. 

For the proof of the missing case, we fix a local minimum $\tau_0 \in \supp \dens \cap I_\theta$ of $\dens$ such that $\dens(\tau_0) \leq \dens_*$. 
We set $\dens \defeq \dens(\tau_0)$. 
Owing to the $1/3$-Hölder continuity of $\dens$ by Proposition~\ref{pro:analyticity_of_m}, there is  
$\eps \sim 1$ such that $\dens(\tau_0+\omega) \sim \dens$ if $\abs{\omega} \leq \eps \dens^3$. 
In particular, $\dens(\tau_0+\omega) >0$ 
and using Lemma~\ref{lem:derivatives_m} with $k=2,3$ to compute the second order Taylor expansion of $\dens$ around $\tau_0$ yields 
\begin{equation} \label{eq:f_gamma_exp1}
 f_{\tau_0}(\omega ) \defeq \dens(\tau_0 + \omega) - \dens(\tau_0) = \frac{c}{\dens^5} \omega^2 + \ord\bigg( \frac{\abs{\omega}^3}{\dens^8} \bigg) 
\end{equation}
for all $\omega\in \R$ satisfying $\abs{\omega} \leq \eps \dens^3$, where $c = c(\tau_0)$ satisfies $0 \leq c \lesssim 1$.

On the other hand, $\tau_0$ is a shape regular point by Proposition~\ref{pro:cubic_for_dyson_equation} and a nonzero local minimum of $\dens$. 
Hence, Proposition~\ref{pro:behaviour_dens_weak_conditions} (i) (d) implies 
\begin{equation} \label{eq:f_gamma_exp2} 
 f_{\tau_0}(\omega) = \dens\Psim\bigg( \Gamma\frac{\omega}{\dens^3} \bigg) + \ord\bigg( \frac{\abs{\omega}}{\dens} \bigg) = \frac{\Gamma^2}{18\dens^5} \omega^2 + \ord\bigg( \frac{\abs{\omega}^3}{\dens^{8}} + \frac{\abs{\omega}}{\dens} \bigg) 
\end{equation}
for $\abs{\omega} \leq \eps \dens^3$, where $\Gamma = \Gamma(\tau_0)$. 
Here, we also used the second order Taylor expansion of $\Psim$ defined in \eqref{eq:def_Psim} in the second step. 
Note that $\Gamma \sim 1$ since $\psi + \sigma^2 \sim 1$ by \eqref{eq:psi_plus_sigma_sim_1} and $\abs{\sigma} \lesssim \dens^2$ by Lemma~\ref{lem:sigma_nonzero_local_minimum}. 

We compare \eqref{eq:f_gamma_exp1} and \eqref{eq:f_gamma_exp2} and conclude 
\[ \frac{c}{\dens^5} \omega^2 = 
\frac{\Gamma^2}{18\dens^5}\omega^2 + \ord\bigg(\frac{\abs{\omega}^3}{\dens^8} +  \frac{\abs{\omega}}{\dens}  \bigg)
\]
for $\abs{\omega} \leq \eps \dens^3$. 
Choosing $\omega = \dens^{7/2}$ and solving for $c$ yield 
\begin{equation} \label{eq:coefficient_parabola_nonzero_local_minimum} 
 c = \frac{\Gamma^2}{18} + \ord(\dens^{1/2}). 
\end{equation}
By starting from the expansion of $f_{\tau_0}$ in \eqref{eq:f_gamma_exp1}, using the Taylor expansion of $\Psim$ and \eqref{eq:scaling_psim}, we obtain \eqref{eq:multiplicative_expansion_nonzero_minimum_general_assums} 
in the last missing regime $\abs{\omega} \lesssim \dens^{7/2}$. 

We now turn to the proof of (ii) of Theorem~\ref{thm:behaviour_dens}.
By Proposition~\ref{pro:cubic_for_dyson_equation}, the conditions of Proposition~\ref{pro:behaviour_dens_weak_conditions} (ii) are satisfied on $I' \defeq I \cap [-3\kappa, 3\kappa]$, where $\kappa \defeq \norm{a} + 2\norm{S}^{1/2}$. 
Since $\norm{a} \lesssim 1$ and $\norm{S} \leq \normtwoinf{S} \lesssim 1$ by Assumptions~\ref{assums:general}, we have $\abs{I'} \lesssim 1$. Moreover, $\supp \dens \subset I'$ by \eqref{eq:supp_v_subset_spec_a}. 
Hence, by Proposition~\ref{pro:behaviour_dens_weak_conditions}, it suffices to estimate the distance $\abs{\gamma_1-\gamma_2}$, where $\gamma_1, \gamma_2 \in \Mb_{\dens_*}$ satisfy $\gamma_1 \neq \gamma_2$.  

Let $\gamma_1, \gamma_2 \in \Mb_{\dens_*}$. 
By \eqref{eq:distance_of_internal_minima} in Proposition~\ref{pro:behaviour_dens_weak_conditions} (ii), 
 we know a dichotomy: either $\abs{\gamma_1 - \gamma_2} \gtrsim 1$ or $\abs{\gamma_1 - \gamma_2} \lesssim \min\{ \dens(\gamma_1), \dens(\gamma_2)\}^4$. 
For $\gamma_1 \neq \gamma_2$, we now exclude the second case by using the expansions obtained in the proof of (i). 
If $\dens_* \sim 1$ is chosen sufficiently small then $c(\gamma_1) \sim 1$ and $c(\gamma_2) \sim 1$ by \eqref{eq:coefficient_parabola_nonzero_local_minimum}. 
Hence, by assuming $\abs{\gamma_1-\gamma_2} \lesssim \min\{ \dens(\gamma_1),\dens(\gamma_2)\}^4$, we obtain
$ \dens(\gamma_2) > \dens(\gamma_1)$ from the expansion of $f_{\tau_0}(\omega)$ in \eqref{eq:f_gamma_exp1} with $\tau_0=\gamma_1$ and $\omega = \gamma_2 - \gamma_1$. 
Similarly, as $c(\gamma_2) \sim 1$, the expansion of $f_{\tau_0}(\omega)$ in \eqref{eq:f_gamma_exp1} with $\tau_0=\gamma_2$ and 
$\omega = \gamma_1 - \gamma_2$ implies $\dens(\gamma_1) > \dens(\gamma_2)$. This is a contradiction. Therefore, the distance of two 
small local minima of $\dens$ is much bigger than $\min\{\dens(\gamma_1), \dens(\gamma_2)\}^4$ and the dichotomy above completes the proof of (ii). 
\end{proof}

\subsection{Characterisations of a regular edge}

In this subsection, we introduce the concept of \emph{regular edges} of the self-consistent support and give several equivalent characterisations relying on the cubic analysis of the previous sections. 
We assume that $S$ is flat and $a$ is bounded, i.e., that \eqref{eq:estimates_data_pair} is satisfied. In particular, owing to Proposition~\ref{pro:regularity_density_of_states}, there is a 
Hölder continuous probability density $\dens \colon \R \to [0,\infty)$ such that 
\[ \avg{m(z)} = \int_\R \frac{\dens(\tau)}{\tau - z } \, \di \tau,  \] 
where $m$ is the solution to the Dyson equation, \eqref{eq:dyson}. 

We now define regular edges of $\dens$ as in \cite{AltEdge}. 

\begin{definition}[Regular edge] \label{def:regular_edge}
We call $\tau_0 \in \pt \supp \dens$ a \emph{regular edge} if the limit 
\[ \lim_{\supp \dens \ni \tau \to \tau_0} \frac{\dens(\tau)}{\sqrt{\abs{\tau-\tau_0}}} = \frac{\gamma_\mathrm{edge}^{3/2}}{\pi} \] 
exists for some $\gamma_\mathrm{edge}$ that satisfies $0 < c_* \leq \gamma_\mathrm{edge} \leq c^* < \infty$ for some constants $c_*$ and $c^*$. 
\end{definition} 

The following proposition provides several equivalent characterisations of a regular edge. 

\begin{proposition}[Characterisations of a regular edge] \label{pro:char_regular_edge} 
Let $a$ and $S$ satisfy \eqref{eq:estimates_data_pair} and $m$ be the solution of the corresponding Dyson equation, \eqref{eq:dyson}. 
Suppose for some $\tau_0 \in \pt \supp \dens$, there are $m_*>0$ and $\delta>0$ such that 
\begin{equation} \label{eq:m_sup_bound_regular} 
 \norm{m(\tau + \ii \eta )} \leq m_* 
\end{equation}
for all $\tau \in [\tau_0 - \delta, \tau_0 + \delta]$ and $\eta \in (0,\delta]$. 
We set $\sigma \defeq \sigma(\tau_0)$. Then the following statements are equivalent:  
\begin{enumerate}[label=(\roman*)]
\item \label{item:regular} The point $\tau_0$ is a regular edge of $\dens$. 
\item \label{item:lim_inf_lim_sup} There are $0 < c_* < c^* < \infty$ such that 
\[ c_* \leq \liminf_{\supp\dens \ni \tau \to \tau_0} \frac{\dens(\tau)}{\sqrt{\abs{\tau-\tau_0}}} \leq \limsup_{\supp\dens \ni \tau \to \tau_0} \frac{\dens(\tau)}{\sqrt{\abs{\tau-\tau_0}}} \leq c^* \] 
\item \label{item:sigma_sim_1} There are positive constants $\sigma_*$ and $\sigma^*$ such that 
\[ \sigma_* \leq \abs{\sigma} \leq \sigma^*.\] 
\item \label{item:square_root_edge_order_one} There is $\delta_* >0$ such that 
\[ 
\dens(\tau_0 + \omega) = \begin{cases} 
\displaystyle \frac{\pi^{1/2}}{\abs{\sigma}^{1/2}} \abs{\omega}^{1/2} + \ord(\abs{\omega}), \quad & \text{if } \sign \omega = \sign \sigma, \\  
0, \qquad  & \text{if }\sign \omega = -\sign \sigma, 
\end{cases} 
\] 
for all $\omega \in [-\delta_*, \delta_*]$. In particular, we have $\gamma_\mathrm{edge} = \pi/\abs{\sigma}^{1/3}$. 
\item \label{item:gap_order_one} There is $\delta_\mathrm{gap} >0$ such that 
\[ \dens(\tau)= 0 \] 
for all $\tau \in [\tau_0, \tau_0 + \delta_\mathrm{gap}]$ or for all $\tau \in [\tau_0 - \delta_\mathrm{gap}, \tau_0]$. 
\end{enumerate}
All constants in \ref{item:regular} -- \ref{item:gap_order_one} depend effectively on each other as well as possibly $c_1$, $c_2$, $c_3$ from \eqref{eq:estimates_data_pair} as well as 
$\delta$ and $m_*$ from \eqref{eq:m_sup_bound_regular}. 
\end{proposition} 

In our recent work \cite{AltEdge} on the universality of the local eigenvalue statistics at regular edges parts of Proposition~\ref{pro:char_regular_edge} have already been proven. 
In fact, in Theorem~4.1 of \cite{AltEdge}, we showed that \ref{item:regular} implies \ref{item:sigma_sim_1} and \ref{item:square_root_edge_order_one}. 
The new implications in Proposition~\ref{pro:char_regular_edge}, however, require the cubic shape analysis of the previous subsections which was not available in \cite{AltEdge}. 
Using our preceding analysis, the proof of Proposition~\ref{pro:char_regular_edge} is quite short. 
In the proof, the comparison relation $\sim$ is understood with respect to $c_1, c_2, c_3$ from \eqref{eq:estimates_data_pair} as well as $\delta$ and $m_*$ from \eqref{eq:m_sup_bound_regular}. 

\begin{proof} 
For the entire proof, we remark that, by Lemma~\ref{lem:q_bounded_Im_u_sim_avg} (ii), the conditions of Proposition~\ref{pro:cubic_for_dyson_equation} are satisfied. 
Moreover, $ \dens(\tau_0)=0$ due to the continuity of $\dens$ and $\tau_0 \in \pt \supp \dens$. 
Before establishing the equivalence of \ref{item:regular} -- \ref{item:gap_order_one}, we show that $\sigma \neq 0$ and there is $c \sim 1$, depending only on the constants in \eqref{eq:estimates_data_pair} 
as well as $\delta$ and $m_*$, such that 
\begin{equation} \label{eq:regular_edge_expansion_aux1}
\dens(\tau_0 + \omega) = \begin{cases}\displaystyle \frac{\pi^{1/2}}{\abs{\sigma}^{1/2}} \abs{\omega}^{1/2} + \ord\Big(\frac{\abs{\omega}}{\abs{\sigma}^{2}}\Big), & \text{if } \sign \omega = \sign \sigma, \\ 0, & \text{if } \sign \omega = -\sign \sigma, 
\end{cases} 
\end{equation} 
for all $\omega \in [-c \abs{\sigma}^3, c \abs{\sigma}^3]$. 

By Proposition~\ref{pro:cubic_for_dyson_equation}, we find $\delta_0 \sim 1$, depending only on the constants in \eqref{eq:estimates_data_pair} as well as $\delta$ and $m_*$, such that 
taking the imaginary part of \eqref{eq:def_admiss_decom_diff_m} and applying $\avg{\genarg}$ to the result yield  
\begin{equation} \label{eq:regular_edge_expansion_aux2}
 \dens(\tau_0 + \omega) = \Im \Big(\Theta(\omega) \pi^{-1}\avg{b}\Big)+ \pi^{-1}\avg{\Im r(\omega)} = \Im \Theta(\omega) + \ord\Big( (\abs{\Theta(\omega)} + \abs{\omega}) \Im \Theta(\omega)\Big) 
\end{equation}
for $\abs{\omega} \leq \delta_0$. 
Here, we used $\avg{b} = \pi$ by \eqref{eq:behaviour_dens_aux1} in the proof of Proposition~\ref{pro:behaviour_dens_weak_conditions} as well as the third bound in 
\eqref{eq:def_admiss_dens_zero_Im_Theta_bounds} in the second step. 

By Proposition~\ref{pro:cubic_for_dyson_equation} the assumptions of Theorem~\ref{thm:abstract_cubic_equation} (ii) are satisfied with $\kappa = \pi$. 
Hence, from Theorem~\ref{thm:abstract_cubic_equation} (ii) (a), \eqref{eq:regular_edge_expansion_aux2} and $\abs{\Theta(\omega)} \lesssim \abs{\omega}^{1/3}$ by \eqref{eq:def_admiss_bound_Theta}, 
we conclude that $\sigma \neq 0$ as $\tau_0 \in \pt\supp \dens$. 
From \eqref{eq:regular_edge_expansion_aux2} and Lemma~\ref{lem:simple_edge}, we, thus, conclude \eqref{eq:regular_edge_expansion_aux1} as $\abs{\sigma} \lesssim 1$, 
$\abs{\Theta(\omega)} \lesssim \abs{\omega/\sigma}^{1/2}$ by Lemma~\ref{lem:simple_edge} and, hence, 
$\abs{\nu(\omega)} \lesssim \abs{\Theta(\omega)} + \abs{\omega} \lesssim \abs{\omega/\sigma}^{1/2}$ by the first bound in \eqref{eq:bound_nu}.
This completes the proof of \eqref{eq:regular_edge_expansion_aux1}. 

We now show that the statements \ref{item:regular} -- \ref{item:gap_order_one} are equivalent. 
Trivially, \ref{item:regular} implies \ref{item:lim_inf_lim_sup}. Moreover, if \ref{item:lim_inf_lim_sup} holds true then \eqref{eq:regular_edge_expansion_aux1} yields 
\ref{item:sigma_sim_1}. Clearly, \ref{item:square_root_edge_order_one} is implied by \ref{item:sigma_sim_1} due to \eqref{eq:regular_edge_expansion_aux1}. 
Furthermore, \ref{item:gap_order_one} is trivially satisfied if \ref{item:square_root_edge_order_one} holds true. 
We now prove that \ref{item:gap_order_one} implies \ref{item:sigma_sim_1}. 
By Proposition~\ref{pro:cubic_for_dyson_equation}, $\tau_0$ is a shape regular point.  
Thus, \ref{item:sigma_sim_1} is a consequence of \ref{item:gap_order_one} by Lemma~\ref{lem:large_gap}. 
Finally, \ref{item:sigma_sim_1} implies \ref{item:regular} due to \eqref{eq:regular_edge_expansion_aux1}. 
This completes the proof of Proposition~\ref{pro:char_regular_edge}. 
\end{proof}

\section{Band mass formula -- Proof of Proposition \ref{pro:band_formula}}  \label{sec:bumps} 

Before proving Proposition \ref{pro:band_formula}, we state an auxiliary lemma which will be proven at the end of this section.

\begin{lemma} \label{lem:bump_proof_aux}
Let $(a,S)$ be a data pair, $m$ the solution of the associated Dyson equation \eqref{eq:dyson} and $\dens$ the corresponding self-consistent density of states.
We assume $\norm{a} \leq k_0$ and $S[x] \leq k_1\avg{x} \id$ for all $x \in \algnon$ and for some $k_0, k_1>0$. Then we have 
\begin{enumerate}[label=(\roman*)]
\item If $\tau \in \R \setminus \supp \dens$ then there is $m(\tau) = m(\tau)^* \in \alg$ such that
\[\lim_{\eta\downarrow 0} \norm{m(\tau + \ii\eta) - m(\tau)} = 0. \]  
Moreover, $m(\tau)$ is invertible and satisfies the Dyson equation, \eqref{eq:dyson}, at $z = \tau$. 
There is $C>0$, depending only on $k_0$, $k_1$ and $\dist(\tau, \supp\dens)$, such that $\norm{m(\tau)} \leq C$ and $\norm{(\Id - (1-t) C_{m(\tau)}S)^{-1}} \leq C$ all $t \in [0,1]$. 
\item Fix $\tau \in \R \setminus \supp\dens$. Let $m_t$ be the solution of \eqref{eq:dyson} associated to the data pair 
\[(a_t, S_t) \defeq (a-tS[m(\tau)], (1-t)S)\]
 for $t \in [0,1]$ and $\dens_t$ the corresponding self-consistent density of states. 
Then, for any $t \in [0,1]$, we have 
\begin{equation} \label{eq:m_t_converges_to_m}
 \lim_{\eta \downarrow 0} \norm{m_t(\tau + \ii\eta) - m(\tau)} = 0. 
\end{equation}
Moreover, there is $c>0$, depending only on $k_0$, $k_1$ and $\dist(\tau, \supp\dens)$, such that $\dist(\tau, \supp \dens_t) \geq c$ for all $t \in [0,1]$.
\end{enumerate}
\end{lemma} 

\begin{proof}[Proof of Proposition \ref{pro:band_formula}]
We start with the proof of (i) and notice that the existence of $m(\tau)$ has been proven in Lemma \ref{lem:bump_proof_aux} (i). 
In order to verify \eqref{eq:band_formula}, we consider the continuous flow of data pairs $(a_t,S_t)$ from Lemma \ref{lem:bump_proof_aux} (ii) and the corresponding solutions $m_t$ of the Dyson equation, \eqref{eq:dyson}, 
and prove 
\begin{equation} \label{eq:band_formula_t}
 \dens_t( (-\infty, \tau)) = \avg{\char_{(-\infty,0)}(m_t(\tau))}
\end{equation}
for all $t \in [0,1]$. Note that $\dist(\tau, \supp\dens_t) \geq c$ for all $t \in [0,1]$ by Lemma \ref{lem:bump_proof_aux} (ii). 

In particular, by Lemma \ref{lem:bump_proof_aux} (ii), $m_t(\tau)=m(\tau)$ is constant along the flow, and with it the right-hand side of \eqref{eq:band_formula_t}. 
The identity \eqref{eq:band_formula_t} obviously holds for $t=1$, because $m_1(z)=(a-Sm(\tau)-z)^{-1}$ is the resolvent of a self-adjoint element and $m(\tau)$ satisfies \eqref{eq:dyson} at $z=\tau$ by 
Lemma \ref{lem:bump_proof_aux} (i). Thus it remains to verify that the left-hand side of \eqref{eq:band_formula_t} stays constant along the flow as well. 
This will show \eqref{eq:band_formula_t} for $t=0$ which is \eqref{eq:band_formula}. 

First we conclude from the Stieltjes transform representation \eqref{eq:Stieltjes_representation} of $m_t$ that 
\begin{equation} \label{contour integral representation of mass}
\rho_{t}((-\infty,\tau))\,=\, -\frac{1}{2 \pi \ii} \oint \avg{m_t(z)}\, \di z \,,
\end{equation}
where the contour encircles $[\min\supp \rho_t,\tau)$ counterclockwise, passing through the real line only at $\tau$ and to the left of $\min\supp \rho_t$, and we extended $m_t(z)$ analytically to a neighbourhood 
of the contour (set $m_t(\bar z) \defeq m_t(z)^*$ for $z \in \Hb$ and use Lemma \ref{lem:behaviour_outside} \ref{item:s_a_solution_on_real_axis} close to the real axis to conclude analyticity 
in a neighbourhood of the contour).

We now show that the left-hand side of \eqref{contour integral representation of mass} does not change along the flow. 
Indeed, differentiating the right-hand side of \eqref{contour integral representation of mass} with respect to $t$ and writing $m_t=m_t(z)$ yield 
\[ \begin{aligned} 
\frac{\dd}{\dd t} \oint \avg{m_t(z)} \dd z & = \oint \avg{\pt_t m_t(z)} \dd z =  \oint \scalar{(C_{m_t^*}^{-1} - S_t)^{-1}[\id]}{S[m(\tau)]-S[m_t]} \dd z  \\ 
 & = \oint \avg{(\pt_z m_t  ) (S[m(\tau)]-S[m_t])} \dd z  = \oint \pt_z \left(\avg{m_t S[m(\tau)]} -  \frac{1}{2} \avg{m_t S[m_t]} \right) \dd z = 0. 
\end{aligned} \]
Here, in the second step, we used $\pt_t m_t(z) = (C_{m_t}^{-1} - S_t)^{-1}[-S[m_t] - S[m(\tau)]]$ obtained by differentiating the Dyson equation, \eqref{eq:dyson}, for the data pair $(a_t,S_t)$ defined 
in Lemma~\ref{lem:bump_proof_aux} (ii) and the definition of the scalar product, \eqref{eq:definition_scalar_product_alg}. 
In the third step, we employed $(C_{m_t^*}^{-1} - S_t)^{-1}[\id] = (\pt_z m_t(z))^*$ which follows from differentiating the Dyson equation, \eqref{eq:dyson}, for the data pair $(a_t,S_t)$ with respect to $z$.  
Finally, we used that $m_t$ is holomorphic in a neighbourhood of the contour. This completes the proof of (i) of Proposition \ref{pro:band_formula}.

For the proof of (ii), we fix a connected component $J$ of $\supp\dens$. Let $\tau_1, \tau_2 \in \R\setminus \supp\dens$ satisfy $\tau_1 < \tau_2$ and $[\tau_1, \tau_2] \cap \supp\dens= J$. 
By \eqref{eq:band_formula}, we have 
\[ n \dens(J) = n \Big(\dens( (-\infty, \tau_2) ) - \dens( (-\infty,\tau_1)) \Big) = \Tr(P_2) - \Tr(P_1) = \rm{rank} P_2 - \rm{rank} P_1, \] 
where $P_i \defeq \pi(\char_{(-\infty,0)}(m(\tau_i)))$ are orthogonal projections in $\C^{n\times n}$ for $i = 1, 2$. 
Hence, $ n\dens(J) \in \Z$. Since $0 < n \dens(J) \leq n$ by definition of $\supp\dens$, we conclude $n \dens(J) \in \{1, \ldots, n\}$, which immediately implies that $\supp\dens$ has 
 at most $n$ connected components. This completes the proof of Proposition \ref{pro:band_formula}.  
\end{proof}

\begin{proof}[Proof of Lemma \ref{lem:bump_proof_aux}]
In part (i), the existence of the limit $m(\tau) \in \alg$ follows immediately from the implication \ref{item:dist_supp} $\Rightarrow$ 
\ref{item:Id_minus_C_m_S_invertible} of Lemma \ref{lem:behaviour_outside}. 
The invertibility of $m(\tau)$ can be seen by multiplying \eqref{eq:dyson} at $z=\tau + \ii\eta$ by $m(\tau + \ii\eta)$ and taking the limit $\eta \downarrow 0$. 
This also implies that $m(\tau)$ satisfies \eqref{eq:dyson} at $z=\tau$. 
In order to bound $\norm{(\Id - (1-t)C_{m(\tau)}S)^{-1}}$, we recall the definitions of $q$, $u$ and $F$ from \eqref{eq:def_q_u} and \eqref{eq:def_F}, 
respectively, and compute
\[ \Id - (1-t)C_m S = C_{q^*,q}(\Id - (1-t)C_u F ) C_{q^*,q}^{-1} \] 
for $m=m(z)$ with $z \in \Hb$.
Hence, by \eqref{eq:char_outside_m_bound_F_condition}, Lemma \ref{lem:q_bounded_Im_u_sim_avg} (i) and Lemma \ref{lem:norm_two_to_inf}, we obtain $\norm{(\Id- (1-t)C_mS)^{-1}} \lesssim (1 - (1-t) \normtwo{F})^{-1} \leq 
(1-\normtwo{F})^{-1} \leq C$ for all $z \in \tau + \ii N$, where the set $N \subset (0,1]$ with an accumulation point at 0 is given in Lemma \ref{lem:behaviour_outside} \ref{item:F_condition}. 
Taking the limit $\eta \downarrow 0$ under the constraint $\eta \in N$ and possibly increasing $C$ yield the desired uniform bound.
This completes the proof of (i). 

We start the proof of (ii) with an auxiliary result. Similarly as in the proof of (i), we see that $\Id - (1-t)C_{m^*,m}S$ is invertible for $m=m(z)$, $z \in \tau + \ii N$ with $N$ as before. 
Since $\normtwo{F(z)} \leq 1 - C^{-1}$ for $z \in \tau + \ii N$ as in the proof of (i), 
Lemma~\ref{lem:inverse_positivity_preserving} implies that $(\Id - (1-t)C_{u^*,u}F)^{-1}$, $F = F(z)$, and, thus, 
$(\Id - (1-t)C_{m^*,m}S)^{-1} = C_{q^*,q}(\Id- (1-t)C_{u^*,u}F)^{-1}C_{q^*,q}^{-1}$  
are positivity-preserving for $z \in \tau + \ii N$. 
Taking the limit $\eta = \imz \downarrow 0$ in $N$ shows that $(\Id - (1-t)C_{m(\tau)}S)^{-1}$ is positivity-preserving for any $t \in [0,1]$. 
Moreover, \eqref{eq:inverse_Id_T_lower_bound} with $x = \id$ yields 
\begin{equation} \label{eq:lower_bound_Id_minus_C_mS_inverse_of_id_aux1}
 (\Id - (1-t)C_{m^*,m}S)^{-1}[\id]  = C_{q^*,q}(\Id- (1-t)C_{u^*,u}F)^{-1}C_{q^*,q}^{-1}[\id] \geq \id.
\end{equation} 
Since~\eqref{eq:lower_bound_Id_minus_C_mS_inverse_of_id_aux1} holds true uniformly for $z \in \tau + \ii N$ and $t \in [0,1]$, taking the limit $\eta = \imz \downarrow 0$ in $N$, we obtain 
\begin{equation}\label{eq:lower_bound_Id_minus_C_mS_inverse_of_id}
 (\Id - (1-t)C_{m(\tau)}S)^{-1}[\id] \geq \id 
\end{equation} 
for all $t \in [0,1]$. 

We fix $t \in [0,1]$. We write $m = m(\tau)$ and define $\Phi_t \colon \alg \times \R \to \alg$ through
\[ \Phi_t(\Delta, \eta) \defeq (\Id- (1-t) C_mS)[\Delta] - \frac{\ii\eta}{2} (m \Delta + \Delta m ) - \ii \eta m^2 - \frac{1}{2}(1-t) ( \Delta S[\Delta] m + m S[\Delta] \Delta)\]  
In order to show \eqref{eq:m_t_converges_to_m}, we apply the implicit function theorem (see e.g.~Lemma \ref{lem:implicit_function} below) to $\Phi_t(\Delta,\eta) = 0$. It is applicable as
$\Phi_t(0,0)=0$ and $\pt_1\Phi_t(0,0) = \Id - (1-t) C_mS$ which is invertible by (i). Hence, we obtain an $\eps>0$ and a continuously differentiable function $\Delta_t \colon (-\eps,\eps) \to \alg$ 
such that $\Phi_t(\Delta_t(\eta), \eta) = 0 $ for all $\eta \in (-\eps,\eps)$ and $\Delta_t(0)=0$. 
We now show that $\Delta_t(\eta) + m(\tau) = m_t(\tau + \ii\eta)$ for all sufficiently small $\eta>0$ by appealing to the uniqueness of the solution to the Dyson equation, \eqref{eq:dyson}, 
with the choice $z = \tau + \ii\eta$, $a = a_t$ and $S=S_t=(1-t)S$. In fact, $m=m(\tau)$ and $m_t=m_t(\tau + \ii\eta)$ with $\eta >0$ satisfy the Dyson equations
\begin{equation} \label{eq:dysons_m_m_t}
 -m^{-1} = \tau - a + S[m], \quad -m_t^{-1} = \tau + \ii\eta - a + t S[m] + (1-t)S[m_t] 
\end{equation}
and $m_t$ is the unique solution of the second equation under the constraint $\Im m_t>0$ (compare the remarks around \eqref{eq:dyson}).
A straightforward computation using the first relation in \eqref{eq:dysons_m_m_t} and $\Phi_t(\Delta_t(\eta),\eta)=0$ reveals that $\Delta_t(\eta) + m(\tau)$ solves the second equation in 
\eqref{eq:dysons_m_m_t} for $m_t$. 
Moreover, differentiating $\Phi_t(\Delta_t(\eta),\eta) = 0$ with respect to $\eta$ at $\eta = 0$ yields
\[ \pt_\eta \Im\Delta_t(\eta = 0) = ( \Id - (1-t) C_mS)^{-1}[ m^2] \geq \norm{m^{-1}}^{-2} (\Id - (1-t)C_mS)^{-1}[\id]\geq \norm{m^{-1}}^{-2} \id. \] 
Here, we used that $(\Id - (1-t)C_mS)^{-1}$ is compatible with the involution $ ^*$ and $m=m^*$ in the first step. Then we employed the invertibility of $m$, $m^2 \geq \norm{m^{-1}}^{-2}\id$ and the positivity-preserving 
property of $(\Id-(1-t)C_mS)^{-1}$ in the second step and, finally, \eqref{eq:lower_bound_Id_minus_C_mS_inverse_of_id} in the last step. 
Hence, $\Im (\Delta_t(\eta) + m(\tau)) = \Im \Delta_t(\eta) >0$ for all sufficiently small $\eta>0$. The uniqueness of the solution to the Dyson equation for $m_t$, the second relation in \eqref{eq:dysons_m_m_t}, 
implies $\Delta_t(\eta) + m(\tau) = m_t(\tau + \ii\eta)$ for all sufficiently small $\eta >0$ and all $t \in [0,1]$. 
Therefore, the continuity of $\Delta_t$ as a function of $\eta$, $\Delta_t(\eta) \to \Delta_t(0)=0$,
 yields~\eqref{eq:m_t_converges_to_m}.

We now conclude from the implication \ref{item:Id_minus_C_m_S_invertible} $\Rightarrow$ \ref{item:dist_supp} of Lemma \ref{lem:behaviour_outside} that $\dist(\tau, \supp\dens_t) \geq \eps$ for some $\eps>0$. 
Lemma \ref{lem:behaviour_outside} is applicable since $\norm{a_t}\leq k_0 + k_1 C$ (cf.~Lemma \ref{lem:norm_two_to_inf} (i) and Lemma \ref{lem:bump_proof_aux} (i)) 
and $S_t[x] \leq S[x] \leq k_1 \avg{x} \id$ for all $x \in \algnon$. 
For any $t \in [0,1]$, statement \ref{item:Id_minus_C_m_S_invertible} in Lemma \ref{lem:behaviour_outside} holds true with the same $m=m(\tau)$ by \eqref{eq:m_t_converges_to_m}
and $S$ replaced by $S_t = (1-t)S$. By (i), $\norm{m} \leq C$ and $\norm{(\Id- (1-t)C_mS)^{-1}} \leq C$ for all $t\in [0,1]$. Hence, owing to Lemma \ref{lem:behaviour_outside} \ref{item:dist_supp}, 
there is $\eps>0$, depending only on $k_0$, $k_1$ and $\dist(\tau, \supp\dens)$, such that $\dist(\tau, \supp\dens_t) \geq \eps$ for all $t \in [0,1]$. Here, $\eps$ depends only on $k_0$, $k_1$ and 
$\dist(\tau, \supp\dens)$ due to the exclusive dependence of $C$ from (i) on the quantities and the effective dependence of the constants in Lemma \ref{lem:behaviour_outside} on each other 
(see final remark in Lemma \ref{lem:behaviour_outside}). 
The uniformity of $\eps$ in $t$ is a consequence of the uniformity of $C$ from (i) in $t$. 
This completes the proof of Lemma~\ref{lem:bump_proof_aux}. 
\end{proof}

\section{Dyson equation for Kronecker random matrices} \label{sec:Kronecker}

In this section we present an application of the theory presented in this work to Kronecker random matrices, i.e., block correlated random matrices with variance profiles within the blocks, and their limits.
In particular, in Lemma \ref{lmm:L2 implies uniform bound general} and Lemma \ref{lmm:L2 implies uniform bound} below, we will provide some sufficient checkable conditions that ensure the flatness of $S$ and the boundedness of $\norm{m(z)}$, 
the main assumptions of Proposition \ref{pro:Hoelder_1_3}, Theorem \ref{thm:singularities_flat} and Theorem \ref{thm:behaviour_v_close_sing}, for the self-consistent density of states of Kronecker random matrices introduced in \cite{AltKronecker}. 

\subsection{The Kronecker setup}
We fix   $K \in \N$ and a probability space $(\frak{X},\pi)$ that we view as a possibly infinite set of indices. We consider the von Neumann algebra
\begin{equation}\label{vN algebra for example}
\cal{A}\,=\,\C^{K \times K} \otimes \rm{L}^\infty(\frak{X})\,,
\end{equation}
with the tracial state
\[
\avg{\kappa \otimes f }\,=\, \frac{ \tr \kappa}{K} \int_{\frak{X}} f \dd \pi\,.
\]
For $K=1$ the algebra $\cal{A}$ is commutative and this setup was previously considered in \cite{AjankiQVE,AjankiCPAM}. Now  let $(\alpha_\mu)_{\mu=1}^{\ell_1},(\beta_\nu)_{\nu=1}^{\ell_2}$ 
be families of matrices in $\C^{K \times K}$ with $\alpha_\mu =\alpha_\mu^*$ self-adjoint and let
$(s^\mu)_{\mu=1}^{\ell_1},(t^\nu)_{\nu=1}^{\ell_2}$  be families of non-negative bounded functions in $\rm{L}^\infty(\frak{X}^2)$ and suppose that all $s^\mu$ are symmetric, $s^\mu(x,y) = s^\mu(y,x)$. Then we define the self-energy operator $S: \cal{A} \to \cal{A}$ as
\begin{equation}
\begin{split}
\label{Kronecker data S}
S \pb{\kappa \otimes f} \,\defeq\, \sum_{\mu = 1}^{\ell_1} \alpha_\mu \kappa \alpha_\mu \otimes S_\mu f
+\sum_{\nu = 1}^{\ell_2} (\beta_\nu \kappa \beta_\nu^* \otimes T_\nu f +\beta_\nu^* \kappa \beta_\nu\otimes T_\nu^* f)\,,
\end{split}
\end{equation}
where the bounded operators $S_\mu,T_\nu,T_\nu^*:\rm{L}^\infty(\frak{X}) \to \rm{L}^\infty(\frak{X})$  act as
\[
(S_\mu f)(x)\,=\, \int_{\frak{X}} s^\mu(x,y) f(y) \pi(\dd y)\,,\quad 
(T_\nu f)(x)\,=\, \int_{\frak{X}} t^\nu(x,y) f(y) \pi(\dd y)\,,\quad
(T_\nu^* f)(x)\,=\, \int_{\frak{X}} t^\nu(y,x) f(y) \pi(\dd y)\,.
\]
Furthermore we fix a self-adjoint $a =a^* \in \cal{A}$. With these data we will consider
the Dyson equation,  \eqref{eq:dyson}.

The following lemma provides sufficient conditions that ensure flatness of $S$ and boundedness of $\norm{m(z)}$ uniformly in $z$ up to the real line. We begin with some preparations.  We use the notation $x \mapsto v_x$ for $x \in \frak{X}$ and an element $v \in \C^{K \times K} \otimes \rm{L}^\infty(\frak{X})$, interpreting it as a function on $\frak{X}$ with values in $\C^{K \times K}$.
We also introduce the functions $\gamma \in \rm{L}^\infty(\frak{X}^2)$ via 
\begin{equation}
\label{definition gamma xy}
\gamma(x,y)\,\defeq\,\pbb{
\int_{\frak{X}} (\abs{s^{\mu}(x,\1\cdot\1)- s^{\mu}(y,\1\cdot\1)}^2+\abs{t^{\nu}(x,\1\cdot\1)- t^{\nu}(y,\1\cdot\1)}^2+\abs{t^{\nu}(\1\cdot\1,x)- t^{\nu}(\1\cdot\1,y)}^2)\dd \pi}^{1/2}
\end{equation}
and $\Gamma:(0,\infty)^2 \to \rm{L}^\infty(\frak{X}), (\Lambda,\tau) \mapsto \Gamma_{\Lambda, \1\cdot\1} (\tau)$ through
\begin{equation}
\label{definition Gamma Lambda x tau}
\Gamma_{\Lambda,x}(\tau)\,\defeq\, \Bigg(\int_{\frak{X}} \pbb{\frac{1}{\tau}+\norm{a_x-a_y} +\gamma(x,y)\Lambda}^{-2}\pi(\dd y)\Bigg)^{1/2}\,.
\end{equation}
Here, we denoted by $\norm{\genarg}$ the operator norm on $\C^{K\times K}$ induced by the Euclidean norm on $\C^K$. 
The two functions $\gamma$ and $\Gamma$ will be important to quantify the modulus of continuity of the data $(a,S)$. 

\begin{lemma}\label{lmm:L2 implies uniform bound general}
Let $m$ be the solution of the Dyson equation, \eqref{eq:dyson}, on the von Neumann algebra $\cal{A}$ from \eqref{vN algebra for example} associated to the data $(a,S)$ with $S$ defined as in \eqref{Kronecker data S}.
\begin{itemize}
\item[(i)] Define $\Gamma(\tau)\defeq C_{\rm{Kr}}\1\essinf_x  \Gamma_{1, x} (\tau)$ with $C_{\rm{Kr}}\defeq (4 + 4 K(\ell_1+\ell_2)\max_{\mu,\nu}(\norm{\alpha_\mu}^2+\norm{\beta_\nu}^2))^{1/2}$, where $\Gamma_{\Lambda, x} (\tau)$ was introduced in \eqref{definition Gamma Lambda x tau} and assume that for some $z \in \mathbb{H}$ the $\rm{L}^2$-upper bound $\norm{m(z)}_2 \le \Lambda$ for some $\Lambda \ge 1$ is satisfied. Then we have the uniform upper bound 
\begin{equation}\label{uniform upper bound general}
\norm{m(z)}\,\le\, \frac{\Gamma^{-1}(\Lambda^{2})}{\Lambda}\,,
\end{equation}
where we interpret the right-hand side as $\infty$ if $\Lambda$ is not in the range of the strictly monotonously increasing function $\Gamma$. 
\item[(ii)] Suppose that the kernels of the operators $S^\mu$ and $T^\nu$, used to define $S$ in \eqref{Kronecker data S}, are bounded from below, i.e., $\essinf_{x,y} s^\mu(x,y) >0$ and $\essinf_{x,y} t^\nu(x,y) >0$. Suppose further that 
\begin{equation}\label{Kronecker flatness condition}
\inf_\kappa \frac{1}{\tr \kappa}\pBB{\sum_{\mu=1}^{\ell_1}\alpha_\mu \kappa \alpha_\mu +\sum_{\nu=1}^{\ell_2}(\beta_\nu \kappa \beta_\nu^*+\beta_\nu^* \kappa \beta_\nu)}\,>\, 0\,,
\end{equation}
where the infimum is taken over all positive definite $\kappa \in \C^{K \times K}$. Then $S$ is flat, i.e., $S \in \Sigma_{\rm{flat}}$ (cf.~\eqref{eq:def_self_energy_flat}). 
\item[(iii)] Let $S$ be flat, hence, $\Lambda \defeq 1 + \sup_{z \in \Hb} \normtwo{m(z)} < \infty$. Then \eqref{uniform upper bound general} holds true with this $\Lambda$. 
\item[(iv)] If $a = 0$ then, for each $\eps>0$, \eqref{uniform upper bound general} holds true on $\abs{z} \geq \eps$ with $\Lambda \defeq 1 + 2\eps^{-1}$. 
\end{itemize}
\end{lemma}

\begin{proof}[Proof of Lemma~\ref{lmm:L2 implies uniform bound general}]
We adapt the proof of Proposition~6.6 in \cite{AjankiQVE} to our noncommutative setting in order to prove (i). 
Recall the definition of $\gamma(x,y)$ in \eqref{definition gamma xy}. Estimating the norm $\norm{m}_2$ from below, we find
\begin{equation}
\begin{split}
\norm{m}_2^2&= \frac{1}{K} \tr\int\frac{\pi(\dd y)}{m_y^{-1}(m_y^*)^{-1}} 
\ge \tr\int_{\frak{X}}\frac{C_{\rm{Kr}}^{2} \2\pi(\dd y)}{m_x^{-1}(m_x^*)^{-1}+  \norm{a_x-a_y}^2+  \gamma(x,y)^2\norm{m}_2^2} 
\ge C_{\rm{Kr}}^{2}  \Big(\Gamma_{\norm{m}_2,x}(\norm{m_x})\Big)^{2}\,,
\end{split}
\end{equation}
for $\pi$-almost all $x \in \frak{X}$, where we used 
\begin{equation}
\begin{split}
\frac{1}{4}m_y^{-1}(m_y^*)^{-1}\,&\le\, m_x^{-1}(m_x^*)^{-1} +  (a_y-a_x)(a_y-a_x)^*+ ((Sm)_x-(Sm)_y) ((Sm)_x-(Sm)_y)^*
\\
\,&\le\,m_x^{-1}(m_x^*)^{-1}+  \norm{a_x-a_y}^2+ K(\ell_1+\ell_2)\max_{\mu,\nu}(\norm{\alpha_\mu}^2+\norm{\beta_\nu}^2) \gamma(x,y)^2\norm{m}_2^2\,.
\end{split}
\end{equation}
We conclude $\Lambda \ge \Lambda^{-1}\Gamma(\Lambda\norm{m_x})$ for any upper bound $\Lambda \ge 1$ on $\norm{m}_2$. In particular, \eqref{uniform upper bound general} follows.

We turn to the proof of (ii). We view a positive element $r \in \cal{A}_+$ as a function $r :[0,1] \to \C^{K \times K}$ with values in positive semidefinite matrices. Then we find 
\[
(S r)_x\,\ge\, c\int_{\frak{X}}\pBB{\sum_{\mu = 1}^{\ell_1} \alpha_\mu r_y \alpha_\mu 
+\sum_{\nu = 1}^{\ell_2} (\beta_\nu r_y \beta_\nu^*+\beta_\nu^* r_y \beta_\nu)}\pi(\dd y)\,,
\]
as quadratic forms on $\C^{K \times K}$ for almost every $x \in \frak{X}$. The claim follows now immediately from \eqref{Kronecker flatness condition}. 
Part (iii) is a direct consequence of (i) and (ii) as well as \eqref{eq:m_Ltwo_bound}. 
For the proof of part (iv), we use part (i) and \eqref{eq:L2_bound} if $a=0$. 
 \end{proof}

\subsection{$N \times N$-Kronecker random matrices}
\label{subsec:NxN Kronecker RM}

As an application of the general Kronecker setup introduced above, we consider the \emph{matrix Dyson equation} associated to Kronecker random matrices.
Let $X_\mu, Y_\nu \in \C^{N \times N}$ be independent centered random matrices such that $Y_\nu=(y_{ij}^\nu)$ has independent entries and $X_\mu=(x_{ij}^\mu)$ has independent entries up to the Hermitian symmetry constraint $X_\mu=X_\mu^*$. Suppose that the entries of $\sqrt{N}X_\mu, \sqrt{N} Y_\nu$ have uniformly bounded moments, $\E(\abs{x^\mu_{ij}}^p+\abs{y^\mu_{ij}}^p) \le N^{-p/2}C_p$ and 
define their variance profiles through
\[
s^\mu(i,j)\,\defeq\, N\E\abs{x^\mu_{ij}}^2\,,\qquad t^\nu(i,j)\,\defeq\, N\E\abs{y^\nu_{ij}}^2\,. 
\]
Then we are interested in the asymptotic spectral properties of the Hermitian \emph{Kronecker  random matrix}
\begin{equation} \label{Kronecker RM}
\begin{split}
H\,\defeq\, A +\sum_{\mu=1}^{\ell_1} \alpha_\mu \otimes X_\mu +\sum_{\nu=1}^{\ell_2} (\beta_\nu \otimes Y_\nu+\beta_\nu^* \otimes Y_\nu^*) \in \C^{K \times K} \otimes \C^{N \times N} \,,
\end{split}
\end{equation}
as $N \to \infty$.
Here the expectation matrix $A$ is assumed to be bounded, $\norm{A}\le C$, and block diagonal, i.e.
\begin{equation}\label{Kronecker expectation matrix}
A\,=\, \sum_{i=1}^N a_i \otimes E_{ii}\,,
\end{equation}
with $E_{ii} = (\delta_{il}\delta_{ik})_{l,k=1}^N\in \C^{N \times N}$ and $a_i \in \C^{K \times K}$. In \cite{AltKronecker} it was shown that the resolvent $G(z) = (H-z)^{-1}$ of the Kronecker matrix $H$ is well approximated by the solution $M(z)$ of a Dyson equation of Kronecker type, i.e., on the von Neumann algebra $\cal{A}$ in \eqref{vN algebra for example} with self-energy $S$ from \eqref{Kronecker data S} and $a=A \in \cal{A}$, when we choose $\frak{X} = \{1, \dots,N\}$ and $\pi$ the uniform probability distribution. In other words, $\rm{L}^\infty(\frak{X})= \C^N$ with entrywise multiplication.

\subsection{Limits of Kronecker random matrices}
Now we consider limits of Kronecker random matrices $H \in \C^{N \times N}$ with piecewise H\"older-continuous variance profiles as $N \to \infty$. In this situation we can make sense of the continuum limit for the solution $M(z)$ of the associated matrix Dyson equation. The natural setup here is  $(\frak{X},\pi)=([0,1], \dd x)$. 
We fix a partition $(I_l)_{l=1}^L$ of $[0,1]$ into intervals of positive length, i.e., $[0,1] = \dot{\cup}_{l}I_l$ and consider non-negative profile functions $s^\mu, t^\nu:[0,1]^2 \to \R$ that are H\"older-continuous with H\"older exponent $1/2$ on each rectangle $I_l\times I_k$. We also fix a function $a :[0,1] \to \C^{K \times K}$ that is $1/2$-H\"older continuous on each $I_l$. 
In this piecewise H\"older-continuous setup
the Dyson equation on $\cal{A}$ with data pair $(a,S)$ describes the asymptotic spectral properties of Kronecker random matrices with fixed variance profiles $s^\mu$ and $t^\nu$, i.e., the random matrices $H$ introduced in Subsection~\ref{subsec:NxN Kronecker RM} if their variances are given by 
\[
\E\abs{x^\mu_{ij}}^2\,=\,  \frac{1}{N}s^\mu \pbb{\frac{i}{N},\frac{j}{N}}\,,\qquad 
\E\abs{y^\nu_{ij}}^2\,=\, \frac{1}{N}t^\nu \pbb{\frac{i}{N},\frac{j}{N}}\,, 
\]
and the matrices $a_i$ in \eqref{Kronecker expectation matrix} by $a_i = a\pb{\frac{i}{N}}$.

\begin{lemma} 
\label{lmm: Kronecker global law}
Suppose that
$a$, $s^\mu$ and $t^\nu$ are piecewise H\"older-continuous with H\"older exponent $1/2$ as described above.
The empirical spectral distribution of the Kronecker  random matrix $H$, defined in \eqref{Kronecker RM}, with eigenvalues $(\lambda_i)_{i=1}^{KN}$ converges weakly in probability to the self-consistent density of states $\rho$ associated to the Dyson equation with data pair $(a,S)$ as defined in  \eqref{Kronecker data S}, i.e., for any $\eps>0$ and $\varphi \in C(\R)$ we have 
\[
\P\pBB{\absBB{\frac{1}{K N} \sum_{i=1}^{KN}\varphi(\lambda_i)\,-\, \int_\R \varphi \, \dd \rho} \,>\, \eps }\,\to\, 0 \,, \qquad N \to \infty\,.
\]
\end{lemma}

\begin{proof}[Proof of Lemma \ref{lmm: Kronecker global law}] 
It suffices to prove convergence of the Stieltjes transforms, i.e., in probability $\frac{1}{NK} \tr_{KN} G(z) \to \avg{m(z)}$ for every fixed $z \in \mathbb{H}$, where $G(z) = (H-z)^{-1}$ is the resolvent of the Kronecker matrix $H$ and $m(z)$ is the solution to the Dyson equation with data $(a,S)$. 

First we use the Theorem~2.7 from \cite{AltKronecker} to show that 
$
\frac{1}{KN} \tr_{KN} G(z) - \frac{1}{N} \sum_{i=1}^N \tr_K m_{i}(z)\,\to\, 0
$ in probability, where $M_N=(m_1, \dots, m_N) \in (\C^{K \times K})^N$ denotes the solution to a Dyson equation formulated on the von Neumann algebra $\C^{K \times K}\otimes \C^{N}$ with entrywise multiplication on vectors in $\C^N$ as explained in Subsection~\ref{subsec:NxN Kronecker RM}.  We recall that in this setup the discrete kernels for $S_\mu$ and $T_\nu$ from the definition of $S$ in \eqref{Kronecker data S} are  given by $N\E\abs{x^\mu_{ij}}^2$ and $N\E\abs{y^\nu_{ij}}^2$, respectively, and $a=\sum_{i=1}^N a\pb{\frac{i}{N}} \otimes e_{i}$. To distinguish this discrete data pair from the continuum limit over $\C^{K \times K}\otimes \rm{L}^\infty[0,1]$, we denote it by $(a_N,S_N)$.  Note that in Theorem~2.7 of \cite{AltKronecker} the test functions were compactly supported in contrast to the function $\tau \mapsto 1/(\tau-z)$ that we used here. However, by Theorem~2.4 of \cite{AltKronecker}  and since the self-consistent density of states is compactly supported (cf.~\eqref{eq:supp_v_subset_spec_a} and $\norm{S} \lesssim 1$) no eigenvalues can be found beyond a certain bounded interval, ensuring that non compactly supported test function are allowed as well.

Now it remains to show that $\avg{M_N} \to \avg{m}$ as $N \to \infty$ for all $z \in \Hb$. 
For this purpose we embed $\C^N$ into $\rm{L}^\infty[0,1]$ via $P v \defeq \sum_{i=1}^N v_i \bf{1}_{[(i-1)/N,i/N)}$. With this identification $M_N$ and $m$ satisfy Dyson equations on the same space $\C^{K \times K}\otimes \rm{L}^\infty[0,1]$. Evaluating these two equations at $z+ \ii\eta$, for a fixed $z \in \Hb$ and any $\eta\geq 0$, and subtracting them from each other yield
\[
B[\Delta]\,=\, m (S_N-S)[m] \Delta + C_m (S_N-S)[\Delta] + m S_N[\Delta]\Delta + C_m (S_N-S)[m]  -m(a_N-a)\Delta - C_m [a_N-a], 
\]
where $m=m(z + \ii \eta)$, $M_N=M_N(z + \ii \eta)$, $B=\Id-C_m S$ and $\Delta =M_N-m$. Using the imaginary part of $z$ we have $\dist(z+ \ii \eta, \supp \rho)\ge \imz>0$. By (3.22), (3.23), (3.11a) and (3.11c) of 
\cite{AltKronecker} we infer $\norm{m}+\norm{B^{-1}}_2 \le C$ for all $\eta\geq0$ with a constant $C$ depending on $\imz$. 
Note that although the proofs in \cite{AltKronecker} were performed on $\C^{N \times N}$ all estimates were uniform in $N$ and all algebraic relations in these proof translate to the current setting on a finite von Neumann algebra.  
Using $\normtwo{S_N-S} \leq \norm{S_N-S}$ as well as $\norm{S_N} \leq C$ and possibly increasing $C$, we thus obtain 
\[ \normtwo{\Delta} \leq C ( \Psi_N + \normtwo{\Delta}^2), \qquad \qquad  \Psi_N \defeq \norm{a_N-a} + \norm{S_N-S}, \]
where $\Delta = \Delta(z + \ii \eta)$, for all $\eta\geq0$. 
We choose $N_0$ sufficiently large such that $2\Psi_NC^2 \leq 1/4$ for all $N\geq N_0$ and define $\eta_* \defeq \sup \{ \eta \geq 0 \colon \,\normtwo{\Delta(z + \ii \eta)} \geq 2C \Psi_N \}$. 
Since $\norm{M_N} + \norm{m} \to 0$ for $\eta \to \infty$, we conclude $\eta_* < \infty$. 

We now prove $\eta_* = 0$. For a proof by contradiction, we suppose $\eta_* > 0$. Then, by continuity, $\normtwo{\Delta(\tau + \ii\eta_*)} = 2 C \Psi_N$. 
Since $2\Psi_NC^2 \leq 1/4$, we have $\normtwo{\Delta(z +\ii\eta_*)} \leq 4C\Psi_N/3 < 2 C \Psi_N = \normtwo{\Delta(z +\ii\eta_*)}$. 
From this contradiction, we conclude $\eta_*=0$. 
Therefore, for $N\geq N_0$, we have
\[
\abs{M_N(z) - m(z)} \leq \normtwo{\Delta(z)} \leq 2 C \Psi_N = 2 C(\norm{S_N-S}+\norm{a_N-a})\,. 
\]
Since the right-hand side converges to zero as $N \to \infty$, due to the piecewise H\"older-continuity of the profile functions, and since $z$ was arbitrary,
we obtain $\avg{M_N}\to \avg{m}$ as $N\to \infty$ for all $z \in \Hb$. This completes the proof of Lemma~\ref{lmm: Kronecker global law}. 
\end{proof}

The boundedness of the solution to the Dyson equation in $L^2$-norm already implies uniform boundedness in the piecewise H\"older-continuous setup.
\begin{lemma} 
\label{lmm:L2 implies uniform bound}
Suppose that
$a$, $s^\mu$ and $t^\nu$ are piecewise $1/2$-H\"older continuous and that $\sup_{z \in \mathbb{D}}\norm{m(z)}_2<\infty$ for some domain $\mathbb{D} \subseteq \mathbb{H}$. Then we have the uniform bound $\sup_{z \in \mathbb{D}}\norm{m(z)}<\infty$.
\end{lemma}
In particular, if the random matrix $H$ is centered, i.e., $a=0$, then $m(z)$ is uniformly bounded as long as $z$ is bounded away from zero; and if $H$ is flat in the limit, i.e., $S$ is flat, then $\sup_{z \in \mathbb{H}}\norm{m(z)}<\infty$. 

\begin{proof} By (i) of Lemma~\ref{lmm:L2 implies uniform bound general}  the proof reduces to checking that $\lim_{\tau \to \infty}\Gamma(\tau) = \infty$ for  piecewise $1/2$-H\"older continuous data in the special case $(\frak{X},\pi)=([0,1], \dd x)$.  But this is clear since in that case $\norm{a_x-a_y}^2 + \gamma(x,y)^2 \le C\abs{x-y}$ implies that the integral in \eqref{definition Gamma Lambda x tau} is at least logarithmically divergent as $\tau \to \infty$.
\end{proof}

\begin{corollary}[Band mass quantization] \label{coro:band_mass_quant_Kronecker}
Let $\rho$ be the self-consistent density of states for the Dyson equation with data pair $(a,S)$ and $ \tau \in \R \setminus \supp \rho$. Then
\[
\rho((-\infty,\tau)) \in \cBB{ \frac{1}{K}\sum_{l =1}^L k_l \abs{I_l} :  k_l =1, \dots K }\,.
\]
In particular, in the $L=1$ case  when $s^\mu,t^\mu$ and $a$ are $1/2$-H\"older continuous on all of $[0,1]^2$ and $[0,1]$, respectively, then $\rho(J)$ is an integer multiple of $1/K$ for every connected component $J$ of $\supp \rho$ and there are at most $K$ such components.
\end{corollary}

\begin{proof} Fix $\tau \in \R \setminus \supp \rho$. We denote by $x \mapsto m_x(\tau)$ the self-adjoint solution $m(\tau)$ viewed as a function of $x \in [0,1]$ with values in $\C^{K \times K}$. As is clear from the Dyson equation this function inherits the regularity of the data, i.e., it is continuous on each interval $I_l$. 
By the band mass formula  \eqref{eq:band_formula}   we have 
\[
\rho((-\infty,\tau))\,=\, 
\frac{1}{K}\sum_{l=1}^L \int_{I_l} \tr{\bf{1}_{(-\infty,0)}(m_x(\tau))}\dd x\,=\, \frac{1}{K}\sum_{l=1}^Lk_l \abs{I_l}\,,
\]
where $k_l=\tr{\bf{1}_{(-\infty,0)}(m_x(\tau))} \in \{0, \dots,K\}$ is continuous in $x \in I_l$ with discrete values and therefore  does not depend on $x$.
\end{proof}

\begin{remark}
We extend the conjecture from Remark~2.9 of \cite{AjankiCPAM} to the Kronecker setting. We expect that in the piecewise $1/2$-H\"older  continuous setting of the current section, the number of connected components of the self-consistent spectrum $\supp \rho$ is at most $K(2L-1)$.
\end{remark}

\section{Perturbations of the data pair}  \label{sec:pert_data_pair}

\newcommand{\Dbdd}{\mathbb{D}_\mathrm{bdd}} 
\newcommand{\Dbulk}{\mathbb{D}_\mathrm{bulk}} 
\newcommand{\Dcusp}{\mathbb{D}_\mathrm{nocusp}} 
\newcommand{\Dout}{\mathbb{D}_\mathrm{out}} 
\newcommand{\Deltam}{\Delta m_t}

In this section, as an application of our results in Sections \ref{sec:regularity} to \ref{sec:shape_analysis}, we show that the Dyson equation, \eqref{eq:dyson}, is stable against small general perturbations of the data pair 
$(a,S)$ consisting of the bare matrix $a$ and the self-energy operator $S$. 
To that end, let $T \subset \R$ contain $0$, $S_t \colon \alg \to \alg$, $t \in T$, be a family of positivity-preserving 
operators and $a_t = a_t^* \in \alg$, $t \in T$, be a family of self-adjoint elements. 
We set $S \defeq S_{t=0}$ and $a \defeq a_{t=0}$ and 
 will always assume that there are $c_1, \ldots, c_5>0$ such that 
\begin{equation} \label{eq:assums_perturbation}
 c_1 \avg{x} \id \leq S[x]\leq c_2 \avg{x}\id, \qquad \norm{a} \leq c_3, \qquad 
\norm{S-S_t} \leq c_4 t, \qquad \norm{a - a_t} \leq c_5 t 
\end{equation}
for all $x \in \algnon$ and for all $t \in T$. 
For any $t \in T$, let $m_t$ be the solution to the Dyson equation associated to the data pair $(a_t,S_t)$, i.e., 
\begin{equation} \label{eq:dyson_Hoelder_S}
 -m_t(z)^{-1} = z\id - a_t + S_t[m_t(z)]  
\end{equation}
for $z \in \Hb$ (cf. \eqref{eq:dyson}). We also set $m \defeq m_{t=0}$. 

The main result of this section, Proposition \ref{pro:Hoelder_continuity_S} below, states that $\norm{m_t(z)-m(z)}$ is small for sufficiently small $t$ and all $z$ away from points, 
where $m(z)$ blows up. 
Depending on the location of $z$, there are three cases for the estimate: we obtain the best estimate of order $\abs{t}$ on $\norm{m_t(z)-m(z)}$ in the bulk,  
the estimate is weaker, of order $\abs{t}^{1/2}$, if $z$ is close to a regular edge and the weakest, of order $\abs{t}^{1/3}$, if $z$ is close to an (almost) cusp point. 

We now introduce these concepts precisely.  
For a given $m_* >0$, we define the set $P_m \defeq P_m^{m_*} \subset \Hb$, where $\norm{m(z)}$ is larger than $m_*$, i.e., 
\[ P_m^{m_*} \defeq \{ \tau \in \R : \sup_{\eta >0} \, \norm{m(\tau + \ii \eta)} > m_* \}.  \] 
For any fixed $m_*>0$ and $\delta>0$, we introduce the set $\Dbdd$ of points of distance at least $\delta$ from $P_m$, i.e., 
\begin{equation} \label{eq:def_away_from_almost_cusp}
 \Dbdd \defeq \Dbdd^{m_*,\delta}\defeq \{ z \in \Hb : \dist(z,P_m) \geq \delta \}. 
\end{equation}
Note that $\norm{m(z)} \leq \max\{m_*, \delta^{-1}\}$ for all $z \in \Dbdd$ as $\norm{m(z)} \leq (\dist(z,\supp\dens))^{-1}$ by \eqref{eq:stieltjes_bound_m}. 

We now introduce the concept of the \emph{bulk}. 
Since $S \in \Sigma_\rm{flat}$, the self-consistent density of states of $m$ (cf. Definition~\ref{def:density_of_states}) has a continuous density $\dens \colon \R \to [0,\infty)$ 
with respect to the Lebesgue measure (cf. Proposition \ref{pro:regularity_density_of_states}). 
We also write $\dens$ for the harmonic extension of $\dens$ to $\Hb$ which satisfies $\dens(z) = \avg{\Im m(z)}/\pi$ for $z \in \Hb$. 
For $\dens_* >0$ and $\delta_s >0$, we denote those points, where $\dens$ is bigger than $\dens_*$ or which are at least $\delta_s$ away from $\supp\dens$, by 
\[ \Dbulk \defeq \Dbulk^{\dens_*} \defeq \{ z \in \Hb : \dens(z) \geq \dens_*\}, \qquad \qquad \Dout \defeq \Dout^{\delta_s} \defeq \{ z \in \Hb \colon \dist(z,\supp \dens) \geq \delta_s\}, \] 
respectively. 
We remark that, for fixed $\dens_*$ and $\delta_s$, we have the inclusion $\Dbulk\cup \Dout \subset \Dbdd$ for all sufficiently large $m_*$ and sufficiently small $\delta$ by \eqref{eq:m_bound_dens}. 

For $\tau \in \R \setminus \supp\dens$, let $\Delta(\tau)$ denote the size of the largest interval that contains $\tau$ and is contained in $\R \setminus \supp \dens$.  
 For $\dens_*>0$ and $\Delta_*>0$, we define the set $P_\mathrm{cusp} = P_\mathrm{cusp}^{\dens_*, \Delta_*} \subset \R$ of \emph{almost cusp points} through 
\begin{equation} \label{eq:def_P_cusp}
P_\mathrm{cusp}^{\dens_*,\Delta_*} \defeq  \{ \tau \in \supp\dens \setminus \pt\supp \dens : \tau \text{ is a local minimum of }\dens, ~\dens(\tau) \leq \dens_*\} \cup \{ \tau \in \R \setminus \supp\dens : \Delta(\tau) \leq \Delta_*\}. 
\end{equation}
For some $\delta_c>0$, we denote those points which are at least $\delta_c$ away from almost cusp points by
\[ \Dcusp \defeq \{ z \in \Hb : \dist(z, P_\mathrm{cusp}) \geq \delta_c \}. \] 
 We remark that $\Dmin = \mathbb{D}_\mathrm{bdd} \cap \mathbb{D}_\mathrm{cusp}$, where $\Dmin$ denotes 
the set of points which are away from $P_m$ and $P_\mathrm{cusp}$. More precisely, for some $\delta >0$, we define
\[ \Dmin \defeq \{ z \in \Hb : \dist(z, P_m) \geq \delta,~ \dist(z,P_\mathrm{cusp}) \geq \delta \}. \]

In this section, the model parameters are given by $c_1, \ldots, c_5$ from \eqref{eq:assums_perturbation} as well as the fixed parameters $m_*$, $\delta$, $\dens_*$, $\delta_s$, $\Delta_*$ and $\delta_c$ 
from the definitions of 
$P_m$, $\Dbdd$, $\Dbulk$, $\Dout$, $P_\mathrm{cusp}$, and $\Dcusp$, respectively. Thus, the comparison relation $\sim$ (compare Convention~\ref{conv:comparsion_1}) is understood with respect to these parameters throughout this section.

\begin{proposition}\label{pro:Hoelder_continuity_S} 
If the self-adjoint element $a=a_{t=0}$, $a_t$ in $\alg$ and the positivity-preserving operators $S=S_{t=0}$, $S_t$ on $\alg$ satisfy \eqref{eq:assums_perturbation} for each $t \in T$ then  
there is $t_* \sim 1$ such that 
\begin{enumerate}[label=(\alph*)]
\item 
Uniformly for all $z \in \Dbdd$ and for all $t \in [-t_*,t_*]\cap T$, we have  
\[ \norm{m_t(z) - m(z)} \lesssim \abs{t}^{1/3}. \]
In particular, $\norm{m_t(z)} \lesssim 1$ uniformly for all $z \in \Dbdd$ and for all $t \in [-t_*,t_*]\cap T$.
\item (Bulk and away from support of $\dens$) 
Uniformly for all $z \in \Dbulk\cup\Dout$ and for all $t \in [-t_*,t_*] \cap T$, we have 
\[ \norm{m_t(z)- m(z)} \lesssim \abs{t}. \] 
\item (Away from almost cusps)
Uniformly for all $z \in \Dcusp \cap \Dbdd$ and for all $t \in [-t_*,t_*] \cap T$, we have  
\[ \norm{m_t(z)- m(z)} \lesssim \abs{t}^{1/2}. \] 
\end{enumerate}
\end{proposition}

In order to simplify the notation, we set $\Delta m_t =\Delta m_t(z) = m_t(z) - m(z)$. 
The behaviour of $\Deltam$ will be governed by a scalar-valued cubic equation (see \eqref{eq:scalar_cubic_S} below). This is the origin of the cubic root $\abs{t}^{1/3}$ in the general 
estimate on $\norm{m_t(z)-m(z)}$ in Proposition~\ref{pro:Hoelder_continuity_S}. In the special cases, $z \in \Dbulk \cup \Dout$ and $z \in \Dcusp$, the cubic equation simplifies to 
a linear or quadratic equation, respectively, which yield the improved estimates $\abs{t}$ and $\abs{t}^{1/2}$, respectively.

We now define two positive auxiliary functions $\wt{\xi}_1(z)$ and $\wt{\xi}_2(z)$ for $z \in \Dbdd$ which will control the coefficients in the cubic equation mentioned above. 
For their definitions, we distinguish several subdomains of $\Dbdd$. 
The slight ambiguity of the definitions due to overlaps between these domains does, however, not affect the validity of the following statements as the different versions 
of $\wt{\xi}_1$ as well as $\wt{\xi}_2$ are comparable with each other with respect to the comparison relation $\sim$ and $\wt{\xi}_1$ as well as $\wt{\xi}_2$ are only 
used in bounds with respect to this comparison relation. 
For $\dens_* \sim 1$ and $\delta_* \sim 1$, we define
\begin{itemize} 
\begin{subequations}\label{eq:def_wt_xi} 
\item \textbf{Bulk:} If $z \in \Dbulk\cup\Dout$ then we set 
\begin{equation} \label{eq:def_wt_xi_bulk}
 \wt{\xi}_1(z) \defeq \wt{\xi}_2(z) \defeq 1. 
\end{equation}
\item \textbf{Around a regular edge:} If $z = \tau_0 + \omega + \ii \eta \in \Dcusp \cap \Dbdd$ with some $\tau_0 \in \pt\supp \dens$, $\omega\in [-\delta_*,\delta_*]$ and $\eta \in (0, \delta_*]$ then we set 
\begin{equation} \label{eq:def_wt_xi_regular_edge}
 \wt{\xi}_1(z) \defeq (\abs{\omega} + \eta)^{1/2}, \qquad \qquad \wt{\xi}_2(z)\defeq 1.
\end{equation}
\item \textbf{Close to an internal edge with a small gap:} Let $\alpha, \beta \in (\pt \supp\dens) \setminus P_m$ satisfy $\beta < \alpha$ and $(\beta, \alpha) \cap \supp\dens = \varnothing$. We set $\Delta \defeq \alpha-\beta$.
If $z \in\Dbdd$ satisfies $z = \alpha - \omega + \ii \eta$ or $z = \beta + \omega + \ii \eta$ for some $\omega \in  [-\delta_*,\Delta/2]$ and $\eta \in (0,\delta_*]$ then we define
\begin{equation} \label{eq:def_wt_xi_internal_edge}
\wt{\xi}_1(z) \defeq (\abs{\omega} + \eta)^{1/2} (\abs{\omega} + \eta + \Delta)^{1/6}, \qquad \qquad \wt{\xi}_2(z) \defeq (\abs{\omega} + \eta + \Delta)^{1/3} 
\end{equation}
\item \textbf{Around a small internal minimum:} If $z = \tau_0 + \omega + \ii \eta \in \Dbdd$, where $\tau_0 \in \supp\dens\setminus \pt\supp\dens$ is a local minimum of $\dens$ with $\dens(\tau_0) \leq \dens_*$, 
$\omega \in [-\delta_*,\delta_*]$ and $\eta \in (0,\delta_*]$ then we define
\begin{equation} \label{eq:def_wt_xi_minimum}
\wt{\xi}_1(z)\defeq (\dens(\tau_0) + ( \abs{\omega} + \eta)^{1/3})^2, \qquad \qquad \wt{\xi}_2(z) \defeq  \dens(\tau_0) + (\abs{\omega}+\eta)^{1/3}. 
\end{equation}
\end{subequations}
\end{itemize}

We remark that $\tau_0 \in \pt \supp \dens$ is a \emph{regular edge} if $\dens(\tau)=0$ for all 
$\tau \in [\tau_0-\eps,\tau_0]$ or $\tau \in [\tau_0,\tau_0+\eps]$ for some $\eps \sim 1$. 
In fact, $\overline{\Dcusp \cap \Dbdd} \cap \pt \supp\dens$ consists only of regular edges. 

In the proof of Proposition \ref{pro:Hoelder_continuity_S}, we will use the following two lemmas, whose 
proofs we postpone until the end of this section. 

\begin{lemma} \label{lem:cubic_equation1}
Let $\Dbdd$ be defined as in \eqref{eq:def_away_from_almost_cusp}.  Let $a$, $S$ and $(a_t)_{t \in T}$ and $(S_t)_{t \in T}$ satisfy \eqref{eq:assums_perturbation}.
Then there is $\eps_1 \sim 1$ such that if $\norm{\Deltam(z)} \leq \eps_1$ for some $z \in\Dbdd$, $t \in T$, then there are $l, b \in \alg$ depending on $z$ such that 
$\Theta_t\defeq \scalar{l}{\Deltam}/\scalar{l}{b}$ satisfies a cubic inequality
\begin{equation} \label{eq:scalar_cubic_S}
 \abs{\Theta_t^3 +{\xi}_2 \Theta_t^2 + {\xi}_1 \Theta_t} \lesssim \abs{t}
\end{equation}
with complex coefficients ${\xi}_1$ and ${\xi}_2$ depending on $z$ and $t$. 
The function $\Theta_t$ depends continuously on $\imz$ 
and we also have $\abs{\Theta_t} \lesssim \norm{\Deltam}$ as well as $\norm{\Deltam} \lesssim \abs{\Theta_t} + \abs{t}$ for all $t \in T$. 

The coefficients, ${\xi}_1$ and ${\xi}_2$, behave as follows: There are $\delta_* \sim 1$, $\dens_* \sim 1$ and $c_*\sim 1$ 
such that, with the appropriate definitions of $\wt{\xi}_1$ and $\wt{\xi}_2$ from \eqref{eq:def_wt_xi}, we have 
\begin{itemize} 
\begin{subequations} \label{eq:scaling_cubic_coefficients} 
\item If $z\in \Dbdd$ satisfies the conditions for \eqref{eq:def_wt_xi_bulk} or \eqref{eq:def_wt_xi_internal_edge} with $\omega \in [c_* \Delta,\Delta/2]$ then we have
\begin{equation} \label{eq:scaling_cubic_coefficients_sim_lesssim} 
\abs{\xi_1(z)} \sim \wt{\xi}_1(z), \qquad \abs{\xi_2(z)} \lesssim \wt{\xi}_2(z). 
\end{equation}
\item If $z\in \Dbdd$ satisfies the conditions for \eqref{eq:def_wt_xi_internal_edge} with $\omega \in [-\delta_*,c_*\Delta]$ or \eqref{eq:def_wt_xi_regular_edge} or \eqref{eq:def_wt_xi_minimum} then we have 
\begin{equation}
 \abs{\xi_1(z)} \sim \wt{\xi}_1(z), \qquad \abs{\xi_2(z)} \sim \wt{\xi}_2(z). 
\end{equation}
\end{subequations}
\end{itemize}
All implicit constants in this lemma are uniform for any $t \in T$. 
\end{lemma} 

\begin{lemma} \label{lem:cubic_equation2}
For $0<\eta_\ast<\eta^\ast<\infty$, let $\xi_1,{\xi}_2\colon [\eta_\ast,\eta^\ast] \to \C$ be complex-valued functions and $\wt{\xi}_1,\wt{\xi}_2, d\colon [\eta_\ast,\eta^\ast]  \to \R^+ $ be continuous. 

Suppose that some continuous function $\Theta\colon [\eta_{\ast},\eta^{\ast}]  \to \C$ satisfies the cubic inequality 
\begin{equation} \label{eq:scalar_cubic_S_condition}
\abs{\Theta^3 + {\xi}_2 \Theta^2  + {\xi}_1 \Theta} \lesssim d 
\end{equation}
on $[\eta_\ast, \eta^\ast]$ as well as 
\begin{equation} \label{eq:condition_cubic_boostraping}
\abs{\Theta}  \,\lesssim\, \min\cbb{d^{1/3}, \frac{d^{1/2}}{\wt{\xi}_2^{1/2}},\frac{d}{\wt{\xi}_1}} 
\end{equation}
at $\eta_\ast$. 
If one of the following two sets of relations holds true:
\begin{enumerate}[label=\arabic*)]
\item 
\begin{enumerate}[label=(\roman*)]
\item $\wt{\xi}_2^3/d$, $\wt{\xi}_1^3/d^2$, $\wt{\xi}_1^2/(d\1\wt{\xi}_2)$ are monotonically increasing functions, 
\item $\abs{\xi_1}\sim \wt{\xi}_1$, $\abs{\xi_2}\sim \wt{\xi}_2$, 
\item $d^2/\wt{\xi}_1^3+d\1\wt{\xi}_2/\wt{\xi}_1^2$ at $\eta^\ast$ is sufficiently small depending on the implicit constants in 1) (ii) as well as \eqref{eq:scalar_cubic_S_condition} and \eqref{eq:condition_cubic_boostraping}.
\end{enumerate} 
\item  
\begin{enumerate}[label=(\roman*)] 
\item $\wt{\xi}_1^3/d^2$ is a monotonically increasing function, 
\item $\abs{{\xi}_1}\sim \wt{\xi}_1$,  $\abs{\xi_2}\lesssim \wt{\xi}_1^{1/2}$. 
\end{enumerate} 
\end{enumerate}
then, on $[\eta_\ast, \eta^\ast]$, we have the bound 
\begin{equation} \label{eq:conclusion_cubic_boostrapping} 
\abs{\Theta}  \,\lesssim\, \min\cbb{d^{1/3}, \frac{d^{1/2}}{\wt{\xi}_2^{1/2}},\frac{d}{\wt{\xi}_1}}. 
\end{equation}
\end{lemma}

\begin{proof}[Proof of Proposition \ref{pro:Hoelder_continuity_S}] 
We start the proof by introducing the control parameter $M(t)$. 
Let $\wt{\xi}_1$ and $\wt{\xi}_2$ be defined as in \eqref{eq:def_wt_xi}. 
For $t \in \R$, we set 
\begin{equation} \label{eq:def_M_t}
M(t) \defeq \min \{ \abs{t}^{1/3} , \wt{\xi}_2^{-1/2}\abs{t}^{1/2}, \wt{\xi}_1^{-1} \abs{t} \}. 
\end{equation}
We remark that $M$ also depends on $z$ as $\wt{\xi}_1$ and $\wt{\xi}_2$ depend on $z$. 

We will prove below that there are $t_* \sim 1$ and $C \sim 1$ such that, for any fixed $t \in [-t_*,t_*]\cap T\setminus \{0\}$ (if this set is nonempty) and $z \in\Dbdd$, we have the implication
\begin{equation} \label{eq:improving_bound}
  \norm{\Deltam(\Re z+ \ii \eta)} \leq \eps_1 \quad \text{ for all } \eta \geq \imz \qquad \Rightarrow \qquad \norm{\Deltam(z)} \leq C M(t), 
\end{equation}
where $\eps_1\sim 1$ is from Lemma \ref{lem:cubic_equation1}. 

Armed with \eqref{eq:improving_bound}, by possibly shrinking $t_* \sim 1$, we can assume that $2Ct_*^{1/3} \leq \eps_1$. 
We fix $\tau\in\R$ and $t \in [-t_*,t_*]\cap T\setminus \{0\}$ and set 
\[ \eta_* \defeq \sup \{ \eta>0 : \norm{\Deltam(\tau + \ii\eta)} \geq 2C M(t) \}. \] 
Here, we use the convention $\eta_* = - \infty$ if the set is empty. 
Note that $\norm{\Deltam(\tau + \ii \eta)} \leq 2 \eta^{-1}$ since $m$ and $m_t$ are Stieltjes transforms. 
Hence, $\eta_* < \infty$ as $t \neq 0$. 

We prove now that $\eta_* \leq \inf \{ \imz : z \in \Dbdd, ~\Re z = \tau\}$. 
For a proof by contradiction, we suppose that there is $z_* \in \Dbdd$ such that $\Re z_* = \tau$ and $\Im z_* = \eta_*$ (note that if $\tau + \ii \eta \in \Dbdd$ then 
$\tau + \ii \eta'\in \Dbdd$ for any $\eta' \geq \eta$). 
Since $\Deltam$ is continuous in $z$, we have $\norm{\Deltam(z_*)} = 2CM(t)$. 
Thus, $\norm{\Deltam(\tau + \ii \eta)} \leq 2Ct_*^{1/3} \leq \eps_1$ for all $\eta \geq \eta_*$ by the choice of $t_*$. 
From \eqref{eq:improving_bound}, we conclude $\norm{\Deltam(z_*)} \leq C M(t)$, which 
contradicts $\norm{\Deltam(z_*)}=2 C M(t)$. Thus, $\eta_* \leq \inf\{ \imz: z \in\Dbdd, ~ \Re z = \tau \}$. 

As $\tau$ was arbitrary, this yields $\norm{\Deltam(z)} \leq 2C M(t)$ for all $z \in \Dbdd$, which proves part (a) of Proposition~\ref{pro:Hoelder_continuity_S} up to \eqref{eq:improving_bound}. 
Since $\wt{\xi}_1(z) \sim 1$ for $z \in \Dbulk\cup\Dout$ and $\wt{\xi}_2(z) \sim 1$ for $z \in \Dcusp\cap \Dbdd$, we also obtain part (b) and (c) from the definition of $M$ in \eqref{eq:def_M_t}. 

Hence, it suffices to show \eqref{eq:improving_bound} to complete the proof of Proposition~\ref{pro:Hoelder_continuity_S}. 
In order to prove \eqref{eq:improving_bound}, we use Lemma~\ref{lem:cubic_equation2}
with $\Theta(\eta) = \Theta_{t}(\Re z + \ii \eta)$, $\eta \geq \eta_* \defeq \Im z$, $d=\abs{t}$, and 
$\xi_1$, $\xi_2$ and $\wt{\xi}_1$, $\wt{\xi}_2$ are chosen as in \eqref{eq:scalar_cubic_S} of Lemma~\ref{lem:cubic_equation1} and \eqref{eq:def_wt_xi}, respectively. 
As $\norm{\Deltam(\Re z + \ii \eta)} \leq \eps_1$ for all $\eta \geq \imz$, we conclude that \eqref{eq:scalar_cubic_S_condition} is satisfied with $d=\abs{t}$ due to~\eqref{eq:scalar_cubic_S}. 

We first consider $z \in\Dbulk\cup\Dout$. 
If $z \in\Dbulk\cup\Dout$ then $\Re z + \ii \eta \in \Dbulk\cup\Dout$ and $\xi_1(\Re z + \ii \eta) = \xi_2(\Re z + \ii \eta) = 1$ for all $\eta \geq \eta_*$ and assumption 2) of Lemma~\ref{lem:cubic_equation2} 
is always fulfilled. 
Since $\norm{\Deltam(\Re z + \ii \eta)} \leq 2 \eta^{-1}$ as remarked above and $t \neq 0$, the condition 
in \eqref{eq:condition_cubic_boostraping} is met for some sufficiently large $\eta>0$. 
Hence, by Lemma~\ref{lem:cubic_equation2}, there is $C \sim 1$ such that $\abs{\Theta_{t}(z)} \leq C M(t)$.
Possibly increasing $C \sim 1$ and using $\abs{t} \leq t_* \sim 1$ yield $\norm{\Deltam(z)} \leq C M(t)$ due 
to $\norm{\Deltam} \lesssim \abs{\Theta_t} + \abs{t}$ from Lemma \ref{lem:cubic_equation1}. 

For each $z \in\Dbdd\setminus \Dbulk\cup\Dout$, due to \eqref{eq:scaling_cubic_coefficients}, we have $\xi_1(z_\delta) \sim 1$ and $\xi_2(z_\delta) \sim 1$ for $z_\delta \defeq \Re z + \ii \delta_*$, 
where $\delta_*\sim 1$ is as in Lemma~\ref{lem:cubic_equation1}. 
Hence, we conclude $\abs{\Theta_t(z_\delta)} \leq C M(t)$ as for $z\in \Dbulk\cup\Dout$. For each $z \in\Dbdd\setminus \Dbulk\cup\Dout$, the validity of assumption 1) or assumption 2) of Lemma~\ref{lem:cubic_equation2} can be read off 
from \eqref{eq:scaling_cubic_coefficients}. 
Lemma~\ref{lem:cubic_equation2}, thus, implies $\abs{\Theta_t(z)} \leq C M(t)$. As before, we conclude $\norm{\Deltam(z)} \leq CM(t)$ from Lemma~\ref{lem:cubic_equation1}. 
This completes the proof of \eqref{eq:improving_bound} and, hence, the one of Proposition \ref{pro:Hoelder_continuity_S}. 
\end{proof}

\begin{proof}[Proof of Lemma \ref{lem:cubic_equation1}]
We remark that a straightforward computation starting from \eqref{eq:dyson} and \eqref{eq:dyson_Hoelder_S} yields 
\begin{equation} \label{eq:quadratic_equation_Hoelder_S}
 B[\Deltam] = A[\Deltam, \Deltam] + K[\Delta^{S}, \Delta^{a}, \Deltam] + T[\Delta^{S}, \Delta^{a}],  
\end{equation}
where $B \defeq \Id - C_mS$, $A[x,y] \defeq (mS[x] y + yS[x] m)/2$ are defined as in \eqref{eq:B_A_general_dyson}, 
$\Delta^{S} \defeq S_t - S$, $\Delta^{a} \defeq a_t - a$ and 
\[ \begin{aligned} 
K[\Delta^{S}, \Delta^{a}, \Deltam] \, =\,  \, &  \frac{1}{2} ( m \Delta^{S}[\Deltam] \Deltam  + \Deltam \Delta^{S}[\Deltam] m + m \Delta^{S}[m] \Deltam 
+ \Deltam \Delta^{S}[m] m ) \\ 
  & - \frac{1}{2} ( m \Delta^{a} \Deltam + \Deltam \Delta^{a} m), \\  
 T[\Delta^{S}, \Delta^{a}] \, = \, \, & m \Delta^{S}[m] m -m \Delta^{a} m . 
\end{aligned} \]

In the following, we will split $\Dbdd$ into two regimes and choose $l$ and $b$ according to the regime. In both cases, we use the definitions
\begin{equation} \label{eq:def_Theta_r}
 \Theta \defeq \Theta_t = \frac{\scalar{l}{\Deltam}}{\scalar{l}{b}}, \qquad r = r_t \defeq Q[\Deltam], \qquad Q \defeq \Id - \frac{\scalar{l}{\genarg}}{\scalar{l}{b}} b .
\end{equation}
In particular, $\Deltam = \Theta b + r$. 
We denote by $\dens(z)$ the harmonic extension of $\dens$, i.e., $\dens(z) = \avg{\Im m(z)}/\pi$. 

If $z$ is close to a regular edge or close to an almost cusp point then 
$\Deltam(z)$ is governed by a quadratic or cubic equation for $\Theta_t$, respectively, where $l$ and $b$ are a left and 
a right eigenvector of $B$, respectively. 
If $z$ is in the bulk or away from $\supp\dens$ then $\Deltam(z)$ can be controlled by $\Theta_t$ with $l=b=\id$
and $\Theta_t$ is the solution of a scalar-valued linear equation. 
Note that in the bulk and away from $\supp \dens$ the choice $l = b= \id$ is arbitrary, 
in fact the splitting $\Deltam= \Theta_t b +r$ is artificial since the stability operator does not have a 
distinguished ``bad'' direction that needs to be treated separately. 
We still use this formalism in order to treat all three cases uniformly for the sake of brevity. For a similar 
reason we will always write the equation for $\Theta_t$ as a cubic equation, sometimes by adding and 
subtracting apparently superfluous (and negligible) terms. 

\emph{Case 1:}
We first assume that $z \in \Dbdd$ satisfies $\dens(z) \geq \dens_*$ for some $\dens_*\sim 1$ or $\dist(z,\supp\dens) \geq \delta$ for some $\delta\sim 1$, i.e., $z \in\Dbulk^{\dens_*} \cup \Dout^{\delta}$. 
This implies that $B$ is invertible and $\norm{B^{-1}} \lesssim 1$ due to \eqref{eq:linear_stability}, $\normtwoinf{S} \lesssim 1$, $\norm{m(z)} \lesssim 1$ and
Lemma~\ref{lem:norm_two_to_inf} (ii).
In this case, we choose $l = b = \id$ and apply $QB^{-1}$ to \eqref{eq:quadratic_equation_Hoelder_S} to obtain 
\[ r = QB^{-1} (A[\Deltam, \Deltam] + K[\Delta^{S},\Delta^{a}, \Deltam] + T[\Delta^{S},\Delta^{a}]) = \ord(\abs{\Theta}^2 + \norm{r} \norm{\Deltam} + \abs{t}), \] 
where we used that $\norm{m} \lesssim 1$ on $\Dbdd$ as well as $\norm{\Delta^{S}}+\norm{\Delta^{a}} \lesssim \abs{t}$. 
Shrinking $\eps_1 \sim 1$, using $\norm{\Deltam} \leq \eps_1$ and absorbing $\norm{r} \norm{\Deltam}$ into the left-hand side yield $\norm{r} \lesssim \abs{\Theta}^2 + \abs{t}$. 
Thus, $\norm{\Deltam} \lesssim \abs{\Theta} + \abs{t}$. 
Hence, applying $B^{-1}$ and $\avg{\genarg}$ to \eqref{eq:quadratic_equation_Hoelder_S} and using $\avg{r}=0$ as well as $\norm{\Deltam} \lesssim \abs{\Theta} + \abs{t}$, we find $\xi_2\in\C$ such that 
$\abs{\xi_2} \lesssim 1 = \wt{\xi}_2$ and 
\[ \Theta = -\xi_2 \Theta^2 + \ord(\abs{t} \abs{\Theta} + \abs{t}) = -{\xi}_2 \Theta^2 + \ord(\abs{t}).  \] 
Adding and subtracting $\Theta^3$ on the left-hand side as well as setting ${\xi}_1 \defeq 1 - \Theta^2$ show \eqref{eq:scalar_cubic_S} in Case 1 for sufficiently small $\eps_1 \sim 1$ 
as $\abs{\Theta} \lesssim \norm{\Deltam} \leq \eps_1$ implies  $\abs{{\xi}_1} \sim 1=\wt{\xi}_1$.  
This completes the proof of \eqref{eq:scaling_cubic_coefficients_sim_lesssim} for $z \in\Dbulk \cup \Dout$. 

\emph{Case 2:} We now prove \eqref{eq:scalar_cubic_S} for $z \in \Dbdd$ satisfying $\dens(z) \leq \dens_*$ and $\dist(z,\supp\dens) \leq \delta$ with sufficiently small $\dens_* \sim 1$ and $\delta \sim 1$. 
For any $\eps_* \sim 1$, we find $\delta \sim 1$ such that $\dens(z)^{-1} \imz \leq \eps_*$ for all $z \in \Hb$ satisfying $\dist(z,\supp\dens) \leq \delta$ 
due to \eqref{eq:dens_reciprocal_imz_equal_zero} and the $1/3$-Hölder continuity of $z \mapsto \dens(z)^{-1} \imz$ 
by Lemma~\ref{lem:q_u_f_u_extension} (ii). 
Therefore, using $\dens(z) \leq \dens_*$, we see that Lemma~\ref{lem:prop_F_small_dens} and Corollary~\ref{coro:eigenvector_expansion} are applicable 
for sufficiently small $\dens_*\sim 1$ and $\delta \sim 1$. 
They yield $l, b \in \alg$ which we use to define $\Theta$ and $r$ as in \eqref{eq:def_Theta_r}, i.e., $\Deltam = \Theta b + r$ and $\Theta = \scalar{l}{\Deltam}/\scalar{l}{b}$. 

In order to derive \eqref{eq:scalar_cubic_S}, we now follow the proof of Lemma \ref{lem:general_cubic_equation} applied to \eqref{eq:quadratic_equation_Hoelder_S} instead of \eqref{eq:quadratic_eq}. 
Here, $\Delta^a$ and $\Delta^S$ play the role of $e$. 
In fact, by Lemma \ref{lem:prop_F_small_dens} and Corollary~\ref{coro:eigenvector_expansion}, the first two 
bounds in \eqref{eq:assums_derivation_cubic_equation} are fulfilled. 
Owing to $\norm{m} \lesssim 1$, the third bound in \eqref{eq:assums_derivation_cubic_equation} is trivially satisfied. 
Instead of the last two bounds in \eqref{eq:assums_derivation_cubic_equation}, we use 
\[ \norm{T[\Delta^S,\Delta^{a}]} \lesssim \norm{\Delta^S} + \norm{\Delta^{a}}, \qquad \norm{K[\Delta^S, \Delta^{a},\Deltam]} \lesssim (\norm{\Delta^S} + \norm{\Delta^{a}}) \norm{\Deltam}, \]
due to $\norm{m} \lesssim 1$ and $\norm{\Deltam} \lesssim 1$. 
In fact, the last bound in \eqref{eq:assums_derivation_cubic_equation} will not hold true for a general $y \in \alg$ but in the proof of Lemma \ref{lem:general_cubic_equation} it is only used with the special choice $y = \Deltam$.
We choose $\eps_1 \leq \eps$ for $\eps$ from Lemma~\ref{lem:general_cubic_equation} and obtain 
the cubic equation \eqref{eq:general_cubic_equation} from Lemma~\ref{lem:general_cubic_equation}
with $\mu_0 = \scalar{l}{T[\Delta^S,\Delta^{a}]}$ and $\norm{e}$ replaced by $\abs{t}$ as $\norm{\Delta^{S}} + \norm{\Delta^{a}}\lesssim \abs{t}$. 
In particular, $\abs{\mu_0} \lesssim \abs{t}$. 
We decompose the error term $\wt{e}=\ord(\abs{\Theta}^4 + \abs{t} \abs{\Theta} + \abs{t}^2)$ from \eqref{eq:general_cubic_equation} 
into $\wt{e} = \wt{e}_1\Theta^3 + \wt{e}_2$ with $\wt{e}_1, \wt{e}_2\in \C$ satisfying $\wt{e}_1 = \ord(\abs{\Theta})$ and $\wt{e}_2 = \ord(\abs{t} \abs{\Theta} + \abs{t}^2)$. 
With the notation of Lemma~\ref{lem:general_cubic_equation}, the cubic equation \eqref{eq:general_cubic_equation} can be written as 
\[ (\mu_3-\wt{e}_1) \Theta^3 + \mu_2  \Theta^2 + \mu_1 \Theta = - \mu_0 +\wt{e}_2 = \ord(\abs{t}). \] 
Since $A$ and $B$ introduced above have the same definitions as in \eqref{eq:B_A_general_dyson} and $\mu_3$, $\mu_2$ and $\mu_1$ in \eqref{eq:Coefficients mu3 to mu0} depend only on $A$ and $B$, 
Lemma~\ref{lem:coefficients_general_cubic} yields the expansions of $\mu_3$, $\mu_2$ and $\mu_1$ in \eqref{eq:coefficients_dyson_general} for sufficiently small $\dens_* \sim 1$ and $\delta \sim 1$. 
By possibly shrinking $\eps_1 \sim 1$, we find $c \sim 1$ such that $\abs{\mu_3-\wt{e}_1} + \abs{\mu_2} \geq 2 c$ as $\abs{\wt{e}_1} \lesssim \abs{\Theta} \lesssim \norm{\Deltam} \leq \eps_1$. 
Here, we also used $\abs{\mu_3} + \abs{\mu_2} \gtrsim \psi + \abs{\sigma}$ by \eqref{eq:coefficients_dyson_general} 
as well as \eqref{eq:psi_plus_sigma_sim_1}. 

Consequently, we obtain \eqref{eq:scalar_cubic_S}, where we introduced 
\[ \begin{aligned} 
 {\xi}_2 & \defeq \bigg(\mu_2 + (\mu_3 -\wt{e}_1-1)\Theta\bigg) \char( \abs{\mu_2} \geq c) + \frac{\mu_2}{\mu_3-\wt{e}_1} \char(\abs{\mu_2} < c), \\
{\xi}_1 & \defeq  \mu_1 \char(\abs{\mu_2} \geq c) + \frac{\mu_1}{\mu_3-\wt{e}_1} \char (\abs{\mu_2} < c) . 
\end{aligned}\] 
Hence, we have $\abs{{\xi}_2} \sim \abs{\mu_2}$ and $\abs{{\xi}_1} \sim \abs{\mu_1}$ for sufficiently small $\eps_1 \sim 1$ as $\abs{\wt{e}_1} \lesssim \abs{\Theta}$ and $\abs{\Theta} \lesssim \norm{\Deltam} \leq \eps_1$.
This completes the proof of \eqref{eq:scalar_cubic_S} in Case 2. 

It remains to show the scaling relations in \eqref{eq:scaling_cubic_coefficients} for $z \in \Dbdd$ satisfying $\dens(z) \leq \dens_*$ and $\dist(z,\supp\dens) \leq \delta$ in order to complete the proof of Lemma~\ref{lem:cubic_equation1}. 
Starting from $\abs{{\xi}_1} \sim \abs{\mu_1}$ and $\abs{{\xi}_2} \sim \abs{\mu_2}$ proven in Case 2, we conclude as in the proof of (10.6) in \cite{AjankiQVE} that  
\[ \abs{{\xi}_1} \sim \dens(z)^2 + \abs{\sigma(z)} \dens(z) + \dens(z)^{-1}\imz, \qquad \abs{{\xi}_2} \sim \dens(z) + \abs{\sigma(z)},  \] 
where $\sigma$ is defined as in \eqref{eq:def_psi_sigma}. 
Here, $\xi_1$ and $\xi_2$ play the role of $\pi_1$ and $\pi_2$, respectively, in \cite{AjankiQVE}. Their definitions differ slightly but this does not affect the straightforward estimates. 
Note that the proof in \cite{AjankiQVE} relies on the expansions of $\mu_1$, $\mu_2$ and $\mu_3$ from (8.33) in \cite{AjankiQVE}. These are the exact analogues of \eqref{eq:coefficients_dyson_general}, 
where $\dens$ plays the role of $\alpha$ from \cite{AjankiQVE}. 

Note that according to Remark~\ref{rmk:scaling im m} the harmonic extension $\dens(z)$ 
for $z \in \Hb$ in the vicinity of the singularities has the same scaling behavior as in Corollary~A.1 of \cite{AjankiQVE}.  
Similarly, the proof of (10.7) in \cite{AjankiQVE} yields 
\begin{equation} \label{eq:scaling_sigma} 
\abs{\sigma(\beta)} \sim \abs{\sigma(\alpha)} \sim (\alpha-\beta)^{1/3}, \qquad \abs{\sigma(\tau_0)} \lesssim \dens(\tau_0)^2,
\end{equation}
where $\alpha, \beta \in (\pt \supp\dens)\setminus P_m$ satisfy $\beta < \alpha$ and $(\beta, \alpha) \cap \supp\dens = \varnothing$ and $\tau_0 \in \supp\dens\setminus \pt\supp\dens$ is 
a local minimum of $\dens$ and $\dens(\tau_0) \leq \dens_*$. 
Here, we use Lemma~\ref{lem:small_gap} above and $\abs{\sigma} \sim \wh{\Delta}^{1/3}$ by Theorem~\ref{thm:abstract_cubic_equation} (ii) (b) 
 instead of Lemma~9.17 in \cite{AjankiQVE} and Lemma~\ref{lem:sigma_nonzero_local_minimum} above instead of Lemma~9.2 in \cite{AjankiQVE}. 
We then follow the proof of Proposition~4.3 in \cite{Ajankirandommatrix} and use the $1/3$-Hölder continuity of $\sigma$ proven in Lemma~\ref{lem:stability_cubic_equation} (i). 
This yields the missing scaling relations in \eqref{eq:scaling_cubic_coefficients}. 

We remark that $\Theta_t$ constructed above is not continuous in $\Im z$ due to the separation into two cases. 
However, there is only one transition between Case 1 and Case 2 for $z \in \Dbdd$ when $\Im z$ is varied while $\Re z$ is kept fixed.
 Therefore, we obtain a continuous version of $\Theta_t$ by a simple interpolation between these 
two cases in the vicinity of this transition point.  
We leave the details of this interpolation argument to the reader. 
This completes the proof of Lemma~\ref{lem:cubic_equation1}. 
\end{proof}

\begin{remark}[Scaling of coefficients]
The proof of Lemma~\ref{lem:cubic_equation1} can equally well be carried out under Assumption~\ref{assums:general} instead of the flatness condition in \eqref{eq:assums_perturbation}. In particular, it shows that in the setting of Theorem~\ref{thm:behaviour_dens}, there are $\delta_* \sim 1, \rho_* \sim 1$ and $c_* \sim 1$ such that the following comparison relations hold  for $z \in I_\theta + \ii [0,\eta_\ast]$:
\begin{itemize} 
\item If $z$ satisfies the conditions for \eqref{eq:def_wt_xi_bulk} or \eqref{eq:def_wt_xi_internal_edge} with $\omega \in [c_* \Delta,\Delta/2]$, then we have
\begin{equation*}
 \dens(z)^2 + \abs{\sigma(z)} \dens(z) + \dens(z)^{-1}\imz \sim \wt{\xi}_1(z), \qquad \dens(z) + \abs{\sigma(z)} \lesssim \wt{\xi}_2(z). 
\end{equation*}
\item If $z$ satisfies the conditions for  \eqref{eq:def_wt_xi_internal_edge} with $\omega \in [-\delta_*,c_*\Delta]$ or \eqref{eq:def_wt_xi_regular_edge} or \eqref{eq:def_wt_xi_minimum} with $\rho(\tau_0) \le \rho_*$,  then we have 
\begin{equation*}
  \dens(z)^2 + \abs{\sigma(z)} \dens(z) + \dens(z)^{-1}\imz \sim \wt{\xi}_1(z) , \qquad \dens(z) + \abs{\sigma(z)} \sim \wt{\xi}_2(z). 
\end{equation*}
\end{itemize}
\end{remark}

\begin{proof}[Proof of Lemma \ref{lem:cubic_equation2}]
By dividing the cubic inequality through $d$ and considering $\frac{\Theta}{d^{1/3}}$ instead of $\Theta$, we may assume that $d=1$. We fix $\eps \in (0,1)$ sufficiently small. First we prove the lemma under assumption 1).  
Owing to the smallness of $\frac{1}{\wt{\xi}_1^3}+\frac{{\wt{\xi}_2}}{\wt{\xi}_1^2}$ at $\eta^\ast$ as well as the monotonicity of $\wt{\xi}_1$ and $\frac{\wt{\xi}_1^2}{\wt{\xi}_2}$ there are $0<\eta_1, \eta_2<\eta^\ast$ with the following properties:
(i) $\wt{\xi}_2 \ge \eps^4 \wt{\xi}_1^2$ on $[\eta_\ast,\eta_1]$; (ii) $\wt{\xi}_2 \le \eps^4 \wt{\xi}_1^2$ on $[\eta_1,\eta^\ast]$; (iii) $\eps\wt{\xi}_1 \le 1$ on  $[\eta_\ast,\eta_2]$; (iv) $\eps\wt{\xi}_1 \ge 1$ on $[\eta_2,\eta^\ast]$.
Here the intervals $[\eta_\ast,\eta_{2}]$ and $[\eta_\ast,\eta_1]$ may be empty. 
We will now assume the bound $\abs{\Theta}  \lesssim \min\cb{1, \frac{1}{\wt{\xi}_2^{1/2}},\frac{1}{{\wt{\xi}_1}}}$ at the initial value $\eta^\ast$ and bootstrap it down to $\eta_\ast$.
Now we distinguish two cases:\\[0.3cm]
\emph{Case 1 ($\eta_1 \ge \eta_2$): }  On $[\eta_1,\eta^\ast]$ we have $\eps\1\wt{\xi}_1 \ge 1$ and $\wt{\xi}_2 \le \eps^4 \wt{\xi}_1^2$. Thus, by the cubic inequality
\[
\abs{\Theta}\,\lesssim \,  \min\cbb{1, \frac{1}{{\wt{\xi}_2^{1/2}}}}\quad \text{implies} \quad \abs{\Theta} \,\lesssim \, \frac{1}{\wt{\xi}_1}\,\lesssim \,  \min\cbb{\eps, \frac{\eps^2}{{\wt{\xi}_2^{1/2}}}}\,.
\]
In particular, there is a gap in the values of $\abs{\Theta}$ and by continuity all values lie below the gap on $[\eta_1,\eta^\ast]$. 

The interval $[\eta_\ast,\eta_1]$ is split again, $[\eta_\ast,\eta_1] = [\eta_\ast,\eta_3]\cup [\eta_3,\eta_1]$, where $\eta_3$ is chosen such that 
(i) $\wt{\xi}_2\eps^2\ge 1$ on $[\eta_3,\eta_1]$; (ii) $\wt{\xi}_2\eps^2\le 1$ on $[\eta_\ast,\eta_3]$. 
Here one or both of these intervals may be empty. Using $\wt{\xi}_2 \ge \eps^4 \wt{\xi}_1^2$ we see that on $[\eta_3,\eta_1]$ the bound 
\[
\abs{\Theta}\,\lesssim \, \min \cbb{\frac{1}{\eps}, \frac{1}{\eps^3\wt{\xi}_1}}\quad \text{implies} \quad
\abs{\Theta} \,\lesssim\, \frac{1}{\eps^{3/2}\wt{\xi}_2^{1/2}}\,\lesssim\, \min\cbb{\frac{1}{\eps^{1/2}},\frac{1}{\eps^{7/2} \wt{\xi}_1}}\,.
\]
Again the gap in the values of $\abs{\Theta}$ allows us to infer from the bound $\abs{\Theta}  \lesssim \min\cb{1, \frac{1}{{\wt{\xi}_2^{1/2}}},\frac{1}{{\wt{\xi}_1}}}$ at $\eta_1$ that $\abs{\Theta}$ satisfies the same bound on $[\eta_3, \eta_1]$ up to an $\eps$-dependent multiplicative constant. 

Finally, on $[\eta_\ast,\eta_3]$ we have $\wt{\xi}_2\le \eps^{-2}$ and $\wt{\xi}_1^2 \le \eps^{-4}\wt{\xi}_2 \le \eps^{-6}$. Using the cubic inequality this immediately implies 
$\abs{\Theta} {\lesssim_\eps} 1 { \lesssim_\eps}\min\cb{1, \frac{1}{{\wt{\xi}_2^{1/2}}},\frac{1}{{\wt{\xi}_1}}}$.
Here and in the following, the notation $\lesssim_\eps$ indicates that the implicit constant in the bound is allowed to depend on $\eps$.  
\\[0.3cm]
\emph{Case 2 ($\eta_1 \le \eta_2$): } On $[\eta_2, \eta^\ast]$ we have $\eps \wt{\xi}_1 \ge 1$ and $\wt{\xi}_2 \le \eps^4 \wt{\xi}_1^2$. So this regime is treated exactly as in the beginning of \emph{Case 1}. On $[\eta_\ast, \eta_2]$ we have $\eps\1\wt{\xi}_1\le 1$ and $\wt{\xi}_2\le \wt{\xi}_2(\eta_2)\le \eps^4\wt{\xi}_1(\eta_2)^2=\eps^2$, which implies $\abs{\Theta} {\lesssim_\eps} 1 {\lesssim_\eps}\min\cb{1, \frac{1}{{\wt{\xi}_2^{1/2}}},\frac{1}{{\wt{\xi}_1}}}$.
\\[0.3cm]
Now we prove the lemma under assumption 2). In this case we choose $0<\eta_1<\eta^\ast$  such that (i) $\eps \wt{\xi}_1 \ge 1$ on $[\eta_1, \eta^\ast]$; (ii) $\eps \wt{\xi}_1 \le 1$ on $[\eta_\ast, \eta_1]$. Here the interval $[\eta_\ast, \eta_1]$ may be empty.  

On $[\eta_1, \eta^\ast]$ the bound 
\[
\abs{\Theta} \,\lesssim\, 1 \quad \text{implies} \quad \wt{\xi}_1\abs{\Theta} \,\lesssim\,1+\wt{\xi}_1^{1/2}\2\abs{\Theta}^2 \,\lesssim\, \eps^{-1/2}+\eps^{1/2}\wt{\xi}_1 \abs{\Theta} \quad \text{implies} \quad\abs{\Theta} \,\lesssim\,\frac{1}{\sqrt{\eps}\2 \wt{\xi}_1}\,\le\, \sqrt{\eps}\,.
\]
From the gap in the values of $\abs{\Theta}$ and its continuity we infer $\abs{\Theta}\lesssim \min\cb{\sqrt{\eps},\frac{1}{\sqrt{\eps}\1\wt{\xi}_1}}$. On $[\eta_\ast, \eta_1]$ we use $\wt{\xi}_1 \le \eps^{-1}$ and 
$\abs{\xi_2} \lesssim \wt{\xi}_1^{1/2}\le \eps^{-1/2}$ to conclude $\abs{\Theta}{\lesssim_\eps} 1{\lesssim_\eps}\min\cb{1, \frac{1}{\wt{\xi}_1}}$. This finishes the proof of the lemma.
\end{proof}

\begin{lemma}[Hölder continuity of $\sigma$ and $\psi$ with respect to $a$ and $S$]\label{lem:Hcont}
Let $T \subset \R$ contain $0$. For each $t \in T$, we assume that the linear operator $S_t \colon \alg \to \alg$ satisfies  
\begin{equation} \label{eq:flatness_cont_sigma_psi} 
  c_1 \avg{x} \id \leq S_t[x] \leq c_2 \avg{x} \id 
\end{equation}
for all $x \in \algnon$ and some $c_2 > c_1 >0$. Moreover, let $a_t = a_t^*  \in\alg$ be self-adjoint such that $S_t$ and $a_t$
satisfy \eqref{eq:assums_perturbation} with $a\defeq a_{t=0}$ and $S \defeq S_{t=0}$. 
Let $m_t$ be the solution to \eqref{eq:dyson_Hoelder_S} and $\dens(z) \defeq \avg{\Im m_0(z)}/\pi$ for $z \in\Hb$. 

If $\sigma_t$ and $\psi_t$ are defined according to \eqref{eq:def_psi_sigma}, where $m$ is replaced by $m_t$, then 
there are $\dens_* \sim 1$ and $t_* \sim 1$ such that 
\[ \abs{\sigma_t(z_1) - \sigma_0(z_1)} \lesssim \abs{t}^{1/3}, \qquad \abs{\psi_t(z_2) - \psi_0(z_2)} \lesssim \abs{t}^{1/3} \] 
for all $ t \in [-t_*, t_*] \cap T$ and all $z_1, z_2 \in \Dbdd\cap \{ z \in \Hb \colon \abs{z} \leq c_6\}$ satisfying $\dens(z_1) \leq \dens_*$ and $\dens(z_2) + \dens(z_2)^{-1}\Im z_2 \leq \dens_*$. 
Here, $c_6>0$ is also considered a model parameter. 
\end{lemma} 

\begin{proof} 
We choose $t_*$ as in Proposition~\ref{pro:Hoelder_continuity_S} and conclude from this result that 
$\norm{m_t(z)} \leq k_3$ for all $t \in [-t_*,t_*] \cap T$, all $z \in\Dbdd$ and some $k_3\sim 1$.
Hence, owing to \eqref{eq:assums_perturbation}, \eqref{eq:flatness_cont_sigma_psi} and Lemma~\ref{lem:q_bounded_Im_u_sim_avg} (ii), 
the conditions of Assumptions~\ref{assums:general} are met on $\Dbdd\cap \{ z \in\Hb \colon \abs{z} \leq c_6\}$. 
Hence, from the proof of Lemma~\ref{lem:q_u_f_u_extension}, it can be read off that, 
 after reducing $\dens_* \sim 1$ and $t_* \sim 1$ if necessary, 
$\mathcal{M}^{(2)} \defeq \{ m_t(z_1) \colon t \in [-t_*, t_*] \cap T \}$ and 
$\mathcal{M}^{(3)} \defeq \{ m_t(z_2) \colon t \in [-t_*,t_*] \cap T\}$ satisfy the conditions of Remark~\ref{rem:derived_quantities_cont_function_of_m} (ii) and (iii), respectively,
uniformly for any $z_1, z_2 \in  \Dbdd\cap \{ z \in\Hb \colon \abs{z} \leq c_6\}$  such that  
$\dens(z_1) \leq \dens_*$ and $\dens(z_2) + \dens(z_2)^{-1} \Im z_2 \leq \dens_*$. 
Therefore, the lemma is a consequence of Remark~\ref{rem:derived_quantities_cont_function_of_m} (ii) and (iii) as well as Proposition~\ref{pro:Hoelder_continuity_S} (a). 
\end{proof} 

\begin{remark}
Combining Lemma~\ref{lem:stability_cubic_equation} and Lemma~\ref{lem:Hcont}, we obtain that $m$, $\sigma$ and $\psi$ 
are jointly H\"older continuous in all three variables $(z, a, S)$ in the following sense. Suppose that $m$ solves the MDE
for some data pair $(a,S)$ satisfying Assumptions \ref{assums:general} on  some $I$ for some $\eta_* \in (0,1]$
and consider a one-parameter family of data pairs $(a_t, S_t)$, $t\in T$, as described in Lemma~\ref{lem:Hcont}.
Then $m=m_t(z)$, as well as of $\sigma_t(z_1)$ and $\psi_t(z_2)$ are uniformly 1/3-H\"older continuous functions of
$t\in [-t_*,t_*] \cap T$ as well as $z\in \Hbtheta$, $z_1 \in \{ \zeta \in \Hbtheta \colon \dens(\zeta) \leq \dens_*\}$  and $z_2 \in \{ \zeta \in \Hbtheta \colon \dens(\zeta) + \dens(\zeta)^{-1} \im \zeta \leq \dens_*\}$, respectively, 
for sufficiently small $t_*\sim 1$ and $\dens_* \sim 1$.
\end{remark}

\begin{remark}[Scaling of $\sigma$] 
Let Assumptions~\ref{assums:general} hold true for some interval $I$ and $\eta_* \in (0,1]$. Let $\theta\in (0,1]$. 
\begin{enumerate}[label=(\roman*)]
\item 
As in the proof of \eqref{eq:scaling_sigma} in the proof of Lemma~\ref{lem:cubic_equation1}, we obtain that 
\[ \abs{\sigma(\tau_0)} \sim \abs{\sigma(\tau_1)} \sim (\tau_1 - \tau_0)^{1/3}, \] 
if $\tau_0, \tau_1 \in \supp \dens \cap I_\theta$ satisfy $\tau_0 < \tau_1$ and $(\tau_0, \tau_1) \cap \supp \dens = \varnothing$. 
Furthermore, there is $\rho_* \sim 1$ such that 
\[ \abs{\sigma(\tau_0)} \lesssim \rho(\tau_0)^2, \] 
if $\tau_0 \in \supp \dens\cap I_\theta$ is a local minimum of $\dens$ satisfying $\dens(\tau_0) \leq \dens_*$. 
\item Owing to the $1/3$-Hölder continuity of $\sigma$ from Lemma~\ref{lem:stability_cubic_equation} (i), we conclude that there 
is $\eps \sim 1$ such that $\abs{\sigma(\tau)} \sim (\tau_1 - \tau_0)^{1/3}$ for all $\tau \in \supp \dens \cap I_\theta$
satisfying $\min\{\abs{\tau - \tau_0}, \abs{\tau -\tau_1}\} \leq \eps (\tau_1-\tau_0)$ for some $\tau_0, \tau_1 \in \supp \dens \cap I_\theta$ such that $\tau_0 < \tau_1$ and $(\tau_0, \tau_1) \cap \supp \dens = \varnothing$.

\item If $\tau \in I_\theta$ satisfies the assumptions of (ii) as well as $\rho(\tau)>0$ 
then we write $\Delta \defeq \tau_1 - \tau_0$ and conclude from (ii) and Lemma~\ref{lem:derivatives_m} that 
\[ \norm{\pt_\tau m(\tau)}  \lesssim \frac{1}{\rho(\tau)(\rho(\tau) + \Delta^{1/3})}. \] 
\end{enumerate}
\end{remark}

\appendix

\section{Stieltjes transforms of positive operator-valued measures} \label{app:positive_operator_valued_measure}

In this appendix, we will show some results about the Stieltjes transform of a positive operator-valued measure on $\alg$. 

We first prove Lemma \ref{lem:existence_measure} by generalizing existing proofs in the matrix algebra setup. Since we have not found the general version in the literature, 
we provide a proof here for the convenience of the reader. 
In the proof of Lemma \ref{lem:existence_measure}, we will use that a von Neumann algebra is always isomorphically isomorphic as a Banach space to the dual space of a Banach space. 
In our setup, this Banach space and the identification are simple to introduce which we will explain now. 
Analogously to $L^2$ defined in Section \ref{sec:regularity}, we define $L^1$ to be the completion of $\alg$ when equipped with the norm $\normone{x} \defeq \avg{(x^*x)^{1/2}} = \avg{\abs{x}}$ 
for $x \in \alg$. 
Moreover, we extend $\avg{\genarg}$ to $L^1$ and remark that $xy \in L^1$ for $x \in \alg$ and $y \in L^1$. 
It is well-known (e.g.~\cite[Theorem 2.18]{TakesakiBook}) that the dual space $(L^1)'$ of $L^1$ can be identified with $\alg$ via the isometric isomorphism
\begin{equation} \label{eq:isomorphism_alg_iso_L1_dual}
 \alg \to (L^1)', ~ x \mapsto \psi_x, \qquad \psi_x \colon L^1 \to \C,~y \mapsto \avg{xy}.
\end{equation}
We stress that the existence of this isomorphism requires the state $\avg{\genarg}$ to be normal. 

\begin{proof}[Proof of Lemma \ref{lem:existence_measure}]
From \eqref{eq:Stieltjes_transform_normalization}, we conclude that 
\[ \lim_{\eta\to\infty}\ii\eta\scalar{x}{h(\ii\eta)x} = - \scalar{x}{x} \] 
for all $x \in \alg$. 
Hence, $z \mapsto \scalar{x}{h(z)x}$ is the Stieltjes transform of a unique finite positive measure $v_x$ on $\R$ with $v_x(\R) = \normone{x^*x}$. 

For any $x \in \alg$, we can find $x_1, \ldots x_4 \in \algnon$ such that $x=x_1 - x_2 + \ii x_3 - \ii x_4$. We define 
\begin{equation} \label{eq:def_varphi_B}
 \varphi_B (x) \defeq v_{\sqrt{x_1}}(B) - v_{\sqrt{x_2}}(B) + \ii v_{\sqrt{x_3}}(B) - \ii v_{\sqrt{x_4}}(B)  
\end{equation}
for $B \in \mathcal{B}$. 
This definition is independent of the representation of $x$. Indeed, for fixed $x \in \alg$, 
any representation $x=x_1 - x_2 + \ii x_3 - \ii x_4$ with $x_1, \ldots, x_4 \in \algnon$ defines a complex measure $\varphi_\cdot(x)$ through $B \mapsto \varphi_B(x)$ on $\R$ via~\eqref{eq:def_varphi_B}. 
However, extending $h$ to the lower half-plane by setting $h(z) \defeq h(\bar z)^*$ for $z \in \C$ with $\Im z<0$, the Stieltjes transform of $\varphi_\cdot (x)$ is given by 
\[ \int_\R \frac{\varphi_{\di \tau} (x)}{\tau - z} = \scalar{\sqrt{x_1}}{h(z) \sqrt{x_1}} - \scalar{\sqrt{x_2}}{h(z) \sqrt{x_2}}  + \ii \scalar{\sqrt{x_3}}{h(z) \sqrt{x_3}} - \ii \scalar{\sqrt{x_4}}{h(z) \sqrt{x_4}} 
= \avg{h(z) x}  \]  
for all $z \in \C \setminus \R$. This formula shows that the Stieltjes transform of $\varphi_\cdot (x)$ is independent of the decomposition $x = x_1 - x_2 + \ii x_3 - \ii x_4$. 
Hence, $\varphi_B(x)$ is independent of this representation for all $B \in \mathcal{B}$ since the Stieltjes transform
uniquely determines even a complex measure. 
A similar argument also implies that, for fixed $B \in \mathcal{B}$, $\varphi_B$ defines a linear functional on $\alg$. 

Since $v_{\sqrt{y}}(\R) = \avg{y}$ for $y \in \algnon$, we obtain for $x = (\Re x)_+ - (\Re x)_- + \ii (\Im x)_+ - \ii (\Im x)_- \in \alg$ 
\[\begin{aligned} 
 \abs{\varphi_B(x)} & \leq v_{\sqrt{(\Re x)_+}}(\R) + v_{\sqrt{(\Re x)_-}}(\R) + v_{\sqrt{(\Im x)_+}}(\R) +v_{\sqrt{(\Im x)_-}}(\R) \\
 &  \leq \avg{(\Re x)_+ + (\Re x)_- + (\Im x)_+ + (\Im x)_-} \leq 2\normone{x}, 
\end{aligned}\] 
where we used that $(\Re x)_++(\Re x)_- = \abs{\Re x}$ and $(\Im x)_+ + (\Im x)_-=\abs{\Im x}$. 
Therefore, $\varphi_B$ extends to a bounded linear functional on $L^1$ as $\alg$ is a dense linear subspace of $L^1$. 
Using the isomorphism in \eqref{eq:isomorphism_alg_iso_L1_dual}, for each $B \in \mathcal{B}$, there exists a unique $v(B) \in \alg$ such that 
\[ \varphi_B(x) = \avg{v(B)x} \] 
for all $ x \in \alg$. For $y \in \alg$, we conclude $v_y(B) = v_{\sqrt{yy^*}}(B) = \varphi_B(yy^*) = \scalar{y}{v(B)y} \geq 0$, where we used that $v_y = v_{\sqrt{yy^*}}$ since they have the same Stieltjes transform. 
Since $\avg{v(B)y} \geq 0$ for all $y \in \algnon$, we have $v(B) \in \algnon$ for all $B \in \mathcal{B}$. 
Moreover, $v_x = \scalar{x}{v(\cdot)x}$, in particular, $\scalar{x}{v(\R)x} = v_x(\R) = \scalar{x}{x}$, for all $x \in \alg$. 
The polarization identity yields that $v$ is an $\algnon$-valued measure on $\mathcal{B}$ satisfying \eqref{eq:Stieltjes_representation_h} and $v(\R) = \id$. This completes the proof of Lemma \ref{lem:existence_measure}. 
\end{proof}

\begin{lemma}[Stieltjes transform inherits Hölder regularity] \label{lem:Stieltjes_transform_Hölder}
Let $v$ be an $\algnon$-valued measure on $\R$ and $h \colon \Hb \to \alg$ be its Stieltjes transform, i.e., $h$ satisfies \eqref{eq:Stieltjes_representation_h} for all $z \in\Hb$. 
Let $f \colon I \to \algnon$ be a $\gamma$-Hölder continuous function on an interval $I \subset \R$ with $\gamma \in (0,1)$ and $f$ be a density of $v$ on $I$ with respect to the 
Lebesgue measure, i.e., 
\[ \norm{f(\tau_1)- f(\tau_2)} \leq C_0 \abs{\tau_1 - \tau_2}^\gamma, \qquad  
 v(A) = \int_A f(\tau) \di \tau \] 
for all $\tau_1, \tau_2 \in I$,  some $C>0$ and for all Borel sets $A \subset I$.
Moreover, we assume that $\norm{f(\tau)} \leq C_1$ for all $\tau \in I$.
Let $\theta \in (0,1]$. 

Then, for $z_1, z_2 \in \Hb$ satisfying $\Re z_1, \Re z_2 \in I$ and $\dist(\Re z_k, \pt I) \geq \theta$, $k=1,2$, 
we have 
\begin{equation} \label{eq:Hoelder_in_close_support}
 \norm{h(z_1) -h(z_2)} \leq \bigg(\frac{21C_0}{\gamma(1-\gamma)} + \frac{4 \norm{v(\R)}}{\theta^{1+\gamma}}+ \frac{14C_1}{\gamma\theta^\gamma}\bigg)
 \abs{z_1 - z_2}^\gamma. 
\end{equation}
Furthermore, for $z_1, z_2 \in \Hb$ satisfying $\dist(z_k, \supp v ) \geq \theta$, $k=1,2$, we have 
\begin{equation} \label{eq:Hoelder_away_from_support}
 \norm{h(z_1) -h(z_2)} \leq\frac{2 \norm{v(\R)}}{\theta^2} \abs{z_1 - z_2}^\gamma. 
\end{equation}
\end{lemma} 

We omit the proof of Lemma \ref{lem:Stieltjes_transform_Hölder} since it is very similar to the one of Lemma A.7 in \cite{AjankiQVE}.

\section{Positivity-preserving, symmetric operators on $\alg$}  \label{app:positive_symmetric_map}

\begin{lemma}  \label{lem:positive_symmetric_maps_aux}
Let $T \colon \alg \to \alg$ be a positivity-preserving, symmetric operator.
\begin{enumerate}[label=(\roman*)]
\item If $T[a] \leq C \avg{a} \id$ for some $C >0$ and all $a\in \algnon$ then $\normtwo{T} \leq 2C$. Moreover, $\normtwo{T}$ is an eigenvalue of $T$ and there is $x \in \algnon\setminus\{0\}$ 
such that $T[x] =\normtwo{T}x$. 
\item 
We assume $\normtwo{T}=1$ and that there are $c, C>0$ such that 
\begin{equation} \label{eq:T_flat_aux}
 c \avg{a}\id \leq T[a] \leq C \avg{a} \id
\end{equation}
for all $a \in \alg_+$. Then $1$ is an eigenvalue of $T$ with a one-dimensional eigenspace. 
There is a unique $x \in \alg_+$ satisfying $T[x] = x$ and $\normtwo{x} = 1$. Moreover, $x$ is positive definite, 
\begin{equation} \label{eq:bounds_eigenfunction}
  cC^{-1/2} \id \leq x \leq C\id. 
\end{equation}
Furthermore, the spectrum of $T$ has a gap of size $\theta \defeq c^6/(2(c^3+2C^2)C^2))$, i.e.,  
\begin{equation} \label{eq:spectral_gap_general_T}
 \spec(T) \subset [-1+\theta, 1- \theta ] \cup \{ 1 \}.
\end{equation}
\end{enumerate}
\end{lemma}

Lemma \ref{lem:positive_symmetric_maps_aux} is the analogue of Lemma 4.8 in \cite{AjankiCorrelated}. 
Here, we explain how to generalize it to the context of von Neumann algebras. 
In the proof of Lemma \ref{lem:positive_symmetric_maps_aux}, we will use the following lemma. 

\begin{lemma} \label{lem:norm_two_to_inf} Let $T \colon \alg \to \alg$ be a linear map. 
\begin{enumerate}[label=(\roman*)]
\item If $T$ is positivity-preserving such that $T[a] \leq C\avg{a} \id$ for all $a \in \algpos$
and some $C>0$ then $\norm{T} \leq \normtwoinf{T} \leq 2C$. 
\item If $T-\omega\Id$ is invertible on $\alg$ for some $\omega \in \C\setminus\{0\}$ and $\normtwo{(T-\omega\Id)^{-1}}< \infty$, $\normtwoinf{T} < \infty$ then we have 
\[ \norm{(T-\omega\Id)^{-1}} \leq \abs{\omega}^{-1}\Big( 1+ \normtwoinf{T} \normtwo{(T-\omega\Id)^{-1}}\Big). \] 
\end{enumerate}
\end{lemma}

We include the short proof of Lemma~\ref{lem:norm_two_to_inf} for the reader's convenience. 
In fact, the first part is obtained as in (4.2) of \cite{AjankiCorrelated} and the second part as in (5.28) of \cite{AjankiQVE}. 

\begin{proof}[Proof of Lemma~\ref{lem:norm_two_to_inf}]
Let $a \in \alg$ be self-adjoint, i.e., $a = a^*$. Thus, $a= a_+ - a_-$ is the sum of its positive and negative part, $a_+, a_- \in \algnon$. We conclude 
\[ T[a] \leq T[a_+] + T[a_-] \leq C \avg{a_+ + a_-} \leq C \normtwo{a} \] 
since $a_+ + a_- = \abs{a}$. 
Hence, $\norm{T[a]} \leq C \normtwo{a}$ as $T[a] \geq - C \normtwo{a}$ is shown similarly. For a general $a \in\alg$, we obtain $\norm{T[a] } \leq 2 C \normtwo{a}$. As $\normtwo{a} \leq \norm{a}$ 
this completes the proof of part (i). 

For the proof of (ii), we take an arbitrary $x \in \alg$. We set $y \defeq (T- \omega \Id )^{-1}[x]$. From the definition of the resolvent, we conclude $\omega y = T[y] - x$. 
This yields 
\[ \norm{y} \leq \abs{\omega}^{-1} ( \normtwoinf{T}\normtwo{y} + \norm{x}) \leq  \abs{\omega}^{-1} ( 1 + \normtwoinf{T}\normtwo{(T-\omega \Id)^{-1}}) \norm{x}, \] 
where we used $\normtwo{x} \leq \norm{x}$ in the last step. 
Since $x$ was arbitrary, we have completed the proof of (ii).  
\end{proof}  

\begin{proof}[Proof of Lemma \ref{lem:positive_symmetric_maps_aux}]
For the proof of (i), we remark that Lemma \ref{lem:norm_two_to_inf} (i) implies $\normtwo{T} \leq \normtwoinf{T} \leq 2C$. 
Without loss of generality, we assume $\normtwo{T} = 1$. Since $T$ is positivity-preserving, we have $T[b] \in \algsa$ for all $b \in \algsa$. 
It is easy to check that, for each $a \in \alg$, one may find $b \in \algsa$ such that $\normtwo{a} = \normtwo{b}$ and $\normtwo{T[a]} \leq \normtwo{T[b]}$. Hence, $\normtwo{T|_\algsa} = \normtwo{T}=1$ 
and $1$ is contained in the spectrum of $T\colon \Ltwo_\rm{sa} \to \Ltwo_\rm{sa}$, where $\Ltwo_\rm{sa} \defeq \overline{\alg_\rm{sa}}^{\normtwo{\genarg}}$, 
due to the variational principle for the spectrum of self-adjoint operators and 
$\abs{\scalar{b}{T[b]}} \leq \scalar{\abs{b}}{T[\abs{b}]}$ for all $ b\in \algsa$. 
This last inequality can be checked easily by decomposing $b = b_+ - b_-$ into positive and negative part. 

Hence, due to the symmetry of $T$, there is a sequence $(y_n)_n$ of approximating eigenvectors in $\algsa$, i.e., $y_n \in \algsa$, $\normtwo{y_n} = 1$ and $T[y_n] - y_n$ converges to 0 in $\Ltwo$ for $n \to \infty$.   
We set $x_n \defeq \abs{y_n}$. 
By using $\normtwo{T|_{\Ltwo_\rm{sa}}} =1$ and $\scalar{b}{T[b]} \leq \scalar{\abs{b}}{T[\abs{b}]}$ for all $ b\in \algsa$, we obtain $\normtwo{T[x_n]-x_n}^2 \leq 2 \normtwo{y_n}\normtwo{T[y_n]-y_n}$ and, thus, 
\begin{equation} \label{eq:approx_eigenvectors}
 \lim_{n\to\infty} \normtwo{T[x_n] - x_n} = 0. 
\end{equation}
Since the unit ball in the Hilbert space $\Ltwo$ is relatively sequentially compact in the weak topology, we can assume by possibly replacing $(x_n)_{n}$ by a subsequence that there is $x \in \Ltwo$ 
such that $x_{n} \rightharpoonup x$ weakly in $\Ltwo$. 
From $T[x_n] \leq C \avg{x_n}\id$, we conclude 
\[ x_{n} \leq (\Id-T)[x_{n}] + C\avg{x_{n}} \id. \] 
Multiplying this by $\sqrt{x_{n}}$ from the left and the right and applying $\avg{\genarg}$ yields 
\[ 1 \leq \scalar{x_{n}}{(\Id-T)[x_{n}]} + C \avg{x_{n}}^2. \] 
Taking the limit $n\to\infty$, we obtain $\avg{x} \geq C^{-1/2}$, due to \eqref{eq:approx_eigenvectors}. 
Hence, $x \neq 0$ and we can replace $x$ by $x/\normtwo{x}$ and $x_{n}$ by $x_{n}/\normtwo{x}$. 
For any $b \in \Ltwo$, we have
\[ \scalar{b}{(\Id-T)[x]} = \lim_{n\to\infty} \scalar{b}{(\Id-T)[x_{n}]} = 0 \] 
due to $x_n\rightharpoonup x$ and \eqref{eq:approx_eigenvectors}. 
Hence, $T[x] = x$. Since $\normtwoinf{T} \leq 2C$, we have $T[b] \in \alg$ for all $b \in \Ltwo$ and thus 
$x = T[x] \in \alg$.  Owing to $x_{n} \rightharpoonup x$ and $x_{n} \in \algnon$, we obtain $x \in \algnon$. 
This completes the proof of (i). 

We start the proof of (ii) by using \eqref{eq:T_flat_aux} with $a =x$ which immediately yields the upper bound in \eqref{eq:bounds_eigenfunction}. 
As $\avg{x} \geq C^{-1/2}$, the first inequality in \eqref{eq:T_flat_aux} then yields the lower bound in~\eqref{eq:bounds_eigenfunction}. 

In order to prove the spectral gap, \eqref{eq:spectral_gap_general_T}, we remark that $\normtwoinf{T} \leq 2 C$ due to the upper bound in \eqref{eq:T_flat_aux} 
and Lemma \ref{lem:norm_two_to_inf} (i). Hence, by Lemma \ref{lem:norm_two_to_inf} (ii), 
the spectrum of $T$ as an operator on $\alg$ is contained in the union of $\{0\}$ and the spectrum of $T$ as an operator on $\Ltwo$. 
Therefore, we will consider $T$ as an operator on $\Ltwo$ in the following and exclusively study its spectrum as an operator on $\Ltwo$. 
Hence, to prove the spectral gap, it suffices to establish a lower bound on $\scalar{y}{(\Id\pm T)[y]}$
for all self-adjoint $y \in \alg$ satisfying $\normtwo{y} =1$ and $\scalar{x}{y} =0$. 
Fix such $y \in \alg$. Since $y$ is self-adjoint we have 
\begin{equation} \label{eq:y_approx_projections}
 y = \lim_{N\to\infty} y^N, \qquad y^N \defeq \sum_{k=1}^N \lambda_k^N p_k^N 
\end{equation}
for some $\lambda_n^N \in \R$ and $p_k^N \in \alg$ orthogonal projections such that $p_k^Np_l^N = p_k^N \delta_{k,l}$. 
Here, the convergence $y^N \to y$ is with respect to $\norm{\cdot}$. 
We can assume that $\normtwo{y^N} = 1$ for all $N$ as well as $\avg{p_k^N} >0$ for all $k$ and $\avg{p_1^N + \ldots + p_N^N} = 1$ for all $N$. 

We will now reduce estimating $\scalar{y}{(\Id \pm T)[y]}$ to estimating a scalar product on $\C^N$. On $\C^N$, we consider the scalar product $\scalar{\genarg}{\genarg}_N$ 
induced by the probability measure $\pi(A) = \sum_{k \in A} \avg{p_k^N}$ on $[N]$, i.e., 
\[ \scalar{\lambda}{\mu}_N = \sum_{k=1}^n \overline{\lambda_k} \mu_k \avg{p_k^N} \] 
for $\lambda = (\lambda_k)_{k=1}^N$, $\mu = (\mu_k)_{k=1}^N \in \C^N$. The norm on $\C^N$ and the operator norm on $\C^{N\times N}$ induced by $\scalar{\genarg}{\genarg}_N$ 
are denoted by $\norm{\genarg}_N$ and $\norm{\genarg}$, respectively.  
Moreover, $\Id_N$ is the identity map on $\C^N$. 
With this notation, we obtain from \eqref{eq:y_approx_projections} that 
\[ \scalar{y}{(\Id\pm T)[y]} = \lim_{N\to\infty} \sum_{k,l=1}^N \lambda_k^N\lambda_l^N \scalar{p_k^N}{(\Id \pm T)[p_l^N]} = \lim_{N\to\infty} \scalar{\lambda^N}{(\Id_N\pm S^N)[\lambda^N]}_N, \]
where we introduced $\lambda^N = (\lambda_k^N)_{k=1}^N \in \C^N$ and the $N\times N$ symmetric matrix $S^N$ viewed as an integral operator
on $([N], \pi)$ with the kernel $s_{kl}^N$ given by 
\[ s_{kl}^N = \frac{\scalar{p_k^N}{T[p_l^N]}}{\avg{p_k^N}\avg{p_l^N}}. \]
Since $\normtwo{y^N} = 1$, we have $\norm{\lambda^N}_N=1$. By the flatness of $T$, we have 
\begin{equation} \label{eq:bound_entries_S}
 c \leq s_{kl}^N \leq C. 
\end{equation}

In the following, we will omit the $N$-dependence of $\lambda_k$, $s_{kl}$ and $p_k$ from our notation. 
By the definition of $\scalar{\cdot}{\cdot}_N$, we have 
\[ \scalar{\lambda}{S\lambda}_N = \sum_{k,l=1}^N \lambda_k \avg{p_k} s_{kl} \avg{p_l} \lambda_l = \scalar{y^N}{T[y^N]}. \] 
Let $s \in \C^N$ be the Perron-Frobenius eigenvector of $S$ satisfying $Ss = \norm{S} s$, $\norm{s}_N = 1$.
From \eqref{eq:bound_entries_S}, we conclude 
\begin{equation} \label{eq:lower_bound_scalar_s_Ss}
 c \leq \scalar{e}{Se}_N \leq \norm{S} = \scalar{s}{Ss}_N \leq \normtwo{T} = 1, 
\end{equation}
where $e = (1, \ldots, 1) \in \C^N$. 
Since $\norm{s}_N=1$ and $c \leq \norm{S}$, we have
\[ \max_i s_i  = \frac{(Ss)_i}{\norm{S}} \leq \frac{C}{c} \sum_{k=1}^N s_k \avg{p_k} \leq \frac{C}{c} \left(\sum_{k=1}^N \avg{p_k} \right)^{1/2}  \left( \sum_{k=1}^N s_k^2 \avg{p_k} \right) ^{1/2} = \frac{C}{c}. \] 
As $\inf_{k,l} s_{k,l} \geq c$ by \eqref{eq:bound_entries_S}, Lemma 5.7 in \cite{AjankiQVE} yields 
\[ \spec(S) \subset \bigg[ - \norm{S} + \frac{c^3}{C^2}, \norm{S} - \frac{c^3}{C^2} \bigg] \cup \{ \norm{S} \}. \] 
We decompose  $\lambda = (1-\norm{w}_N^2)^{1/2} s + w$ with $w \perp s$ and obtain 
\begin{equation} \label{eq:estimate_lambda_S_lambda_leq_w}
 \abs{\scalar{\lambda}{S\lambda}_N} \leq \norm{S} (1 - \norm{w}_N^2)  + \bigg( \norm{S} - \frac{c^3}{C^2} \bigg) \norm{w}_N^2 \leq 1 - \frac{c^3}{C^2} \norm{w}_N^2,
\end{equation}
where we used $\norm{S} \leq 1$ in the last step. 
Hence, it remains to estimate $\norm{w}_N$.

Recalling $T[x] = x$, we set $\tilde{x} = (\avg{xp_k}/\avg{p_k})_{k=1}^N$ and compute 
\[ \scalar{x}{y^N} = \sum_{k} \lambda_k \avg{xp_k} = \scalar{\tilde{x}}{\lambda}_N. \] 
Since the left-hand side goes to $\scalar{x}{y} = 0$ for $N\to\infty$, we can assume that $\abs{\scalar{\tilde{x}}{\lambda}_N} \leq \sqrt{\eps/2}$ for any fixed $\eps \sim 1$ 
and all sufficiently large $N$. 
As $\tilde{x}_k \geq c/\sqrt{C}$ by \eqref{eq:bounds_eigenfunction}, we obtain 
\begin{equation} \label{eq:initial_estimate_w}
 (1- \norm{w}_N^2) \frac{c^2}{C} \left(\sum_k s_k \avg{p_k}\right)^2 \leq (1- \norm{w}_N^2) \scalar{\tilde{x}}{s}_N^2 = ( \scalar{\tilde{x}}{\lambda}_N- \scalar{\tilde{x}}{w}_N)^2 \leq 2 \norm{\tilde{x}}_N^2 \norm{w}_N^2 + \eps.   
\end{equation}
Now, we use $c \leq \scalar{s}{Ss}_N$ from \eqref{eq:lower_bound_scalar_s_Ss} to get
\[ c \leq \scalar{s}{Ss}_N = \sum_{k,l} s_k s_{kl} s_l \avg{p_k}\avg{p_l} \leq C \left(\sum_{k} s_k \avg{p_k}\right)^2. \] 
By plugging this and $\norm{\tilde{x}}_N^2 \leq \norm{x}^2\sum_k \avg{p_k} = 1 $ into \eqref{eq:initial_estimate_w}, solving the resulting estimate for $\norm{w}_N^2$ and choosing $\eps = c^3/(2C^2)$, we obtain 
\[ \norm{w}_N^2 \geq \frac{c^3}{2(c^3 + 2C^2 )}. \]
Therefore, from \eqref{eq:estimate_lambda_S_lambda_leq_w}, we conclude  
\[ \abs{\scalar{\lambda}{S\lambda}_N} \leq 1 - \frac{c^6}{2(c^3 + 2C^2)C^2} \] 
uniformly for all sufficiently large $N \in \N$. 
We thus obtain that 
\[ \scalar{y}{(\Id \pm T)[y]} \geq \frac{c^6}{2(c^3 + 2C^2)C^2} \] 
if $y \perp x$ and $\normtwo{y} = 1$. 
We conclude \eqref{eq:spectral_gap_general_T}, which completes the proof of the lemma. 
\end{proof}

\begin{lemma} \label{lem:inverse_positivity_preserving} 
If $T \colon \alg \to \alg$ is a positivity-preserving operator such that $\normtwo{T} < 1$ and $\normtwoinf{T} < \infty$ then 
$\Id - T$ is invertible as a bounded operator on $\alg$ and $(\Id - T)^{-1}$ is positivity-preserving with 
\begin{equation} \label{eq:inverse_Id_T_lower_bound}
 (\Id - T)^{-1}[x^*x] \geq x^*x 
\end{equation}
for all $x \in \alg$. 
\end{lemma} 

\begin{proof} 
Since $\normtwo{T} < 1$, $\Id - T$ is invertible on $\Ltwo$ and we conclude the invertibility of $\Id-T$ on $\alg$ from Lemma~\ref{lem:norm_two_to_inf} (ii). 

Moreover, for $y \in \alg$ with $\normtwo{y^*y} < 1$, we expand the inverse as a Neumann series using $\normtwo{T} <1$ and obtain 
\[ (\Id - T)^{-1}[y^*y] = y^*y  + \bigg(\sum_{k=1}^\infty T^k[y^*y] \bigg) \geq y^*y. \] 
The series converges with respect to $\normtwo{\genarg}$. 
In the last inequality, we used that $T^k$ is a positivity-preserving operator for all $k \in \N$. 
Hence, by rescaling a general $x \in \alg$, we see that $(\Id -T)^{-1}$ is a positivity-preserving operator on $\alg$ which satisfies \eqref{eq:inverse_Id_T_lower_bound}. 
\end{proof}

\section{Non-Hermitian perturbation theory} 

Let $B_0 \colon \alg \to \alg$ be a bounded operator with an isolated, single eigenvalue $\beta_0$ and an associated eigenvector $b_0$, $\normtwo{b_0} = 1$, i.e.,
\[ B_0[b_0] = \beta_0 b_0. \]
Moreover, we denote by $P_0$ and $Q_0$ the spectral projections corresponding to $\beta_0$ and $\spec(B_0)\setminus\{\beta_0\}$. 
Note that $P_0 + Q_0 =\Id$ but they are not orthogonal projections in general.
If $l_0$ is a normalized eigenvector of 
$B_0^*$ associated to its eigenvalue $\ol{\beta}_0$, then we obtain 
\begin{equation} \label{eq:explicit_formuala_P_0}
 P_0 = \frac{\scalar{l_0}{\genarg}}{\scalar{l_0}{b_0}}b_0. 
\end{equation}
For some bounded operator $E \colon \alg \to \alg$, we consider the perturbation 
\[ B = B_0 + E. \] 
We assume $E$ to be sufficiently small such that there is an isolated, single eigenvalue $\beta$ of $B$ close to $\beta_0$ and that $\beta$ and $\beta_0$ are 
separated from $\spec(B)\setminus\{\beta\}$ and $\spec(B_0)\setminus \{\beta_0\}$ by an amount $\Delta >0$. 
Let $P$ be the spectral projection of $B$ associated to $\beta$.

\begin{lemma} \label{lem:non_hermitian_perturbation_theory}
We define $b \defeq P[b_0]$ and $l \defeq P^*[l_0]$. Then $b$ and $l$ are eigenvectors of $B$ and $B^*$ corresponding to $\beta$ and $\bar{\beta}$, respectively. 
Moreover, we have 
\begin{equation}\label{eq:expansion_b_and_l}
 b = b_0 + b_1 + b_2 + \ord(\norm{E}^3),\qquad  l=l_0 + l_1 + l_2 + \ord(\norm{E}^3), 
\end{equation}
where we introduced
 \begin{align*} 
b_1 = \; &  - Q_0(B_0-\beta_0\Id)^{-1} E[b_0], \\
b_2 = \; & Q_0 (B_0-\beta_0\Id)^{-1} E (B_0-\beta_0\Id)^{-1}Q_0 E[b_0] - Q_0(B_0-\beta_0\Id)^{-2} E P_0 E[b_0] - P_0 E Q_0 (B_0 - \beta_0\Id)^{-2} E [b_0], \\
l_1  =\;  & - Q_0^*(B_0^*-\bar{\beta}_0\Id)^{-1} E^*[l_0], \\
l_2 = \; &  Q_0^* (B_0^*-\bar{\beta}_0\Id)^{-1} E^* (B_0^*-\bar{\beta}_0\Id)^{-1}Q_0^* E^*[l_0] - Q_0^* (B_0^*-\ol{\beta}_0\Id)^{-2} E^* P^*_0 E^*[l_0] - P_0^* E^* Q_0^*(B_0^* - \ol{\beta}_0\Id)^{-2} E^*[l_0]. 
\end{align*}
In particular, we have $b_i,l_i = \ord(\norm{E}^i)$ for $i=1,2$. 
Furthermore, we obtain 
\begin{equation} \label{eq:expansion_beta_scalar_l_b_abstract}
 \beta \scalar{l}{b} = \beta_0 \scalar{l_0}{b_0} + \scalar{l_0}{E[b_0]} - \scalar{l_0}{EB_0(B_0-\beta_0\Id)^{-2}Q_0E[b_0]} + \ord(\norm{E}^3). 
\end{equation}
The implicit constants in the error terms depend only on the separation $\Delta$.
\end{lemma}

\begin{proof}
In this proof, the difference $B - \omega$ with an operator $B$ and a scalar $\omega$ is understood as $B- \omega\Id$. 
We first prove that 
\begin{equation} \label{eq:expansion_P}
P = P_0 + P_1 + P_2 + \ord(\norm{E}^3), 
\end{equation}
where we defined
\begin{align*} 
P_1 \defeq \; & - \frac{Q_0}{B_0-\beta_0} E P_0 - P_0 E \frac{Q_0}{B_0-\beta_0}, \\
 P_2 \defeq \; & P_0 E \frac{Q_0}{B_0 -\beta_0} E \frac{Q_0}{B_0-\beta_0} + \frac{Q_0}{B_0 -\beta_0}E P_0 E \frac{Q_0}{B_0-\beta_0}   + \frac{Q_0}{B_0 -\beta_0} E \frac{Q_0}{B_0-\beta_0} E P_0  \\ 
 & -\frac{Q_0}{(B_0-\beta_0)^2} E P_0 E P_0 - P_0 E \frac{Q_0}{(B_0 -\beta_0)^2}E P_0 - P_0 E P_0 E \frac{Q_0}{(B_0 -\beta_0)^2}. 
\end{align*}
The analytic functional calculus yields that 
\begin{equation} \label{eq:P_anal_func_calculus}
 P = -\frac{1}{2\pi\ii} \oint_\Gamma \frac{\dd \omega }{B-\omega} = \frac{1}{2\pi\ii}\oint_\Gamma \bigg(-\frac{1}{B_0 - \omega} + \frac{1}{B_0 -\omega}E \frac{1}{B_0-\omega} - \frac{1}{B_0-\omega} E \frac{1}{B_0-\omega}E \frac{1}{B_0-\omega}
  \bigg)\di\omega + \ord(\norm{E}^3), 
\end{equation}
where $\Gamma$ is a closed path that encloses only $\beta$ and $\beta_0$ both with winding number $+1$ but no other element of the spectra of $B$ and $B_0$. 
Integrating the first summand in the integrand of \eqref{eq:P_anal_func_calculus} yields $P_0$. In the second and third summand, we expand $\Id = P_0 + Q_0$ in the numerators. 
Applying an analogue of the residue theorem yields $P_1$ and $P_2$ for the second and third summand, respectively.
For example, for the second summand, we obtain 
\[ P_1 = \frac{1}{2\pi\ii}\oint_\Gamma \frac{1}{B_0-\omega}E \frac{1}{B_0 - \omega} \di \omega =  - \frac{Q_0}{B_0-\beta_0} E P_0 - P_0 E \frac{Q_0}{B_0-\beta_0}. \] 
The other two combinations of $P_0$, $Q_0$ vanish. Using a similar expansion for the third term, we get \eqref{eq:expansion_P}.

Starting from \eqref{eq:expansion_P} as well as observing $b_i = P_i[b_0]$ and $l_i = P_i^*[l_0]$ for $i=1,2$, the relations \eqref{eq:expansion_b_and_l} are a direct consequence of 
the definitions $b=P[b_0]$ and $l=P^*[l_0]$ and \eqref{eq:explicit_formuala_P_0}. 

We will show below that 
\begin{equation} \label{eq:expansion_B}
BP = B_0P_0 + B_1 + B_2 + \ord(\norm{E}^3), 
\end{equation}
where we defined 
\begin{align*}
B_1  \defeq \; & P_0 E P_0 - \beta_0 \bigg( \frac{Q_0}{B_0 - \beta_0} E P_0 + P_0 E \frac{Q_0}{B_0 - \beta_0} \bigg), \\
B_2  \defeq \;  & \beta_0 \bigg( P_0 E \frac{Q_0}{B_0 -\beta_0} E \frac{Q_0}{B_0-\beta_0}  
 +  \frac{Q_0}{B_0 -\beta_0}E P_0 E \frac{Q_0}{B_0-\beta_0} + \frac{Q_0}{B_0 -\beta_0} E \frac{Q_0}{B_0-\beta_0} E P_0   \bigg)  \\ 
 &  - \frac{B_0Q_0}{(B_0-\beta_0)^2} E P_0 E P_0  - P_0 E \frac{B_0Q_0}{(B_0 -\beta_0)^2}E P_0 - P_0 E P_0 E \frac{B_0Q_0}{(B_0 -\beta_0)^2}.
\end{align*}
Now, we obtain \eqref{eq:expansion_beta_scalar_l_b_abstract} by applying \eqref{eq:expansion_b_and_l} as well as \eqref{eq:expansion_B} to $\beta \scalar{l}{b} = \scalar{l}{BPb}$.

In order to prove \eqref{eq:expansion_B}, we use the analytic functional calculus with $\Gamma$ as defined above to obtain 
\[ BP = -\frac{1}{2\pi\ii}\oint_\Gamma \frac{\omega\dd \omega}{B-\omega} = 
\frac{1}{2\pi\ii}\oint_\Gamma \omega \bigg(-\frac{1}{B_0 - \omega} + \frac{1}{B_0 -\omega}E \frac{1}{B_0-\omega} - \frac{1}{B_0-\omega} E \frac{1}{B_0-\omega}E \frac{1}{B_0-\omega}
  \bigg)\di \omega + \ord(\norm{E}^3). 
\]
Proceeding similarly as in the proof of \eqref{eq:expansion_P} yields \eqref{eq:expansion_B} and thus completes the proof of Lemma \ref{lem:non_hermitian_perturbation_theory}.
\end{proof}

\section{Characterization of $\supp\dens$} 
\label{app:proof_characterization_outside}

The following lemma gives equivalent characterizations of $\supp\dens$ in terms of $m$. 
Note $\supp \dens = \supp v$ due to the faithfulness of $\avg{\genarg}$. 
We denote the disk of radius $\eps>0$ centered at $z \in\C$ by $D_\eps(z) \defeq \{ w \in \C\colon \abs{z-w} <\eps\}$. 

\begin{lemma}[Behaviour of $m$ on $\R\setminus \supp\dens$] \label{lem:behaviour_outside} 
Let $m$ be the solution of the Dyson equation \eqref{eq:dyson} for a data pair $(a,S) \in \alg_\rm{sa} \times \Sigma$ with $\norm{a} \leq k_0$ and $S[x] \leq k_1\avg{x} \id$ for all $x \in \algnon$ and 
some $k_0, k_1>0$.
Then, for any fixed $\tau \in \R$, the following statements are equivalent: 
\begin{enumerate}[label=(\roman*)]
\item \label{item:limsup_eta_dens} There is $c>0$ such that 
\[ \limsup_{\eta \downarrow 0} \eta\norm{\Im m(\tau + \ii\eta)}^{-1} \geq c. \] 
\item \label{item:F_condition} There are $C>0$ and $N \subset (0,1]$ with an accumulation point $0$ such that 
\begin{equation} \label{eq:char_outside_m_bound_F_condition}
 \hspace{-0.25cm} \norm{m(z)} \leq C, \quad \norm{m(z)^{-1}} \leq C, \quad C^{-1} \avg{\Im m(z)} \id \leq \Im m(z) \leq C \avg{\Im m(z)} \id, \quad \normtwo{F(z)}\leq 1- C^{-1}
\end{equation}
for all $z \in \tau + \ii N$. (The definition of $F$ was given in \eqref{eq:def_F}.)
\item \label{item:Id_minus_C_m_S_invertible} There is $m=m^*\in \alg$ such that 
\begin{equation} \label{eq:char_outside_m_limit_eta}
 \lim_{\eta \downarrow 0} \norm{m(\tau + \ii\eta)-m} =0.
\end{equation}
Moreover, there is $C>0$ such that $\norm{m} \leq C$ and $\norm{(\Id - C_mS)^{-1}} \leq C$. 
\item \label{item:s_a_solution_on_real_axis} 
There are $\eps >0$ and an analytic function $f \colon D_\eps(\tau) \to \alg$ such that $f(z) = m(z)$ for all $z \in D_\eps(\tau) \cap \Hb$ and $f(z) = f(\bar z)^*$ for all $z \in D_\eps(\tau)$. 
In particular, $f(z) = f(z)^*$ for $z \in D_\eps(\tau) \cap \R$. 

In other words, $m$ can be analytically extended to a neighbourhood of $\tau$. 
\item \label{item:dist_supp} There is $\eps>0$ such that $\dist(\tau, \supp \dens)=\dist(\tau,\supp v) \geq \eps$. 
\item \label{item:liminf_eta_dens} There is $c >0$ such that 
\[ \liminf_{\eta\downarrow 0} \eta \norm{\Im m(\tau + \ii\eta)}^{-1} \geq c. \] 
\end{enumerate}
All constants in \ref{item:limsup_eta_dens} -- \ref{item:liminf_eta_dens} depend effectively on each other as well as possibly $k_0$, $k_1$ and an upper bound on $\abs{\tau}$. 
For example, in the implication \ref{item:Id_minus_C_m_S_invertible} $\Rightarrow$ \ref{item:dist_supp}, $\eps$ in \ref{item:dist_supp} 
can be chosen to depend only on $k_1$ and $C$ in \ref{item:Id_minus_C_m_S_invertible}. 
\end{lemma} 

We remark that $m$ in \ref{item:Id_minus_C_m_S_invertible} above is invertible and satisfies \eqref{eq:dyson} at $z=\tau$. 

As a direct consequence of the equivalence of \ref{item:limsup_eta_dens} and \ref{item:dist_supp}, we spell out the following simple characterization of $\supp \dens$. 

\begin{corollary}[Characterization of $\supp\dens$] \label{coro:characterization_support}
Under the conditions of Lemma \ref{lem:behaviour_outside}, we have 
\begin{equation} \label{eq:lim_dens_inverse_eta_abstract}
 \lim_{\eta\downarrow 0}\eta \norm{\Im m(\tau + \ii\eta)}^{-1}=0. 
\end{equation}
if and only if $\tau \in \supp \dens(=\supp v)$. 
\end{corollary}

\begin{remark}
In the proof of Lemma \ref{lem:behaviour_outside}, the condition $S[x] \leq k_1 \avg{x} \id$ for all 
$x \in \algnon$ is only used to guarantee the following two weaker consequences:
First, this condition implies $\normtwoinf{S} \leq 2 k_1$. Moreover, this condition yields, by Lemma \ref{lem:positive_symmetric_maps_aux} (i), that 
$F=F(\tau +\ii\eta)$ has an eigenvector $f \in \algnon$ corresponding to $\normtwo{F}$, 
$Ff = \normtwo{F} f$, for any fixed $\tau \in \R\setminus \supp \dens$ and any $\eta \in (0,1]$. 
If both of these consequences are verified, then the condition $S[x] \leq k_1 \avg{x} \id$ may be dropped from Lemma \ref{lem:behaviour_outside} 
without any changes in the proof. 
\end{remark} 

\begin{lemma}[Quantitative implicit function theorem] \label{lem:implicit_function}
Let $X, Y, Z$ be Banach spaces, $U \subset X$ and $V \subset Y$ open subsets with $0 \in U, V$. 
Let $\Phi \colon U \times V \to Z$ be continuously Fr\'echet-differentiable map such that the derivative $\pt_1 \Phi(0,0)$ 
with respect to the first variable has a bounded inverse in the origin and $\Phi(0,0)=0$. 
Let $\delta >0$ such that $B_\delta^X\subset U$, $B_\delta^Y \subset V$ and 
\begin{equation} \label{eq:implicit_function_condition}
 \sup_{(x,y) \in B_\delta^X \times B_\delta^Y} \norm{\Id_X - (\pt_1 \Phi(0,0))^{-1} \pt_1 \Phi(x,y)} \leq \frac{1}{2}, 
\end{equation}
where $B_\delta^X$ and $B_\delta^Y$ denote the $\delta$-ball around $0$ in $X$ and $Y$, respectively. We also assume that 
\[ \norm{(\pt_1\Phi(0,0))^{-1}} \leq C_1, \qquad \sup_{(x,y) \in B_\delta^X \times B_\delta^Y}\norm{\pt_2 \Phi(x,y)} \leq C_2\]
for some constants $C_1$, $C_2$, where $\pt_2$ denotes the derivative of $\Phi$ with respect to the second variable. 
Then there is a constant $\eps >0$, depending only on $\delta$, $C_1$ and $C_2$, and a unique function $f\colon B_\eps^Y\to B_\delta^X$ such that 
$\Phi(f(y),y) = 0$ for all $y \in B_\eps^Y$. Moreover, $f$ is continuously Fr\'echet-differentiable and if $\Phi(x,y) = 0$ for some $(x,y) \in B_\delta^X\times B_\eps^Y$ then 
$x=f(y)$. If $\Phi$ is analytic then $f$ will be analytic. 
\end{lemma}  

\begin{proof}
The proof is elementary and left to the reader. 
\end{proof}

For $x, y \in \alg$ and $\omega \in \C$, we define 
\begin{equation} \label{eq:diff_dyson} 
 \Phi_x(y, \omega) \defeq  (\Id - C_xS)[y] -  \omega x^2 - \frac{\omega}{2} \Big( x y + y x\Big)  - \frac{1}{2}\Big( x S[y]y + y S[y]x\Big). 
\end{equation}

We remark that $\Phi_{m(z)}(m(z + \omega) - m(z),\omega) = 0$ for all $z \in \Hb$ and $z + \omega  \in \Hb$ (see \eqref{eq:stability_dyson}).

\begin{proof}[Proof of Lemma \ref{lem:behaviour_outside}]
Lemma \ref{lem:norm_two_to_inf} (i) yields $\normtwoinf{S} \lesssim 1$ due to $S[x] \leq k_1 \avg{x} \id$ for all $x \in \algnon$. 
Therefore, $\norm{a} \lesssim 1$ and $\norm{S} \leq \normtwoinf{S} \lesssim 1$ imply that $\supp v = \supp \dens$ is bounded, i.e., $\sup\{ \abs{\tau} \colon \tau \in \supp \dens \} \lesssim 1$ by~\eqref{eq:supp_v_subset_spec_a}. 

First, we assume that \ref{item:limsup_eta_dens} holds true. 
We set $N \defeq \{ \eta \in (0,1] \colon \eta \norm{\Im m(\tau + \ii\eta)}^{-1} \geq c/2\}$. By assumption, $N$ is nonempty and has $0$ as an accumulation point. 
In particular, we have 
\begin{equation} \label{eq:char_support_aux1}
 \norm{\Im m(z)} \leq \frac{2\eta}{c}, \qquad \eta\id \lesssim \Im m(z) \lesssim \frac{\eta}{c}\id 
\end{equation}
for all $z \in \tau + \ii N$. The first bound is a direct consequence of the definition of $N$. The second bound follows from \eqref{eq:Stieltjes_representation} and the bounded 
support of $v$. Moreover, the first bound immediately implies the third bound. By averaging the two last bounds in \eqref{eq:char_support_aux1} and using $\Im m(\tau + \ii\eta) \lesssim \eta$ for $\eta \in N$, 
we obtain the third and fourth estimates in \eqref{eq:char_outside_m_bound_F_condition}. In particular, $\dens(z) \sim \norm{\Im m(z)}$ for $z \in \tau + \ii N$. 
Owing to \eqref{eq:Stieltjes_representation}, for any $z \in \Hb$ and $x, y \in \Ltwo$, we have
\[ \abs{\scalar{x}{m(z)y}} \leq \frac{1}{2} \int_\R  \frac{\scalar{x}{v(\di \tau)x} + \scalar{y}{v(\di \tau)y}}{\abs{\tau - z}} \lesssim \frac{1}{\eta} \Big( \scalar{x}{\Im m(z) x} + \scalar{y}{\Im m(z) y} \Big)
\leq \frac{2}{c} \Big(\normtwo{x}^2 + \normtwo{y}^2\Big). \] 
Here, we used that $v$ has a bounded support and \eqref{eq:Stieltjes_representation} in the second step and the first bound in \eqref{eq:char_support_aux1} in the last step. 
This proves the first bound in \eqref{eq:char_outside_m_bound_F_condition}. 
The second estimate in \eqref{eq:char_outside_m_bound_F_condition} is a consequence of \eqref{eq:dyson} as well as $\norm{a} \lesssim 1$, $\norm{S} \leq \normtwoinf{S} \lesssim 1$ and 
the first bound in  \eqref{eq:char_outside_m_bound_F_condition}. 
We recall the definitions of $q=q(z)$ and $u=u(z)$ in \eqref{eq:def_q_u}. 
Owing to Lemma \ref{lem:q_bounded_Im_u_sim_avg} (i), the bounds in \eqref{eq:char_outside_m_bound_F_condition} yield 
\begin{equation} \label{eq:char_support_aux2}
 \norm{q} \lesssim 1, \qquad \norm{q^{-1}} \lesssim 1, \qquad \Im u \sim \avg{\Im u} \id \sim \dens \id 
\end{equation}
uniformly for all $z \in \tau + \ii N$. 
Thus, for all $x\in\algnon$ and $z  = \tau + \ii \eta$ and $\eta \in N$, $F = F(z)$ satisfies $F[x] \lesssim \avg{x} \id$ due to $S[x] \lesssim \avg{x} \id$.
Hence, Lemma \ref{lem:positive_symmetric_maps_aux} (i) yields the existence of an eigenvector $f \in \algnon$, i.e., $Ff = \normtwo{F} f$.  
By taking the imaginary part of \eqref{eq:dyson_second_version} and then the scalar product with $f$ as well as using the symmetry of $F$, we get 
\begin{equation} \label{eq:char_support_aux4}
 1 - \normtwo{F} = \eta \frac{\scalar{f}{qq^*}}{\scalar{f}{\Im u}} \sim \eta \norm{\Im m(z)}^{-1}\gtrsim c  
\end{equation}
for $z = \tau + \ii \eta$ and $\eta \in N$ (compare \eqref{eq:normtwo_F}).  Here, we also used $f \in \algnon$, \eqref{eq:char_support_aux2}, $\dens(z) \sim \norm{\Im m(z)}$ and the definition of $N$.  
This completes the proof of \ref{item:limsup_eta_dens} $\Rightarrow$ \ref{item:F_condition}. 

Next, let \ref{item:F_condition} be satisfied. As before, Lemma \ref{lem:q_bounded_Im_u_sim_avg} (i) implies \eqref{eq:char_support_aux2} for all $z \in \tau + \ii N$ 
due to the first four bounds in \eqref{eq:char_outside_m_bound_F_condition}. 
Thus, inspecting the proofs of Lemma \ref{lem:q_bounded_Im_u_sim_avg} (iii) and Proposition \ref{pro:linear_stability} and using $\normtwoinf{S} \lesssim 1$ via Lemma \ref{lem:norm_two_to_inf} (ii) 
yield 
\begin{equation} \label{eq:char_support_aux3}
 \norm{(\Id - C_{m(z)}S)^{-1}} \lesssim 1
\end{equation}
uniformly for all $z \in \tau + \ii N$. 
Thus, we can apply the implicit function theorem, Lemma \ref{lem:implicit_function}, to $\Psi_\eta(\Delta,\omega) \defeq \Phi_{m(\tau + \ii \eta)}(\Delta,\omega)$ ($\Phi$ has been defined in \eqref{eq:diff_dyson})
for each $\eta \in N$ with $\omega \in \C$. 
Since $\Psi_\eta(0,0)=0$ for all $\eta \in N$, there are $\eps>0$ and unique analytic functions $\Delta_\eta \colon D_\eps(0) \to B_\delta^\alg$ by Lemma \ref{lem:implicit_function} 
such that $\Psi_\eta(\Delta_\eta(\omega),\omega)=0$ for all $\omega \in D_\eps(0)$ and all $\eta \in N$. 
We now explain why $\eps$ can be chosen uniformly for all $\eta\in N$. 
By \eqref{eq:char_outside_m_bound_F_condition} and \eqref{eq:char_support_aux3}, there are bounds on $m(z)$ and $(\Id-C_{m(z)} S)^{-1}$ which hold uniformly for $z \in \tau + \ii N$. 
Hence, it is easy to find $\delta >0$ such that \eqref{eq:implicit_function_condition} holds true uniformly for all $\eta \in N$. These uniform bounds yield the uniformity of $\eps$. 
Since 0 is an accumulation point of $N$, there is $\eta_0 \in N$ such that $\eta_0 < \eps$. 
We set $z \defeq \tau + \ii \eta_{0}$. 
An easy computation using \eqref{eq:dyson} at spectral parameters $z$ and $z+\omega$ shows $\Psi_{\eta_{0}}(m(\omega + z) - m(z),\omega) = 0$ for all $\omega \in \C$ such that $\omega + z \in \Hb$. 
Owing to the continuity of $m$, we find $\eps' \in (0,\eps)$ such that $m(\omega+ z)- m(z) \in B_\delta^\alg$ for all $\omega \in D_{\eps'}(0)$. 
Thus, by the uniqueness of $\Delta_{\eta_0}$ (cf.~Lemma \ref{lem:implicit_function}), $\Delta_{\eta_0}(\omega) = m(\omega + z ) - m(z)$ for all $\omega\in D_{\eps'}(0)$. 
As $\Delta_{\eta_0}$ and $m(\cdot + z)$ are analytic, owing to the identity theorem, we obtain $\Delta_{\eta_0}(\omega) + m(z)= m(\omega + z)$ for all $\omega \in D_{\eps}(0)$ satisfying $\omega +z \in \Hb$. 
Using $\eta_0 < \eps$, we set $m\defeq \Delta_{\eta_0}(-\ii\eta_0) + m(z)$. 
For this choice of $m$, the continuity of $\Delta_{\eta_0}(\omega)$ for $\omega \to -\ii\eta_0$ and $\Delta_{\eta_0}(\omega) + m(z) = m(\omega + z)$ yield \eqref{eq:char_outside_m_limit_eta}. 
It remains to show that $m$ is self-adjoint. Since \eqref{eq:char_support_aux2} holds true under \ref{item:F_condition} as we have shown above, we obtain 
\[ \eta \norm{\Im m(z)}^{-1} \sim 1 - \normtwo{F} \geq C^{-1} \]
for $z = \tau + \ii \eta$ and $\eta \in N$ as in \eqref{eq:char_support_aux4}. Thus, $\liminf_{\eta\downarrow 0} 
\norm{\Im m(\tau + \ii \eta)} \leq 0$. Hence, we obtain $\Im m = 0$, i.e., $m = m^*$. 
This completes the proof of \ref{item:F_condition} $\Rightarrow$ \ref{item:Id_minus_C_m_S_invertible}. 

If \ref{item:Id_minus_C_m_S_invertible} holds true then $\Id- C_mS$ has a bounded linear inverse on $\alg$ for $m$. 
Hence, we can apply the implicit function theorem, Lemma \ref{lem:implicit_function}, to $\Phi_m(\Delta,\omega) = 0$ (see \eqref{eq:diff_dyson} for the definition of $\Phi$) as $\Phi_m(0,0)=0$ 
and $\pt_1\Phi_m(0,0) = \Id - C_mS$. It is easy to see that there is $\delta>0$ such that \eqref{eq:implicit_function_condition} is satisfied. 
Therefore, there are $\eps >0$ and an analytic function $\Delta \colon D_\eps(0) \to B_\delta^\alg$ such that $\Phi_m(\Delta(\omega),\omega) = 0$ for all $\omega \in D_\eps(0)$. 
In particular, $f \colon D_\eps(\tau) \to\alg$, $f(w) \defeq \Delta(w-\tau) + m$ is analytic. 
From \eqref{eq:char_outside_m_limit_eta} and \eqref{eq:dyson}, we see that $m$ is invertible and satisfies \eqref{eq:dyson} at $z=\tau$. 
Thus, a straightforward computation using \eqref{eq:dyson} at $z = \tau$ and at $z = \tau + \ii\eta$ yields $\Phi_m(m(\tau + \ii\eta)-m,
\ii\eta) = 0$ for all $\eta \in (0,\eps]$. Therefore, $m(\tau + \ii\eta) = \Delta(\ii\eta) + m = f(\tau + \ii\eta)$ for all $\eta \in (0,\eta_*]$ and some $\eta_* \in (0,\eps]$
due to the uniqueness part of Lemma \ref{lem:implicit_function} and~\eqref{eq:char_outside_m_limit_eta}. 
Since $m$ and $f$ are analytic on $D_\eps(\tau) \cap \Hb$, the identity theorem implies $m(z) = f(z)$ for all $z \in D_\eps(\tau) \cap \Hb$. 
A simple computation shows $\Phi_m(\Delta(\bar\omega)^*,\omega) = \Phi_m(\Delta(\bar \omega),\bar\omega)^* = 0$ for all $\omega \in D_\eps(0)$ as $m = m^*$. Hence, $\Delta(\omega) = \Delta(\bar\omega)^*$ 
for all $\omega \in D_\eps(0)$ by the uniqueness part of Lemma \ref{lem:implicit_function}. Thus, $f(w) = f(\bar w)^*$ for all $w \in D_\eps(\tau)$ and $f(w) = f(w)^*$ for all $w \in D_\eps(\tau) \cap \R$. 
This proves \ref{item:Id_minus_C_m_S_invertible} $\Rightarrow$ \ref{item:s_a_solution_on_real_axis}. Clearly, \ref{item:s_a_solution_on_real_axis} implies \ref{item:dist_supp} 
by \eqref{eq:Stieltjes_representation}.

If the statement in \ref{item:dist_supp} holds true then $\dist(\tau, \supp \dens) \geq \eps$. In particular, by \eqref{eq:stieltjes_bound_m}, we have 
\[ \liminf_{\eta \downarrow 0}\eta\norm{\Im m(\tau + \ii\eta)}^{-1}\geq \liminf_{\eta \downarrow 0}\dist(\tau + \ii\eta, \supp \dens)^2 \geq \eps^2 \]
for all $\eta>0$. Here, we used \eqref{eq:stieltjes_bound_m} in the first step. This immediately implies \ref{item:liminf_eta_dens} with $c = \eps^2$. 
Moreover, \ref{item:limsup_eta_dens} is immediate from \ref{item:liminf_eta_dens}. 

Inspecting the proofs of the implications above shows the additional statement about the effective dependence of the constants in \ref{item:limsup_eta_dens} -- \ref{item:liminf_eta_dens}. 
In particular, the application of the implicit function theorem, Lemma \ref{lem:implicit_function}, in the proof of \ref{item:s_a_solution_on_real_axis} shows that $\eps$ can 
be chosen to depend only on $k_1$ and $C$ from \ref{item:Id_minus_C_m_S_invertible}. 
This completes the proof of Lemma \ref{lem:behaviour_outside}. 
\end{proof}

\bibliographystyle{amsplain}

\begin{thebibliography}{10}

\bibitem{AjankiQVE}
O.~Ajanki, L.~Erd{\H o}s, and T.~Kr{\" u}ger, \emph{Quadratic vector equations
  on complex upper half-plane}, to appear in Mem. Amer. Math. Soc.,
  arXiv:1506.05095v4, 2015.

\bibitem{AjankiCPAM}
\bysame, \emph{Singularities of solutions to quadratic vector equations on the
  complex upper half-plane}, Comm. Pure Appl. Math. \textbf{70} (2017), no.~9,
  1672--1705.

\bibitem{Ajankirandommatrix}
\bysame, \emph{Universality for general {W}igner-type matrices}, Prob. Theor.
  Rel. Fields \textbf{169} (2017), no.~3-4, 667--727.

\bibitem{AjankiCorrelated}
\bysame, \emph{Stability of the matrix {D}yson equation and random matrices
  with correlations}, Prob. Theor. Rel. Fields (2018),
  doi:10.1007/s00440--018--0835--z (Online first).

\bibitem{Aleksandrov2016}
A.~B. Aleksandrov and V.~V. Peller, \emph{Operator {L}ipschitz functions},
  Russian Math. Surveys \textbf{71} (2016), no.~4, 605.

\bibitem{AltSingGram}
J.~Alt, \emph{Singularities of the density of states of random {G}ram
  matrices}, Electron. Commun. Probab. \textbf{22} (2017), 13 pp.

\bibitem{AltKronecker}
J.~Alt, L.~Erd\H{o}s, T.~Kr{\" u}ger, and Yu. Nemish, \emph{Location of the
  spectrum of {K}ronecker random matrices}, to appear in Ann. Inst. H.
  Poincaré Probab. Statist., arXiv:1706.08343, 2017.

\bibitem{AltEdge}
J.~Alt, L.~Erd\H{o}s, T.~Kr{\" u}ger, and D.~Schr{\" o}der, \emph{Correlated
  {R}andom {M}atrices: {B}and {R}igidity and {E}dge {U}niversality},
  arXiv:1804.07744, 2018.

\bibitem{AltGram}
J.~Alt, L.~Erd\H{o}s, and T.~Krüger, \emph{Local law for random {G}ram
  matrices}, Electron. J. Probab. \textbf{22} (2017), no.~25, 41 pp.

\bibitem{Anderson2006}
G.~W. Anderson and O.~Zeitouni, \emph{A {CLT} for a band matrix model}, Probab.
  Theory Related Fields \textbf{134} (2006), no.~2, 283--338.

\bibitem{Bai1999}
Z.~D. Bai and J.~W. Silverstein, \emph{Exact separation of eigenvalues of
  large-dimensional sample covariance matrices}, Ann. Probab. \textbf{27}
  (1999), no.~3, 1536--1555. \MR{1733159}

\bibitem{Bekerman2017}
F.~Bekerman, T.~Leblé, and S.~Serfaty, \emph{{CLT} for fluctuations of
  $\beta$-ensembles with general potential}, arXiv:1706.09663, 2017.

\bibitem{Berezin1973}
F.~A. Berezin, \emph{{Some remarks on the {W}igner distribution}}, Theoret.
  Math. Phys. \textbf{17} (1973), 1163--1171.

\bibitem{Bhatia_matrix_analysis}
R.~Bhatia, \emph{Matrix analysis}, Graduate Texts in Mathematics, vol. 169,
  Springer-Verlag, New York, 1997. \MR{1477662}

\bibitem{Erdos2018Cusp2}
G.~Cipolloni, L.~Erd\H{o}s, T.~Kr{\" u}ger, and D.~Schr{\" o}der, \emph{{C}usp
  {U}niversality for {R}andom {M}atrices {II}: {T}he {R}eal {S}ymmetric
  {C}ase}, arXiv:1811.04055, 2018.

\bibitem{Claeys2010}
T.~Claeys, I.~Krasovsky, and A.~Its, \emph{Higher-order analogues of the
  {T}racy-{W}idom distribution and the {P}ainlevé {II} hierarchy}, Comm. Pure
  Appl. Math. \textbf{63}, no.~3, 362--412.

\bibitem{Claeys2018}
T.~Claeys, A.~B.~J. Kuijlaars, K.~Liechty, and D.~Wang, \emph{Propagation of
  singular behavior for gaussian perturbations of random matrices},
  Communications in Mathematical Physics \textbf{362} (2018), no.~1, 1--54.

\bibitem{DEIFT1998388}
P.~Deift, T.~Kriecherbauer, and K.~T.-R McLaughlin, \emph{New results on the
  equilibrium measure for logarithmic potentials in the presence of an external
  field}, J. Approx. Theory \textbf{95} (1998), no.~3, 388 -- 475.

\bibitem{Erdos2017Correlated}
L.~Erd\H{o}s, T.~Kr{\" u}ger, and D.~Schr{\" o}der, \emph{Random matrices with
  slow correlation decay}, arXiv:1705.10661, 2017.

\bibitem{Erdos2018Cusp}
\bysame, \emph{{C}usp {U}niversality for {R}andom {M}atrices {I}: {L}ocal {L}aw
  and the {C}omplex {H}ermitian {C}ase}, arXiv:1809.03971, 2018.

\bibitem{Erdos2017book}
L.~Erd{\H o}s and H.-T. Yau, \emph{A dynamical approach to random matrix
  theory}, Courant Lecture Notes in Mathematics, vol.~28, Courant Institute of
  Mathematical Sciences, New York; American Mathematical Society, Providence,
  RI, 2017.

\bibitem{Garg2016ADP}
A.~Garg, L.~Gurvits, R.~Mendes~de Oliveira, and A.~Wigderson, \emph{A
  deterministic polynomial time algorithm for non-commutative rational identity
  testing}, 2016 IEEE 57th Annual Symposium on Foundations of Computer Science
  (FOCS) (2016), 109--117.

\bibitem{girko2012theory}
V.~L. Girko, \emph{Theory of stochastic canonical equations: Volumes {I} and
  {II}}, Mathematics and Its Applications, Springer Netherlands, 2012.

\bibitem{GUIONNET2002341}
A.~Guionnet, \emph{Large deviations upper bounds and central limit theorems for
  non-commutative functionals of {G}aussian large random matrices}, Ann. Inst.
  H. Poincaré Probab. Statist. \textbf{38} (2002), no.~3, 341 -- 384.

\bibitem{Haagerup2006}
U.~Haagerup, H.~Schultz, and S.~Thorbj{\o}rnsen, \emph{A random matrix approach
  to the lack of projections in {$C^*_{\mathrm{red}}(\mathbb{F}_2)$}}, Adv.
  Math. \textbf{204} (2006), no.~1, 1--83. \MR{2233126}

\bibitem{Haagerup2005}
U.~Haagerup and S.~Thorbj{\o}rnsen, \emph{A new application of random matrices:
  {$\mathrm{Ext}(C^*_{\mathrm{red}} (\mathbb{F}_2))$} is not a group}, Ann. of
  Math. (2) \textbf{162} (2005), no.~2, 711--775. \MR{2183281}

\bibitem{He2017}
Y.~He, A.~Knowles, and R.~Rosenthal, \emph{Isotropic self-consistent equations
  for mean-field random matrices}, Probab. Theory Related Fields (2017).

\bibitem{Helton2018}
J.~W. Helton, T.~Mai, and R.~Speicher, \emph{Applications of realizations (aka
  linearizations) to free probability}, J. Funct. Anal. \textbf{274} (2018),
  no.~1, 1--79. \MR{3718048}

\bibitem{Helton01012007}
J.~W. Helton, R.~Rashidi~Far, and R.~Speicher, \emph{Operator-valued
  semicircular elements: Solving a quadratic matrix equation with positivity
  constraints}, Int. Math. Res. Not. IMRN (2007), no.~22, Art. ID rnm086.

\bibitem{KhorunzhyPastur1994}
A.~M. Khorunzhy and L.~A. Pastur, \emph{{On the eigenvalue distribution of the
  deformed Wigner ensemble of random matrices}}, Spectral operator theory and
  related topics, Adv. Soviet Math., 19, Amer. Math. Soc., Providence, RI,
  1994, pp.~97--127.

\bibitem{Knowles2017}
A.~Knowles and J.~Yin, \emph{Anisotropic local laws for random matrices},
  Probab. Theory Related Fields \textbf{169} (2017), no.~1-2, 257--352.
  \MR{3704770}

\bibitem{MAI20171080}
T.~Mai, R.~Speicher, and M.~Weber, \emph{Absence of algebraic relations and of
  zero divisors under the assumption of full non-microstates free entropy
  dimension}, Adv. Math. \textbf{304} (2017), 1080 -- 1107.

\bibitem{Mai2018}
T.~Mai, R.~Speicher, and S.~Yin, \emph{The free field: zero divisors, {A}tiyah
  property and realizations via unbounded operators}, arXiv:1805.04150, 2018.

\bibitem{mingo2017free}
J.A. Mingo and R.~Speicher, \emph{Free probability and random matrices}, Fields
  Institute Monographs, Springer New York, 2017.

\bibitem{PasturShcherbina2011}
L.~Pastur and M.~Shcherbina, \emph{Eigenvalue distribution of large random
  matrices}, Mathematical Surveys and Monographs, vol. 171, American
  Mathematical Society, Providence, RI, 2011.

\bibitem{Paulsen2002}
V.~I. Paulsen, \emph{Completely bounded maps and operator algebras.}, Cambridge
  Studies in Advanced Mathematics, no. Vol. 78, Cambridge University Press,
  2002.

\bibitem{Shlyakhtenko1996}
D.~Shlyakhtenko, \emph{Random {G}aussian band matrices and freeness with
  amalgamation}, Int. Math. Res. Not. IMRN \textbf{1996} (1996), no.~20,
  1013--1025.

\bibitem{freebandrigidity}
D.~Shlyakhtenko and P.~Skoufranis, \emph{Freely independent random variables
  with non-atomic distributions}, Trans. Amer. Math. Soc. \textbf{367} (2015),
  no.~9, 6267--6291.

\bibitem{Speicher1998}
R.~Speicher, \emph{Combinatorial theory of the free product with amalgamation
  and operator-valued free probability theory}, Mem. Amer. Math. Soc.
  \textbf{132} (1998), no.~627.

\bibitem{TakesakiBook}
M.~Takesaki, \emph{Theory of operator algebras {I}}, Encyclopaedia of
  mathematical sciences, no. 124, Springer, 1979.

\bibitem{Voiculescu1994}
D.~Voiculescu, \emph{The analogues of entropy and of {F}isher's information
  measure in free probability theory. {II}}, Invent. Math. \textbf{118} (1994),
  no.~3, 411--440. \MR{1296352}

\bibitem{Voiculescu1995}
\bysame, \emph{Operations on certain non-commutative operator-valued random
  variables}, Ast\'erisque (1995), no.~232, 243--275, Recent advances in
  operator algebras (Orl\'eans, 1992). \MR{1372537}

\bibitem{Voiculescu1998}
\bysame, \emph{The analogues of entropy and of {F}isher's information measure
  in free probability theory. {V}. {N}oncommutative {H}ilbert transforms},
  Invent. Math. \textbf{132} (1998), no.~1, 189--227. \MR{1618636}

\bibitem{Voiculescu2000}
\bysame, \emph{The coalgebra of the free difference quotient and free
  probability}, Int. Math. Res. Not. IMRN \textbf{2000} (2000), no.~2, 79--106.

\bibitem{doi:10.1112/S0024609301008992}
\bysame, \emph{Free entropy}, Bull. Lond. Math. Soc. \textbf{34} (2002), no.~3,
  257--278.

\bibitem{Wegner1979}
F.~J. Wegner, \emph{{Disordered system with $ n $ orbitals per site: $ n
  =\infty $ limit}}, Physical Review B \textbf{19} (1979).

\end{thebibliography}
\providecommand{\bysame}{\leavevmode\hbox to3em{\hrulefill}\thinspace}
\providecommand{\MR}{\relax\ifhmode\unskip\space\fi MR }
\providecommand{\MRhref}[2]{%
  \href{http://www.ams.org/mathscinet-getitem?mr=#1}{#2}
}
\providecommand{\href}[2]{#2}

\end{document}